\documentclass[12pt,a4paper,twoside,openright]{book}

\usepackage[inner=32mm,outer=24mm,top=32mm,bottom=32mm]{geometry}

\usepackage{hyperref}
\hypersetup{colorlinks,linkcolor={blue},citecolor={blue},urlcolor={blue}}
\usepackage[T1]{fontenc}
\usepackage[utf8]{inputenc}
\usepackage{amsmath}
\usepackage{amssymb}
\usepackage{pgfplots}
\usepackage{xcolor}
\usepackage{pgfplots}
\usepackage{multirow}
\usepgfplotslibrary{groupplots}
\usetikzlibrary{pgfplots.groupplots}
\usetikzlibrary{fadings}
\usepackage{booktabs}
\usepackage{enumitem}
\usepackage{algorithm}
\usepackage{algorithmicx}
\usepackage{algpseudocode}
\usepackage{graphics}
\usepackage{enumitem}
\usepackage{geometry}
\usepackage{subcaption}
\usepackage{inputenc}
\usepackage{verbatim}

\newlength\figH
\newlength\figW
\newlength\textSize

\usepackage{minitoc}

\usepackage{afterpage}

\usepackage{stmaryrd}
\usepackage{mathrsfs}

\usepackage{amsthm}

\newtheorem{theorem}{Theorem}[chapter]
\newtheorem{definition}{Definition}[chapter]
\newtheorem{lemma}{Lemma}[chapter]
\newtheorem{remark}{Remark}[chapter]
\newtheorem{proposition}{Proposition}[chapter]
\newtheorem{corollary}{Corollary}[chapter]

\newtheorem{ex}{Example}[chapter]

\def\cop#1{\underline{#1}}

\def\ind{{\mathchoice {\rm 1\mskip-4mu l} {\rm 1\mskip-4mu l}
		{\rm 1\mskip-4.5mu l} {\rm 1\mskip-5mu l}}}

\newcommand\suite[1]{\left\{#1;\,n\in\N\right\}}

\newcommand\suiten[1]{\left\{#1;\,n\in\N\right\}}
\def\mumin{\mu_{\mbox{\scriptsize{min}}}}
\def\mumax{\mu_{\mbox{\scriptsize{max}}}}
\def\mumini{\mu_{\mbox{\emph{\scriptsize{min}}}}}
\def\mumaxi{\mu_{\mbox{\emph{\scriptsize{max}}}}}

\newcommand\pr[1]{{\mathbb P}\left[#1\right]}
\def\P{{\mathbb P}}  
\newcommand\td[1]{\overline{#1}}
\def\bw{\mathbf w}

\newcommand\gre{\textbf{e}}
\newcommand\grx{\textbf{x}}

\def\esp#1{{\mathbb E}\left[#1\right]}
\def\cop#1{\underline{#1}}

\newcommand{\pae}[1]{\mbox{$\lfloor \kern-1pt #1 \kern-1pt \rfloor$}}
\newcommand{\paep}[1]{\mbox{$\lceil \kern-1pt #1 \kern-1pt \rceil$}}

\newcommand\ccc{\circledcirc}

\def\v{{\--}}
\def\pv{{\not\!\!\--}}
\def\N{{\mathbb N}}
\def\R{{\mathbb R}}
\def\I{{\mathbb I}}
\def\mbS{{\mathbb S}}
\def\mbH{{\mathbb H}}
\def\mbG{{\mathbb G}}
\def\mbX{{\mathbb X}}
\def\mbW{{\mathbb W}}
\def\Z{{\mathbb Z}}
\def\D{{\mathbb D}}
\def\C{{\mathbb C}}

\def\maP{{\mathcal P}}
\def\maE{{\mathcal E}}
\def\maS{{\mathcal S}}
\def\maV{{\mathcal V}}
\def\maH{{\mathcal H}}
\def\maI{{\mathcal I}}
\def\maU{{\mathcal U}}
\def\maF{{\mathcal F}}
\def\maB{{\mathcal B}}
\def\maT{{\mathcal T}}
\def\maJ{{\mathcal J}}
\def\maM{{\mathcal M}}
\def\maN{{\mathcal N}}
\def\maZ{{\mathcal Z}}
\def\maA{{\mathcal A}}

\def\ll{[\![}
\def\rr{]\!]}

\def\V{\mathcal V}
\def\A{\mathcal A}
\def\B{\mathcal B}
\def\bV{\mathbf V}
\def\oV{\overline{\V}}
\def\w{\mathbf w}
\def\E{\mathcal E}

\newcommand{\boundellipse}[3]
{(#1) ellipse (#2 and #3)
}
\def\fcfm{\textsc{fcfm}}

\usepackage[utf8]{inputenc}

\usepackage{fancyhdr}
\begin{document}
\pagestyle{empty}
\begin{titlepage}
	\begin{figure}[htp]
		\vspace{-0.5cm}
		\begin{subfigure}{.5\textwidth}
			\centering
			\includegraphics[width=.25\linewidth]{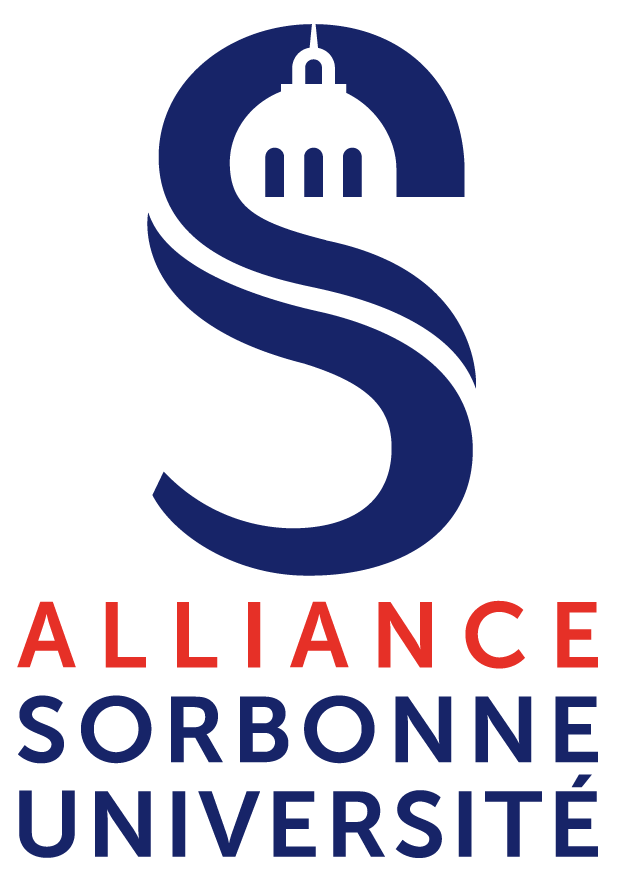}\qquad\qquad\qquad\qquad\qquad\qquad\qquad
		\end{subfigure}
		\begin{subfigure}{.5\textwidth}
			\centering
			\qquad\qquad\includegraphics[width=.7\linewidth]{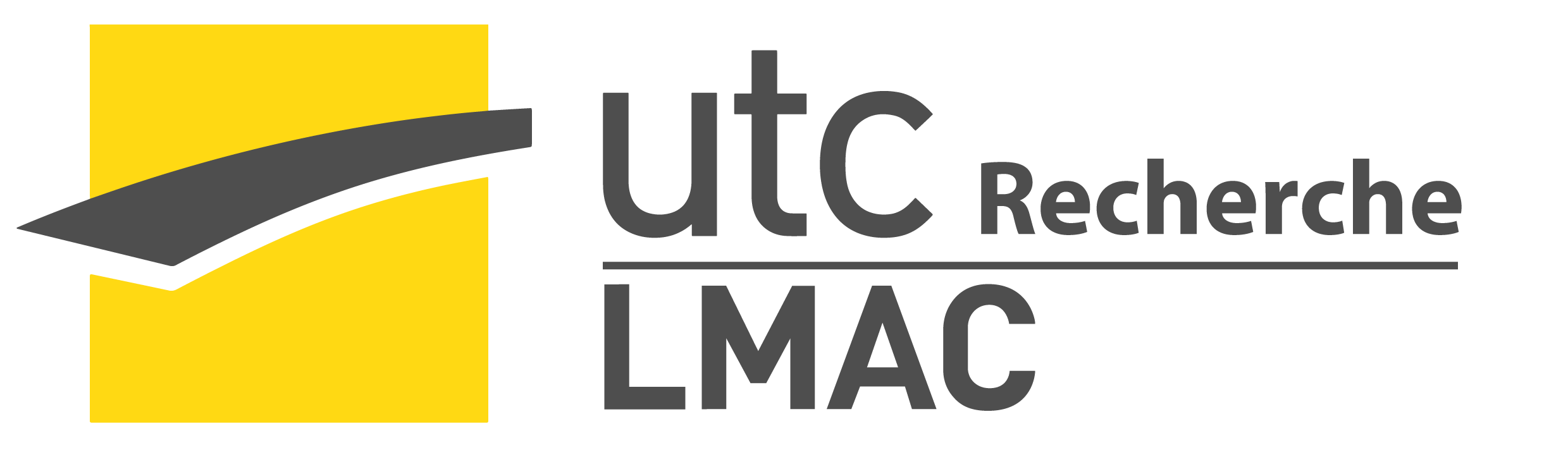}
		\end{subfigure}
	\end{figure}
	\begin{center}
		\vspace{0.7cm}
		
		\textbf{\huge Stochastic matching model on the general graphical structures}
		
		\vspace{1.1cm}
		{\large \textit{A thesis presented for the degree of Doctor of Philosophy}} \\
		\vspace{1.1cm}
		
		\textbf{\Large Youssef RAHME}
		
		\vspace{0.6cm}
		\begin{center}
			{\large April 8, 2021}
		\end{center}
		\vspace{1.1cm}
		\textcolor{black}{\Large Compi\`{e}gne University of Technologie}\\
		\vspace{0.6cm}
		{\large Department of Computer Engineering}\\
		\vspace{0.6cm}
		\textcolor{black}{\large Laboratory of Applied Mathematics of Compi\`{e}gne} \\
		
		\vspace{1.5cm}
		{\large\textbf{Jury Members}}
	\end{center}
	\vspace{0.3cm}
	\begin{center}
		\begin{tabular}{llll}
			\textbf{Supervisor:} &Pr. &MOYAL Pascal & University of Lorraine\\ 
			\\
			\textbf{Reviewers:}&Pr.& GAUJAL Bruno& INRIA Grenoble\\
			&Dr.& DEACONU Madalina& INRIA Nancy\\
			\\
			\textbf{Examiners:}& Pr.& GAYRAUD Ghislaine& Compi\`{e}gne University of Technologie\\
			&Pr.& KHRAIBANI Zaher& Lebanese University {\qquad$\;$ (co-advisor)}\\
			&Dr. &ROBIN Vincent& Compi\`{e}gne University of Technologie \\
			&Dr. &BUKE Burak&  University of Edinburgh\\
			&Dr. &BU$\check{\mbox{S}}$I\'C Ana& INRIA Paris
		\end{tabular}
	\end{center}
\end{titlepage}
\clearpage
\newpage	
\begin{titlepage}
	\begin{figure}[htp]
		\vspace{-0.5cm}
		\begin{subfigure}{.5\textwidth}
			\centering
			\includegraphics[width=.25\linewidth]{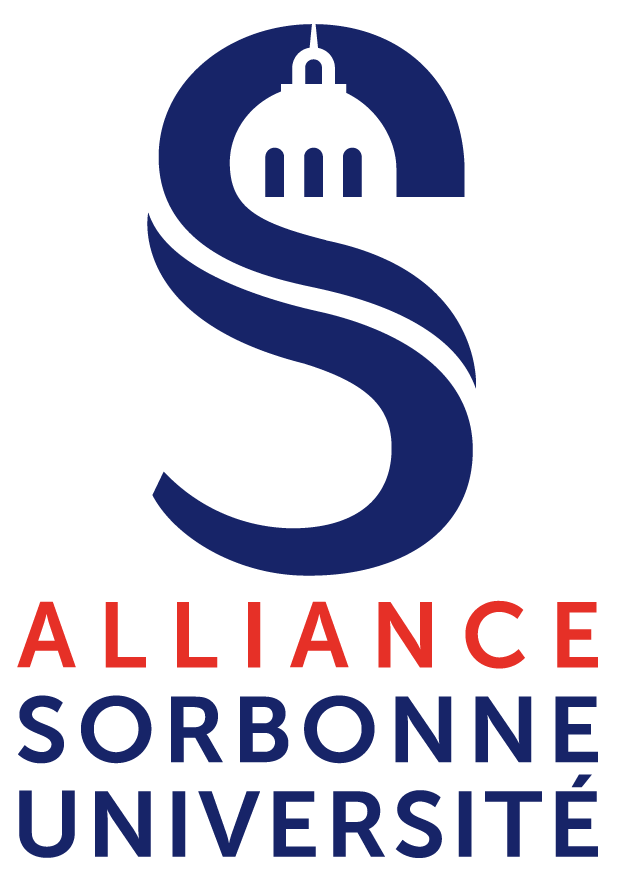}\qquad\qquad\qquad\qquad\qquad\qquad\qquad
		\end{subfigure}
		\begin{subfigure}{.5\textwidth}
			\centering
			\qquad\qquad\includegraphics[width=.7\linewidth]{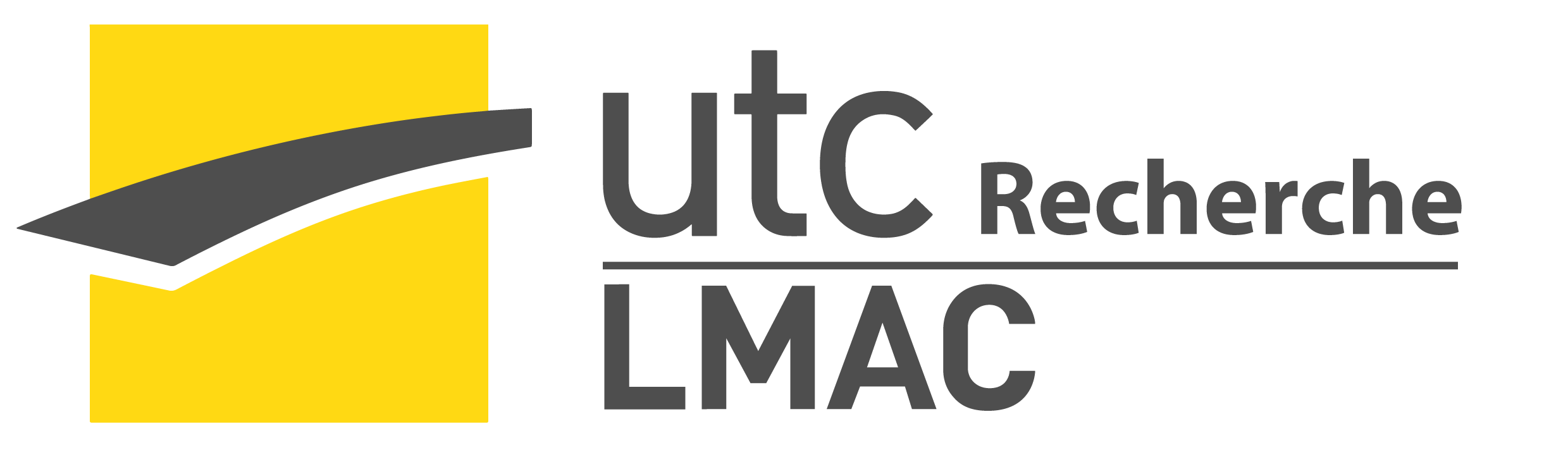}
		\end{subfigure}
	\end{figure}
	\begin{center}
		\vspace{0.7cm}
		
		{\large\textbf{\huge Modèle d'appariement aléatoire sur des structures graphiques générales}}
		
		\vspace{1.1cm}
		{\large{\it Th\`{e}se pr\'{e}sent\'{e}e  pour l'obtention du grade de Docteur}} \\
		
		\vspace{1.1cm}
		
		\textbf{\Large Youssef RAHME}

		\vspace{0.6cm}	
		{\large Avril 8, 2021}\\
		\vspace{0.6cm}
		{\Large Universit\'{e} de Technologie de Compi\`{e}gne}\\
		\vspace{0.6cm}
		{\large Département de Génie Informatique}\\
		\vspace{0.6cm}
		{\large Laboratoire de Math\'{e}matiques Appliquées de Compi\`{e}gne} \\

		\vspace{1.5cm}
		{\large \textbf{Membres du Jury}}
	\end{center} 
	\vspace{0.3cm}
	
	\begin{center}
		\begin{tabular}{llll}
			\textbf{Directeur:} & Pr. & MOYAL Pascal & Université de Lorraine\\ 
			\\
			\textbf{Rapporteur:} & Pr. & GAUJAL Bruno & INRIA Grenoble \\  
			& Dr.& DEACONU Madalina& INRIA Nancy  \\
			\\
			\textbf{Examinateur:} &Pr.& GAYRAUD Ghislaine& Université de Technologie de Compiègne\\
			&Pr. &KHRAIBANI Zaher&Université Libanaise {\qquad \quad (co-encadrant)}\\
			&Dr.& ROBIN Vincent& Université de Technologie de Compiègne \\
			&Dr.& BUKE Burak&  Université de Edinburgh\\
			&Dr.& BU$\check{\mbox{S}}$I\'C Ana&  INRIA Paris
		\end{tabular}
	\end{center}
\end{titlepage}
	\clearpage

\dominitoc
\clearpage
\pagestyle{fancy}
\pagenumbering{roman}
\pagestyle{empty}
\pagenumbering{roman}
\thispagestyle{empty}
\newpage
\addcontentsline{toc}{chapter}{Remerciements}

\chapter*{Acknowledgment}


My thesis is a result of a challenging journey to which many people have contributed and given their support. It is a pleasant task to express my thanks to all those people who made this thesis possible and unforgettable yet a life-changing experience for me. 

First of all, I would like to express my gratitude to the French government for its generosity also for granting me the honor to study at the Compi\`{e}gne University of Technologie (UTC) which is a founding member of Sorbonne University association, and one of the best Universities in France. Moreover, I would also like to thank the Lebanese government and the general directorate of the Lebanese internal security forces to support me  seeking my goal and completing my Ph.D.

I would like to sincerely thank the Laboratory of Applied Mathematics of Compiègne for allowing me to complete my thesis under the best conditions, also, I am grateful for the funding received through `national research agency' ANR to support me.

My deep appreciation goes out to my director, Professor Pascal MOYAL who honored me by agreeing to guide and advise me throughout this work, for he has been a constant source of inspiration and the guide light for all my Endeavours. His expertise was invaluable, I appreciate his kindness, availability to listen and all the confidence he has instilled in me, thank you for everything you have done to help me reach where I am now.

I equally address a greater thanks to Professor Zaher KHRAIBANI who accompanied and supported me throughout this thesis, I am deeply indebted and sincerely grateful to you for your never-failing guidance and encouragement.

I have always been blessed with magnificent friends in my life who kept me going by providing a stimulating and fun-filled environment. Words are short to express my gratitude towards my following friends, Hani EL HAJJAR for hosting me at his house and making me feel so welcomed, my colleagues at university Josephine MERHI BLEIK, Joanna AKROUCHE, and Samer TAOUM who helped me all in numerous ways during my stay in France, my colleagues at work in Lebanon who always provided me with moral support whenever the need arises required, either directly or indirectly. 

The biggest support, motivation, and encouragement a man can receive is form his parents. My father has always given me the courage to dream and try to fulfill challenging dreams and my mother has always taught me the importance of being patient and humble in all stages of life. The expression ``thank you''  does not seem sufficient but it is said with appreciation and respect to my parents, Mr. Elias and Mrs. Mona for their unwavering support, care, love, motivation, understanding, and prayers. I would not have reached this stage in my life without your blessings and assistance. I would also like to extend a huge loving thanks to my brother Gabriel and my sister Galia who always had faith in my abilities and were always by my side.

The best outcome of these past years was finding my soul-mate, life partner, and lovely wife. I married the best girl out there for me! I feel fortunate to have her by my side and there are no words to convey how much I love her. I truly thank my wife Mrs. Saide BAHJA  for giving me strength and patience in my hard times. I could not have completed this without having you and our little girl Tatiana by my side whom I will seek to make her proud.

Thank you all for making it possible for me to complete what I started, for making this dream come true, it has been three years already and I have learned many things and created great memories that will stay with me forever.
\addcontentsline{toc}{chapter}{Publications}
 
\chapter*{Publications}
	\vspace{1.2cm}
\begin{enumerate}
	
\item \textbf{\large Submitted journal articles}
\begin{itemize}
\item {Y. Rahme and P. Moyal (2019).} ``A stochastic matching models on hypergraphs''. To appear in \textit{Advances in Applied Probability} 53.4 (December 2021). ArXiv math.PR/1907.12711, 2019.
	\vspace{0.5cm}
\item {J. Begeot, I. Marcovici, P. Moyal, Y. Rahme (2020).} ``A general stochastic matching model on multigraphs''. To appear in \textit{ALEA}. ArXiv  math.PR/2011.05169, 2020.
\end{itemize}
\vspace{0.7cm}
\item \textbf{\large Preparation of a journal article}
\begin{itemize}
\item {Y. Rahme and P. Moyal (2021).} ``Comparison of models for organ transplant applications''. In the process of submitting.
\end{itemize}
\end{enumerate}
\addcontentsline{toc}{chapter}{Abstract}
\chapter*{Abstract}

Motivated by a wide range of assemble-to-order systems and systems of the collaborative economy applications, we introduce a stochastic matching model on hypergraphs and multigraphs, extending the model introduced by Mairesse and Moyal 2016.\\

In this thesis, the stochastic matching model $(\mbS,\Phi,\mu)$ on general graph structures are defined as follows: given a compatibility general graph structure $\mbS=(\maV,\maS)$ which of a set of nodes denoted by $\maV$ that represent the classes of items and by a set of edges denoted by $\maS$ that allows matching between different classes of items. Items arrive at the system at a random time, by a sequence (assumed to be $i.i.d.$) that consists of different classes of $\maV,$ and request to be matched due to their compatibility according to $\mbS.$  The compatibility by groups of two or more (hypergraphical cases) and by groups of two with possibilities of matching between the items of the same classes (multigraphical cases). The unmatched items are stored in the system and wait for a future compatible item and as soon as they are matched they leave it together. Upon arrival, an item may find several possible matches, the items that leave the system depend on a matching policy $\Phi$ to be specified.\\

 We study the stability of the stochastic matching model on hypergraphs, for different hypergraphical topologies. Then, the stability of the stochastic matching model on multigraphs using the maximal subgraph and minimal blow-up to distinguish the zone of stability.
 \clearpage
\pagestyle{empty}
\addcontentsline{toc}{chapter}{R\'{e}sum\'{e}}
\chapter*{R\'{e}sum\'{e}}


Motivé par des applications à large éventail des systèmes d'assemblage à la commande et des systèmes de l'économie collaborative, nous introduisons un modèle d'appariement aléatoire sur les hypergraphes et sur les multigraphes, étendant le modèle par Mairesse et Moyal 2016.\\

\begin{sloppypar}
Dans cette thèse, le modèle d'appariement aléatoire $(\mbS,\Phi,\mu)$ sur les structures graphiques générales est défini comme suit: étant donné une structure graphique générale de compatibilité $\mbS=(\maV,\maS)$ qui est constituée d'un ensemble de nœuds noté par $\maV$ qui représentent les classes d'éléments et par un ensemble d'arêtes noté par $\maS$ qui permettent d'apparier entre les différentes classes. Les éléments arrivent au système à un moment aléatoire, par une séquence (supposée être $i.i.d.$) constituée de différentes classes de $\maV,$ et demandent d'être appariés selon leur compatibilité dans $\mbS. $ La compatibilité par groupe de deux ou plus (cas hypergraphique) et par groupe de deux avec les possibilités  d'apparier entre les éléments de même classe (cas multigraphique). Les éléments, qui ne sont pas appariés, sont stockés dans le système et en attente d'un futur élément compatible et dès qu'ils sont appariés, ils quittent le système ensemble. À l'arrivée, un élément peut trouver plusieurs d'appariements possibles, les éléments qui quittent le système dépendent d'une politique d'appariement $\Phi$ à spécifier. \\
\end{sloppypar}

\begin{sloppypar}
Nous étudions la stabilité du modèle d'appariement aléatoire sur l'hypergraphe, pour des différentes topologies hypergraphiques puis, la stabilité du modèle d'appariement aléatoire sur les multigraphes en utilisant son sous-graphe maximal et sur-graphe minimal étendu pour distinguer la zone de stabilité.
\end{sloppypar}
\pagestyle{empty}
\listoffigures
\newlist{abbrv}{itemize}{1}
\setlist[abbrv,1]{label=,labelwidth=1in,align=parleft,itemsep=0.1\baselineskip,leftmargin=!}
\addcontentsline{toc}{chapter}{List of Acronyms}
\pagestyle{empty}
\chapter*{List of Acronyms}
\pagestyle{fancy}
\chaptermark{List of Acronyms}
\begin{abbrv}
\item[\textsc{alis}] Assign the Longest Idling Server.
\item[BM]Bipartite Matching graph.
\item[\textsc{cftp}] Coupling From The Past.
\item[CTMC] Continuous-Time Markov Chain.
\item[DTMC] Discrete-Time Markov Chain.
\item[EBM] Extended Bipartite Matching graph.
\item[\textsc{fcfm}] First Come, First Matched.
\item[\textsc{fcfs}] First Come, First Served.
\item[FWLLN] Fluid  Weak Law of Large Numbers. 
\item[GM] General Matching graph.
\item[$i.i.d.$ (IID)] Independent and Identically Distributed.
\item[\textsc{lcfm}] Last Come, First Matched.
\item[\textsc{ml}] Match Longest.
\item[\textsc{ms}] Match Shortest.
\item[\textsc{mw}] Max-Weight.
\item[M/M/1]Discipline represents the queue length in a system having a single server, where arrivals are determined by a Poisson process and job service times have an exponential distribution. 
\item[r.v.] Random variable.
\item[RCLL] Right Continuous and have Limits from the Left everywhere.
\item[SLLN] Strong Law of Large Numbers.


\end{abbrv}

\addcontentsline{toc}{chapter}{List of Symbols}
\pagestyle{empty}

\tableofcontents
\addcontentsline{toc}{chapter}{Introduction}
\pagestyle{empty}

\chapter*{Introduction}
\pagenumbering{arabic}
\pagestyle{fancy}
\addcontentsline{toc}{section}{1. Context and Motivation}
\section*{1. Context and Motivation}
\label{sec:Context and Motivation}

Matching models have recently received a growing interest in the literature of queueing models in which compatibilities between the requests need to be taken into account. 
This is a natural enrichment of service systems in which the requests must be matched, or put in relation, rather than being served. 
Among other fields of applications, this is a natural representation of peer-to-peer networks, interfaces of the collaborative economy (such as car and ride-sharing, dating websites, 
and so on), assemble-to-order systems, job search applications, and healthcare systems (blood banks and organ transplant networks). 
All of these applications share the same common ground: \textit{elements/items/agents} enter a system that is just an interface to put them in relation, and relations are possible only if the ``properties'' (whatever this means) of the elements make them compatible. \\

In \cite{CKW09} (see also \cite{AW11}), a variant of such skill-based systems was introduced, which are now commonly referred to as \textit{Bipartite Matching models} (BM): couples customer/server enter the system at each time point, and customers and servers play symmetrical roles: exactly like customers, servers come and go into the system. Upon arrival, they wait for a compatible customer, and as soon as they find one, leave the system together with it. Otherwise, items remain in the system waiting for compatible arrivals (in particular, there are no service times). These settings are suitable to various fields of applications, among which, blood banks, organ transplants, housing allocation, job search, dating websites, and so on.  In both references, compatible customers and servers are matched according to the {\sc fcfs} `First Come, First Served' service discipline. \\

In \cite{ABMW17}, a subtle dynamic reversibility property is shown, entailing that the stationary state of such systems under {\sc fcfs}, can be obtained in a product form. Moreover, a sub-additivity property is proved, allowing (under stability conditions) the construction of a unique stationary bi-infinite matching of the customers and servers, by a coupling-from-the-past ({\sc cftp}) technique. Interestingly, the product form of the stationary state can then be adapted to various skill-based queueing models as well, and in particular, those applying (various declinations of) the so-called {\sc fcfs-alis} (Assign the Longest Idling Server) service discipline - see e.g. \cite{AW14}, and various extensions of BM models in \cite{AKRW18,BC15,BC17}. \\

In \cite{BGM13}, the settings of  \cite{CKW09,AW11} are generalized to more general service disciplines (termed `matching policies' in this context), and necessary and sufficient conditions for the stability of the system are introduced. 
Moreover, the results in \cite{BGM13} do not assume the independence between the types of the entering customer and the entering server. The system is then called \textit{Extended Bipartite Matching model} (EBM, for short), and suits applications in which independence between the classes of the customers and servers entering simultaneously cannot be assumed. \\

In \cite{MBM18}, a {\sc cftp} result is obtained, showing the existence of a unique bi-infinite matching in various cases for EBM models, and for a broader class of matching policies than {\sc fcfs}, thereby generalizing the results of \cite{ABMW17}.  

\addcontentsline{toc}{section}{2. Problem statement}
\section*{2. Problem statement}
\label{sec:Problem statement}
To model concrete systems, the need then arose to extend these different models. Indeed, in many applications, the assumption of pairwise arrivals may appear somewhat artificial, and it is more realistic to assume that arrivals are simple. Also, all the aforementioned references assume that the compatibility graph is bipartite, namely, there are easily identifiable classes of {\em servers} and classes of {\em customers}. 
For instance, in dating websites, it is a priori not possible to split items into two sets of classes (customers and servers) with no possible matches within those sets. 
In particular, if one considers blood types as a primary compatibility criterion, the compatibility graph between couples is naturally non-bipartite. \\
Motivated by these observations, a variant model was introduced in \cite{MaiMoy16}, in which items arrive one by one and the compatibility graph in general, i.e., not necessarily bipartite: specifically, in this so-called {\em General Matching model} (GM for short), items enter one by one in discrete-time in a buffer, and belong to determinate classes in a finite set $\maV$. Upon each arrival, the class of the incoming item is drawn independently of everything else, from a distribution $\mu$ having full support $\maV$. A connected graph $\mbG$ whose set of nodes is precise $\maV$ determines the compatibility among classes. Then, an incoming item is either immediately matched, if there is a compatible item in the line, or else stored in a buffer. It is the role of the matching policy $\Phi$ to determine the match of the incoming item in case of a multiple choice. Then, the two matched items immediately leave the system forever.

The {\em stability region} of the model, given $\mathbb{G}$ and $\Phi$, is then defined as the set of measures $\mu$ such that the model is positive recurrent.  A necessary condition for the stability $\textsc{Ncond}$ of GM models are provided in \cite{MaiMoy16}. Also, is proven that the matching policy `Match the Longest' has a maximal stability region, that is, the latter necessary condition $\textsc{Ncond}$ is also sufficient (we then say that the latter policy is {\em maximal}).  Further, the model with a complete $p$-partite (separable) graph is also stable for all matching policy $\Phi.$ However, the study of a particular model on a non-separable graph (see \cite{MaiMoy16}, p.14) shows that $\textsc{Ncond}$ is not sufficient in general for non-separable graphs. This raises the question of whether the sufficiency of $\textsc{Ncond}$ is true only for separable graphs. In \cite{MoyPer17} was proved that, except for a particular class of graphs, there always exists a matching policy rendering the stability region strictly smaller than the set of arrival intensities satisfying the necessary condition for stability $\textsc{Ncond}.$

\addcontentsline{toc}{section}{3. Objectives and Contributions}
\section*{3. Objectives and Contributions}
The main purpose of this thesis is to study the long-run stability of stochastic matching models, in the sense defined above, on hypergraphical and mutligraphical compatibility matching structures, and to illustrate the potential applications of these results to concrete settings.

\subsection*{3.1 Hypergraphical compatibility matching structures} 
{Two closest references to the stochastic matching model on hypergraphs are  \cite{GW14} and \cite{NS16}: in both cases, a general matching model is addressed 
	(in continuous time in the former, and discrete-time - allowing batches of arrivals - in the latter) on an hypergraphical matching structure (notice that \cite{NS16} also allows matchings including several items of the same class). \\
	In \cite{GW14} a matching control is introduced, that asymptotically minimizes the holding cost of items in an unstable system. \cite{NS16} introduces an algorithm that is a variant of the ``Primal-dual algorithm'', allowing to essentially optimize a given objective function {\em provided that stability can be achieved}. Then the objective function can incorporate stability (setting utility 0), in a way that stability is achieved by the essentially optimal algorithm, whenever it is achievable at all. Both references allow {\em idling policies}, i.e., scheduling algorithms allowing to perform no matching at all despite the presence of matchable items in the system, to wait for more profitable future matches.  
	Allowing idling policy makes sense in applications such as assemble-to-order systems, advertisement, or operations scheduling, but is much less suitable to kidney transplant networks, in which case the practitioners always perform a transplant whenever one is possible. In this thesis, \textit{all the matching policies we consider are non-idling}, i.e. entering items are always matched right away if this is possible at all. Thus, the model studied in the present thesis is a special case of the model studied in \cite{NS16}, for simple arrivals, no same-class matchings, and non-idling matching policies. \\
	
	Our approach see \cite{RM19}, is in fact, complementary to that in \cite{NS16} and \cite{GW14}: generalizing the approach of \cite{MaiMoy16} to hypergraphs instead of graphs, 
	in this thesis we are mostly concerned with the structural properties of the underlying hypergraph of the matching model, and determine classes of hypergraph for which there does, or does not, exist non-idling policy that can stabilize the system. In a sense, the present work addresses an upstream problem to that of implementing a performant matching algorithm: we provide simple and comprehensive criteria, {\em based only on the structural properties of the considered hypergraph} for the (non)-existence of a stabilizing non-idling policy.}\\

We address the problem of the existence of a steady-state for the system: we formally define the stability region of the system as the set of measures on the set of nodes, rendering the natural Markov chain of the system positive recurrent, for a given compatibility hypergraph and a given matching policy. Also, we assess the form of the stability region of specific stochastic matching models, as a function of the geometry of the underlying hypergraphs. In a nutshell, we show that such systems are not easily stabilizable, by exhibiting wide classes of models having an empty stability region, whatever the non-idling matching policy is. Finally, we provide or give bounds for, the stability region of particular stabilizable systems. 

\subsection*{3.2 Multigraphical compatibility matching structures}
Motivated again by concrete applications, we present a further extension of the GM model. Indeed, in various contexts, among which dating websites and peer-to-peer interfaces, it is natural to assume that items {\em of the same class} can be matched together. 
Hence, the need to generalize the previous line of research to the case where the matching architecture is a {\em multigraph} (a graph admitting {\em self-loops}, that is, edges connecting nodes to themselves), rather than just a graph. 

This generalization is the core of Chapter \ref{chap5:Multigraph} of the present thesis (see \cite{BMMR20}). We show how several stability results of \cite{MaiMoy16,MBM17,JMRS20} can be generalized to the case of a multigraphical matching structure. As is easily seen, the buffer of a matching model on a multigraph is hybrid by essence: nodes admitting self-loops (if any)  admit at most one item in the line, whereas nodes with no self-loops (if any) have unbounded queues. A matching model on a multigraph typically has a larger stability region than the corresponding model on a graph on which all self-loops are erased (the {\em maximal subgraph} of the latter -  see Definition \ref{def:restricted}), but the interplay between self-looped nodes and their non-self-looped neighbors needs to be clearly understood: intuitively, the arrival flows to self-looped nodes appear as auxiliary flows helping their neighboring non-self-looped nodes to stabilize their queues - provided that the arrivals to self-looped nodes don't match too often with one another. 



\addcontentsline{toc}{section}{4. Thesis outline}
\section*{4. Thesis outline}
\label{sec:Thesis outline}
This manuscript comprises two parts: the first lays prerequisites and background of the thesis while the second presents our main contributions. The hierarchy of the report is based on seven chapters as indicated in what follows.\\

Part \ref{part:background} is devoted to present basic knowledge related to our subject, the basic notions, and the literature review.\\

Chapter \ref{chap1:Definitions and Fundamental Concepts} provides the scientific context for our work. Chapter \ref{chap2:Definition of general model} is devoted to present the dynamic of the stochastic matching model on general graph structures, matching policies, and Markov representation of the model.  Chapter \ref{chap2: state of art} presents the state of arts that is devoted to the related work to stochastic matching model. We end this chapter with the positioning of our work compared to others.\\

Part \ref{part:contributions} is organized into four chapters, that consist of the contributions of this thesis and an application.\\

 Chapter \ref{chap4:Hypergraph} provides the first contribution of this thesis, which is the study of the stochastic matching model on hypergraphs, provide necessary conditions of stability then we identify classes of hypergraphs that has an empty stability region. However, we show that stable matching models on hypergraphs exist. To show how stability can be shown in concrete examples, we provide two case studies of simple hypergraphs, that is, complete $3$-uniform hypergraphs, and sub-hypergraphs of the latter where several hyperedges are erased. We finish with the discussion of the results of the chapter. 
 Chapter \ref{chap5:Multigraph} provides the second contribution of this thesis which is the study of the stochastic matching model on multigraph among which, the maximality and the explicit product form of the stationary probability for {\sc fcfm} policy, and the maximality of Max-Weight policies. Also, we provide a few examples to illustrate our main results.  We finish with the discussion of the results of the chapter. 
 Chapter \ref{chap6:FluidLimits} is devoted to developing the stability of particular cases for multigraph and hypergraph using the fluid limits techniques rather than the Lyapunov-Foster Theorem. 
 Chapter \ref{chap7:Apllication} present an application that compare the models for organ transplantation concerned to compatibilities of blood types that illustrate the importance of studied the stability of the model for complete 3-uniform hypergraphs instead of studied the stability of the complete 3-partite graphs according to some distributions.\\
 
 A general conclusion recapitulating the basic concepts and contributions of the thesis, as well as future work and perspectives, are given at last.


\clearpage
\pagestyle{empty}
\part{Background}\label{part:background}
\pagestyle{empty}
\clearpage
\chapter{Definitions and Fundamental Concepts}\label{chap1:Definitions and Fundamental Concepts}
\pagestyle{fancy}

In this chapter, we introduce the main definitions and fundamental concepts used in this work. First, we start with some preliminaries, we introduce the \textit{general graph structures}. Then,  we present the definitions and specific properties of \textit{graph structures}, \textit{hypergaph structures} and \textit{multigraph structures}.\\

In this introductory chapter, we provide an intuitive background to the material that will be used in the coming chapters. 
 

\section{Preliminaries}
\label{subsec:preliminary}
\subsection{Classical notations}
We adopt the usual $\R$, $\Z,$ and $\N$ notation for the sets of real numbers, of integers and natural integers, respectively. We let $\R^+$ and $\Z^+$ be the non-negative real numbers and non-negative integers respectively. Also, we denote by $\R^{++}$ and $\Z^{++}$ (or $\N^+$) the strictly positive real numbers and strictly positive integers, respectively. For $y\in\R$, we denote by $\lfloor y\rfloor$ is the integer part of $y.$ For $a$ and $b$ in $\N$, we denote by $\llbracket a,b \rrbracket$ the integer interval $[a,b]\cap \N$.  We let $a\wedge b$ and $a\vee b$ denote the minimum and the maximum of two numbers $a,b\in\R$ respectively.

Let $A$ be a finite set. The cardinality of $A$ is denoted by $|A|$  and for any $k \in\Z^{++},\;A^k$ denotes the set of $k$-dimensional vectors with components in $A$. For any $i\in \llbracket 1,q \rrbracket$, let $\gre_i$ denote the vector of $\mathbb{N}^{q}$ of components $(\gre_i)_j=\delta_{ij},\;j\in \llbracket 1,q \rrbracket$.  
Let $q\in \N^+$. The null vector of $\mathbb{N}^{q}$ is denoted by $\mathbf 0$. The norm of any vector $u \in \N^q$ is denoted by $\parallel u \parallel =\sum_{i=1}^{q} u_i$.  For any subset $J \subset \llbracket 1,k \rrbracket$ and $x \in \R^k$, we use the notation $x_J$ for the restriction of $x$ to its coordinates corresponding to the indices of $J$.

\subsection{Aphabet and words}
An alphabet is a finite non-empty set denoted by $A.$ A word form is writing the numerical/number as you would say it in words. We let $A^*$ denote the free monoid associated with $A$, i.e., the set of finite words over the alphabet $A$. 
The length of a word $w\in A^*$ is denoted by $|w|$. We write any word $w\in A^*$ as $w=w(1)w(2)...w(|w|)$. As a convention, let us denote by {\bf 0} the empty word.
We denote, for any $a\in A$, $|w|_a$ the number of occurrences of the letter $a$ in the word $w$. 
Having set an ordering on $A$, and denoting by $1,2,...,|A|$ the elements of $A$ in increasing order, the commutative image of a word $w\in A$ is the $\N^{|A|}$-valued vector $[w]$ defined by 
$[w]=\left(|w|_1,...,|w|_{|A|}\right),$ i.e., the vector whose $i$-th coordinate is the number of occurrences of the letter $i$ in the word $w$. The concatenation of $k$ words $w^1,w^2,\dots,w^k$ of $A^*$, that is, the word $w$ in which appear successively from left to right, the words $w^1, w^2, \dots, w^k$, is denoted by $w=w^1w^2\dots w^k$. Also, for any $w\in A^*$ of length $|w|=q$, $(w=w_1w_2\cdots w_q),$ and any $i,j,\cdots,k\in\llbracket 1,q \rrbracket,$ we denote by $w_{[i,j,...,k]}$ the word of length $|w|-|\{i,j,...,k\}|$ obtained from $w$ by deleting its $i$-th, $j$-th, $\cdots,$ and $k$-th letters. For any integer $q,$ the vectors of $\N^q$ are denoted as $x=\left(x(1),\cdots,x(q)\right),$ and denoted $\bw=(w(1),\dots,w(q)).$ Define for any subset $B$ of $\maV$, $x(B)$ to be the class-content of elements of $B$ as $x(B) = \sum_{i\in B}x(i).$
\subsection{Probability}
\begin{sloppypar}
All the random variables (r.v.'s, for short) are defined on a common probability space $(\Omega,\maF, \mathbb{P})$. Given a finite set $S$, we denote by $\mathscr M(S)$ the set of probability measures on $S$ having $S$ as exact support. 
Denote by $\bar S$ the complement set of $S$ (within a set of reference that is fixed by the context).
\end{sloppypar}

For an interval $I\subset [0,\infty),$ let $\D^d(I)$ denote the space of $\R^d$-valued functions on $I$ that are right continuous and have limits from the left everywhere denoted by `RCLL', given with the standard Skorohod $J_1$ topology [4]. To simplify notation,
we write, e.g., $\D^d(a, b):=\D^d((a, b))$, and $\D(I):=\D^1(I).$ However, for the convergence in $\D^d$ that holds over an arbitrary compact subinterval of $[0,\infty)$ we omit the interval from the notation. 


We write $\Rightarrow$ to denote convergence in distribution and denote $\{Y^n;\, n \in \Z^{++}\}$ the sequence of real-valued random variables. For any $M > 0,$ if $P(Y^n > M) \rightarrow 1$ as $n \rightarrow\infty,$ then we write $Y^n \Rightarrow \infty$. Denote by $\bar{Y}^n\; :=\; Y^n/n$ the fluid-scaled version of a sequence of stochastic processes $\{Y^n;\,n \in \Z^{++}\}.$\\

In section \ref{sec:generalGraphStructure} below, we present the specific properties of general graph structures.

\section{General graph structures}
\label{sec:generalGraphStructure}
We consider different types of general graph structures $\mbS$ defined  as a couple of $(\maV,\mathcal S),$ where:
\begin{itemize}
	\item The finite set $\maV$ is the set of nodes (vertices) of $\mbS$. We let $q(\mbS)$ be the cardinality of $\maV$, and the general graph structures are of order $q(\mbS)$. 
	\item A finite set $\mathcal S:=\left\{S_1,...,S_{m(\mbS)}\right\}$
	of subsets of  $\maV$ such that $\bigcup_{i=1}^{m(\mbS)}S_i=\maV,$ whose elements are called edges of $\mbS$ (in case of hypergraphical is called hyperedges).   
\end{itemize} 
  Whenever no ambiguity is possible we denote the general graph structures by \textbf{matching structures}, and we often write $q:=q(\mbS)$, $m:=m(\mbS)$. The degree of a node $i\in \maV$ is the number of edges $i$ belongs to, i.e., $d(i)=\sum_{\ell=1}^{m(\mbS)}\ind_{S_{\ell}}(i)$. If there exists a constant $d$ such that $d(i)=d$ for any $i\in\maV$, then $\mathbb S$ is said $d$-{regular}.

For any set $A \subset \maV$, we denote 
\begin{equation}
\label{eq:defSA}
\maS(A) = \left\{S \in \maS\;:\; S \cap A\ne \emptyset\right\},
\end{equation}
i.e., the set of edges that intersects with $A$. With some abuse, for any node $i\in\maV$, we write $\maS(i):=\maS(\{i\})$. 
\begin{definition}
	We say that $I \subset \maV$ is an {\em independent set} of $\,\mbS$ if $I$ does not include any edge of $\,\mbS$, i.e, for any 
	$S\in \maS$, $S\cap \bar I \ne \emptyset$. We also let $\mathbb I(\mbS)$ be the set of all independent sets of $\,\mbS$. 
\end{definition}
An independent set is said \textit{maximal} if it is not strictly included in another independent set.

\begin{definition}
	\label{def:transverse}
	A set $T\subset \maV$ is a {transversal} of	$\,\mathbb{S}$ if it meets all its edges, that is, $T\cap S \neq\emptyset, \mbox{ for any } S\in \maS.$ 
	The set of transversals of $\,\mbS$ is denoted by $\mathcal T(\mbS)$. 
	A transversal $T$ is said minimal if it is of minimal cardinality among all transversals of $\,\mbS$. 
	The transversal number of the matching structures $\mbS$ is the cardinality of its minimal transversals. It is denoted $\tau(\mbS)$. 
\end{definition}


 Different kinds of matching structures such as a \textit{graph}, \textit{hypergraph} and \textit{multigraph} in which each of them has specific properties (see Figure \ref{fig:graphHyperMulti}).

{\scriptsize \begin{figure}[htp]
	\begin{center}
		\def\fourthellip{(-2.8, 2) ellipse [x radius=2.cm, y radius=0.9cm, rotate=90]}
		\def\fifthellip{(-5.2, 2) ellipse [x radius=2.cm, y radius=0.9cm, rotate=90]}
		\def\firstellip{(-4, 2) ellipse [x radius=2.05cm, y radius=1cm, rotate=90]}
		\def\secondellip{(-4, 3.38) ellipse [x radius=2.6cm, y radius=0.7cm, rotate=-180]}
		\def\thirdellip{(-4, 0.62) ellipse [x radius=2.6cm, y radius=0.7cm, rotate=-180]}
		\begin{tikzpicture}
		
		\fill (-10,0) circle (2pt)node[above]{3};
		\fill (-8,0) circle (2pt)node[above]{4};
		\fill (-9,1.5) circle (2pt)node[right]{2};
		\fill (-9,3) circle (2pt)node[above]{1};

		\draw[-] (-8,0) -- (-9,1.5);
		\draw[-] (-8,0) -- (-10,0);
		\draw[-] (-10,0) -- (-9,1.5);
		\draw[-] (-9,1.5) -- (-9,3);
		
		\filldraw (-2.7,3.5) circle (2pt) node [below] {8};
		\filldraw (-4,3.5) circle (2pt) node [below] {7};
		\filldraw (-5.3,3.5) circle (2pt) node [below] {6};
		
		\filldraw (-3.5,2) circle (2pt) node [right] {5};
		\filldraw (-4.5,2) circle (2pt) node [left] {4};
		
		\filldraw (-5.3,0.5) circle (2pt) node [above] {1};
		\filldraw (-4,0.5) circle (2pt) node [above] {2};
		\filldraw (-2.7,0.5) circle (2pt) node [above] {3};

		\draw \firstellip {};
		
		\draw \secondellip {};
		
		\draw \thirdellip {};
		\draw \fourthellip {};
		\draw \fifthellip {};

		\fill (1,3) circle (2pt)node[above]{1};
		\fill (1,1.5) circle (2pt)node[right]{2};
		\fill (0,0) circle (2pt)node[above]{3};
		\fill (2,0) circle (2pt)node[above]{4};

		\draw[-] (1,3) -- (1,1.5);
		\draw[-] (1,1.5) -- (0,0);
		\draw[-] (1,1.5) -- (2,0);
		\draw[-] (0,0) -- (2,0);
		
		\draw[thick,->] (1,1.5) to [out=200,in=110,distance=15mm] (1,1.5);
		\end{tikzpicture}
		\caption{Left: Graph. Middle: Hypergraph. Right: Multigraph.}
		\label{fig:graphHyperMulti}
	\end{center}
\end{figure}
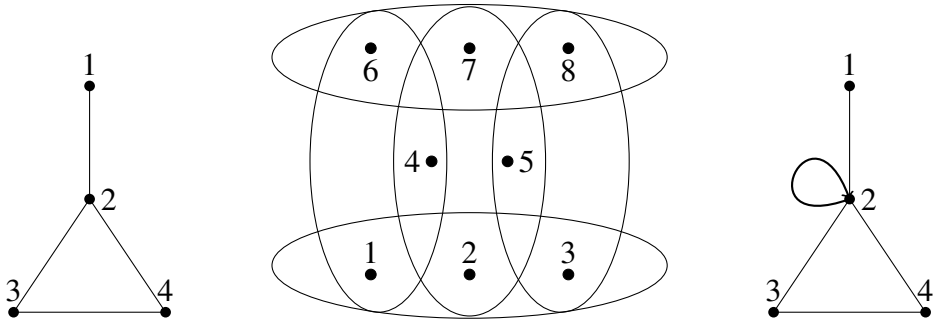
}
\section{Graphs}
	\label{sec:Graph}
In this section, we consider that $\mbS$ be a graph under the form $\mathbb G=(\maV,\maE).$  For easy reference, let us introduce the basics that will be used in this thesis. 
 
 A (simple) graph $\mathbb G$ is defined  as a couple $(\maV,\maE).$ Two vertices $u,v\in\maV$ are said to be \textit{adjacent}, if there is an edge between $u$ and $v$. We write $u\v v$ (or $v\v u$) for $\{u,v\}\in\mathcal E$ and $u \pv v$ (or $v\pv u$) else. The \textit{neighborhood} of a vertex $v$ is the subgraph of $\mathbb{G}$ induced by all vertices adjacent to $v.$
 
 As the equation (\ref{eq:defSA}), and specifically for any graph $\mathbb G=(\maV,\mathcal E)$ and any $U\subset\maV,$ we denote    
	\[\mathcal E(U):=\left\lbrace v\in\maV\;:\;\exists u\in U, u\v v\right\rbrace,\]
the neighborhood of $U$, and for $u\in \maV$, we write for short $\maE(u)=\maE(\{u\})$. The  {\it degree} of a vertex is the number of edges connecting it.
 A {\it walks} is a way of getting
	from one vertex to another, and consists of a sequence of edges, one following after another. A walk in which no vertex appears more than once is called a {\it path}.  A \textit{cycle} is a non-empty trail in which the only repeated vertices are the first and last vertices. For example, given the graph depicted in Figure \ref{fig:graphHyperMulti} we have, 1 —> 2 —> 3 is a path of length 2 and 1 —> 2 —> 3 —> 4 —> 2 is a walk of length 4. A walk of the form 2 —> 3 —> 4 —> 2 is called a {\it cycle}.\\
	 A \textit{chain} is a sequence of vertices from one vertex to another using the edges. A chain is \textit{closed} if the first and last vertex are the same. A graph is called \textit{connected} if there is a chain between every pair of vertices in the
 graph.
	
	 Throughout the presentation of this thesis all considered graphs are simple and connected.
	 
	
	\begin{definition}
		\label{def:Cycle}		
A cycle or circular graph is a graph that consists of a single cycle, or in other words, a number of vertices (at least 3) connected in a closed chain. The cycle graph with $q$ vertices is called $C_q$. The number of vertices in $C_q$ equals the number of edges, and every vertex has a degree 2; that is, every vertex has exactly two edges incident with it.
\end{definition}
		A cycle with an even number of vertices is called an {\em even cycle}; a cycle with an odd number of vertices is called an {\em odd cycle}. For example, in Figure \ref{fig:Graph.q-partiteAndComplementaire} the graph dedicated on the left is an odd cycle $C_3$.


\begin{definition}
\label{def:CompleteGraph}
A graph, in which each pair of distinct vertices is adjacent, is a {\it complete
graph} i.e., $(\forall i,j\in\maV,\;i\v j$). We denote the complete graph on $q$ vertices by $K_q$ and it has $q(q-1)/2$ edges. 
\end{definition}

\begin{definition}
\label{def:bipartiteGraphe}
A {\it $k$-partite} graph is a graph whose vertices are or can be partitioned into $k$ different independent sets. If $k=2,$ we say that the graph is a bipartite graph.
A {\it complete $k$-partite} graph (also its called separable graph in \cite{MaiMoy16}) is a k-partite graph in which there is an edge between every pair of vertices from different independent sets.
\end{definition}

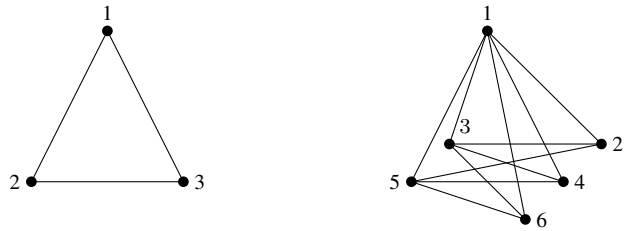
\begin{figure}[htb]
	\begin{center}
		\begin{tikzpicture}
		
		\fill (1,2) circle (2pt)node[above]{\scriptsize{1}};
		\fill (0,0) circle (2pt)node[left]{\scriptsize{5}};
		\fill (0.5,0.5) circle (2pt)node[above]{\scriptsize{$\;\;\;\;\;3$}};
		\fill (2,0) circle (2pt)node[right]{\scriptsize{4}};
		\fill (2.5,0.5) circle (2pt)node[right]{\scriptsize{2}};
		\fill (1.5,-0.5) circle (2pt)node[right]{\scriptsize{6}};
		
		\draw[-] (2.5,0.5) -- (1,2);
		\draw[-] (0.5,0.5) -- (1,2);
		\draw[-] (2,0) -- (1,2);
		\draw[-] (0,0) -- (1,2);
		\draw[-] (0,0) -- (2,0);
		\draw[-] (0,0) -- (2.5,0.5);
		\draw[-] (0.5,0.5) -- (2,0);
		\draw[-] (0.5,0.5) -- (2.5,0.5);
		\draw[-] (1.5,-0.5) -- (0.5,0.5);
		\draw[-] (1.5,-0.5) -- (0,0);
		\draw[-] (1.5,-0.5) -- (1,2);

		\fill (-4,2) circle (2pt)node[above]{\scriptsize{1}};
		\fill (-5,0) circle (2pt)node[left]{\scriptsize{2}};
		\fill (-3,0) circle (2pt)node[right]{\scriptsize{3}};

		\draw[-] (-4,2) -- (-3,0);
		\draw[-] (-4,2) -- (-5,0);
		\draw[-] (-5,0) -- (-3,0);

		\end{tikzpicture}
		\caption{Complete $3$-partite graphs.}
		\label{fig:Graph.q-partiteAndComplementaire}
	\end{center}
\end{figure}
\begin{ex}
	\label{ex:DefGraphKPartite}\rm
	The two graphs depicted in Figure \ref{fig:Graph.q-partiteAndComplementaire} are complete 3-partite graphs. The graph on the left is $K_3$ (i.e., for any two vertices $i,j$ we have $i\v j).$ The independent sets are $I_1=\{1\},\;I_2=\{2\}$ and $I_3=\{3\}.$ However, the graph on the right is not a complete graph. The independent sets are $I_1=\{1\},\;I_2=\{2,4,6\}$ and $I_3=\{3,5\}.$
	
\end{ex}

In the Figure \ref{fig:N-NN-W Graph}, we present the famous types of bipartite graph, such as, `\textbf{N}' graph, `\textbf{NN}' graph and `\textbf{W}' graph.
\begin{figure}[htp]
	\begin{center}
		\begin{tikzpicture}
		
		\fill (1,2) circle (2pt)node[right]{\scriptsize{$1^\prime$}};
		\fill (1,0) circle (2pt)node[right]{\scriptsize{1}};
		\fill (3,0) circle (2pt)node[right]{\scriptsize{2}};
		\fill (3,2) circle (2pt)node[right]{\scriptsize{$2^\prime$}};
		
		\fill (5,0) circle (2pt)node[right]{\scriptsize{1}};
		\fill (5,2) circle (2pt)node[right]{\scriptsize{$1^\prime$}};	
		\fill (7,0) circle (2pt)node[right]{\scriptsize{2}};
		\fill (7,2) circle (2pt)node[right]{\scriptsize{$2^\prime$}};
		\fill (9,0) circle (2pt)node[right]{\scriptsize{3}};
		\fill (9,2) circle (2pt)node[right]{\scriptsize{$3^\prime$}};

		\fill (11,2) circle (2pt)node[right]{\scriptsize{$1^\prime$}};	
		\fill (12,0) circle (2pt)node[right]{\scriptsize{1}};
		\fill (13,2) circle (2pt)node[right]{\scriptsize{$2^\prime$}};
		\fill (14,0) circle (2pt)node[right]{\scriptsize{2}};
		\fill (15,2) circle (2pt)node[right]{\scriptsize{$3^\prime$}};

		
		\draw[-] (1,0) -- (1,2);
		\draw[-] (1,2) -- (3,0);
		\draw[-] (3,0) -- (3,2);
		
		\draw[-] (5,0) -- (5,2);
		\draw[-] (5,2) -- (7,0);
		\draw[-] (7,0) -- (7,2);
		\draw[-] (7,2) -- (9,0);
		\draw[-] (9,0) -- (9,2);
		
		\draw[-] (11,2) -- (12,0);
		\draw[-] (12,0) -- (13,2);
		\draw[-] (13,2) -- (14,0);
		\draw[-] (14,0) -- (15,2);

		\end{tikzpicture}
		\caption{Left: `\textbf{N}' graph. Middle: `\textbf{NN}' graph. Right: `\textbf{W}' graph.}
		\label{fig:N-NN-W Graph}
	\end{center}
\end{figure}
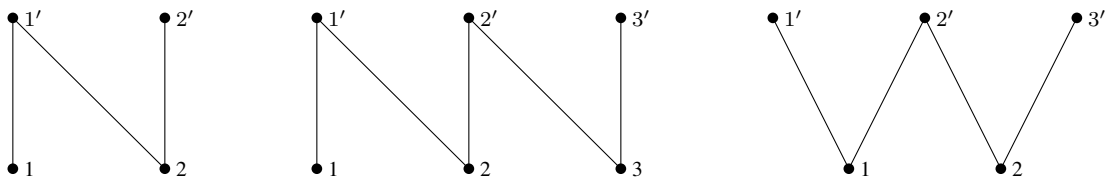



\section{Hypergraphs}
\label{subsec:prelimhypergraphs}
In this section we consider that $\mbS$ be a hypergraph under the form $\mbH=(\maV,\maH).$ For easy reference, let us first introduce the basics of hypergraph theory will be used in this thesis. \\

A throughout presentation of the topic can be found e.g., in \cite{Ber89}. \\

	A hypergraph $\mathbb H$ is defined as a couple ($\maV$,$\maH$), where $\maV$ is a finite set of nodes of $\mathbb H$ and  $\maH:=\left\{H_1,...,H_{m(\mbH)}\right\}$ is a finite set whose elements are called \textit{hyperedges} of $\mbH$. 
	
	We say that the hypergraph is {\it simple} (or a {Sperner family}) if $H_i\subset H_j$ implies $i=j$ for all $i,j \in \llbracket 1,m(\mbH) \rrbracket$, i.e.,  
	no hyperedge is included in another one (if not, say that the hypergraph is {\it multiple} hypergraphs) see Figure \ref{fig:completeHyper}. We assume hereafter that all hypergraphs are simple.  
	A {\it subhypergraph} of $\mbH$ is a hypergraph $\mbH'=(V,\maH')$ such that $\maH' \subset \maH$. 

\begin{definition}\label{Def:rankdegree}
	Let $\mathbb H=(\maV,\maH)$ be a hypergraph. 
	The {rank} of $\mathbb H$ is the largest size of a hyperedge, i.e., the integer $r(\mathbb H)=\max_{j\in\llbracket 1,m(\mbH)\rrbracket}|H_j|$; 
	the {anti-rank} of $\mathbb H$ is defined as $a(\mathbb H)=\min_{j\in\llbracket 1,m(\mbH)\rrbracket}|H_j|$, i.e., the smallest size of a hyperedge. 
	If there exists a constant $r$ such that $r(\mathbb H)= a(\mathbb H)=r$, then $\mathbb H$ is said $r$-{uniform}. 
	
\end{definition}
\begin{remark}
\label{rq:Hyper2UnifoIsGra}\rm
As is easily seen, any $2$-uniform hypergraph is a graph whose edges are the elements of $\maH$, and any simple, connected hypergraph contains no isolated node, i.e., has 
anti-rank at least 2. 
\end{remark}

\begin{figure}[!h]
	\begin{center}
		\def\firstellip{(0.15, 2.5) ellipse [x radius=4cm, y radius=0.8cm, rotate=180]}
		\def\secondellip{(0.15, 1.6cm) ellipse [x radius=4cm, y radius=0.8cm, rotate=180]} 
		\def\thirdellip{(0.8, 1.7) ellipse [x radius=3cm, y radius=1.3cm, rotate=90]} 
		\def\fourthellip{(-0.8, 1.7cm) ellipse [x radius=3cm, y
			radius=1.3cm, rotate=-90]}	
		
		\def\fifthellip{(8, 2) ellipse [x radius=3cm, y radius=1.5cm, rotate=-180]}
		\def\sixthellip{(8.3, 2) ellipse [x radius=2cm, y radius=1cm, rotate=180]}
		\begin{tikzpicture}
		\draw \firstellip {};
		
		\draw \secondellip {};
		
		\draw \thirdellip {};
		
		\draw \fourthellip {};
		
		\filldraw 
		(1.8,2) circle (2pt) node [left] {$\;4$};
		\filldraw 
		(0,2.8) circle (2pt) node [right] {$1$};
		\filldraw 
		(-1.8,2) circle (2pt) node [right] {$\;3$};
		\filldraw 
		(0,1.2) circle (2pt) node [right] {$2$};
		
		\draw \fifthellip {};
		
		\draw \sixthellip {};

		\fill (8.1,2) circle (2pt)node[above]{3};
		\fill (9.5,2) circle (2pt)node[above]{2};
		\fill (6.7,2) circle (2pt)node[above]{4};
		\fill (6,2.8) circle (2pt)node[right]{5};
		
		\fill (6,1.3) circle (2pt)node[right]{1};
		
		\end{tikzpicture} 
		\caption{\label{fig:completeHyper} Left: Complete $3$-uniform hypergraph of order $4.$ Right: multiple hypergraph.}
	\end{center}
\end{figure}
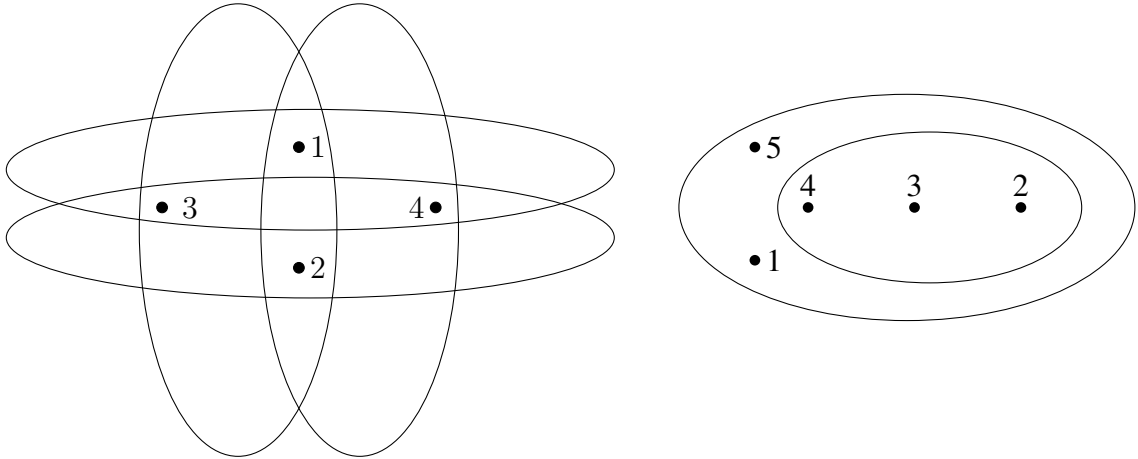
\begin{ex}
	\label{ex:completeHyper}\rm
Consider the structure depicted in Figure \ref{fig:completeHyper} (left), it represents a hypergraph $\mathbb{H}=(\maV,\maH)$ with $\maV=\{1,2,3,4\}$ and $\maH=\left\lbrace\{1,2,3\},\{1,2,4\},\{1,3,4\},\{2,3,4\}\right\rbrace$. The cardinal of $\maV$ is equal to 4, then $\mbH$ is of order 4; it is simple because no hyperedges is included in another one; it is $3$-uniform since all hyperedges are of cardinality $3$, 
	and $3$-regular, because all nodes are of degree 3 (they all belong to exactly 3 hyperedges). As all hyperedges of cardinality $3$ appear 
	in $\maH$, this hypergraph is said {\em Complete $3$-uniform of order 4}. 
\end{ex}

\begin{definition}\label{Def:GraphRepresentatif}
	The {representative graph} of a hypergraph $\mathbb H=(\maV,\maH)$ is the graph $L(\mathbb H)=(\maH,\mathcal E)$ whose nodes are the elements of $\maH$, 
	and such that $(H_i,H_j) \in \mathcal E$ (i.e., $H_i$ and $H_j$ share an edge in the graph) if and only if $H_i\cap H_j\neq\emptyset$. 
	The hypergraph $\mathbb H$ is said {connected} if $L(\mathbb H)$ is connected. 
\end{definition} 

\begin{figure}[htp!]
	\centering
	\includegraphics[width=.45\linewidth]{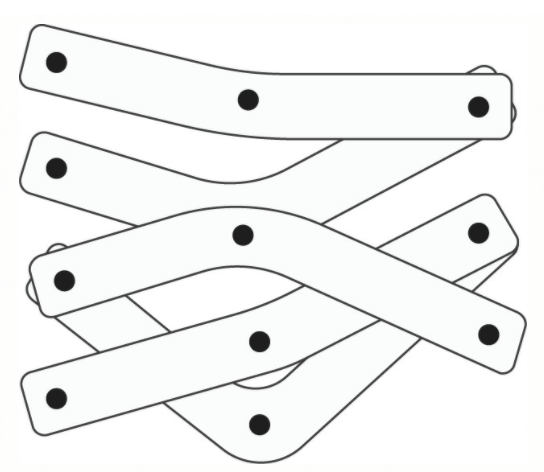}
	\caption{A 3-uniform 3-partite hypergraph.}
	\label{fig:3-uniform3-partite}
\end{figure}
\begin{definition}
	{An $r$-uniform $(r\geq 2)$ hypergraph $\mathbb{H}=(\maV,\maH)$ is said to be $r$-partite if there exists a partition $V_1, V_2,\cdots\,, V_r$ of $\maV$ such that every hyperedge in 
		$\mbH$ meets each of the $V_i$'s at precisely one vertex, i.e., for any $H\in\maH$ and any $i\le r$, $\left|H\cap V_i\right|=1.$ A 3-uniform 3-partite hypergraph depicted on Figure \ref{fig:3-uniform3-partite}.} 
	With some abuse, we say that an $r$-uniform hypergraph $\mbH$ is $r$-uniform bipartite, if there exists a partition $V_1,V_2$ of $\maV$ such that for any $H\in\maH$, $|H\cap V_1|=1$ and $|H\cap V_2|=r-1.$ 
\end{definition}

{
	\begin{remark}
		\label{remark:bipartite}
		Notice, first, that in the case $r=2$, $\mbH$ being $2$-partite means exactly that it is bipartite. Second, any $2$-uniform bipartite hypergraph cannot be $2$-partite 
		unless it is a bipartite graph.
\end{remark}}

\begin{definition}
	A hypergraph $\mbH=(\maV,\maH)$ satisfies Hall's condition if $|V_2| \geq |V_1|$ for any disjoint subsets $V_2$ and $V_1$ of $\maV$ satisfying 
	$|H\cap V_2|\geq |H\cap V_1|$ for all hyperedges $H\in\maH$.
	 
\end{definition}

\begin{ex}\rm
	Consider a 4-uniform hypergraph $\mbH=(\maV,\maH)$ such that $\maV=\{1,2,3,4,5\}$ and $\maH=\left\lbrace\{1,2,4,5\},\{1,3,4,5\},\{2,3,4,5\}\right\rbrace.$ There exists a partition of $\maV$ into two disjoint sets $V_1=\{1,2,3\}$ and $V_2=\{4,5\}$ such that $|V_2|<|V_1|$ and for any $H\in\maH,$ we have $|H\cap V_2|\geq |H\cap V_1|.$ Then $V_1$ and $V_2$ violating Hall's condition.
\end{ex}
\begin{figure}[htp]
	\begin{center}
		\def\tfourthellip{(0, 3.6cm) ellipse [x radius=2.5cm, y radius=0.7cm, rotate=180]} 
		\def\firstellip{(1.7, 3.1) ellipse [x radius=2.5cm, y radius=0.85cm, rotate=-33]}
		\def\secondellip{(2.9, 1.8cm) ellipse [x radius=2.5cm, y radius=0.8cm, rotate=-58]} 
		\def\thirdellip{(3.55, 0.05) ellipse [x radius=2.5cm, y radius=0.65cm, rotate=-90]}
		\def\fourthellip{(3.1, -1.8cm) ellipse [x radius=2.5cm, y radius=0.7cm, rotate=60]} 
		\def\sfirstellip{(1.7, -3.1) ellipse [x radius=2.5cm, y radius=0.85cm, rotate=33]}
		
		\def\ssecondellip{(0, -3.6cm) ellipse [x radius=2.5cm, y radius=0.7cm, rotate=180]} 
		\def\sthirdellip{(-1.7, 3.1) ellipse [x radius=2.5cm, y radius=0.85cm, rotate=33]}
		\def\sfourthellip{(-2.9, 1.8cm) ellipse [x radius=2.5cm, y radius=0.8cm, rotate=58]} 
		\def\tfirstellip{(-3.55, 0.05) ellipse [x radius=2.5cm, y radius=0.65cm, rotate=90]}
		\def\tsecondellip{(-3.1, -1.8cm) ellipse [x radius=2.5cm, y radius=0.7cm, rotate=-60]} 
		\def\tthirdellip{(-1.7, -3.1) ellipse [x radius=2.5cm, y radius=0.85cm, rotate=-33]}

		\def\ytfourthellip{(11, 2.2cm) ellipse [x radius=3.3cm, y radius=2cm, rotate=180]} 
		\def\ysecondellip{(13, 1cm) ellipse [x radius=3.3cm, y radius=2cm, rotate=-60]} 
		\def\yfourthellip{(13, -1cm) ellipse [x radius=3.3cm, y radius=2cm, rotate=60]} 
		
		\def\yssecondellip{(11, -2.2cm) ellipse [x radius=3.3cm, y radius=2cm, rotate=180]} 
		\def\ysfourthellip{(9, 1cm) ellipse [x radius=3.3cm, y radius=2cm, rotate=60]} 
		\def\ytsecondellip{(9, -1cm) ellipse [x radius=3.3cm, y radius=2cm, rotate=-60]}

		\begin{tikzpicture}[thick, scale=0.6]
		\filldraw (-0*360/12:4) circle (3pt) node [above]{$$};
		\filldraw (-1*360/12:4) circle (3pt) node [above]{$$};
		\filldraw (-2*360/12:4) circle (3pt) node [above]{$$};
		\filldraw (-3*360/12:4) circle (3pt) node [above]{$$};
		\filldraw (-4*360/12:4) circle (3pt) node [above]{$$};
		\filldraw (-5*360/12:4) circle (3pt) node [above]{$$};
		\filldraw (-6*360/12:4) circle (3pt) node [above]{$$};
		\filldraw (-7*360/12:4) circle (3pt) node [above]{$$};
		\filldraw (-8*360/12:4) circle (3pt) node [above]{$$};
		\filldraw (-9*360/12:4) circle (3pt) node [above]{$$};
		\filldraw (-10*360/12:4) circle (3pt) node [above]{$$};
		\filldraw (-11*360/12:4) circle (3pt) node [above]{$$};
		\draw \firstellip node [label={[xshift=1.5cm, yshift=0.8cm]$$}] {};
		\draw \secondellip node [label={[xshift=2.2cm, yshift=0cm]$$}] {};
		\draw \thirdellip node [label={[xshift=2.0cm, yshift=-0.8cm]$$}] {};
		\draw \fourthellip node [label={[xshift=2.1cm, yshift=-0.6cm]$$}] {};
		\draw \sfirstellip node [label={[xshift=2.1cm, yshift=-0.8cm]$$}] {};
		\draw \ssecondellip node [label={[xshift=0cm, yshift=-2cm]$$}] {};
		\draw \sthirdellip node [label={[xshift=-0.3cm, yshift=1.2cm]$$}] {};
		\draw \sfourthellip node [label={[xshift=-2.2cm, yshift=0.4cm]$$}] {};
		\draw \tfirstellip node [label={[xshift=-1.6cm, yshift=-0.4cm]$$}] {};
		\draw \tsecondellip node [label={[xshift=-1.5cm, yshift=-0.9cm]$$}] {};
		\draw \tthirdellip node [label={[xshift=-1.0cm, yshift=-1.8cm]$$}] {};
		\draw \tfourthellip node [label={[xshift=0cm, yshift=1.1cm]$$}] {};
		
		
		
		
		
		\filldraw 
		(9,1.2) circle (4pt) node [left] {};
		\filldraw 
		(9,-1.2) circle (4pt) node [right] {};
		\filldraw 
		(13,1.2) circle (4pt) node [right] {};
		\filldraw 
		(13,-1.2) circle (4pt) node [right] {};
		
		\filldraw 
		(11,2.2) circle (4pt) node [right] {};
		\filldraw 
		(11,-2.2) circle (4pt) node [right] {};
		
		\draw \ysecondellip node [label={[xshift=1cm, yshift=-1cm]$$}] {};
		\draw \yfourthellip node [label={[xshift=2.1cm, yshift=-0.6cm]$$}] {};
		\draw \yssecondellip node [label={[xshift=0cm, yshift=-2cm]$$}] {};
		\draw \ysfourthellip node [label={[xshift=-2.2cm, yshift=0.4cm]$$}] {};
		\draw \ytsecondellip node [label={[xshift=-1.5cm, yshift=-0.9cm]$$}] {};
		\draw \ytfourthellip node [label={[xshift=0cm, yshift=1.1cm]$$}] {};
		
		\end{tikzpicture}
		\caption{Left: A $3$-uniform $2$-cycle of order $12.$ Right: A $3$-uniform $2$-cycle of order $6.$ \label{fig:cycle}}
	\end{center}
\end{figure}
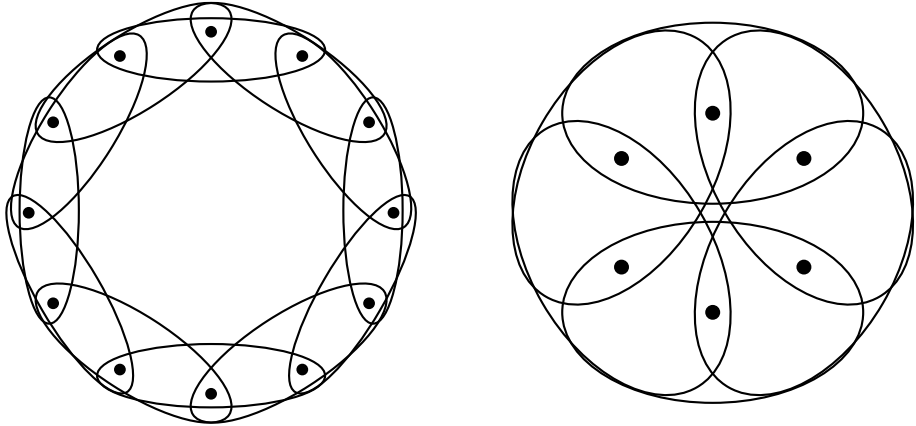
\begin{definition}
	An $r$-uniform ($r\geq 2$) hypergraph $\mbH$ is called an $\ell$-(Hamiltonian) cycle $(0<\ell<r)$, if there exists an ordering $\maV=\left(v_1,v_2,\cdots\,,v_{q(\mbH)}\right)$ of the nodes of 
	$\maV$ such that:
	\begin{itemize}
		\item {E}very hyperedge of $\maH$ consists of $r$ consecutive nodes modulo $q(\mbH)$; 
		\item Any couple of consecutive hyperedges (in an obvious sense) intersects in exactly $\ell$ vertices. 
	\end{itemize}
\end{definition}

In Figure \ref{fig:cycle} we have a $3$-uniform $2$-(hamiltonian) cycle of order $12$ and a $3$-uniform $2$-cycle of order $6$.

\begin{definition}
	A $3$-uniform hypergraph $\mbH=(\maV,\maH)$ is said to be {\em complete $k$-partite}, if there exists a partition of $\maV$ into 
	$k$ independent sets $I_1,...,I_k$ such that $\maH$ contains exactly all subsets of cardinality 3 of the form $\{v_1,v_2,v_3\}$, where $v_1\in I_{i_1}$, 
	$v_2\in  I_{i_2},$ and $v_3\in I_{i_3}$, for three distinct independent sets $I_{i_1}$, $I_{i_2},$ and $I_{i_3}$. 
\end{definition}
\begin{ex}\rm
	Consider a 3-uniform hypergraph $\mbH=(\maV,\maH)$ such that $\maV=\{1,2,3,4,5\}$ and $\maH=\left\lbrace\{1,2,3\},\{1,2,4\},\{1,2,5\},\{1,3,4\},\{1,3,5\},\{2,3,4\},\{2,3,5\}\right\rbrace.$ There exists a partiton of $\maV$ into four independent sets $I_1=\{1\},\;I_2=\{2\},\;I_3=\{3\}$ and $I_4=\{4,5\}.$ Then $\mbH$ is complete 3-uniform $4$-partite hypergraph. 
\end{ex}

{\bf Summaries:} Consider a hypergraph $\mathbb{H}=(\maV,\maH).$ The rank (respectively anti-rank) of $\mathbb{H}$ is the largest (respectively smallest) size of hyperedge. If the rank and anti-rank are equal to $k,$ we then say $\mbH$ is $k$-uniforme. A $k$-uniform hypergraph $\mbH=(\maV,\maH)$ of order $q$ is said,
\begin{itemize}
	\item \textbf{complete $k$-uniform} if all hyperedges of cardinality $k$ appear in $\maH.$
	\item \textbf{$k$-uniform bipartite} if there exists $V_1$ and $V_2$ a partition of $\maV$ such that $|H\cap V_1|=1$ and $|H\cap V_2|=k-1.$ 
	\item  \textbf{$k$-uniform $k$-partite} if there exists a partition $V_1, V_2,\cdots\,, V_k$ of $\maV$ such that for any $H\in\maH$ and any $i\le k$, $\left|H\cap V_i\right|=1$.
\end{itemize}

Throughout this thesis, all considered hypergraphs are connected and simple.
\section{Multigraphs}
In this section consider that $\mbS$ be a multigraph under the form $G=(\maV,\maE).$ For easy reference, let us introduce the basics that we will be used in this thesis. 

\begin{definition}
A {\em multigraph} is a graph that given by a couple $G=(\maV,\maE)$, where $\maV$ is the (finite) set of nodes, and $\maE$ is the set of edges, which is permitted to have {\it multiple edges} (also called parallel edges), else it is called simple edges and also a permitted to have {\it self-loops}, that is, an edge which starts and ends at the same nodes. Elements of the form $\{v\} \in \maE$, are called {\em self-loops}. 
We write $u \v v$ or $v \v u$ for $\{u,v\} \in \maE$, 
and $u \pv v$ (or $v \pv u$) else. 
\end{definition}
As the equation (\ref{eq:defSA}), and specifically for any multigraph $G=(\maV,\maE)$ and any $U\subset\maV$, we denote 
\[
\maE(U) := \{ v \in \maV \; : \; \exists u \in U, \ u
\-- v\},
\]
the set $\maV$ can then be partitioned in $\maV =\maV_1 \cup \maV_2$, where $\maV_1:= \{u\in\maV\;:\;u \v u \}$ and $\maV_2:= \{u\in\maV\;:\; u\pv u\}$, i.e., $\maV_1$ contains all nodes from which a self-loop emanates, if any, and $\maV_2$ is the complement set of $\maV_1$ in $\maV$. Observe that, concerning the classical notion of multigraphs, we assume hereafter that all edges are simple. \\
A multigraph having no self-loop, that is, a couple $G=(\maV,\maE)$ such that $\maV_1=\emptyset$, is simply a {\em graph}. A multigraph is connected if for any $u,v \in \maV$, there exists a subset $\{v_0= u, v_1, v_2,\dots, v_p= v\}\subset \maV$ such that $v_i \v v_{i+1}$, for any $i \in\llbracket 0,p-1 \rrbracket$.\\ 
For any multigraph $G=(\maV=\maV_1\cup\maV_2,\maE)$ and any $U\subset \maV$, the {\em
	subgraph induced by} $U$ in $G$ is the multigraph $(U,\maE\cap U)$. 
%
Observe that $\forall \:I\in \mathbb{I}(G), \;$ we get $\, I\cap\mathcal{V}_1=\emptyset$, i.e., $I\subset\mathcal{V}_2$.

\begin{remark}\rm
	\label{rem:DiffrenceMultigraphAndHypergraph}
	A multigraph is different from a hypergraph, which is a graph in which an hyperedge can connect any number of nodes, not just two.
\end{remark}


Throughout this thesis, all considered multigraphs are connected and without the possibility of multiple edges.
\begin{definition}
	\label{def:restricted}
	Let $G=(\maV,\maE)$ be a multigraph. The {\em maximal subgraph} of $G$ is the graph $\check G=(\maV,\check{\maE})$ 
	obtained by deleting all self-loops in $G$, that is 
	\begin{equation}
	\label{eq:defcheckE}
	\check{\maE}=\maE\setminus\left\lbrace(i,i):i\in\maV_1\right\rbrace.
	\end{equation}
	See an example of a Figure \ref{fig:GgraphGLmultigZAndGTilde}. 
\end{definition}

\begin{definition}
	\label{def:extended}
	Let $G=(\maV_1\cup\maV_2, \maE)$ be a multigraph. 
	The {\em minimal blow-up graph} of $G$ is the graph $\hat{G}=(\hat{\maV},\hat{\maE})$ defined as follows:
	\begin{equation}
	\label{eq:defGtilde}
	\hat{\maV}=\maV\cup\cop{\maV_1}\quad \textrm{and}\quad\hat{\maE}=\check{\maE} \cup\underline{\maE_1},
	\end{equation}
	where $\cop{\maV_1}$ is an independent copy of $\maV_1$, $\check{\maE}$ is defined by (\ref{eq:defcheckE}) and $$\underline{\maE_1}=\{(\underline{i},j):(i,j)\in\maE,\,i\in\maV_1,\,j\in\maV\}.$$
	In other words, $ \hat G$ is obtained from $G$ by duplicating each node having a self-loop by two nodes having the same neighborhood and replacing each self-loop by an edge between 
	the node and its copy. See an example of a Figure \ref{fig:GgraphGLmultigZAndGTilde}. 
	
	The maximal subgraph $\check G$ of $G$ is then called {\em reduced graph} of $\hat G$. \\

    For any set $\maA \subset \maV_1$, we denote by $\cop{\maA}$ the set of all copies of elements of $\maA$, that is 
\[\cop{\maA} = \left\{\underline i : i\in \maA\right\}.\]
\end{definition}
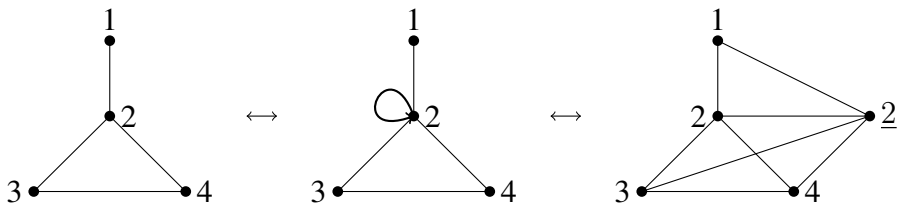
\begin{figure}[htb]
	\begin{center}
		\begin{tikzpicture}

		\fill (0,2) circle (2pt)node[above]{1};
		\fill (0,1) circle (2pt)node[right]{2};
		\fill (-1,0) circle (2pt)node[left]{3};
		\fill (1,0) circle (2pt)node[right]{4};
		
		\draw[-] (0,1) -- (1,0);
		\draw[-] (0,1) -- (-1,0);
		\draw[-] (0,1) -- (0,2);
		\draw[-] (-1,0) -- (1,0);
		
		\draw[<->] (1.8,1) -- (2.2,1);
		
		
		\fill (4,2) circle (2pt)node[above]{1};
		\fill (4,1) circle (2pt)node[right]{2};
		\fill (3,0) circle (2pt)node[left]{3};
		\fill (5,0) circle (2pt)node[right]{4};
		
		\draw[-] (4,1) -- (5,0);
		\draw[-] (4,1) -- (3,0);
		\draw[-] (4,1) -- (4,2);
		\draw[-] (3,0) -- (5,0);

		\draw[<->] (5.8,1) -- (6.2,1);
		
		\draw[-] (8,1) -- (9,0);
		\draw[-] (8,1) -- (7,0);
		\draw[-] (8,1) -- (8,2);
		\draw[-] (7,0) -- (9,0);

		\fill (8,2) circle (2pt)node[above]{1};
		\fill (8,1) circle (2pt)node[left]{2};
		\fill (7,0) circle (2pt)node[left]{3};
		\fill (9,0) circle (2pt)node[right]{4};
		\fill (10,1) circle (2pt)node[right]{$\cop{2}$};

		\draw[-] (10,1) -- (9,0);
		\draw[-] (10,1) -- (7,0);
		\draw[-] (10,1) -- (8,2);
		\draw[-] (10,1) -- (8,1);
		
		\draw[thick,->] (4,1) to [out=110,in=200,distance=10mm] (4,1);
		\end{tikzpicture}
		\caption{Middle: A multigraph $G.$ Left: Its maximal subgraph $\check G.$ Right: Its minimal blow-up graph $\hat{G}.$}
		\label{fig:GgraphGLmultigZAndGTilde}
	\end{center}
\end{figure}

\section{Main useful probabilistic results}

In Chapter \ref{chap2:Definition of general model}, we will define the stochastic matching model and the corresponding Markov chain representation, while in this Section we present two famous techniques (Lyapunov-Foster Theorem and Fluid limits) that help us to find the stability of the model (i.e., its Markov chain is positive recurrent).

\subsection{Lyapunov-Foster Theorem}
This section is taken from (\cite{Bre99}, \S  5.1):\\

The following Theorem provides an ergodicity criterion for countable Markov chain valued in countable state space $E$,
\begin{theorem}
	\label{the:LyapunovFosterTheorem}
	Let the transition matrix $P=\{p_{ij}\}_{i,j\in E}$ on the countable state space $E$ be irreducible and suppose that there exists a function $h: E \longrightarrow \R$ such that $\inf\limits_i h(i) > -\infty$ and
	\begin{align}
	\sum_{k\in E} p_{ik}h(k ) &< \infty \textrm{ for all } i\in F,\\
	\sum_{k\in E} p_{ik}h(k ) &\leq h(i)-\varepsilon \textrm{ for all } i\notin F,
	\end{align}
	for some finite set $F$ and some $\varepsilon > 0.$ Then the corresponding Markov chain is positive recurrent.
\end{theorem}
The stationary distribution criterion of positive recurrence of an irreducible chain requires solving the balance equation, a too-often hopeless enterprise except in a few textbook situations. The above sufficient condition is more tractable, and indeed quite powerful.
\subsection{Classification of random walks in $(\Z^+)^2$}
This section is taken from (\cite{FMM95}, \S  3.3):\\

Consider a discrete time homogeneous irreducible and aperiodic Markov chain $\mathcal{L}=\{\maE_n;\;n\geq 0\}.$ Its state space is the lattice in the positive quarter-plane $(\Z^+)^2=\{(i,j)\;:\;i,j\geq 0,\textrm{ integers}\}$ and it satisfies the recursive equation
$$\maE_{n+1}=[\maE_n+\theta_{n+1}]^+,$$
where the distribution of $\theta_{n+1}$ depends only on the position of $\maE_n$ in the following way (maximal space homogeneity):
$$p\{\theta_{n+1}=(i,j)/\maE_n=(k,l)\}=\left\lbrace\begin{array}{cl}
p_{ij},&\textrm{for }k,l\geq1,\\	
p_{ij}',&\textrm{for }k\geq1,l=0,\\	
p_{ij}'',&\textrm{for }k=0,l\geq1,\\
p_{ij}^0,&\textrm{for }k,l=0,
\end{array}\right.$$
where $p_{ij},\;p_{ij}',\;p_{ij}''$ and $p^0_{ij}$ are probabilities belong to the interval $[0,1].$
Moreover, for the one-step transition probabilities, making the following assumptions:

	\textbf{Condition A} (Lower boundedness)
	$$\left\lbrace\begin{array}{ccll}
		p_{ij}=0,&\textrm{ if }&i<-1 &\textrm{or }j<-1,\\	
		p_{ij}'=0,&\textrm{ if }&i<-1 &\textrm{or }j<0,\\		
		p_{ij}''=0,&\textrm{ if }&i<0 &\textrm{or }j<-1,\\	
	\end{array}\right.$$
	
	\textbf{Condition B} (First moment condition)
	$$\esp{||\theta_{n+1}||/\maE_n=(k,l)}\leq C<\infty,\;\forall(k,l)\in(\Z^+)^2,$$
	where $||z||,\;z\in(\Z^+)^2,$ denotes the euclidean norm and $C$ is an arbitrary but strictly positive number.\\
\textbf{Notation:} Using lower case greek letters $\alpha,\beta,\cdots$ to denote arbitrary points of $(\Z^+)^2,$ and then $p_{\alpha\beta}$ will mean the one-step transition probabilities of the Markov chain $\mathcal{L},\;\alpha>0$ means
$$\alpha_x>0,\:\alpha_y>0,\textrm{ for }\alpha=(\alpha_x,\alpha_y).$$
Also, from the homogeneity conditions, one can write
$$\theta_{n+1}=(\theta_x,\theta_y), \textrm{given that }\maE_n=(x,y).$$
Define the vector 
$$M(\alpha)=(M_x(\alpha),M_y(\alpha))$$
of the one-step mean jumps (drifts) from the point $\alpha.$ setting 
$$\alpha=(\alpha_x,\alpha_y),\;\beta=(\beta_x,\beta_y),$$
we have 
$$\begin{array}{ccc}
M_x(\alpha)&=\sum\limits_{\beta}p_{\alpha\beta}(\beta_x-\alpha_x),\\
M_y(\alpha)&=\sum\limits_{\beta}p_{\alpha\beta}(\beta_y-\alpha_y).\\
\end{array}$$
Condition \textbf{B} ensures the existence of $M(\alpha),$ for all $\alpha\in(\Z^+)^2.$ By the homogeneity condition \textbf{A},  only four drift vectors are different from zero:
$$M(\alpha)=\left\lbrace\begin{array}{cl}
M,&\textrm{for }\alpha_x,\alpha_y>0,\\
M',&\textrm{for }\alpha=(\alpha_x,0),\alpha_x>0;\\
M'',&\textrm{for }\alpha=(0,\alpha_y),\alpha_y>0;\\
M_0,&\textrm{for }\alpha=(0,0).\\
\end{array}\right.$$
\begin{remark}\rm
\begin{enumerate}
	\item[(i)] All our results remain valid if a finite number of transition probabilities are arbitrarily modified.
	\item[(ii)] Given $\maE_n=\alpha,$ the components of $\theta_{n+1}$ might be taken bounded from below not by -1, but by some arbitrary number $-K>-\infty,$ provided that:
	
	First, we keep the maximal homogeneity for the drift vectors $M(\alpha)$ introduced above (i.e., four of them only different);
	
	Secondly, the second moments and the covariance of the one-step jumps inside $(\Z^+)^2,$ i.e., from any point $\alpha>0,$ are kept constant.\\
	 These last fact will emerge more clearly in the course of the study.
\end{enumerate}
\end{remark}
\begin{theorem}
	\label{theo:TheoremFayolMalMain}
	Assume conditions \textbf{A} and \textbf{B} are satisfied.
\begin{enumerate}
\item[(a)] If $M_x<0,\;M_y<0,$ then the Markov chain $\mathcal{L}$ is 
\begin{enumerate}
	\item[(i)] ergodic if 
	$$\left\lbrace\begin{array}{cc}
	M_xM_y'-M_yM_x'<0,\\
	M_yM_x''-M_xM_y''<0;\\
	\end{array}\right.$$
	\item[(ii)] non-ergodic if either\\
	$$M_xM_y'-M_yM_x'\geq 0\textrm{ or }	M_yM_x''-M_xM_y''\geq 0.$$
\end{enumerate}

\item[(b)] If $M_x\geq 0,\;M_y<0,$ then the Markov chain $\mathcal{L}$ is 
\begin{enumerate}
	\item[(i)] ergodic if 
	$$M_xM_y'-M_yM_x'<0;$$

	\item[(ii)] transient if 
	$$M_xM_y'-M_yM_x'> 0.$$
\end{enumerate}
\item[(c)] (Case symmetric to case (b)) If $M_y\geq0,\;M_x<0,$ then the Markov chain $\mathcal{L}$ is 
\begin{enumerate}
	\item[(i)] ergodic if 
	$$M_yM_x''-M_xM_y''<0;$$
	\item[(ii)] transient if 
	$$M_yM_x''-M_xM_y''> 0.$$
\end{enumerate}
\item[(d)] If $M_x\geq 0,\;M_y\geq 0,\;M_x+M_y>0,$ then the Markov chain is transient.
\end{enumerate}
\end{theorem}

	Consider the following real functions on $(\Z^+)^2:$
	$$\left\lbrace\begin{array}{ccl}
	Q(x,y)&=&ux^2+vy^2+\omega xy,\\
	f(x,y)&=&Q^{1/2}(x,y),\\
	\Delta f(x,y)&=&Q^{1/2}(x+\theta_x,y+\theta_y)-Q^{1/2}(x,y),
		\end{array}\right.$$
		where $(x,y)\in(\Z^+)^2$ and $u,v,\omega$ are unspecified constants, to be properly chosen later, but subject to the constraints $u,v>0,\;4uv>\omega^2,$ so that the quadratic form $Q$ is positive definite. 
		\begin{lemma}
			\label{lem:UsInProofTheorem3.3.1}
			We have 
			$$	\esp{\Delta f(x,y)}=\displaystyle\frac{x[2u\esp{\theta_x}+\omega\esp{\theta_y}]+y[\omega\esp{\theta_x}+2v\esp{\theta_y}]}{2f(x,y)}+o(1),$$
			where $o(1)\rightarrow 0$ as $(x^2+y^2)\rightarrow \infty.$
		\end{lemma}
		\begin{proof}[Proof of Theorem \ref{theo:TheoremFayolMalMain} :]
		First, we shall prove ergodicity in the case (a(i)). Lemma \ref{lem:UsInProofTheorem3.3.1} shows that, if there exists $u,v>0$ and $\omega^2<4uv,$ such that, for somme $\varepsilon_2>0$ and all $(x,y)\in(\Z^+)^2\backslash A,$ where $A$ in a finite set, 
		\begin{equation}
		\label{eq:uvwRelation}
		\left\lbrace\begin{array}{ccc}
		2u\esp{\theta_x}+\omega\esp{\theta_y}&<-\varepsilon_2,\\
		\omega\esp{\theta_x}+2v\esp{\theta_y}&<-\varepsilon_2,
		\end{array}\right.
		\end{equation}
		then, for some $\mathcal{D},$ there exists $\varepsilon>0$ such that for all $(x,y)$ with $x^2+y^2>\mathcal{D}^2,$ we have 
		\begin{equation}
		\label{eq:uvwEspRelation}
		\esp{\Delta f(x,y)}<-\varepsilon.
		\end{equation}
		Therefore, when (\ref{eq:uvwEspRelation}) holds, the random walk is ergodic, by using Lyapunov-Foster Theorem. Let us rewirte inequalities (\ref{eq:uvwEspRelation}) in terms of the drifts on the axes and in the internal part of $(\Z^+)^2,$
		\begin{equation}
			\label{eq:uvwMRelation}
			\left\lbrace\begin{array}{ccc}
			2uM_x+\omega M_y&<-\varepsilon_2,\\
			2vM_y+\omega M_x&<-\varepsilon_2,\\
			2uM_x'+\omega M_y'&<-\varepsilon_2,\\
			2vM_y''+\omega M_x''&<-\varepsilon_2.\\
			\end{array}\right.
		\end{equation}
		It is easy to show that, if 
		\begin{equation}
		\label{eq:uvwConlusion}
		\left\lbrace\begin{array}{lc}
		M_x<0,\\
		M_y<0,\\
		M_xM_y'-M_yM_x'<0,\\
		M_yM_x''-M_sM_y''<0,
		\end{array}\right.
		\end{equation}
		then there exists $u=-M_x/2,\;v=-M_y/2>0$ and $\max\left(M_x;M_y\right)<\omega<\min\left(\frac{M_yM'_x}{M'_y};\frac{M_xM''_y}{M''_x}\right)$ then $\omega^2<4uv,$ such that $(\ref{eq:uvwMRelation})$ is satisfied for some $\varepsilon_2>0,$ thus poving case a(i). 
		The cases (b(i)) and (c(i)) are analogus to (a(
		i)). Indeed, if 
		$$\left\lbrace\begin{array}{lc}
		M_x\geq 0,\\
		M_y<0,\\
		M_xM_y'-M_yM_x'<0,
		\end{array}\right.\textrm{ or }\qquad\left\lbrace\begin{array}{lc}
		M_x<0,\\
		M_y\geq 0,\\
		M_yM_x''-M_xM_y''<0,
		\end{array}\right.$$ 
		we show that there exists $u,v>0$ and $\omega^2<4uv,$ such that (\ref{eq:uvwMRelation}) holds, so that the chain is ergodic in both cases.
		
		Now the prove of non-ergodicity in (a(ii)). Assume that
		$$\left\lbrace\begin{array}{lc}
		M_x< 0,\\
		M_y<0,\\
		M_xM_y'-M_yM_x\geq 0.
		\end{array}\right.$$
		There exists a linear function $f(x,y)=ax+by,$ such that, for all $\alpha=(x,y)$ with $ax+by\geq C,$ we have 
		$$f(\alpha+M(\alpha))\geq f(\alpha)+\varepsilon,\textrm{ for some }C, \varepsilon\geq 0,$$
		and the non-ergodicity is immediately deduced.  The remain proof of the transience in (b(ii)), (c(ii)) and (d) is more difficult and can be found in (\cite{FMM95}, \S  3.3).
\end{proof}
\subsection{Rescaled Markov processes and Fluid Limits}
\label{subsec:Rescaled Markov Processes and Fluid Limits}
This section is taken from \cite[Chapters 5 and 9]{R13} :

In this section, limit results consist in speeding up time and scaling appropriately the process itself with some parameter. The behavior of such rescaled stochastic processes is analyzed when the scaling parameter goes to infinity. In the limit one gets a sort of caricature of the initial stochastic process which is defined as a \textit{fluid limit}. These ideas of rescaling stochastic processes has emerged  in the analysis of stochastic networks, to study their ergodicity properties in particular, see \cite{RS92}.
In statistical physics, these methods are quite classical, see \cite{C91}.


In the following, $(X(x,t))$ denotes an irreducible `RCLL' continuous-time Markov chain on a countable state space $\maU$ starting from $x\in\maU$, i.e., such that $X(x,0) = x\in\maU$. 
Denotes by $\maN_{\varepsilon}(\omega,dx), \omega\in\Omega$, a Poisson point process on $\R$ with parameter $\varepsilon\in\R^+$, all Poisson processes used are assume to be a priori independent. The topology on the space of probability distributions induced by the Skorokhod topology on the space of `RCLL' functions $\mathbb{D}\left([0, T], \R^d\right)$ is used.

\subsubsection{\large Rescaled Markov Processes}
Throughout this subsection, assumed that the state space $\maU$ can be embedded in a subset of some normed space $\R^d,$ $\parallel.\parallel$ denotes the associated norm.

\begin{definition}
For $x\in\maU,\;\left(\bar{X}(x,t)\right)$ denotes the process $\left(X(x,t)\right)$ renormalized so that for $t\geq0,$ 
$$\bar{X}(x,t)=\displaystyle\frac{1}{\parallel x\parallel}X(x,\parallel x\parallel t).$$ 
\end{definition}


Only continuous-time Markov processes are considered in this section. In discrete-time, if $(X_n)_{n\in\N}$ is a Markov chain, the corresponding rescaled process can be also defined by
$$\bar{X}(x,t)=\displaystyle\frac{X_{\lfloor\parallel x\parallel t\rfloor}}{\parallel x\parallel},$$
where $X_0 = x\in\maU$ and $t\in\R^+.$ 

\subsubsection{\large Fluid limits}

Fluid limits are the results of a scaling of the number of customers of an M/M/1 queue. The scaling considered here consists in speeding up the time scale with the size of its initial state and in scaling the process with the same
quantity. The procedure suppresses some random fluctuations around what appears to be the main trajectory of the process. For the M/M/1 queue the behavior of the rescaled process is very simple.

Initially, there are $x_N$ customers in the queue and the sequence $(x_N)$ is such that
$$\lim_{N\rightarrow +\infty}\displaystyle\frac{x_N}{N}=x\in\R^+.$$
Suppose that $L(t)$ is the number of customers of the queue at time $t\geq 0.$ Assumed that $L(t)$ be a stochastic process. The renormalized process $\left(\bar{L}_N(t)\right)$ is defined by,
$$\bar{L}_N(t)=\displaystyle\frac{L(Nt)}{N},$$
notice that $\bar{L}_N$ lives on a very rapid time $t\rightarrow Nt,$ arrivals and services are sped up by a factor $N.$ The scaling by $1/N$ compensates the acceleration of time.\\

\begin{figure}[htp!]
	\centering
	\includegraphics[width=.45\linewidth]{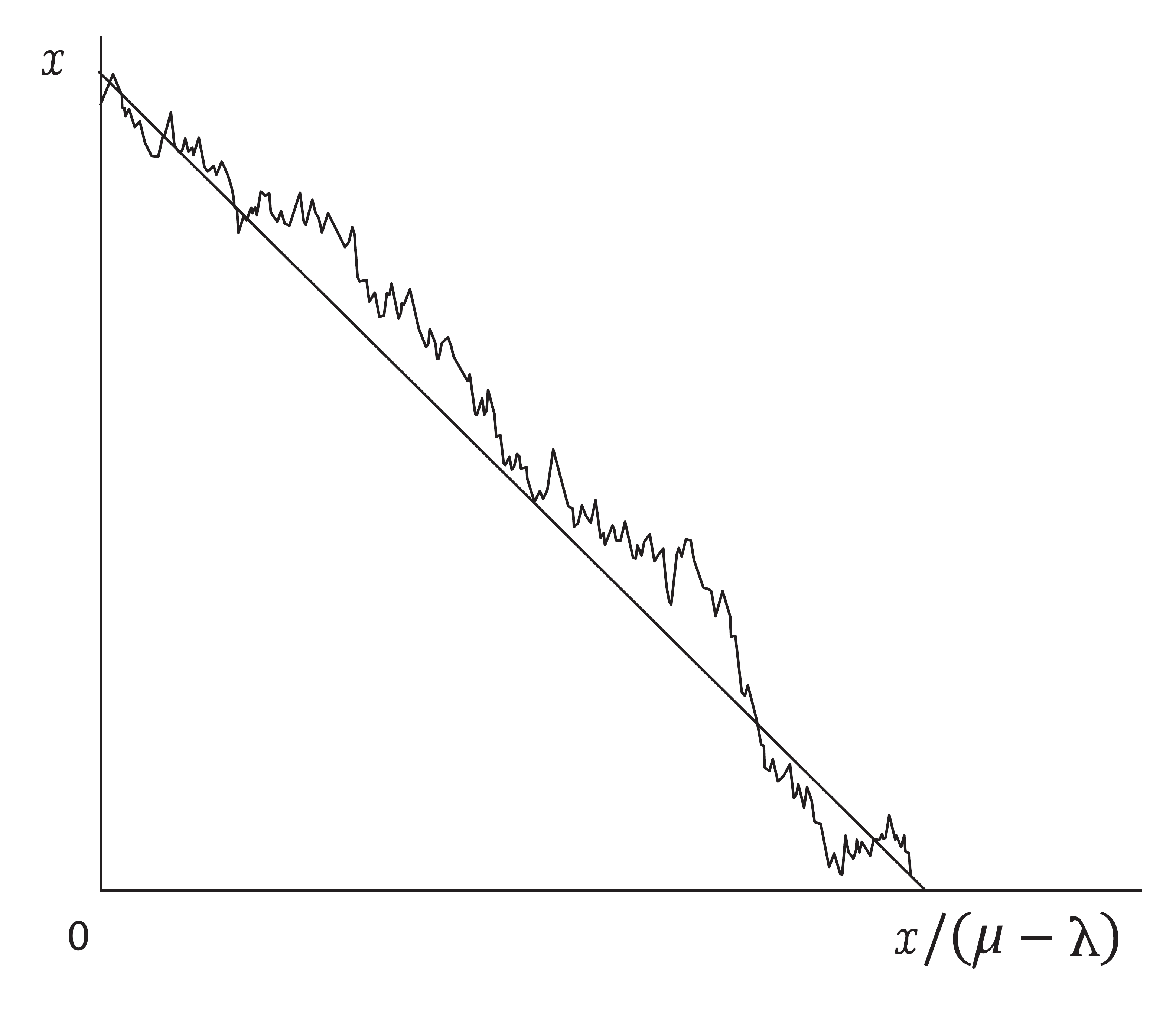}
	\caption{The renormalized process $(\bar{L}(t))$ and the fluctuations of $(\bar{L}_N(t))$.}
	\label{fig:renormalization}
\end{figure}

Consider the scaling depicted on Figure \ref{fig:renormalization}, the stochastic process $(L(t))$ is reduced to a deterministic drift $(\lambda-\mu)t.$ 

If $\lambda<\mu$, once the renormalized process hits 0, it remains at 0. This property is, in some sense, characteristic of ergodic Markov processes. Intuitively, it can be argued as follows: Once the process $(L(t))$ hits 0, with the coupling argument of \cite[Proposition 5.8]{R13}, it is approximately at equilibrium and it is mainly living in bounded neighborhood of 0. The time scale, linear with $N$, does not allow the visit of large values since an exponential time scale is necessary for this purpose, of the order $(\mu/\lambda)^{Ny}$ to reach the value $Ny$. The scaling factor $1/N$ suppresses these small variations. This explains that the renormalized process is stuck at 0.

	

\begin{definition}
	\label{def:fluidLimitAndQueues}
	A fluid limit associated with the Markov process $(X(t))$ is a stochastic process which is one of the limits of the process
 	$$\left(\bar{X}(x,t)\right)=\left(\displaystyle\frac{X(x,\parallel x\parallel t)}{\parallel x\parallel}\right)$$
	when $\parallel x\parallel$ goes to infinity.
\end{definition}
Strictly speaking, if $Q_x$ is the distribution of $\left(\bar{X}(x,t)\right)$ on the space of `RCLL' functions $\mathbb{D}\left(\R^+,\maU\right),$ a fluid limit is a probability distribution $\tilde{Q}$ on $\mathbb{D}\left(\R^+,\R^d\right)$ such that 
$$\tilde{Q}=\lim_n{Q_x}_n$$
for some sequence $(x_n)$ of $\maU$ whose norm converges to infinity. By choosing an appropriate probability space, it can be represented as a `RCLL' stochastic process $(W(t))$ whose distribution is $\tilde{Q}$. A fluid limit is thus an asymptotic description of sample paths of a Markov process with a large initial state.

\begin{ex}
	\label{ex:5.16Robert}\em
	(Fluid limits of the M/M/1 queue).  The arrival rate is $\lambda$ and the service rate $\mu,\;L(t)$ is the number of customers of the queue at time $t\geq 0.$ The renormalized process $(\bar{L}_N(t))$ converges to a deterministic function, piecewise linear, i.e., if $\mathbb{P}$-almost surely, the convergence of processes associated with the uniform norm on compact sets holds
	$$\lim_{N\rightarrow +\infty}\left(\bar{L}_N(t)\right)=\left(x+(\lambda-\mu)t\right)^+.$$
	
	In addition, for $\varepsilon,\delta>0,$ there exists $N_0\in\N$ such that, if $N\geq N_0$ then
	\begin{equation}
	\label{eq:ProbaEnx}
	\inf_{|x/N-x|<\delta/2}\mathbb{P}_x\left(\sup_{0\leq x \leq t}|\bar{L}_N(s)-(x+(\lambda-\mu)s)^+|<\delta\right)\geq 1-\varepsilon.
	\end{equation}
The function $\left(x+(\lambda-\mu)t\right)^+$ is therefore the unique fluid limit of this Markov process. If we assume that $\lambda<\mu,$ this implies in particular that the Markov process $L(t)$ is ergodic.

\end{ex}


\pagestyle{empty}
\chapter{Stochastic matching model}\label{chap2:Definition of general model}
\pagestyle{fancy}

Stochastic matching techniques aim to study the dynamic systems resulting from group matching of individuals or agents on a microscopic scale. Different applications of this technique, in economics or finance, provide tools for studying over-the-counter contracts, labor and housing markets, co-operative sites, peer-to-peer networks, and so on. These applications stimulate the introduction of dynamic probabilistic modeling, at a discrete or continuous-time representing the evolution of stochastic matching between agents. The discrete-time models, initially introduced for healthcare systems (blood banks, organ allocations) are also adapted with continuous-time to specific applications. 
For this class of models, it is possible to describe the restrictive behavior of the size of a subset that meets such or that criteria, as well as market price formations, approximating them through a reduced system of non-linear differential equations.\\

The objective of this chapter is to describe the stochastic matching model with the approach of the Markov chain in discrete and continuous-time. In Section \ref{sec:GenModel} we present the stochastic matching model on matching  structures. In Section \ref{sec:pol} we present further information for matching policies. In Section \ref{sec:Markov} we formulate a Markov representation of the general model. In Section \ref{sec:defstab} we define the stability and the instability of the model. 
Finally, in Section \ref{sec:MatchingQueuContiounsTime} we present the matching queue and the stability of the continuous model which we will be used in Chapter \ref{chap6:FluidLimits}.

\section{Stochastic model on matching structures}
\label{sec:GenModel}
A (discrete-time, matching structure) stochastic matching model is specified by a triple $(\mathbb{S},\Phi,\mu)$, such that:
\begin{itemize}
	\item $\mathbb S=(\maV,\maS)$ is a  \textit{connected matching structure} which can be a non-bipartite graph, hypergraph or multigraph;
	\item $\Phi$ is a \textit{matching policy}, which defines the new buffer-content given the pair formed by the old buffer-content and the arriving item, precisely defined in section \ref{sec:pol} below; 
	\item $\mu$ is an element of $\mathscr M(\maV),$ the \textit{common law} of the independent and identically distributed $(i.i.d.)$ classes of the arriving items.
\end{itemize}

\subsection{The models}
\label{sec:TheModel}

\noindent The \textit{matching model} $(\mathbb S,\Phi,\mu)$ is then defined as follows. At each time point $n\in\N^+$, 
\begin{enumerate}
	\item An item enters the system. Its class $V_n$ is drawn from the measure $\mu$ on $\maV$, independently of everything else. 
	\big(Thus the sequence of classes of incoming items $\suite{V_n}$ is $i.i.d.$ of common distribution $\mu$\big).
	\item The incoming item then faces the following alternatives: 
	\begin{itemize}
		\item[(i)] If there exists in the buffer, at least one set of items whose respective set of classes forms, together with $V_n$, an edge of $\maS$, then it is the role of the matching policy $\Phi$ to select one of these sets of classes, say 
		$\{i_1,...,i_m\},\;m\geq 1$ \big(in the cases of graph and multigraph $(m=1),$ and for hypergraph $(m\geq 1)$\big). Then the $m+1$ items of respective classes $i_1,...,i_m,V_n$ are matched together and leave the system right away. Denoting  $S_j:=\{i_1,...,i_m,V_n\} \in \maS,$ for $j\in\ll 1,|\maS|\rr$ \big(in the case of multigraph for the self-loop edges, we have the possibility of having $i_1$ and $V_n$ coincides\big), we then say that $V_n$ {\em completes} a {\em matching} of type $S_j$ at time $n$, and we denote $S(n)=S_j$, the matching performed at $n$. 
		\item[(ii)] {E}lse, the item is stored in the buffer of the system, waiting for a future match, and we write $S(n)=\emptyset$. 
	\end{itemize}
\end{enumerate}

\begin{figure}[htp]
	\begin{center}
		\begin{tikzpicture}
		\draw[-] (1,2) -- (13,2);
		\fill (1,2) circle (2pt) node[below] {\small{2}};
		\node[draw] at (2,3.3) {234};
		\draw[-, thick] (1,2) -- (2,3);
		\draw[-, thick] (2,2) -- (2,3);
		\draw[-, thick] (3,2) -- (2,3);
		\fill (2,2) circle (2pt) node[below] {\small{3}};
		\fill (3,2) circle (2pt) node[below] {\small{4}};
		\fill (4,2) circle (2pt) node[below] {\small{1}};
		\node[draw] at (5,3.3) {123};
		\draw[-, thick] (4,2) -- (5,3);
		\draw[-, thick] (6,2) -- (5,3);
		\draw[-, thick] (7,2) -- (5,3);
		\fill (5,2) circle (2pt) node[below] {\small{1}};
		\fill (6,2) circle (2pt) node[below] {\small{2}};
		\fill (7,2) circle (2pt) node[below] {\small{3}};
		\fill (8,2) circle (2pt) node[below] {\small{3}};
		\fill (9,2) circle (2pt) node[below] {\small{4}};
		\fill (10,2) circle (2pt) node[below] {\small{2}};
		\node[draw] at (7,3.3) {134};
		\draw[-, thick] (5,2) -- (7,3);
		\draw[-, thick] (8,2) -- (7,3);
		\draw[-, thick] (9,2) -- (7,3);
		\fill (11,2) circle (2pt) node[below] {\small{2}};
		\node[draw] at (12,3.26) {...};
		\node[draw] at (13,3.26) {...};
		\draw[-, thick] (10,2) -- (12,3.09);
		\draw[-, thick] (11,2) -- (13,3.09);
		\end{tikzpicture}
		\caption{The matching model in action, on the matching hypergraph of Figure \ref{fig:completeHyper} (left).}
		\label{fig:Dyn}
	\end{center}
\end{figure}
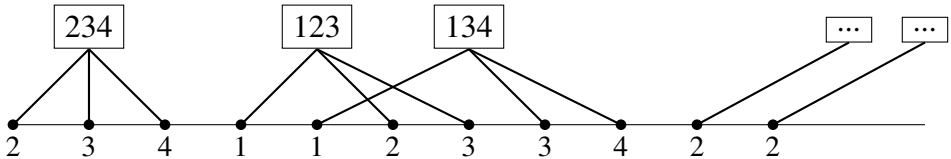
\begin{ex}
	\rm 
	Consider the complete $3$-uniform hypergraph $\mathbb{H}=(\maV,\maH)$ of size $4,$  with $\maV=\{1,2,3,4\}$ and $\maH=\left\lbrace\{1,2,3\},\{1,2,4\},\{1,3,4\},\{2,3,4\}\right\rbrace$. The dynamic matchings of the realization $\suite{V_n(\omega)}=2,3,4,1,1,2,3,3,4,2,2,....$ is represented in Figure \ref{fig:Dyn}. 

	\subsection{State spaces}
	\label{subsec:state}
	We reproduce here the state description of the model introduced in \cite{MaiMoy16} for the stochastic model on general graphs, and then 
	\cite{MBM18} for the same model under the matching policy {\sc fcfm}. 
	Fix a connected matching structure $\mathbb{S}=(\maV,\maS)$, in the sense specified above, until the end of this section. 
	Fix an integer $n_0 \ge 1$, a realization $v_1,\dots,v_{n_0}$ of $V_1,\dots,V_{n_0}$, and define the word $z= v_1\dots v_{n_0} \in \maV^*$.  
	Then, for any matching policy $\Phi$, there exists a unique {\em matching} of the word $z$, that is, a matching structure having a set of nodes 
	$\left\{v_1,\dots,v_{n_0}\right\}$ and whose edges represent the matches performed in the system until time $n_0$, if the successive arrivals are given by $z$.   
	This matching is denoted by $M^\Phi(z)$. 
	The state of the system is then defined as the word $W^\Phi(z)\in \maV^*$, whose letters are the classes of the unmatched items at time $n_0$, 
	i.e., the isolated vertices in the matching $M^{\Phi}(z)$, in their order of arrivals. The word $W^\Phi(z)$ is called {\bf queue detail} at time $n_0$. 
	Then, any admissible queue detail belongs to the set 
	\begin{equation}
	\mathbb W = \Bigl\{ w\in \maV^*\; : \; \forall  i\neq j\neq\cdots \neq k \; \text{s.t.} \; \{i,j,\cdots,k\} \in \maS, \; |w|_i|w|_j\cdots|w|_k=0 \Bigr\}.\label{eq:defmbW}
	\end{equation} 
	
	As will be seen below, depending on the service discipline $\Phi$, we can also restrict the available information on the state of the system at time $n_0$, to a vector only keeping track of 
	the number of items of the various classes remaining unmatched at $n_0$, that is, of the number of occurrences of the various letters of the alphabet $\maV$ in the word $W^\Phi(z)$.    
	This restricted state thus equals the commutative image of $W^{\Phi}(z)$ and is called {\bf class detail} of the system. It takes values in the set 
	\begin{align}
	\label{eq-css}
	\mathbb X &= \Bigl\{x \in \N^{|\maV|}\,:\,\forall i\neq j\neq\cdots\neq k \mbox{ s.t. }\{i,j,\cdots,k\}\in\maS,\; x(i)x(j)\cdots x(k)=0\Bigl\}\nonumber\\
	&=\Bigl\{\left[w\right]\;:\,w \in \mathbb W\Bigl\}.
	\end{align} 
	\begin{remark}
		\label{rq:restrictionFor Multig}\rm
		Denote, in case that $\mathbb{S}=(\maV,\maS)$ be a multigraphs whose vertex divided into two subsets $\maV=\maV_1\cup\maV_2$ such that $\maV_1\neq\emptyset,$ we must add the following restrictions:
		
		\begin{itemize}
			\item In equation (\ref{eq:defmbW}) we have $\forall i\in\maV_1,\;|w|_i\leq 1;$
			\item In equation (\ref{eq-css}) we have $\forall i\in\maV_1,\;x(i)\leq 1.$ 
		\end{itemize} 
		
		Let us order the elements of $\maV$ and identify them with $\llbracket 1,|\maV| \rrbracket$, in a way that the $|\maV_1|$ first elements are those of $\maV_1$ and the remaining $|\maV_2|$ elements are those of $\maV_2$. We index the elements of $\mbX$ accordingly: namely, for any $x\in\mbX$ and any $i\in\llbracket 1,|\maV| \rrbracket$, $x(i)$ is the queue size of node $i$. With these conventions, for instance, 	the coordinate $x(|\maV_1|+3)$ corresponds to the queue size of the third node of $\maV_2$. 
	\end{remark}
	
	\section{Matching policies}
	\label{sec:pol}
In this section, we present and formally define the set of matching policies that can be taken into consideration. 
	
	\begin{definition}
		A matching policy $\Phi$ is said {\em admissible} if the choice of the match of an incoming item depends 
		{solely} on the queue detail upon the arrival (and possibly on an independent uniform draw, in case of a tie). 
	\end{definition} 
	An admissible matching policy can be formally characterized by an action $\odot_{\Phi}$ of $\maV$ on $\mathbb W$, defined as follows: 
	if $w$ is the queue detail at a given time and the input is augmented by the arrival of $v \in \maV$ at that time, then the new queue detail $w'$ and $w$ satisfies the relation 
	\begin{equation}
	\label{eq:defodot}
	w'= w\odot_{\Phi} v.
	\end{equation}
	Notice that the action $\odot_\Phi$ is possibly random.  
	
	\subsubsection{\large Matching policies that depend on the arrival times}  
	In `First Come, First Matched' ({\sc fcfm}), the oldest item in line is chosen, so the map  $\odot_{\textsc{fcfm}}$ is 
	given, for all $w \in \mathbb W$ and all $v\in\mathcal{V}$, by 
	$$
	w \odot_{\textsc{fcfm}} v =
	\left \{
	\begin{array}{lll}
	wv & \textrm{if for any }S\in\maS(v)\textrm{ there exists }i\neq v\in S\textrm{ s.t }|w|_{i} = 0,\\
	w_{\left [\Phi(w,v)\right]} & \textrm{else, }\textrm{where }\Phi(w,v) =\{i,j,\cdots,k\}\textrm{ such that}\\
	&\qquad i=\min \{\ell\in[\![1;|w|]\!] : \textrm{ there exists }S\in\maS(v)\textrm{ s.t }w_{\ell}\in S\},
	\end{array}
	\right.
	$$
	where ties are broken from the minimum of the next term $j$ as $i$ whenever the above is non-unique.\\
	
	In `Last Come, First Matched' ({\sc lcfm}), the updating map $\odot_{\textsc{lcfm}}$ is analog to $\odot_{\textsc{fcfm}}$, for 
	$\Phi(w,v) =\{i,j,\cdots,k\}\textrm{ such that }i=\max \{\ell\in[\![1;|w|]\!] :\textrm{ there exists }S\in\maS(v)\textrm{ s.t }w_{\ell}\in S\}.$ 
	
\end{ex}

\subsubsection{\large Class-admissible matching policies}
A matching policy $\Phi$ is said to be {\em class-admissible} if it can be implemented upon the sole knowledge of 
the class detail of the system. Let us define, for any $x\in \mathbb X,$ and any $v \in \maV$,
\begin{equation*}
\maP(x,v) =\Bigl\{S\in\maS(v)\,:\;\forall i\in S\backslash\{v\},\;x\left(i\right) > 0\Bigl\},\label{eq:setP2}
\end{equation*}
the set of classes of available compatible items with the entering class $v$-item, if the class detail of the system is given by $x$. 
Then, a class-admissible policy $\Phi$ is fully characterized by a mapping $p_\Phi$, such that $p_\Phi(x,v)$ denotes the class of the match chosen 
by the entering $v$-item under $\Phi$, in a system of class detail $x$, such that $\mathcal P(x,v)$ is non-empty. 
Then, the arrival of $v$ entails the following action on
the class detail:
\begin{equation}
\label{eq:defccc}
x \ccc_{\Phi} v = \left \{
\begin{array}{ll}
x+\gre_v &\mbox{ if }\mathcal P(x,v)=\emptyset,\\
x-\gre_{p_\Phi(x,v)}&\mbox{ else}. 
\end{array}
\right .
\end{equation}
\begin{remark}
	\label{rem:equiv}
	\rm
	As is easily seen, to any class-admissible policy $\Phi$ corresponds  an admissible policy, if one makes precise the rule of choice of a match for the 
	incoming items {\em within} the class that is chosen by $\Phi$, in the case where more than one item of that class is present in the system. 
	In this thesis, we always assume that within classes, the item chosen is always the {\em oldest} in the line, i.e., we always apply an \textsc{fcfm} policy {\em within classes}.  
	Under this convention, any class-admissible policy $\Phi$ is admissible, that is, the mapping $\ccc_\Phi$ from 
	$\mathbb X\times \maV$ to $\mathbb X$ can be detailed into a map $\odot_{\Phi}$ from 
	$\mathbb W \times \maV$ to $\mathbb W$, as in (\ref{eq:defodot}), that is, such that for any queue detail $w$ and any $v$,
	\[\left[w\odot_\Phi v\right] = \left[w\right]\ccc_\Phi v.\]    
\end{remark}


\medskip

\paragraph{{\em Fixed priority policies}} 
In the context of fixed priorities, each vertex $i \in \maV$ is assigned a full ordering of the edges and chooses to be matched with the first matchable edge following this order. Formally, to each node $i$ is associated a permutation $\sigma_i$ of the index set $\llbracket 1,d(i) \rrbracket$, and if we denote $\maS(i)=\left\{S_{i_1},S_{i_2},...,S_{i_{d(i)}}\right\}$, 
then at any time $n$, 
\begin{equation}
\label{eq:deffixedpriority}
p_{\Phi}(x,v)= S_{i_{\sigma_i(j)}},\mbox{ where }j=\mbox{min\,}\left\{k\in \llbracket 1,d(i) \rrbracket\,:\,S_{i_{\sigma_i(k)}}\in\mathcal{P}(x,v)>0\right\}.
\end{equation}

\paragraph{{\em Random policies}} 
For this matching policy, the priority order defined above is not fixed and is drawn uniformly at random upon each arrival, i.e., for any $v$, $p_{\Phi}(x,v)$ is defined as in equation (\ref{eq:deffixedpriority}), for a permutation $\sigma_i$ that is drawn, independently of everything else, uniformly at random among all permutations of $\llbracket 1,d(i) \rrbracket$.


\medskip

\paragraph{{\em Max-Weight policies}}
The Max-Weight policies are an important class of class-admissible policies, in which matches are based upon the queue length and 
a fixed reward that is associated with each match. 
Formally, for any $S\in \maS$ , we let $w_{S}$ be the reward associated to the match of the items of $S$ together, and fix a real parameter $\beta$. 

Then, in a system of class detail $x$, the match of the incoming $v$-item is given by 
\begin{equation*}
p_{\Phi}(x,v)=\mbox{argmax}\left\{\beta x(S) + w_{S}\,:\,S\in \maP(x,v)\right\},
\end{equation*}
where ties are broken uniformly at random whenever the above is non-unique. In other words, $S$ maximizes a linear combination of the queues-size and the rewards.  Several particular cases are to be mentioned:
\begin{enumerate}
	\item[(i)] If $\beta>0$ and the rewards are constant (i.e., $w_{S} = w_{S'}$, for any $S,\,S'\in\maS$), then the matching policy is `Match the Longest' ({\sc ml}), i.e., the incoming $v$-item is matched upon the arrival with the items of the compatible classes having the longest queues size (ties being broken uniformly at random).
	\item[(ii)] If $\beta<0$ and the rewards are constant, then the matching policy is `Match the Shortest' ({\sc ms}), i.e., the incoming $v$-item is matched upon the arrival with the items of the compatible classes having the shortest queues size (ties being broken uniformly at random).
	\item[(iii)] If $\beta=0$ and $w_{S}\ne w_{S'}$ for any $i\in \maV$ and any $S\ne S' \in \maS(i)$ (implying that there is a strict ordering of rewards for all possible matches of any given class), then the matching policy is of a priority type, defined above. 
\end{enumerate}

\section{Primary Markov representations}
\label{sec:Markov}
The Markov representations of the model are similar to general matching models on graphs. 
Denote, for all $w\in\mathbb W$ and all $n\ge 1$, by $W^{\{w\}}_n$, the buffer-content at time $n$ 
(i.e., just after the arrival of the item $V_n$) if the buffer-content at time 0 was set to $w$. In other words,
\[\left\{\begin{array}{ll}
W^{\{w\}}_0 &=0,\\
W^{\{w\}}_n &= W^\Phi\left(wV_1\dots V_{n}\right),\quad n\in \N_+.\end{array}\right.\]
It readily follows from (\ref{eq:defodot}) that the buffer-content sequence $\suite{W^{\{w\}}_n}$ is a Markov chain.
Indeed, for any $w\in\mathbb W$, we have 
\[W^{\{w\}}_{n+1} =W^{\{w\}}_n \odot_\Phi V_{n+1},\,\forall n\in\N.\]

Secondly, we deduce from (\ref{eq:defccc}) that for any class-admissible matching policy $\Phi$ (e.g., $\Phi=\textsc{u}, \textsc{ml}$ or $\textsc{ms}$) 
and any initial condition as above, the $\mathbb X$-valued sequence $\suite{X_n}$ of class details also is a Markov chain, as for any initial condition $x \in \mathbb X$, 
we get 
\[X^{\{x\}}_{n+1} =X^{\{x\}}_n \ccc_\Phi V_{n+1},\,\forall n\in\N.\]
For a fixed initial condition and all $n\in\N$, we denote, for all $B \subset \maV$, by $W_n(B)$, the number of items in the line of classes in $B$ at time $n$, and 
by $|W_n|$, the total number of items in the system at time $n$. In other words, 
\[\left\{\begin{array}{ll}
W_n(B) &= \displaystyle\sum_{i\in B} W_n(i),\\
|W_n|  &= W_n(\maV) = \displaystyle\sum_{i\in\maV} W_n(i). 
\end{array}\right.\]




\section{Stability of the matching model}
\label{sec:defstab}
We say that the matching model $(\mathbb S,\Phi,\mu)$ is stable if the Markov chain $\suite{W_n}$ \big(and thereby $\suite{X_n}$\big) is positive recurrent. \\
Consider a matching structure $\mathbb S=(\maV,\maS)$ and a matching policy $\Phi$, we define the \textbf{stability region} associated to $\mathbb S$ and $\Phi$ as the set of probability measures on $\maV$ rendering the model $(\mathbb S,\Phi,\mu)$ stable, i.e., 
\[\textsc{Stab}(\mathbb S,\Phi)=\left\{\mu \in\mathscr M(\maV):\suite{W_n} \mbox{ is positive recurrent}\right\}.\]

\begin{remark}
	\label{rq:MultigraphSelfLoop}
	\rm
	If the matching structure $\mathbb{S}=(\maV,\maS)$ is a multigraph such that $\maV = \maV_1$, i.e., all nodes of the multigraph have a self-loop, we say, for obvious reasons, that the considered matching models are {\em finite}. Then any matching model on $\mathbb{S}$ is necessarily stable, that is, for any admissible $\Phi$ we have that
	\[\textsc{Stab}(\mathbb{S},\Phi)=\mathscr M(\maV).\]
	Indeed, the Markov chain $\suite{W_n}$ is 
	irreducible on the finite state space $\mathbb W$, containing only words having a size less or equal to the cardinality of the largest independent set of $\mathbb{S}$. 
\end{remark}
\begin{definition}
	A connected matching structure $\mathbb{S}$ is said to be,	\begin{itemize}
		\item \textbf{stabilizable} if $\textsc{Stab}(\mathbb S,\Phi)$ is non-empty for some matching policy $\Phi$,
		\item \textbf{non-stabilizable} if $\textsc{Stab}(\mathbb S,\Phi)$ is empty.
	\end{itemize}
	\label{def:StabilisableOrNot}
\end{definition}

\section{Continuous-Time Markov chain processes}
\label{sec:MatchingQueuContiounsTime}
Now, we present a stochastic matching model in continuous-time where each class of items arrive at the system according to an independent Poisson process of intensity $\lambda>0$.

A throughout presentation of this section can be found e.g. in \cite{MoyPer17}.
\subsection{Matching Queues}
\label{sec:MatchQueAndStabConTime}
The matching queue associated with a matching structure $\mbS=(\maV,\maS)$, an arrival-rate vector $\lambda := (\lambda_1, \cdots , \lambda_{|\maV|})$ and the matching policy $\Phi$, is defined as follows:
\begin{itemize}
\item Each node of $\maV$ is associated with a class of items;
\item Items of each class $i\in\maV$ arrive to the system according to an independent Poisson process $N_i$ of intensity $\lambda_i> 0 $;
\item  A class-$i$ items can be matched with class-$j$, $\cdots,$  class-$k$ items if and only if there is an edge $S=\{i, j,\cdots,k \}\in\maS;$
\item Upon arrival at time $t$, a class-$i$ item is either matched exactly with the classes $\,j,\cdots,k$ such that $S(t)=\{i,j,\cdots,k\}$, if any such items are available it  leave the system immediately, or are placed in an infinite buffer.
\end{itemize}
Let us define the following summation $\bar{\lambda}$:
\begin{equation}
\label{eq:SomeOfLambda}
\bar{\lambda}:=\sum\limits_{i\in\maV}\lambda_i\qquad\bar{\lambda}_A:=\sum\limits_{i\in A}\lambda_i,\qquad A\subset\maV.
\end{equation}

\begin{remark}
	\label{rq:assumptionForPoissonArr}\rm
	The Poisson process can be obtained by evaluating the following assumptions for arrivals during an infinitesimal short period of time $\delta t:$
	\begin{itemize}
	\item The probability that one arrival occurs between $t$ and $t+\delta t$ is  $t + o(t),$ independent of the time $t$, and independent of arrivals in earlier intervals.
	\item The number of arrivals in non-overlapping intervals are statistically independent.
	\item The probability of two or more arrivals happening during $[t,t+\delta t]$ is negligible compared to the probability of zero or one arrival, i.e., it is of the order $o(t)$.
	\end{itemize}

\end{remark}

\subsection{System dynamics}
  Each vertex $i$ associated to a buffer content called class-$i$ queue, and denote the associated class-$i$ queue process by $Q_i := \{Q_i(t)\; :\; t \geq 0\}$. More precisely, for all $t \geq 0,\; Q_i(t)$ is the number of the class-$i$ items in the queue at time $t$. The $|\maV|$-dimensional queue process of the system denoted as follows:
\begin{equation}
\label{eq:TheQueueQ}
Q=(Q_1,\cdots,Q_{\maV}).
\end{equation}
 For $t \geq 0$ and $A \subset \maV$, we let $Q_A(t)$ be the restriction of $Q(t)$ to its coordinates in $A$.  Upon arrival to the system, an element of the class-$i$ can find several possible matches. A matching policy is a rule specifying how to match in such cases. We say that a matching policy $\Phi $  is admissible if matches always occur when possible, and decisions are made only on the value of the $Q$ queue process at arrival times.
 
 For a matching structure $\mathbb{S}$, an arrival-rate vector $\lambda$ and under the admissible matching policy $\Phi,$ the queue process $Q$ is as follows:
 \begin{itemize}
 	\item Its a \textbf{Continuous-Time Markov Chain} denoted by  `CTMC';
 	\item The initial queue length $Q(0);$
 	\item For all $t \geq 0,\;Q_i(t)Q_j(t)\cdots Q_k(t) = 0$ where $i, j,\cdots, k\in\maV$ such that $S(t)=\{i,j,\cdots k\}\in\maS$.
 \end{itemize}
 
We thus characterize the system by the triple $(\mbS,\Phi, \lambda)_C$ (where we append the subscript $C$ to denote a {\em continuous-time model}, as opposed to the one in discrete-time, which will omit the subscript $C$).
\begin{ex}
	\rm
In \cite{NS16} was shown an example of a matching system with $4$ item types depicted on Figure \ref{fig:CTMCFigure}. The items arrive as a random process, as individual items, or in batches. The average arrival rate of type $i$ items is $\alpha_i.$ There exist three possible matchings; e.g., $\langle1, 2\rangle$ respectivily $\langle2,3\rangle$ is a matching which matches one item of type $1$ with one item of type $2$ resp one item of type $2$ with one item of type $3$. $\langle2, 3, 4\rangle$ is another matching which matches one item of types $2$, $3$ and $4$.

A matching can only be applied if all contributing items are present in the system; and if it is applied, the contributing items instantaneously leave the system.

\end{ex}
 
 \begin{figure}[ht]
 	\centering
 	\includegraphics[width=.9\linewidth]{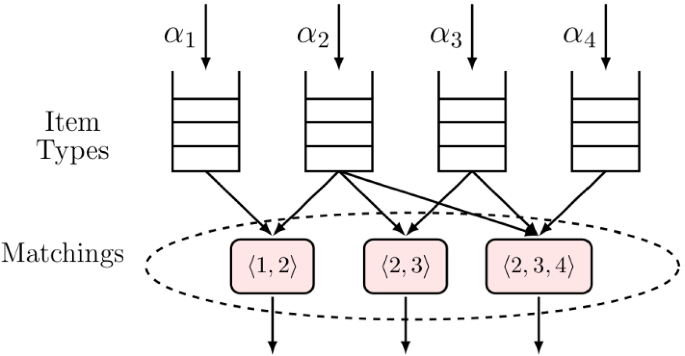}
 	\caption{Matching Model (CTMC).}
 	\label{fig:CTMCFigure}
 \end{figure}
 


\subsection{Stability of a matching queue}
\label{sec:StabOfMatchingQueue}
The matching queue $(\mbS,\Phi, \lambda)_C$ is said to be stable if the corresponding CTMC $Q$ is positive recurrent, and unstable
otherwise.

\begin{definition}
	\label{def:stabConti}
	The stability region corresponding to the connected matching structures
	$\mbS$ and the matching policy $\Phi$ is the set
	\[\textsc{Stab}(\mbS,\Phi)=\left\lbrace\lambda\in (\R^{++})^{|\maV|}\;:\; (\mbS,\Phi,\lambda)_C\;\textrm{ is stable}\right\rbrace.\]
\end{definition}

We also say that node $i\in\maV$ is stable if, for some initial condition, the\textit{ mean time for its associated queue to empty is finite}. Otherwise, the node is unstable.\\


Indeed, the advantage of the fluid limit techniques in continuous-time is to facilitate stability analysis.
Thus, as was shown in \cite[Theorem 2]{MoyPer17}, the stability region of a discrete-time stochastic model can be studied by embedding it in an appropriate continuous-time mode. Then, the continuous-time counterparts of the results in \cite{RM19,MaiMoy16,JMRS20} hold for our matching queues and vice versa.

\clearpage

\pagestyle{empty}

\chapter{Literature review}\label{chap2: state of art}
\pagestyle{fancy}

Nowadays, the \textbf{matching model} is considered one of the major challenges that are of interest in various sectors (healthcare systems, peer-to-peer networks, interfaces of the collaborative economy, assemble-to-order systems, job search applications, and so on). Other references address specific models for designated applications: \cite{DMKPFL15} on organ transplantation,  \cite{BDPS11} on kidney transplants, \cite{TW08} on housing allocations systems, or \cite{OW19} on ride-sharing models. A more recent application for the matching model results in modeling sharing-economy (collaborative consumption) platforms, with the most relevant examples being car-sharing platforms, such as Uber and Lyft, lodging services, such as Airbnb, and virtual call centers (namely call centers with home-based agents), as considered in, e.g., \cite{GLM16, I16}. Since a platform operating in a sharing-economy market must match supply and demand at every instance, possibly in a multi-region setting, matching queues can be used to model and optimize such platforms; see \cite{SRJ15} for an application in the car-sharing setting.

Before passing to part \ref{part:contributions}  which contains the main subject and the results of our contributions, in this chapter, we present some of the related work for the matching models.
In Section \ref{sec:SkillBased}, we present the first  natural representation of service systems in which customers and servers are of different classes called skill-based queueing systems. In Section \ref{sec:BipartiteMatchingModel(BM)} we start from \cite{CKW09} where they have introduced the matching model referred to the bipartite matching model and we present some of related works. In Section \ref{sec:Extended bipartite matching model} we describe the extended bipartite matching model (EBM). In Section \ref{sec:matching model} we represent all dedicated studies for the general matching model (GM). In Section \ref{sec:Optimisation} 
we address a point of view of stochastic optimization of the matching models. In Section \ref{sec:OtherExtention} we present other extensions of the matching model. Finally, in Section \ref{sec:ProbAndPosi} we represent the related study with our contributions.

\section{Skill-based service systems}
\label{sec:SkillBased}
Over the past decade, an increasing interest has been dedicated to stochastic systems in which incoming elements are matched according to specified compatibility rules. 
This is, first, a natural representation of service systems in which customers and servers are of different classes, and where designated classes of servers can serve 
designated classes of customers. For this general class of queueing models, termed skill-based queueing systems, it is then natural to investigate the conditions for the existence of a stationary state, 
and under these conditions, to \textit{design and control the model at best}, for given performance metrics (end-to-end delay, matching rates, fairness, and so on.) Such models are classical queueing systems, in the sense that 
there is a dissymmetry between customers and servers: customers come and depart the system, whereas servers are part of the `hardware', remain in the system, and switch to the service of another customer when they have completed one (with possible vacation times in-between services) see Figure \ref{fig:Example} (a).  In \cite{AAM07, GM2000} various types of customers call are routed to various groups of skill-based servers.  

Generally, these studies consist of analyzing the stability of the model under which matching policy is optimal and the probability measures that makes it  stable. It should be noted that a lot of works have advanced research and have been developed within the framework of stochastic matching model.

\begin{figure}[htp]
	\begin{subfigure}{.5\textwidth}
		\centering
		\includegraphics[width=.9\linewidth]{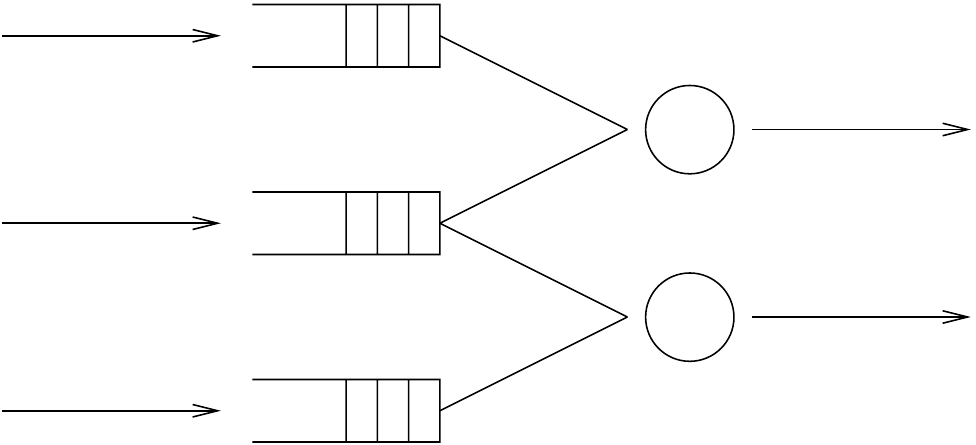}
		\caption{Modelisation of call center `Skill-based'.}
		\label{fig:QueueingModelCallCenter}
	\end{subfigure}
	\begin{subfigure}{.5\textwidth}
		\centering
		\includegraphics[width=.9\linewidth]{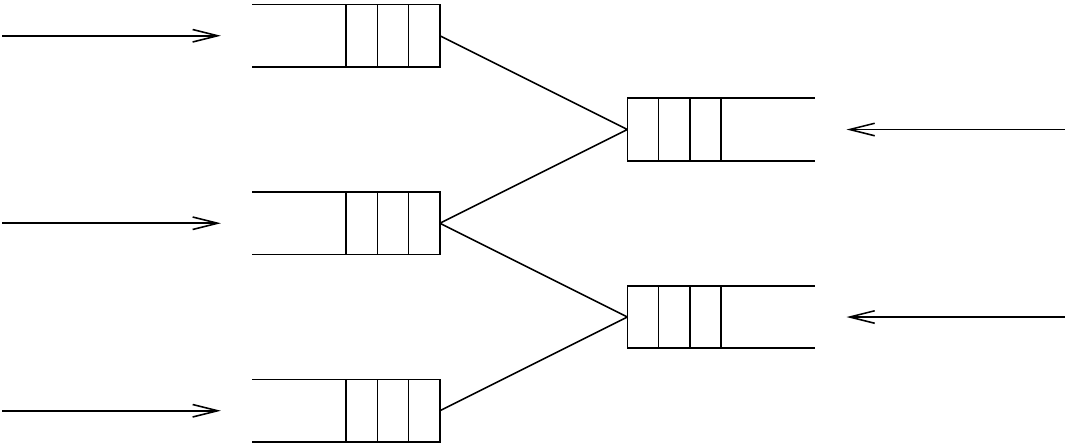}
		\caption{Bipartite matching model.}
		\label{fig:QueueingModeTwoEntrance}
	\end{subfigure}
	\caption{Left: A skill-based queueing system. Right: A bipartite matching model.}
	\label{fig:Example}
\end{figure}


\section{Bipartite Matching model (BM)}
\label{sec:BipartiteMatchingModel(BM)}
In \cite{CKW09,AW11}, a variant of such skill-based systems have introduced, which are now commonly referred to as `Bipartite Matching models' (BM): couples customer/server enter the system at each time point, and customers and servers play symmetric roles: exactly like customers, servers come and go into the system see Figure \ref{fig:Example} (b). Upon arrival, they wait for a compatible customer, and as soon as they find one, leave the system together with it.  These settings are suitable to various fields of applications, among which, blood banks, organ transplants, housing allocation, job search, dating websites, and so on.  

The mathematical setting is the following: in the two aforementioned references, it is assumed that the incoming customers are of various classes in the set	$C = \{1,2,\cdots,I\}$, and that server are of classes in the set $S =\{1,2,\cdots,J\}$. It is assumed that a server of type $j\in S$ can serve a subset of customer types $C(j)\subset C$, and that a customer of type $i\in C$ can be served by a subset of server types $S(i)\subset S$. 

Caldentey \& al. \cite{CKW09} have considered the types of customers and servers if the infinite sequences are \textit{random}, \textit{independent identically distributed}, and customers and servers are matched according to their order in the sequence, on a First Come, First Matched ({\sc fcfs}) basis. This service system can be represented by a bipartite graph $\mathbb{G}=(\mathcal C+\maS,\maE).$

Consider the sequences $(c^n; s^n)_{n\geq 1}\in\mathcal C^\infty\times\maS^\infty$ have probability distribution $\mathbb P((c^n; s^n)=(i,j))=\alpha_i\beta_j$, for probability vectors,
$$\alpha=(\alpha_i)\in(\R^+)^I \textrm{ and }   \beta=(\beta_j)\in(\R^+)^J.$$ 
Define \textit{matching rates}, by counting for each $n$ the number of $(c_i; s_j)$ matches created between $c_1;\cdots;c_n$ and $s_1;\cdots;s_n$, and divide by $n$ to get $r^n_{c_i;s_j}.$ For a given $\mathbb{G}$, $\alpha$, $\beta,$ the matching rates $r_{c_i;s_j} =\lim\limits_{n\rightarrow +\infty}r^n_{c_i;s_j}$ if these limits exist almost surely.\\

The construction of the matching is to add one pair (\textit{independent}) of a server and a customer at a time and to match those to the earliest unmatched customer or server that they find, or leave them unmatched, waiting for subsequent pairs. \\

They have proposed two following simplifications:\\
\textbf{1. First simplification} which leads to their model is that there are \textit{no service times}. \\
\textbf{2. Second simplification} to the considered model, by \textit{ignoring the arrival times}.\\

They said that the system has a \textit{balanced infinite matching}: if the fraction of customers of type $i$ among the first $n$ customers, which are matched by one of the first $n$ servers, converge almost surely to $\alpha_i$, and the fraction of servers of type $j$ which are matched by one of the first $n$ customers, converge almost surely to $\beta_j$ as $n\longrightarrow +\infty.$\\
They have found a \textit{necessary condition for the system} to be balanced, that is:
\begin{equation}
\label{eq:NecCondForBalance}
\alpha(C) \leq\beta(S(C)),\;\;\beta(S)\leq\alpha(C(S))\;;\; C\subseteq\mathcal C,\qquad S\subseteq\maS.
\end{equation}
Also, they provided the following conjecture that is, a sufficient condition for ergodicity (existence of matching rates),
\begin{equation}
\label{eq:ConjCKWNecErgo}
\alpha(C)<\beta(S(C)),\;\;\beta(S)<\alpha(C(S))\;;\;\; C\subset\mathcal C,\;C\neq\mathcal C,\qquad S\subset\maS,\;S\neq\mathcal S.
\end{equation}

They have studied specific models such as, `N' model, (respectively `W' model) (i.e., the structure of the model is of the form `N' (respectively `W') graph) and an almost complete graph case (in which each server type can be matched to all except at most one customer type and vice versa). For these, they proved the Conjecture (\ref{eq:ConjCKWNecErgo}) above, solved the balance equations and obtained the matching rates. However, they studied the ergodicity for the  `NN' model (i.e., the matching structure of the model is of the form `NN' graph), and then the Conjecture (\ref{eq:ConjCKWNecErgo}) holds for any bipartite graph in which every server type is connected to all but at most 2 of the customer types and vice versa.

	A dynamic representation of  `W' model is given in Figure \ref{fig:Example of W model}. \\

\begin{figure}[htp]
	\begin{subfigure}{.5\textwidth}
		\centering
		\includegraphics[width=.9\linewidth]{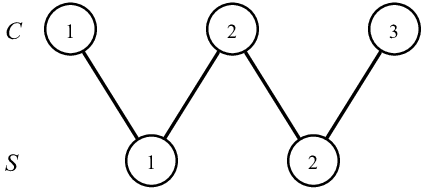}
		\caption{The bipartite graph for the `W' model}
		\label{fig:The-bipartite-graph-for-the-W-model}
	\end{subfigure}
	\begin{subfigure}{.5\textwidth}
		\centering
		\includegraphics[width=.7\linewidth]{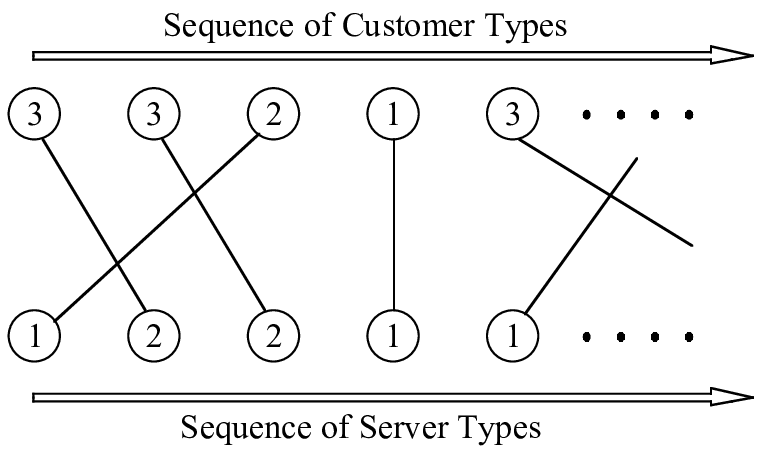}  
		\caption{The {\sc fcfs} infinite matching for the `W' model}
		\label{fig:The-FCFS-infinite-matching-for-the-W-model}
	\end{subfigure}
	\caption{The bipartite matching graph for the `W' model.}
	\label{fig:Example of W model}
\end{figure}

Adan and Weiss \cite{AW11} have considered a bipartite matching model in which multi-type customer multi-type server models take into account the special needs of customers as well as the aptitudes and capabilities of the servers. Through the overlap of the client and server subsets, the main interest is in the ability of providing individually tailored service while still allowing for cooperation and pooling of the servers. The Markovian model suggested in \cite{CKW09} to describe the {\sc fcfs} infinite matching turns out to be intractable in general, and could therefore not be used to prove the conjecture in general or to calculate the rates.  
Also, they have solved the balance conditions (\ref{eq:NecCondForBalance}) by an \textit{explicit product form} stationary distribution.
Further, they have given some examples and demonstrated the calculation of the matching rates for some special system graphs.\\

	Adan \& al. \cite{ABMW17} have proved the fundamental structure of the model in the following three steps:
	\begin{itemize}
		\item  Derive a Loynes' scheme, which enables to get to stationarity through sample path dynamics, and to prove the existence of a \textit{unique} {\sc fcfs} matching over $\Z$ and not solely for $\N$.
		\item  Define a pathwise transformation in which they interchanged the positions of the two items in a matched pair, and they proved the dynamic ``reversibility'' of the model under this transformation.
		\item  Construct ``primitive'' Markov chains whose product form stationary distributions are obtained directly from the dynamic reversibility. Using these as building blocks, they have drived product form stationary distributions for multiple `natural' Markov chains associated with the model, and they computed various non-trivial performance measures as a by-product.\\
		They illustrated these results for the `NN' model that was described in \cite{CKW09}, and that couldn't be fully analyzed.
	\end{itemize}
	They used the following \textit{reversibility result}: starting from two independent $i.i.d.$ sequences over $\Z$ with {\sc fcfs} matching between them, and performing the exchange transformation on all the links, they obtained two sequences of exchanged customers and servers and matching between them. It is then true that the sequences are again independent $i.i.d.$, and the matching between them is {\sc fcfs} in the reversed time direction.\\
	
Adan and Weiss \cite{AW14} have considered a queueing system with $J$ parallel servers $S = \{m_1,\cdots,m_J\}$ (fixed set of server), and with customer types $C=\{a, b, c,\cdots\}$. A bipartite graph $\mathbb{G}$ describes which pairs of server-customer types are compatible. Further, they considered {\sc fcfs-alis} policy: \textit{A server always picks the first, longest waiting compatible customer, and a customer is always assigned to the longest idle compatible server}. ALIS is the best way to equalize the efforts of the servers, and thus it encourages diligent service.

Assume that arrivals are Poisson and service is exponential. Customers of type $c$ arrive at the system in independent Poisson streams with rates $\lambda_c$, $c\in C$. Service times of server $m_j$ are independent and exponentially distributed with rate $\mu_{m_j}$, $j\in S.$ Note that service durations of customers depend on the server providing the service, and not on the customer type.

Also, they calculated fluid limits of the system under overload, to show that a local steady state exists. They distinguished the case of complete resource pooling when all the customers are served at the same rate by the pooled servers, and the case when the system has a unique decomposition into subsets of customer types, each of which is served at its rate by a pooled subset of the servers. 

Finally, they discussed the possible behavior of the system with generally distributed abandonments, under many server scaling. 


\section{Extended Bipartite Matching model (EBM)}
\label{sec:Extended bipartite matching model}

	In \cite{BGM13}, the settings of  \cite{CKW09,AW11} are generalized to more general service disciplines (termed `matching policies' in this context), and necessary and sufficient conditions for the stability of the system are introduced, which are functions of the compatibility graph and of the matching policy. Define the following condition on $\mu,$ see conditions (\ref{eq:ConjCKWNecErgo}) : 
	\begin{equation}
		\label{eq:NcondEBM}
\left\lbrace\begin{array}{ccccc}		\mu_C(U) &< &\mu_S(S(U)),& \forall\, U\, \subset C\\
		\mu_S(V) &< &\mu_C(C(V )),& \forall\, V \,\subset S.
		\end{array}\right.
			\end{equation}
		The above conditions are then shown necessary and sufficient for the stability of the system for various graph geometries 
		 and have a natural interpretation. Let $\mu_C$ and $\mu_S$ be the	marginals of the arrival probability $\mu$. Customers from $U$ need to be matched with servers from 	$S(U)$. The first line in (\ref{eq:NcondEBM}) asks for strictly more servers in average from $S(U)$ than customers from $U$. The second line has a dual interpretation.\\

		In other words, the measure $\mu$ is not of the form $\mu_C\otimes \mu_S$. Instead, the arrival scenario is characterized by a subset $F\subset C\times S$ representing the possible arrivals of couples, and a measure $\mu$ on $C\times S$ having support $F$.  The system is then called the Extended Bipartite Matching model (EBM, for short), and suits applications in which \textit{independence between the classes of the customers and servers entering simultaneously cannot be assumed}. In the applications to organ transplants and blood transfusions, this extension of the settings of the BM is justified by the possible correlations between the blood types of the arriving couples, who may be parents of one another. Also, in \cite{BGM13} the authors have proven the following results:  
	\begin{itemize}
		\item \textit{Sufficient conditions are obtained}, under which any admissible matching policy can make the system stable,
		\item For the `NN' model, the {\sc ms} policy and some priority policies do not have a maximal stability region - in the sense that the conditions  $(\ref{eq:NcondEBM})$ \textit{are not sufficient for stability}. 
		\item For any bipartite graph, the {\sc ml} {\it policy has a {maximal} stability region}.
	\end{itemize}
	However, \textit{the maximality} of {\sc fcfs} \textit{is left as an open problem.}\\

Moyal \& al. \cite{MBM18} have found an explicit construction of the stationary state of Extended Bipartite Matching (EBM) models as defined in \cite{BGM13}. They used a Loynes-type backward scheme allowing to show the existence and uniqueness of a bi-infinite perfect matching under various conditions, for a large class of matching policies and of bipartite matching structures.



\section{General Matching model (GM)}
\label{sec:matching model}
To model concrete systems, the need then arose to extend these different above models. Indeed, in many applications, the assumption of pairwise arrivals may appear somewhat artificial, and it is more realistic to assume that arrivals are simple. Also, all the aforementioned references assume that the compatibility graph is bipartite, namely, there are easily identifiable classes of {\em servers} and classes of {\em customers}, whatever these mean: donors/receivers, houses/applicants, jobs/applicants, and so on. However, in many cases, the context requires that the compatibility graph take a general (i.e., not necessarily bipartite) form. For instance, in dating websites, it is a priori not possible to split items into two sets of classes (customers and servers) with no possible matches within those sets. Similarly, in kidney exchange programs, intra-incompatible couples donor/receiver enters the system, looking for a compatible couple to perform a `crossed' transplant. 
Then, it is convenient to represent couples donor/receiver as  {\em single} items, and compatibility between couples means that a kidney exchange can be performed between the two couples (the donor of the first couple can give to the receiver of the second, and the donor of the second can give to the receiver of the first). In particular, if one considers blood types as a primary compatibility criterion, the compatibility graph between couples is naturally non-bipartite. \\
	Motivated by these above observations, a variant model was introduced in \cite{MaiMoy16}, in which \textbf{items arrive one by one} and the compatibility graph is general, i.e., not necessarily bipartite: specifically, in this so-called {\em General Matching model} (GM for short), its a particular case of the extended matching model (EBM), it has a queueing model flavor, with the crucial specificity that \textit{items play the roles of both customers and servers}. Then, an incoming item is either immediately matched, if there is a compatible item in the line, or else stored in a buffer. It is the role of the matching policy $\Phi$ to determine the match of the incoming item in case of a multiple choice. Then, the two matched items immediately leave the system forever. Indeed, consider a matching model with a graph $(\maV,\maE)$ and sequence of arriving items $(v_n)_n.$ Let $\tilde{\maV}$ be a disjoint copy of $\maV.$ Define a bipartite matching model with classes $\maV,$ server classes $\tilde{\maV},$ possible matches $\{(u,\tilde{v})\;|\;u\v v\},$ and arriving sequence $(v_n,\tilde{v}_n)_n.$ If the matching policies are the same, then at any time, the buffer-content of the bipartite matching model $(U,\tilde{U}),$ if the buffer-content of the original matching model is $U.$ Then several result transferred from the (EBM) model to (GM) model.\\
	 

Given a connected graph $\mathbb{G}=(\maV,\maE),$ and a matching policy $\Phi.$  Let $\mu$ be a probability measure on $\maV,$ they  defined the following conditions on $\mu$ :
\begin{equation}
\label{eq:NcondGr}
\textsc{Ncond}(G):\left\lbrace
\mu\in\mathscr M(\maV)\;\;;\;\mu(U)<\mu(\maE(U)),\;\forall\,U \subset \maV
\right\rbrace.
\end{equation}
These $\textsc{Ncond}$ are necessary stability conditions. An analog result holds in (\ref{eq:NcondEBM}). In particular, the latter condition is empty if and only the compatibility graph is bipartite (which justifies why items enter by pairs in BM and EBM models - otherwise the model could not be stabilizable). 

	The {\em stability region} of the model denoted by $\textsc{Stab}(\mathbb{G},\Phi)$, is then defined as the set of measures $\mu$ such that the model is 
	positive recurrent (see Section \ref{sec:defstab}). Also, in \cite{MaiMoy16} the authors  have proven that:
	\begin{itemize}
		\item the matching model may be stable if and only if the matching graph in non bipartite, 
		\item $\mathbb{G}$ graph non bipartite then the model is always stable under $\Phi$={\sc ml}, i.e., $$\textsc{Stab}(\mathbb{G},\textsc{ml})=\textsc{Ncond}(\mathbb G),$$
		\item $\mathbb{G}$ complete $p$-partite graph, $p\geq 3$ (which is called in \cite{MaiMoy16} separable graph of order $p)$, then for all $\Phi,\;\mu\in\textsc{Ncond}(\mathbb G),$ the model $(\mathbb{G},\Phi,\mu)$ is stable, i.e., $$\forall \Phi,\;\textsc{Stab}(\mathbb{G},\Phi)=\textsc{Ncond}(\mathbb G).$$
	\end{itemize} 
	However, \cite{MoyPer17} shows that, aside from a particular class of graphs, random policies are never maximal, and that there always exists a strict priority policy that isn't maximal either. Then, by adapting the dynamic reversibility argument of \cite{ABMW17} to the GM models, \cite{MBM17} shows that the matching policy First Come, First Matched ({\sc fcfm}) is maximal and derives the stationary probability in a product form. More recently, following the work of \cite{NS16}, 
	matching policies of the broader {\em Max-Weight} type (including `Match the Longest') are shown to be also maximal and drift inequalities allow to bound the speed of convergence to the equilibrium, and the first two moments of the stationary state. \\

On another hand, since a matching queue is easily seen, the stability region of a discrete-time stochastic model can be studied by embedding it in an appropriate continuous-time model. Thus, the continuous-time counterparts of the results in \cite{MaiMoy16} hold for matching queues and vice versa. The advantage of the continuous-time setting is that powerful fluid-limit techniques can be employed, which greatly facilitate the stability analysis.\\

Fluid models are arguably the most effective tool to prove that a queueing
network is stable, and can also be employed to prove the instability of such
networks. Specifically, following \cite{RS92,JD95} they have found that, under mild regularity conditions, if all the (subsequential) fluid limits of the queues, for all possible initial conditions, converge to 0 in a finite time w.p.1, then the system is stable, in the sense that the underlying queue process is positive recurrent. \\

Moyal and Perry \cite{MoyPer17} have studied the interesting feature of the fluid limits that is to obtain their dynamics determined by the stationary distribution of a “fast” CTMC. Specifically, if the fluid queue associated with one of the nodes is positive, then the relevant time scale for this queue is slower than the time scale for the fluid queues that are null. In the limit, the effect of the “fast” (i.e., null) queues on the evolution of the positive fluid queues are averaged-out instantaneously, a phenomenon known as a stochastic averaging principle (AP) in the literature. See \cite{PW11, PW13} and the references therein, as well as \cite{LZ13, WZZ18} for recent examples of fast averaging in queueing networks.\\
Also, in \cite{MoyPer17} was shown a necessary condition for stability of a matching queue: for any matching graph $\mathbb{G},$ 
\begin{equation}
\label{eq:NcondCDuGraphConti}
\textsc{Ncond}_C(\mathbb G):=\left\lbrace \lambda\in (\R^{++})^{|\maV|}\;:\;\bar{\lambda}_{ I}<\bar{\lambda}_{\maE( I)}\textrm{ for all } I\in\mathbb{I}(G)\right\rbrace.
\end{equation}
That condition can be thought of as an analog to the usual traffic condition for traditional queueing networks see equation (\ref{eq:NcondGr}), and it is thus natural to study whether it is also sufficient. \\

 Except for a particular class of graphs, there \textit{always exists a matching policy rendering the stability region strictly smaller than the set of arrival intensities} satisfying $\textsc{Ncond}_C$ and they are showing explicitly, via fluid-limit arguments, that the stability regions of two basic models {\em pendant graph} and 5-{\em cycle graph} depicted in  Figure \ref{fig:PendantAndCyleGraphs} is strictly included in $\textsc{Ncond}_C$. They generalized this result to any graph $\mathbb{G}$ that is not complete $p$-partite there always \textit{exists a policy of the strict priority type that does not have a maximal stability region}, and that the `Uniform' random policy (natural in the case where no information is available to the entering items on the state of the system) never has a maximal stability region. 
 \begin{definition}
	\cite{MoyPer17} A connected matching structure $\mathbb{S}$ is said to be,	\begin{itemize}
		\item  \textbf{matching-stable} if $\textsc{Ncond}_C(\mathbb{S})$ is non-empty and all admissible matching	policies on $\mathbb{S}$ are maximal;
		\item \textbf{matching-utable
			nstable} if the set $\textsc{Ncond}_C(\mathbb{S})$ is empty.
	\end{itemize}
	\label{def:MatchingStableUnstable}
\end{definition}

Let $\mathcal{G}_7$ denote the set of all connected graphs inducing an odd cycle of size $7$ or more, but no pendant graph and no 5-cycle, and let $\mathcal{G}_7^c$ denote its complement in the set of connected graphs. Also, a result was proven in \cite{MoyPer17} that:
\begin{itemize}
	\item the pendant and  5-cycle graphs depicted in Figure \ref{fig:PendantAndCyleGraphs}, it is shown that those graphs are matching-unstable, (i.e., never maximal).
	\item The only matching-stable graphs in $\mathcal{G}^c_7$ are separable of order 3 or more.
\end{itemize} 
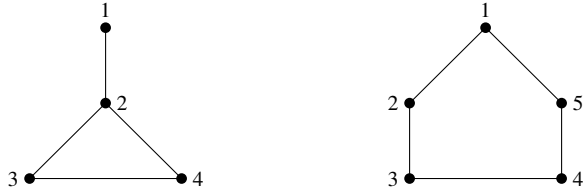
\begin{figure}[htb]
	\begin{center}
		\begin{tikzpicture}
		
		\fill (2,2) circle (2pt)node[above]{\scriptsize{1}};
		\fill (2,1) circle (2pt)node[right]{\scriptsize{2}};
		\fill (1,0) circle (2pt)node[left]{\scriptsize{3}};
		\fill (3,0) circle (2pt)node[right]{\scriptsize{4}};
		
		\draw[-] (2,1) -- (3,0);
		\draw[-] (2,1) -- (1,0);
		\draw[-] (2,1) -- (2,2);
		\draw[-] (1,0) -- (3,0);
		
		\fill (7,2) circle (2pt)node[above]{\scriptsize{1}};
		\fill (6,1) circle (2pt)node[left]{\scriptsize{2}};
		\fill (8,1) circle (2pt)node[right]{\scriptsize{5}};
		\fill (6,0) circle (2pt)node[left]{\scriptsize{3}};
		\fill (8,0) circle (2pt)node[right]{\scriptsize{4}};
		
		\draw[-] (7,2) -- (6,1);
		\draw[-] (7,2) -- (8,1);
		\draw[-] (6,1) -- (6,0);
		\draw[-] (8,1) -- (8,0);
		\draw[-] (6,0) -- (8,0);		
		\end{tikzpicture}
		\caption{Left: Pendant graph. Right: 5-cycle graph.}
		\label{fig:PendantAndCyleGraphs}
	\end{center}
\end{figure}

Variants of the GM model to the case of graphical systems with reneging  	are investigated, respectively in \cite{GW14,NS16,RM19} and \cite{JMRS20} (see also \cite{BDPS11}).\\

\section{Optimizations of the matching model}
\label{sec:Optimisation}
In another line of research, such stochastic matching architectures are addressed from the point of view of stochastic optimization in \cite{BC17}, \cite{GW14} and \cite{NS16}, among others. \\

Buke and Chen \cite{BC17} have focused on the \textit{infinite-horizon average-cost optimal control problem}. In which, they considered a \textit{control policy} determines which are matched at each time by considering a discrete-time bipartite matching model with random arrivals of units of {\it supply} and {\it demand} that can wait in queues located at the nodes in the network.
For a parameterized family of models in which the network load approaches capacity, a new matching policy  for the relaxation admits a closed-form expression is shown to be approximately optimal, with bounded regret, even though the average cost grows without bound.\\

\begin{figure}[htp]
	\centering
	\includegraphics[width=.4\linewidth]{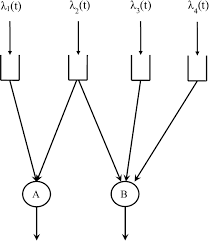}
	\caption{A queueing network view of a system with four input streams and two matchings.}
	\label{fig:A queueing network view of a system with four input}
\end{figure}

Gurvich and Ward \cite{GW14} have considered a model in which each item arrive in a dedicated queue, and wait to be matched with items that exist in other queues (possibly multiple). Once a decision has been made to match, the match itself is instantaneous and the corresponding items leave the system. Upon arrival, an item may find several possible matching, in this cases, an \textit{exsisting controller} must decide which matchings to execute given multiple options. They considered the problem of \textit{minimizing finite-horizon cumulative holding costs}. In principle, the controller may choose to wait until some “{inventory}” of items builds up to facilitate more profitable matches in the future. 

In the example depicted in Figure \ref{fig:A queueing network view of a system with four input}, there are 4 classes of items, and items of class $i$ arrive according to a time-varying Poisson process $A_i$ having instantaneous rate $\lambda_i(t),\;i =1,\,2,\,3,\,4.$ Items of class 1 can be matched to items of class 2. Items of class 2 can be also matched with items of classes 3 and 4. This matching structure is reflected in the graph in Figure \ref{fig:A queueing network view of a system with four input} where each rectangle corresponds to an item class and each of the circles $A$ and $B$ to matching types. When a class 1 item is matched with a class 2 item they both leave the system: matchings are instantaneous. An item of class 4 must be matched to both a class 3 and a class 2 item to depart.\\

Nazari and Stolyar \cite{NS16} have introduced an algorithm that is a variant of the ``Primal-dual algorithm", allowing to achieve stability if this is feasible at all, for a very large class of models. The proposed algorithm furthermore optimizes utility functions that are convex functions of the average matching rates. \cite{GW14} and \cite{NS16} \textbf{allow idling policies}, i.e., scheduling algorithms \textit{allowing to perform no matchin}g at all despite the presence of matchable items in the system,\textit{ to wait for more profitable future matches.} Allowing idling policy makes sense in applications such as assemble-to-order systems, advertisement or operations scheduling, but is much less suitable to kidney transplant networks, in which case the practitioners always perform a transplant whenever one is possible.


\section{Other extensions of matching models}
\label{sec:OtherExtention}
 Specific comparison results concerning single-server queueing systems with impatient customers are also provided in \cite{M08} and \cite{M13}. Moreover, it is well known since the seminal work of Propp and Wilson \cite{PW96}, that coupling from the past algorithms, which mostly use backward coupling convergence, provides a powerful tool for simulating in many cases (monotonicity, stochastic bounds of Markov chains) the steady state of the system.
\\

Along these lines, the above is devoted to the explicit construction of a stationary queue with $S_r$ servers $(S_r \geq 1)$ and impatient customers, by a scheme à la Loynes. Models with impatience (or abandonment, reneging) have been introduced in the queueing literature to represent a strong real-time constraint on the system: the requests have a due date, before which their treatment must be initiated, or completed. Specifically, assume hereafter that any incoming customer is either served if a server becomes available before its deadline or else eliminated forever once the deadline has elapsed. Observe that a loss system of $S_r$ systems (i.e., there is no waiting room, so the incoming customers are either served provided that a server is immediately available or immediately lost otherwise) is a particular case of the present model, for identically null patience. There are $S_r$ servers obeying the First Come, First Served ({\sc fcfs}) rule to serve impatient customers, and the sequences of inter-arrival times, service times, and patience times of the customers are assumed stationary and ergodic, but not necessarily independent.\\



Buke and Chen \cite{BC15} have introduced a new queueing model, called \textit{probabilistic matching system}, to model the traffic in web portals. This queueing model consists of two user classes, in which users wait in the system to match a candidate from the other class, instead of accessing a resource. 
They have stabilized \textit{four admissible matching policies} for the probabilistic matching systems which are:
\begin{itemize}
	\item[1-] the simple threshold policy,
	\item[2-] accept-the-shortest-queue policy, 
	\item[3-] functional threshold policy, 
	\item[4-] the one-sided threshold policy.
\end{itemize}
This study of \cite{BC15} is followed by \cite{BC17} which have proposed approximation methods and analyzed its properties based on fluid and diffusion
limits. They performed numerical experiments to gain insight into probabilistic matching systems.
Also, they showed that some performance measures are insensitive to the matching probability, agreeing with the existing results.\\

Adan \& al. \cite{AKRW18} have considered \textit{three parallel service models} in which customers of several types are served by several types of servers subject to a  \textit{bipartite graph}, and the \textit{service policy is First Come, First Served}. Two of the models have a fixed set of servers. 
\begin{itemize}
	\item The first is a queueing model in which arriving customers are \textit{assigned to the longest idling compatible server} if available, or else queue up in a single queue, and servers that become available pick the longest waiting compatible customer as studied in \cite{AW14}.
	\item  The second is a \textit{redundancy service mode}l where arriving customers split into copies that queue up at all the compatible servers, and are served in each queue on {\sc fcfs} basis, and leave the system when the first copy completes service.
	\item The third model is a matching queueing model with a random stream of arriving servers. Arriving customers queue in a single queue and arriving servers match with the first compatible customer and leave immediately with the customer, or they leave without a customer.
\end{itemize}
They {\it studied the relations between these models}, and showed that they are closely related to the {\sc fcfs} infinite bipartite matching model, in which two infinite sequences of customers and servers of several types are matched {\sc fcfs} according to a bipartite compatibility graph. They also introduced a directed bipartite matching model in which they embed the queueing systems. This leads to a generalization of Burke’s theorem to parallel service systems.

\section{Problem statement and positioning}
\label{sec:ProbAndPosi}
A stochastic matching model, as we said before, is a system of components and each of these components could have more than one state of functioning.\\

The main purpose of this thesis is devoted to three contexts:\\

 First, we study the long-run stability of stochastic matching models, in the sense defined above, on a hypergraphical compatibility matching structure generalizing the approach of \cite{MaiMoy16} to hypergraphs instead of graphs. By doing so, the two closest references to the present work are \cite{GW14} and \cite{NS16}: in both cases, a general matching model is addressed on a hypergraphical matching structure (notice that \cite{NS16} also allows matchings including several items of the same class). The first reference addresses continuous-time models; the second considers discrete-time models, however, most of the results therein can easily be extended to the continuous-time settings. In \cite{GW14} a matching control is introduced, that asymptotically minimizes the holding cost of items in an unstable system (we justify the instability of such systems under the assumptions of \cite{GW14} in Remark \ref{rem:GW}). In the present thesis all the matching policies we consider are non-idling, i.e., entering items are always matched right away if this is possible at all. Thus, the model studied in Chapter \ref{chap4:Hypergraph} is a special case of the model studied in \cite{NS16}, for simple arrivals, \textit{no same-class matchings}, and \textit{non-idling matching policies}.\\

Secondly, we showed how several stability results of \cite{MaiMoy16,MBM17,JMRS20} can be generalized to the case of a multigraphical matching structure which is motivated again by concrete applications, among which dating websites and peer-to-peer interfaces, it is natural to assume that items {\em of the same class} can be matched together. Hence, the need to generalize the previous line of research to the case where the matching architecture is a {\em multigraph} (a graph admitting {\em self-loops}, that is, edges connecting nodes to themselves), rather than just a graph.\\

 On other hand, we study in continuous-time different examples of multigraph $G=(\maV,\maE)$ and hypergraph $\mathbb{H}=(\maV,\maH)$ corresponding respectively matching queues $(G,\Phi,\lambda)_C$ and $(\mathbb H,\Phi,\lambda)_C.$ We deduce the precise stability regions of the corresponding stochastic matching models using the fluid limit techniques.\\

Finally, we present an application for organ transplantation of a stochastic model on hypergraphs, in which we compare the behavior in the long-run stability of the model of complete 3-uniform hypergraphs matching (three-by-three) with the model of complete 3-partite graph matching (two-by-two). Then, according to the distribution of the items, we deduce what is the best matching procedure between two-by-two and three-by-three matchings.


\pagestyle{empty}
\part{Contributions}\label{part:contributions}
\chapter{Hypergraphs} \label{chap4:Hypergraph}
\pagestyle{fancy}
\section*{Introduction}\label{chap4: intro}
In the Stochastic Matching model (introduced in the bipartite case by Caldentey and al. in \cite{CKW09} and generalized by Mairesse and Moyal in \cite{MaiMoy16}), items enter the system randomly and may be matched or not according to their classes. The compatibility between classes is given by a fixed matching structure. In this chapter, we study the long-run stability of stochastic matching
models, in the sense defined above, on a hypergraphical compatibility matching structure.\\

Several applications should naturally incorporate the possibility of matching items by groups of more than two. Let us exemplify this on a concrete example: in organ transplants, (in)-compatibility between givers and receivers are given by a variety of factors, and 
mostly by blood types and immunological factors. 
In kidney exchange programs, items represent intra-incompatible couples $(A,B)$ 
(e.g., a patient $A$ waiting for a transplant and $B$ a parent of his/hers, incompatible with $A$ for a potential organ donation), 
entering a system to find another intra-incompatible couple $(A',B')$ that is compatible with it, in the sense that 
$A$ can receive an organ from $B'$ and $A'$ can receive from $B$. Then the ability of such a system to accommodate all requests and to maximize 
the number of successful transplants and avoid congestion, is translated into the positive recurrence of a stochastic process representing the stochastic system over time. 
Then if we view the items as the {\em couples}, and translate the ``cross-compatibility'' 
(i.e., $A$ can receive from $B'$ and $A'$ can receive from $B$) into the existence of an edge between node $(A,B)$ and node $(A',B')$, such a system is a typical application of the GM introduced in \cite{MaiMoy16}.

Let us now consider the case where such exchanges $(A,B) \leftrightarrow (A',B')$ and $(A',B') \leftrightarrow (A'',B'')$ 
cannot be realized, but $A$ can receive from $B'$, $A'$ can receive from ${B''}$ and ${A''}$ can receive from $B$. 
Then it is natural to consider the possibility of executing the three transplants contemporarily, i.e., to match the 
triplet $(A,B)$, $(A',B')$ and $(A'',B'')$ altogether. 
In several countries including the U.S., such ``exchanges'' by groups of 3 (or more) are allowed, which raises the issue of maximizing ``matchings'' that do not coincide with sets of edges, but of sets of subsets of nodes of cardinality $3$ or more. Hence the need to consider matching models on compatibility matching structures that are {\em hypergraphs} rather than graphs, i.e., a set of nodes $\maV$ equipped with a set of subsets of $\maV$ of cardinality 3 or more. 

The hypergraphical stochastic matching model addressed in this chapter is formally defined as follows: items enter the system by single arrivals and get matched by groups of 2 or more, following compatibilities that are represented by a given hypergraph. 
A matching policy determines the matchings to be executed in the case of a multiple-choice, and the unmatched items are stored in a buffer, waiting for a future match.


At the border between discrete mathematics and probability theory, the main scientific aim of this chapter is to study this widely applicable class of models. In the first step, addressing the crucial question of stability of the system will lead us to study the structural properties of hypergraphs (connectivity, independent sets, rank, anti-rank, degree, size of the transversals, existence of cycles, and so on.) 
In a second step, we will address the weak approximation of the natural Markov process of the model, to better apprehend its main characteristics in steady state, its long-run simulation, and possibly, the estimation of its parameters.\\

This chapter is organized as follows: we start in Section \ref{sec:Ncond} by providing necessary conditions of stability for the present class of systems: 
as will be developed therein, and unlike the particular case of the GM on graphs (see \cite{MaiMoy16}), for which a natural necessary condition could be obtained, we introduce various necessary conditions that depend on distinct geometrical properties of the considered hypergraphs. We then deduce from this, classes of hypergraphs for which the corresponding matching model cannot be stable, see Section \ref{sec:unstable}. Finally, in Section \ref{sec:stable}, we provide the precise stability region in the particular cases where the compatibility hypergraph is complete $3$-uniform, complete $k$-partite  $3$-uniform and then complete up to a partition of its hyperedges (see the precise definitions of these objects below). 
We conclude and discuss this chapter in Section \ref{discussion of results1}.\\

Throughout this chapter, let us consider that the matching structure $\mathbb{S}$ be a hypergraph $\mathbb{H}=(\maV,\maH).$ Recall that $V_n$ is an item which enters the system from the measure $\mu$ on $\maV$ at time $n$ and $H(k)$ is the realized matching at time $k$.
\section{Necessary conditions of stability}
\label{sec:Ncond} 

Fix a matching model $(\mathbb H,\Phi,\mu)$ on a hypergraph $\mathbb H=(\maV,\maH)$. 
Denote for any $n$, $B \subset \maV$ and $\maB \subset \maH$, by $A_n(B)$ the number of arrivals of elements in $B$ and by $M_n(\maB)$ the number of matchings of hyperedges in $\maB$ realized up to $n$, i.e., 
\begin{align*}
A_n(B) &= \sum_{k=1}^n \ind_{\{V_k \in B\}};\\
M_n\left(\maB\right) &= \sum_{k=1}^n \ind_{\{H(k) \in \maB\}};
\end{align*}
and with some abuse, denote $A_n(i)=A_n(\{i\})$ and $M_n(H)=M_n(\{H\})$ for any $i\in \maV$ and $H \in \maH$. 
Observe that the following key relation holds for all $B \subset \maV$, 
\begin{equation}
X_n(B) = A_n(B) - \sum\limits_{H \in \maH} \left|H \cap B \right| M_n\left(H\right)\ge 0,\,\,n\in\N, \label{eq:base}
\end{equation}
since the number of items of classes in $B$ at any time $n$ is precisely the number of arrivals of such items up to time $n$, minus 
the number of these items that leave the system upon each matching of a hyperedge that intersects with $B$.

\subsection{General conditions}
\label{subsec:Ncond}
\noindent We start by introducing several `universal' stability conditions. Fix a hypergraph $\mathbb H=(\maV,\maH)$ throughout the section. 
\begin{definition}
	We say that $I \subset \maV$ is an {\em independent set} of $\mbH$ if $I$ does not include any hyperedge of $\mbH$, i.e, for any 
	$H\in \maH$, $H\cap \bar I \ne \emptyset$. We recall that $\mathbb I(\mbH)$ be the set of all independent sets of $\mbH$. 
\end{definition}
{Let us define for any $\mu\in\mathscr M(\maV)$, and any $B\subset \maV$, the set 
	\begin{equation}
	\label{eq:deflmu}
	L_\mu(B) = \mbox{argmin} \left\{\mu(j)\,:\,j\in B\right\}, \quad B \subset \maV.
	\end{equation}
	To clarify the exposition of the Lemma \ref{lemma:muminimal} (stated below), we need to introduce the following notion, 
	\begin{definition}
		\label{def:muminimal}
		For any $\mu\in\mathscr M(\maV)$ we say that the independent set 
		$I\in\mathbb I(\mbH)$ is $\mu$-{\em minimal} if the intersection of any hyperedge $H\in \maH$ with $I$ is either empty, 
		or reduced to a singleton $\{v_H\}$ that is such that: 
		\begin{itemize}
			\item $v_H$ is of degree 1, i.e., $H$ is the only hyperedge $v_H$ belongs to;
			\item $L_{\mu}(H)=\{v_H\}$, i.e., $v_H$ is the only minimum of $\mu$ over the set $H$.
		\end{itemize}
		%
		An independent set $I\in\mathbb I(\mbH)$ that is not $\mu$-minimal is said {\em non-$\mu$-minimal}. 
		We let $\mathbb I_\mu(\mbH)$ be the set of $\mu$-minimal independent sets of $\mbH$, and ${\mathbb I^\prime_\mu(\mbH)}$ be 
		the set of non-$\mu$-minimal independent sets of $\mbH$, that is, the complement set of $\mathbb I_\mu(\mbH)$ in $\mathbb I(\mbH)$. 
	\end{definition}
	In other words, a $\mu$-minimal independent sets gathers nodes that are the only minimum of $\mu$ over the only hyperedge they belong to. 
	Notice that the collection $\mathbb I_\mu(\mbH)$ can be empty. This is the case if and only if all nodes are of degree at least 2, or all nodes 
	of degree 1 are not the only minimum of $\mu$ on the single hyperedge they belong to. Observe the following characterization, 
	\begin{lemma}
		\label{lemma:muminimal}
		Let $\mbH=(\maV,\maH)$ be a hypergraph, and denote $\maH=\{H_1,...,H_m\}$. 
		An independent set $I=\{v_1,...,v_p\}$ is $\mu$-minimal if and only if for all $n$ and all $k_1,...,k_m$ such that $k_j\in L_\mu(H_j)$ for all $j$,  
		\begin{equation}
		\label{eq:mumin2}
		A_n(I) = \sum_{j=1}^m|H_j\cap I|A_n(k_j). 
		\end{equation}
	\end{lemma}
	\begin{proof}
		First, it is clear that if $I\in\mathbb I_\mu(\mbH)$, then $m\ge p$ and the mapping 
		\begin{equation}
		\label{eq:defvarphi}
		\varphi:\left\{
		\begin{array}{ll}
		\left\{j\in \llbracket 1,m \rrbracket\,:\,I\cap H_j \ne \emptyset\right\} &\longrightarrow \llbracket 1,p \rrbracket\\
		j &\longmapsto i\,:\,L_\mu(H_j)=H_j\cap I =  \{v_{i}\}
		\end{array}\right.\end{equation}
		is bijective. 
		Thus we have a.s. for all $n$, 
		\[A_n(I) = \sum_{j=1}^m A_n(v_{\varphi(j)})=\sum_{\substack{j\in\llbracket 1,m\rrbracket :\\H_j\cap I\ne\emptyset}} |H_j\cap I|A_n(v_{\varphi(j)})= 
		\sum_{j=1}^m |H_j\cap I|A_n(k_j).\]
		Let us now assume that $I\in\mathbb I^\prime_\mu(\mbH)$. Then, 
		\begin{itemize}
			\item If for some hyperedge $H_j$ is such that $|H_j\cap I|\ge 2$, then upon each arrival of an element of class $k_j$, the 
			right-hand side of (\ref{eq:mumin2}) increases {by} $|H_j\cap I|$ while the left-hand side increases {by} 
			1 if $k_j \in I$, or 0 else; 
			\item If for some hyperedge $H_j$ intersecting with $I$, there exists $k_j \in L_\mu(H_j) \cap \bar I$, then upon each arrival of a class $k_j$-item the 
			right-hand-side of (\ref{eq:mumin2}) increases while the left-hand side does not; 
			\item Finally, if for all $j\in\llbracket 1,m \rrbracket$, $|H_j\cap I| \le 1$ and for all $j$ such that $|H_j\cap I| = 1$, $L_\mu(H_j)=\{v_{\varphi(j)}\}$ (defining again $\varphi$ by (\ref{eq:defvarphi})), then if $\varphi(j)=\varphi(l)$ for some $l\ne j$, 
			upon each arrival of a class $v_{\varphi(j)}$-item, the right-hand side of (\ref{eq:mumin2}) increases {by} 2 while the left-hand side increases {by} 1. 
		\end{itemize}
		In all cases, (\ref{eq:mumin2}) cannot hold for all $n$, which concludes the proof.
	\end{proof}
	Now define the following set of measures, 
	\begin{equation*}
	\mathscr N_1(\mathbb H) = \left\{\mu\in\mathscr M(\maV):\mbox{\small{for all }}I\in\mathbb I^\prime_\mu(\mbH),\,\mu(I)\;<\sum_{H\in \maH}\left|H \cap I\right|\min_{k \in H}\mu(k)\right\}.
	\end{equation*}}

\noindent We have the following result, 
\begin{proposition}
	\label{prop:Ncond1}
	For any connected hypergraph $\mathbb H$ and any admissible matching policy $\Phi$,
	\begin{equation*}
	\textsc{Stab}(\mbH,\Phi) \subset \mathscr N_1(\mathbb H).
	\end{equation*}
\end{proposition}

\begin{proof}
	Fix $\mbH=(\maV,\maH)$ and an admissible policy $\Phi$. Denote by $H_1,...,H_m$ the hyperedges of $\mbH$. 
	Suppose that $\mu \in \mathscr M(\maV)$ is such that there exists an independent set $I\in \mathbb I^\prime_\mu(\mbH)$ such that 
	\begin{equation}
	\label{eq:contrNcond1}
	\mu(I) > \sum\limits_{H\in\maH}\left|H \cap I\right|\min_{k \in H} \mu(k). 
	\end{equation}
	For any $i\in\llbracket 1,m \rrbracket$ and any $k_i\in L_\mu(H_i)$ we have that 
	$$M_n(H_i) \le \min_{k \in H_i} A_n(k) \le A_n\left(k_i\right),\quad n\ge 0.$$
	Thus, from the equality in (\ref{eq:base}), for any $k_1,...,k_m$ such that $k_i\in L_\mu(H_i)$ for all $i$, we have that 
	\begin{equation}
	\label{eq:compareXY}
	{X_n(I) \over n} \ge {A_n(I) \over n} - \sum\limits_{i=1}^m\left|H_i \cap I\right| {A_n\left(k_i\right) \over n},\quad n\ge 1.
	\end{equation}
	Applying the SLLN to the right-hand side of (\ref{eq:compareXY}) implies that for any such $k_1,...,k_m$, 
	\[\limsup_{n} {X_n(I) \over n}\ge \mu(I) - \sum\limits_{i=1}^m\left|H_i \cap I\right| \mu\left(k_i\right) =  \mu(I) - \sum\limits_{H\in\maH}\left|H \cap I\right| \min_{k\in H}\mu\left(k\right)>0,\] 
	implying that $X_n(I)$ goes a.s. to infinity and thereby (as $X_n=[W_n]$ for all $n$), the transience of $\suite{W_n}$.
	
	Assume now that $\mu$ is such that for some independent set $I\in \mathbb I^\prime_\mu(\mbH)$, an equality holds in (\ref{eq:contrNcond1}). 
	Then, for any $k_1,...,k_m$ such that $k_j\in L_\mu(H_j)$ for all $j$,
	the Markov chain $\suite{Y_n}$ defined as
	\begin{equation*}
	Y_n =A_n(I)\;-\;\sum\limits_{j=1}^m\left|H_j \cap B\right|A_n\left(k_j\right),\quad n\in\N,
	\end{equation*}
	is a random walk with drift 0 that is different from the identically null process, in view of Lemma \ref{lemma:muminimal}. 
	Hence $\suite{Y_n}$ is null recurrent. Would the chain $\suite{W_n}$ be positive recurrent, the sequence 
	$\suite{X_n}$ would visit the state 
	$\mathbf 0$ infinitely often, with inter-passage time at $\mathbf 0$ of finite expectation. 
	Thus from (\ref{eq:compareXY}), the sequence $\suite{Y_n}$ would be positive recurrent, an absurdity. This concludes the proof. 
	%
\end{proof}

{Define the following sets of measures, 
	\begin{align*}
	\mathscr N^{\tiny{+}}_1(\mathbb H)&= \left\{\mu\in\mathscr M(\maV):\, \forall I \in \mathbb I(\mbH),\,\mu(I)\;<\sum_{H\in \maH}\left|H \cap I\right|\min_{k \in H \cap \bar I}\mu(k)\right\};\\
	\mathscr N_1^{\tiny{++}}(\mbH) &= \left\{\mu\in\mathscr M(\maV):\,\forall B \subset \maV,\,\mu(B)\;\le \sum_{H\in \maH}\left|H \cap B\right|\min_{k \in H}\mu(k)\right\}.
	\end{align*}
	We have the following result,
	\begin{corollary}
		\label{cor:StabN+N++}
		For any connected hypergraph $\mathbb H$ and any admissible matching policy $\Phi$,
		\begin{equation*}
		\textsc{Stab}(\mbH,\Phi) \subset \mathscr N_1^{\tiny{+}}(\mathbb H)\cap \mathscr N_1^{\tiny{++}}(\mathbb H). 
		\end{equation*}
\end{corollary}}

\begin{proof}
	{We just show that $\mathscr N_1(\mathbb H)$ is included in $\mathscr N_1^{\scriptsize{+}}(\mathbb H)\cap \mathscr N_1^{\tiny{++}}(\mathbb H)$. Set again $\maH=\{H_1,...,H_m\}$ and fix $\mu \in \mathscr N_1(\mathbb H)$. 
		To show that $\mu\in\mathscr N_1^{\tiny{+}}(\mathbb H)$, first observe that for any independent set $I=\{v_1,...,v_p\}\in\mathbb I_\mu(\mbH)$, for any hyperedge $H_j$ 
		intersecting with $I,$ we have that $\min_{k\in H_j}\mu(k) = \mu(v_{\varphi(j)}) < \min_{k\in H_j\cap \bar I}\mu(k).$ Therefore, 
		\[\mu(I) = \sum_{i=1}^p \mu(v_i) = \sum_{j=1}^m |H_j\cap I| \mu(v_{\varphi(j)}) < \sum_{H \in\maH} |H\cap I| \min_{k\in H\cap \bar I}\mu(k),\]
		whereas if $I\in\mathbb I^\prime_\mu(\mbH)$, as $\mu\in\mathscr N_1(\mathbb H)$ we have that 
		\[\mu(I)  < \sum_{H \in\maH} |H\cap I| \min_{k\in H}\mu(k) \le \sum_{H \in\maH} |H\cap I| \min_{k\in H\cap \bar I}\mu(k),\]
		hence $\mu\in \mathscr N_1^{\tiny{+}}(\mathbb H)$.}
	
	\medskip
	
	{It remains to show that $\mu\in \mathscr N_1^{\tiny{++}}(\mathbb H)$, and for this, we first observe that 
		\begin{equation}
		\label{eq:animaux}
		\mbox{for all } I \in \mathbb I(\mbH),\,\mu(I)\;\le \sum_{H\in \maH}\left|H \cap I\right|\min_{k \in H}\mu(k).
		\end{equation}
		To see this, it suffices to observe that for any independent set $I\in\mathbb I_\mu(\mbH)$, recalling (\ref{eq:defvarphi}), 
		\[\mu(I) = \sum_{j=1}^m \mu(v_{\varphi(j)})=\sum_{\substack{j\in\llbracket 1,m\rrbracket :\\H_j\cap I\ne\emptyset}} |H_j\cap I|\mu(v_{\varphi(j)})= 
		\sum_{j=1}^m |H_j\cap I|\mu(k_j)=\sum_{H\in \maH} |H\cap I|\min_{k\in H}\mu(k),\]
		hence (\ref{eq:animaux}). 
		Now fix $B$, a subset of $\maV$ that is not an independent set of $\mbH$. 
		Then, we construct by induction the family of sets $B:=B_0 \supset B_1 \supset B_2 \supset ...\supset B_r$, where $r$ is properly defined below, as follows: for any $i\ge 0$, if $B_i$ is not an independent set of $\mathbb I(\mbH)$, 
		then we take an arbitrary hyperedge $H_{j_i} \in\maH$ such that $H_{j_i}\subset B_i$, and set $B_{i+1}=B_i \setminus\{k_i\}$, for an arbitrary $k_i\in L_\mu(H_{j_i})$. 
		Then, there exists an integer $r \le |B|-1$ such that $B_r$ is an independent set of $\mathbb I(\mbH)$, and we stop the construction at this point. 
		Observe that for any $i\in\llbracket 0,r-1 \rrbracket$, 
		\begin{equation}
		\label{eq:sasha}
		\mu(B_{i+1})\;\le \sum_{H\in \maH}\left|H \cap B_{i+1}\right|\min_{k \in H}\mu(k) \,\,\Longrightarrow \,\,\mu(B_{i})\;\le \sum_{H\in \maH}\left|H \cap B_{i}\right|\min_{k \in H}\mu(k).
		\end{equation}
		To see this, fix $i$ and suppose that the left-hand side of the above holds true. Then, we have 
		\begin{equation}
		\label{eq:sasha11}
		\mu(B_{i}) = \mu(B_{i+1}) +\mu(k_i),
		\end{equation}
		and on the other hand, 
		\begin{multline}
		\label{eq:sasha1}
		\sum_{H\in \maH}\left|H \cap B_{i}\right|\min_{k \in H}\mu(k)\\
		\begin{aligned}
		&= \sum_{H\in \overline{\maH(k_i)}}\left|H \cap B_{i}\right|\min_{k \in H}\mu(k) +  \sum_{H\in \maH(k_i)}\left|H \cap B_{i}\right|\min_{k \in H}\mu(k)\\
		&= \sum_{H\in \overline{\maH(k_i)}}\left|H \cap B_{i+1}\right|\min_{k \in H}\mu(k) +  \sum_{H\in \maH(k_i)}\left(\left|H \cap B_{i+1}\right|+1\right)\min_{k \in H}\mu(k)\\
		&= \sum_{H\in \maH}\left|H \cap B_{i+1}\right|\min_{k \in H}\mu(k) +  \sum_{H\in \maH(k_i)}\mu(k_H),
		\end{aligned}
		\end{multline}
		where for any $H\in\maH(k_i)$, $k_H$ is an arbitrary element of $L_\mu(H)$. 
		But $\mu(k_i)$ is less or equal than the second term of the latter sum because $\mu(k_i)=\mu(k_{H_i})$, and we assumed that 
		$\mu(B_{i+1})$ is less than the first term $\left(\mu(B_{i+1})\leq\sum_{H\in \maH}\left|H \cap B_{i+1}\right|\min_{k \in H}\mu(k)\right)$. This completes the proof of (\ref{eq:sasha}) in view of (\ref{eq:sasha1}). 
		To conclude, as $B_r \in \mathbb I(\mbH)$ and in view of (\ref{eq:animaux}), we have that 
		$\mu(B_r) \le \sum_{H\in \maH}\left|H \cap B_{r}\right|\min_{k \in H}\mu(k)$, which implies by an immediate induction using (\ref{eq:sasha}), that 
		$$\mu(B) \le \sum_{H\in \maH}\left|H \cap B\right|\min_{k \in H}\mu(k).$$ This completes the proof.}
\end{proof}
{
	\begin{remark}\rm[Graphical case]
		Let us consider the special case where $\mbH=(\maV,\maH)$ is a graph. Then, it is shown in Proposition 2 of \cite{MaiMoy16} that the stability region of the model is included in the set 
		\[\textsc{Ncond}(\mbH)=\left\{\mu \in \mathscr M(\maV)\,:\;\forall I\in\mathbb I(\mbH),\,\mu(I)<\mu(\maE(I))\right\},\]
		where for any set $B\subset \maV$, $\maE(B)=\left\{j\in \maV\,:\,(i,j)\in \maH\mbox{ for some }i\in B\right\}.$ 
		It is then easy to check by hand that 
		{\sc Ncond}$(\mbH)$ is included in $\mathscr N^{++}_1(\mbH)$. Indeed, 
		if we let $\mu\in \textsc{Ncond}(\mbH)$ and $I\in\mathbb I(\mbH)$ (meaning that $I$ is an independent set of the graph 
		$\mbH$, in the usual sense), then, 
		for any edge $H\in\maH$, $|H\cap I| =1$ if $I$ contains a vertex of the edge $H$, and 0 else, so we get that 
		\[\sum_{H\in \maH}\left|H \cap I\right|\min_{j \in H \cap \bar I}\mu(j) = \sum_{(i,j)\in \maH\,:\,i\in I}\mu(j) \ge \sum_{j\in\maE(I)} \mu(j) = \mu(\maE(I)),\]
		where the inequality above is an equality whenever each element of $\maE(I)$ shares an edge with a single element of $I$, and else a strong inequality. 
		Thus $\mu \in \mathscr N_1^{++}(\mbH)$. 
		In fact, it is necessarily the case that $\textsc{Ncond}(\mbH) \subset \mathscr N_1(\mbH)$ because if it was not true, 
		there would exist in particular a $\mu\in \overline{\mathscr N_1(\mbH)} \cap \textsc{Ncond}(\mbH)$, making 
		the system $(\mbH,\textsc{ml},\mu)$ unstable (in view of Proposition \ref{prop:Ncond1}) despite the fact that 
		$\mu \in \textsc{Ncond}(\mbH)$, a contradiction to Theorem 2 in \cite{MaiMoy16}. 
		Observe however that $\textsc{Ncond}(\mbH) \ne \mathscr N_1(\mbH)$ in general. 
		To see this, consider the case where $\mbH$ is the cycle of size $5$, $1-2-3-4-5-1$. 
		For a small enough $\epsilon>0$, set 
		\[\left\{\begin{array}{ll}
		\mu(1) &={1\over 2}-{3\epsilon \over 4};\\
		\mu(2) = \mu(5) &={1\over 4}-{\epsilon \over 8};\\
		\mu(3) &={4\epsilon \over 5};\\
		\mu(4) &={\epsilon \over 5}.
		\end{array}\right.\]
		It is then easily checked that $\mu\in\mathscr N_1(\mbH)$. However $\mu\not\in \textsc{Ncond}(\mbH)$, since the independent set $I=\{1,3\}$ is such that 
		\[\mu(I)={1\over 2} +{\epsilon \over 20} > {1\over 2} - {\epsilon \over 20} = \mu(\{2,4,5\}) = \mu(\maE(I)).\]
		{\it As a conclusion, if $\mbH$ is a graph the necessary condition ``$\mu\in \textsc{Ncond}(\mbH)$'' is \textbf{stronger than} the necessary condition ``$\mu\in \mathscr N_1(\mbH)$''.} 
	\end{remark}
}

%
%
%

Let us now define the following set of measures, 
\begin{equation*}
\mathscr N_2(\mathbb H)= \left\{\mu\in\mathscr M(\maV):\,\quad \forall T \in \maT(\mathbb H)\;,\,\mu(T)\;>{1\over r(\mbH)}\right\}.
\end{equation*}
We also have that
\begin{proposition}
	\label{prop:Ncond2}
	For any connected hypergraph $\mathbb H$ and any admissible matching policy $\Phi$,
	\begin{equation*}
	\textsc{Stab}(\mbH,\Phi) \subset \mathscr N_2(\mathbb H).
	\end{equation*}
\end{proposition}
\begin{proof}
	Suppose that there exists a transversal $T\in\maT(\mbH)$ such that $\mu(T)\;\le{1\over r(\mbH)}$. 
	{As each match contains at least one element whose class is an element of $T$, at any time 
		the overall number of completed matches cannot exceed the number of arrivals of elements 
		whose class belongs to $T$,}
	in other words $M_n(\maH)\leq A_n(T)$  for all $n$. Thus, for all $n$ we have that 
	\begin{equation*}
	{X_n(\maV)\over n}\ge {1\over n}\left(A_n(\maV)-r(\mbH)M_n(\maH)\right)\ge {1\over n}\left(A_n(\maV)-r(\mbH)A_n(T)\right).
	\end{equation*}
	Taking $n$ to infinity in the above yields 
	\begin{equation*}
	\limsup_n {X_n(\maV)\over n} \ge 1 - r(\mbH)\mu(T),
	\end{equation*}
	and we conclude as in the previous proof. 
\end{proof}

\begin{remark}
	\label{rem:unifsubN2}
	\rm
	As an immediate consequence of Proposition \ref{prop:Ncond2}, if $\mbH=(\maV,\maH)$ is of {order} $q$, 
	and such that $\tau(\mbH) \le {q \over r(\mbH)}$, 
	then $\textsc{Stab}(\mbH,\Phi)$ does not contain the uniform measure $\mu_{\textsc{u}}=(1/q,...,1/q)$ on $\maV$, in other words the model $(\mbH,\Phi,\mu_{\textsc{u}})$ is unstable 
	for any $\Phi$. Indeed, for any minimal transversal $T$ of $\mbH$ we have that 
	\[\mu_{\textsc{u}}(T) = {\tau(\mbH) \over q} \le {1\over r(\mbH)}.\]
\end{remark}

\medskip

We now introduce two necessary conditions of stability based on the anti-rank of the considered hypergraph. 
We first introduce the following sets of measures, 
\begin{align}
\label{eq:Ncond3+}
\mathscr N^{\tiny{+}}_3(\mathbb H) &= \left\{\mu \in \mathscr M(\maV):\;\forall i \in \maV,\,\mu(i) \le {1 \over a(\mbH)}\right\};\\
\label{eq:Ncond3-}
\mathscr N^{-}_3(\mathbb H)&= \left\{\mu \in \mathscr M(\maV):\;\forall i \in \maV,\,\mu(i) < {1 \over a(\mbH)}\right\}. 
\end{align}
We have the following, 
\begin{proposition}
	\label{prop:Ncond3}
	For any connected hypergraph $\mathbb H=(\maV,\maH)$ and any admissible policy $\Phi$, 
	\begin{equation}
	\label{eq:Ncondantirank}
	\textsc{Stab}(\mathbb H,\Phi) \subset \mathscr N^{\tiny{+}}_3(\mathbb H). 
	\end{equation}
	If the hypergraph $\mathbb H=(\maV,\maH)$ is $r$-{uniform} (i.e., $a(\mbH)=r(\mbH)=r$) we have that 
	\begin{equation}
	\label{eq:Ncondkunifcomplet}
	\textsc{Stab}(\mathbb H,\Phi) \subset \mathscr N^{-}_3(\mathbb H).
	\end{equation}
	in other words the model $(\mbH,\Phi,\mu)$ cannot be stable unless $\mu(i)<1/r$ for any $i\in \maV$. 
\end{proposition}

\begin{proof}
	To prove the first statement, we argue again by contradiction. Suppose that 
	$\mu(i_0) > {1 \over a(\mbH)}$ for some node $i_0$. 
	As the function 
	\[
	\begin{cases}
	\R^+ &\longrightarrow \R^+\\
	x &\longmapsto {r(\mbH)-a(\mbH)+x \over xa(\mbH)}
	\end{cases}\]
	strictly decreases to ${1\over a}$, there exists $x_0 > 0$ such that 
	\begin{equation}
	\label{eq:i0}
	\mu(i_0) >  {r(\mbH)-a(\mbH)+x_0 \over x_0a(\mbH)}.
	\end{equation}
	Then, applying the inequality in (\ref{eq:base}) to $B \equiv \maV \setminus \{i_0\}$, we readily obtain that a.s. for all $n$, 
	\begin{multline}
	{r(\mbH)+x_0\over a(\mbH)}A_n\left(\maV\setminus\{i_0\}\right)\\
	\begin{aligned}
	&\ge {r(\mbH)+x_0\over a(\mbH)}\left(\sum\limits_{H \in \maH(i_0)} \left|H -1\right| M_n\left(H\right)+\sum\limits_{H \in \overline{\maH(i_0)}} \left|H \right| M_n\left(H\right)\right)\\
	&\ge \left(r(\mbH)+x_0-{r(\mbH)+x_0\over a(\mbH)}\right)M_n\left(\maH(i_0)\right) + (r(\mbH)+x_0)M_n\left(\overline{\maH(i_0)}\right). \label{eq:conditionM}
	\end{aligned}
	\end{multline}
	Likewise, applying the equality of (\ref{eq:base}) to $\{i_0\}$ and then $\maV\setminus \{i_0\}$ also yields to 
	\begin{multline*}
	X_n\left(\maV\setminus \{i_0\}\right)+\left(x_0 +1 - {r(\mbH)+x_0 \over a(\mbH)}\right)X_{n}(i_0)\\
	\shoveleft{
		= A_n\left(\maV\setminus \{i_0\}\right)-\sum\limits_{H \in \maH(i_0)} \left| H -1\right| M_n\left(H\right)-\sum\limits_{H \in \overline{\maH(i_0)}} \left| H\right| M_n\left(H\right)}\\
	\shoveright{+\left(x_0 +1- {r(\mbH)+x_0 \over a(\mbH)}\right)\left(A_n(i_0)-M_n\left(\maH(i_0)\right)\right)}\\
	\shoveleft{> A_n\left(\maV\setminus \{i_0\}\right)+\left(x_0 +1- {r(\mbH)+x_0 \over a(\mbH)}\right)A_n(i_0)}\\
	-\left(r(\mbH)+ x_0 - {r(\mbH)+x_0 \over a(\mbH)}\right)M_n(\maH(i_0)) - (r(\mbH)+x_0)M_n\left(\overline{\maH(i_0)}\right).
	\end{multline*}
	Combining this with (\ref{eq:conditionM}), implies that a.s. for all $n$, 
	\begin{equation*}
	X_n\left(\maV\right)+\left(x_0 - {r(\mbH)+x_0 \over a(\mbH)}\right)X_{n}(i_0) 
	>  \left(1- {r(\mbH)+x_0 \over a(\mbH)}\right)A_n\left(\maV\right)+x_0A_n(i_0).
	\end{equation*}
	Therefore we have that 
	\begin{equation}
	\label{eq:boundN3}
	\limsup_n {1\over n}\left(X_n\left(\maV\right)+\left(x_0 - {r(\mbH)+x_0 \over a(\mbH)}\right)X_{n}(i_0)\right) \ge 
	1- {r(\mbH)+x_0 \over a(\mbH)}+x_0\mu(i_0),
	\end{equation}
	hence the chain {$\suite{W_n}$} is transient since the right-hand side of the above is positive from (\ref{eq:i0}).  
	
	\medskip
	It remains to check that in the case where the hypergraph is $r$-uniform, the model cannot be stable whenever 
	$\mu(i_0) \ge {1\over a(\mbH)}={1\over r}$ 
	for some $i_0\in \maV$. For this, notice that, as $r(\mbH)=a(\mbH)=r$ a weak inequality holds true in (\ref{eq:i0}) for any $x_0>0$. 
	Then, it readily follows from (\ref{eq:boundN3}) that for any $x_0$, 
	$$\limsup_n {1\over n}\left(X_n\left(\maV\right)+\left(x_0 - {r(\mbH)+x_0 \over a(\mbH)}\right)X_{n}(i_0)\right) \ge 0,$$
	and we conclude, as in the proof of Proposition \ref{prop:Ncond1}, that the chain $\suiten{W_n}$ is at best null recurrent. 
\end{proof}

\section{{Non-stabilizable hypergraphs}}
\label{sec:unstable} 
Having {Corollary \ref{cor:StabN+N++}, and} Propositions \ref{prop:Ncond1}, \ref{prop:Ncond2} and \ref{prop:Ncond3} in hand, one can identify classes of hypergraphs 
$\mbH$ such that $(\mbH,\Phi,\mu)$ has an empty stability region for any admissible $\Phi$.

We start with the following elementary observation, 
\begin{proposition}
	\label{prop:isolated}
	If a hyperedge of $\mbH=(\maV,\maH)$ contains two isolated nodes, i.e., there exist $H\in \maH$ and $i,j \in H$ such that $d(i)=d(j)=1$, then 
	the model cannot be stable, i.e., $\textsc{Stab} (\mbH,\Phi)=\emptyset$ for any admissible $\Phi.$ 
\end{proposition}

\begin{proof}
	Let $\mu \in  \mathscr N^{\tiny{+}}_1(\mbH)$. Then, considering successively the sets $\{i\}$ and $\{j\}$, 
	as $j \in H\cap \bar{\{i\}}$ and $i \in H \cap \bar{\{j\}}$ we obtain that $\mu(i) < \mu(j)\mbox{ and }\mu(i) > \mu(j),$ an absurdity. 
\end{proof}

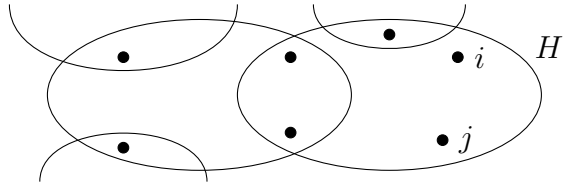
\begin{figure}[htp]\begin{center}
		\def\firstellip{(1, 2) ellipse [x radius=2cm, y radius=1cm, rotate=180]}
		\def\secondellip{(3.5, 2cm) ellipse [x radius=2cm, y radius=1cm, rotate=180]} 
		\def\thirdellip{(-1.5, 2) ellipse [x radius=2cm, y radius=1cm, rotate=-180]}
		\begin{tikzpicture}
		
		
	\filldraw (-2.5,1.3) circle (2pt) node [right] {};
	\draw[-] (-3.6,0.85) to[bend left=90] (-1.4,0.85);
	\filldraw (-2.5,2.5) circle (2pt) node [right] {};
	\draw[-] (-4,3.2) to[bend right=90] (-1,3.2);
	\filldraw (-0.3,2.5) circle (2pt) node [right] {};
	\filldraw (-0.3,1.5) circle (2pt) node [right] {};
	\filldraw (1.9,2.5) circle (2pt) node [right] {$\,i$};
	\draw[-] (0,3.2) to[bend right=90] (2,3.2);
	\filldraw (1,2.8) circle (2pt) node [right] {};
	\filldraw (1.7,1.4) circle (2pt) node [right] {\,$j$};
	\draw \firstellip node [label={[xshift=2.1cm, yshift=0.2cm]$H$}] {};
	\draw \thirdellip node [label={[xshift=-2.0cm, yshift=-0.8cm]}] {};
		\end{tikzpicture} 
		\caption{\label{fig:Ex0} {Any hypergraph with two isolated nodes is non-stabilizable}.}
	\end{center}
\end{figure}

\subsection{Stars}
\label{subsec:starcycles}
First recall that, as for any bipartite graph (see Theorem 2 in \cite{MaiMoy16}), graphical matching models on 
trees are always unstable. This is true in particular if the matching graph is a ``star'', i.e., a connected graph in which all but one vertices are of degree one. 
The following two results can be seen as generalizations of this fact to hypergraphical models, 

\begin{proposition}
	\label{prop:superstar}
	If {an} $r$-uniform hypergraph $\mbH=(\maV,\maH)$ has transversal number $\tau(\mbH)=1$, then it is non-stabilizable. 
\end{proposition}

\begin{proof}
	Fix $\Phi$ and $\mu$ in $\textsc{Stab} (\mbH,\Phi)$. Let $T$ be a transversal of cardinality $1$, i.e., $T=\{i_0\}$, where the vertex $i_0$ belongs 
	to all hyperedges in $\maH$. Then from Proposition \ref{prop:Ncond3}, we have that $\mu(i_0) < 1/a(\mbH) = 1/r$.  
	However, Proposition \ref{prop:Ncond2} implies that $\mu(i_0)>1/r(\mbH)=1/r,$ an absurdity. 
\end{proof}

In other words, any uniform hypergraph whose hyperedges all contain the same node $i_0$ cannot make the corresponding system stable. Moreover, 

\begin{proposition} \label{prop:IpHi}
	Suppose that there exists a subset $B \subset \maV$ in the hypergraph $\mbH=(\maV,\maH)$ such that: 
	\begin{itemize}
		\item all hyperedges of $\maH(B)$ contain at least one node of degree 1; 
		\item at least one of these nodes of degree 1 lies outside of $B$.
	\end{itemize} 
	Then $\mbH$ is non-stabilizable. 
\end{proposition}

\begin{proof}
	Let $k=|\maH(B)|$, i.e., the number of hyperedges intersecting with $B$. Denote by $H_{1},...,H_{k}$ these intersecting 
	hyperedges, and for any $l \in\llbracket 1,k \rrbracket$, by $i_l\in \maV$, a node of degree one belonging to $H_{l}$. 
	Observe that the nodes $i_1,...,i_k$ are not necessarily distinct. On the one hand, for 
	any $l\in\llbracket 1,k \rrbracket$ we have that 
	\begin{equation*}
	X_n(i_l)=A_n(i_l)-M_n(H_{l}).
	\end{equation*}
	Thus, applying again the inequality in (\ref{eq:base}) we get that for all $n$, 
	\begin{equation*}A_n(B)\geq \sum\limits_{l=1}^k |H_{l}\cap B| M_n(H_l)
	=\sum\limits_{l=1}^{k}|H_l\cap B|(A_n(i_l)-X_n(i_l)).\end{equation*}
	This entails that if $\mu \in \mathscr N_1^{\tiny{++}}(\mathbb{H})$, 
	\begin{equation*}
	\limsup\limits_{n\to\infty}{1 \over n}  \sum\limits_{l=1}^{k}|H_l\cap B|X_n(i_l)\geq\sum\limits_{l=1}^{k}|H_l\cap B|\mu(i_l)-\mu(B) \ge 0.
	\end{equation*}
	If the above inequality is strong, then the chain $\suiten{W_n}$ is transient. If the inequality is weak, then as above we can stochastically lower-bound the chain by a zero-drift chain $\suiten{\tilde Y_n}$, defined by 
	\begin{equation*}
	\tilde Y_n =\left(A_n(B)\;-\;\sum\limits_{l=1}^{k}|H_l\cap B|A_n(i_l)\right),\quad n\in\N,
	\end{equation*}
	which is not identically null from the assumption that at least one of the nodes $i_l$, $l=1,...,k$ is not an element of $B,$ which concludes the proof. 
\end{proof}

\begin{ex}\label{ex:chaine}
	\rm
	Any hypergraph $\mbH=(\maV,\maH)$ such that there exist two hyperedges $H_1$ and $H_2$ 
	with $H_1 \cap H_2 \ne \emptyset$ and two nodes $i_1 \in H_1 \cap \overline{H_2}$, $i_2 \in H_2 \cap \overline{H_1}$ and 
	$d\left(i_1\right)=d\left(i_2\right)=1$ is non-stabilizable (see Figure \ref{fig:Ex2}). 
	To see this, take $B = H_1 \cap H_2$ in Proposition \ref{prop:IpHi}. 
	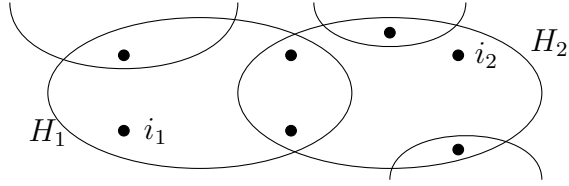
\begin{figure}[htp]\begin{center}
			\def\firstellip{(1, 2) ellipse [x radius=2cm, y radius=1cm, rotate=180]}
			\def\secondellip{(3.5, 2cm) ellipse [x radius=2cm, y radius=1cm, rotate=180]} 
			\def\thirdellip{(-1.5, 2) ellipse [x radius=2cm, y radius=1cm, rotate=-180]}
			\begin{tikzpicture}
			
			
			\filldraw (-2.5,1.5) circle (2pt) node [right] {$\;i_1$};
			\filldraw (-2.5,2.5) circle (2pt) node [right] {};
			\draw[-] (-4,3.2) to[bend right=90] (-1,3.2);
			\filldraw (-0.3,2.5) circle (2pt) node [right] {};
			\filldraw (-0.3,1.5) circle (2pt) node [right] {};
			\filldraw (1.9,2.5) circle (2pt) node [right] {$\,i_2$};
			\draw[-] (0,3.2) to[bend right=90] (2,3.2);
			\filldraw (1,2.8) circle (2pt) node [right] {};
			\filldraw (1.9,1.25) circle (2pt) node [right] {};
			\draw[-] (1,0.85) to[bend left=90] (3,0.85);
			\draw \firstellip node [label={[xshift=2.1cm, yshift=0.2cm]$H_2$}] {};
			\draw \thirdellip node [label={[xshift=-2.0cm, yshift=-1cm]$H_1$}] {};
			\end{tikzpicture} 
			\caption{\label{fig:Ex2} Two intersecting hyperedges containing each, an isolated node outside of their intersection, make the system unstable.}\end{center}
	\end{figure}
\end{ex}

{\begin{remark}[About the DI condition in  \cite{GW14}]
		\label{rem:GW}\em
		Most results of \cite{GW14} hold under the Assumption 1 therein, stating that the Dedicated Item DI condition is satisfied; namely, each hyperedge contains an isolated 
		node. The above example shows that any matching model $(\mbH,\Phi,\mu)$ on a hypergraph $\mbH$ satisfying the DI condition, is unstable 
		for any admissible $\Phi$ (the case where $\mbH$ contains a single hyperedge $H$ is trivial). 
\end{remark}}


\subsection{$r$-partite hypergraphs}
We now turn to hypergraphical generalizations of bipartite graphs. 



\begin{proposition}
	\label{prop:Hbipartiteunstable}
	Any $r$-uniform bipartite hypergraph $\mathbb{H}=(\maV,\maH)$ is non-stabilizable. \end{proposition}
\begin{proof}
	Applying (\ref{eq:base}) successively to $V_1$ and $V_2$ readily implies that for all $n$, 
	\begin{equation*}
	X_n(V_1) =A_n(V_1)-M_n(\maH)\ge 0 \quad \mbox{ and }\quad
	X_n(V_2) = A_n(V_2)-(r-1) M_n(\maH)\geq 0,
	\end{equation*} 
	and thus \[X_n(V_1) \geq A_n(V_1)-\displaystyle\frac{1}{r-1}A_n(V_2).\]
	Then, the usual SLLN-based argument implies that the model cannot be stable unless $\mu(V_2) \geq (r-1)\mu(V_1)$. 
	But as $\mu(V_1)+\mu(V_2)=1$ we have that $\mu(V_1)\leq\displaystyle\frac{1}{r}$, 
	hence $\mu \not\in \mathscr N_2(\mbH)$ since $V_1$ is a transversal. 
\end{proof}


\begin{figure}[htp!]
	\begin{center}
		\def\firstellip{(4, 3.9) ellipse [x radius=5cm, y radius=0.8cm, rotate=180]}
		\def\secondellip{(2,8) ellipse [x radius=5cm, y radius=0.77cm, rotate=65]} 
		\def\thirdellip{(4, 8.2) ellipse [x radius=4.33cm, y radius=0.7cm, rotate=90]} 
		\def\fourthellip{(6, 8cm) ellipse [x radius=5cm, y
			radius=0.77cm, rotate=-65]}	
		\def\fifthdellip{(3,6) ellipse [x radius=4.33cm, y radius=0.6cm, rotate=33]} 
		\def\sixthdellip{(5,6) ellipse [x radius=4.33cm, y radius=0.6cm, rotate=-33]} 
		
		\begin{tikzpicture}[thick, scale=0.5]
		\filldraw 
		(4,12) circle (2pt) node [label={[xshift=0.41cm, yshift=-0.2cm]$1$}] {};
		\filldraw 
		(4,6.5) circle (2pt) node [label={[xshift=0cm, yshift=-0.25cm]$2$}] {};
		\filldraw 
		(4,4.3) circle (2pt) node [label={[xshift=0.1cm, yshift=-0.75cm]$4$}] {};
		\filldraw 
		(0.1,4) circle (2pt) node [label={[xshift=-0.2cm, yshift=-0.9cm]$6$}] {};
		\filldraw 
		(2,8) circle (2pt) node [label={[xshift=-0.2cm, yshift=-0.75cm]$5$}] {};
		\filldraw 
		(7.9,4) circle (2pt) node [label={[xshift=0.4cm, yshift=-0.85cm]$3$}] {};
		(
		\filldraw 
		(6,8) circle (2pt) node [label={[xshift=0.27cm, yshift=-0.7cm]$7$}] {};
		\draw \firstellip node [label={[xshift=4cm, yshift=-0.5cm]}] {};
		\draw \secondellip node [label={[xshift=4.5cm, yshift=-0.6cm]}] {};
		\draw \thirdellip node [label={[xshift=4.5cm, yshift=-0.5cm]}] {};
		\draw \fourthellip node [label={[xshift=3cm, yshift=-2.5cm]}] {};
		\draw \fifthdellip node [label={[xshift=3cm, yshift=-2.5cm]}] {};
		\draw \sixthdellip node [label={[xshift=3cm, yshift=-2.5cm]}] {};
		\end{tikzpicture} 
		\caption{\label{fig:FanoPlane}The Fano plane minus the hyperedge $\{4,5,7\}$.}
	\end{center}
\end{figure}
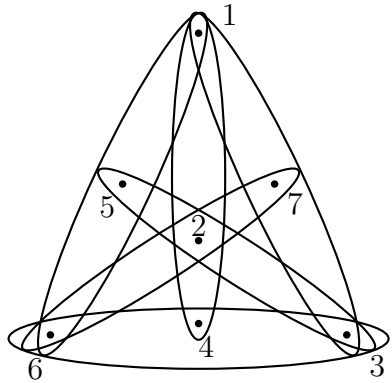

\begin{ex}\rm
	{The so-called Fano plane is a well-known object in discrete geometry. It is the smallest {\em projective plane}, namely, the smallest set of {\em points} and {\em lines} such that 
		any two points share a line, any two lines intersect at a single point, and on every line lies the same number of points. In the settings of hypergraphs (points being nodes and 
		{lines} being hyperedges), the Fano plane is thus the smallest uniform hypergraph $\mbH$ in which each pair of nodes belongs to a single hyperedge, and each pair of hyperedges  intersects at a single node. It can be checked that $\mbH=(\maV,\maH)$ is of order 7, for $\maV=\llbracket 1,7 \rrbracket$ and {e.g.}}
	$$\maH=\left\{\{1,2,4\},\{1,5,6\},\{1,3,7\},\{2,3,5\},\{4,5,7\},\{4,3,6\},\{6,2,7\}\right\}.$$ 
	Supported by simulations, we conjecture that Fano planes are stabilizable. However, 
	if $\mathbb{H}'=(\maV,\maH')$  is the subhypergraph defined by $\maH'=\maH\backslash H$, where $H$ is an arbitrary hyperedge of $\maH$, then 
	it is easily seen that $\mathbb{H}'$ is a 3-uniform bipartite hypergraph with $V_1=H$ and $V_2=\maV\backslash H$. So we deduce from 
	Proposition \ref{prop:Hbipartiteunstable} that $\mathbb{H}'$ is non-stabilizable. A Fano plane minus the hyperedge 
	$\{4,5,7\}$ is represented in Figure \ref{fig:FanoPlane}. 
\end{ex}


{We know from Theorem 2 in \cite{MaiMoy16} that 
	bipartite graphs are not stabilizable. The next result shows that this can be generalized to $r$-partite hypergraphs (which generalize bipartite graphs - see Remark \ref{remark:bipartite}), 
	\begin{proposition}
		\label{prop:Hr-partiteunstable}
		Any $r$-partite hypergraph $\mbH$ is non-stabilizable. 
\end{proposition}}
\begin{proof}
	As in the above proof we get that for any $i\ne j$ and any $n$, 
	\begin{equation*}
	X_n(V_j) =A_n(V_j)-M_n(\maH)\ge 0 \quad \mbox{ and }\quad X_n(V_i) =A_n(V_i)-M_n(\maH)\ge 0,
	\end{equation*}
	implying that $X_n(V_i)\ge A_n(V_i)-A_n(V_j)$, and in turn, that the model cannot be stable unless $\mu(V_i) \le \mu(V_j)$. 
	By symmetry, this implies that $\mu(V_i)=\mu(V_j)$. As the $V_i$'s are disjoint, we thus have that $\mu(V_i)=1/r$ for all $i$. 
	Thus, as any $V_i$ is a transversal of $\mbH$, $\mu $ is not an element of $\mathscr N_2(\mbH)$. 
\end{proof}

It is well known (see \cite{Hall35} for the particular case of graphs, and the general result in \cite{PH95}) that Hall's condition is necessary and sufficient for the existence of a perfect matching on $\mbH$, i.e., a spanning subhypergraph of $\mbH$ in which all nodes have degree 1, in the case where the hypergraph is balanced, i.e., it does not contain any odd 
strong cycle. It is intuitively clear that the construction of stable stochastic matching models on hypergraphs is somewhat reminiscent of that of perfect matchings on a growing hypergraph that replicates the matching hypergraph a large number of times in the long run (in the case of graphs, see the discussion in Section 7 of \cite{MoyPer17}). This connexion has a simple illustration in the next Proposition, which provides a family of probability measures, naturally including the uniform measure on $\maV$, that cannot stabilize a matching model on the hypergraph $\mbH$ unless the latter satisfies Hall's condition. In what follows we denote for any $\mbH=(\maV,\maH)$ and any measure $\mu\in\mathscr M(\maV)$, 
\begin{equation}
\label{eq:defmuminmax}
\mumin=\min\left\{\mu(i)\,:\,i\in \maV\right\}\quad\mbox{ and }\quad\mumax=\max\left\{\mu(i)\,:\,i\in \maV\right\}.
\end{equation}
\begin{proposition}
	\label{pro:Hall}
	For any hypergraph $\mathbb{H}=(\maV,\maH)$ that violates Hall's condition, 
	any matching policy $\Phi$ and any $\mu\in \mathscr M(\maV)$ such that 
	\begin{equation}
	\label{eq:condHall}
	{\mumini \over \mumaxi} > {\left\lfloor {q(\mbH) +1 \over 2}\right\rfloor -1 \over \left\lfloor {q(\mbH) +1 \over 2}\right\rfloor}, 
	\end{equation}
	the model $(\mbH,\Phi,\mu)$ is {unstable}. In particular,  $(\mbH,\Phi,\mu_{\textsc{u}})$ is {unstable} for $\mu_{\textsc{u}}$ the uniform distribution on $\maV$. 
\end{proposition}

\begin{proof}
	Fix $\mbH$, $\Phi$, and a measure $\mu$ satisfying (\ref{eq:condHall}). We first show that $\mu$ is monotonic with respect to the counting measure on $\maV$, i.e., 
	\begin{equation}
	\label{eq:monotonemu}
	\forall E,F \subset \maV,\quad |E| < |F| \Longrightarrow \mu(E) < \mu(F). 
	\end{equation}
	Let $E$ and $F$ be such that $|E| < |F|$, and let $k = |F|$. Let also $\alpha$ be a bijection from $\llbracket 1,q(\mbH) \rrbracket$ to $\maV$ such that 
	\begin{equation}
	\label{eq:defalpha}
	\mumin=\mu(\alpha(1)) \le \mu(\alpha(2)) \le ... \le \mu(\alpha(q(\mbH)))=\mumax,
	\end{equation}
	in other words 
	$\left(\mu(\alpha(1)),\mu(\alpha(2)),...,\mu(\alpha(q(\mbH)))\right)$ is an ordered (in increasing order) version of the family $\left\{\mu(i);\, i \in \maV\right\}$. 
	As $|E|\le k-1$ we clearly have 
	\begin{equation}
	\label{eq:mad0}
	\mu(F) - \mu(E) \ge \sum_{i=1}^k \mu(\alpha(i)) - \sum_{i=q-k+2}^q \mu(\alpha(i)). 
	\end{equation}
	First, if $k \le \left\lfloor {q(\mbH) +1 \over 2}\right\rfloor$, (\ref{eq:condHall}) entails that $k\mumin >(k-1) \mumax,$ whence 
	\begin{equation}
	\label{eq:mad1}
	\sum_{i=1}^k \mu(\alpha(i)) - \sum_{i=q-k+2}^q \mu(\alpha(i)) \ge k \mumin - (k-1)\mumax >0,
	\end{equation}
	If $k>\left\lfloor {q(\mbH) +1 \over 2}\right\rfloor$, then the index sets $\llbracket 1,k \rrbracket$ and  $\llbracket q-k+2, q \rrbracket$ intersect precisely 
	on $\llbracket q-k+2,k \rrbracket$. 
	Thus 
	\begin{align}
	\sum_{i=1}^k \mu(\alpha(i)) - \sum_{i=q-k+2}^q \mu(\alpha(i)) &= \sum_{i=1}^{q-k+1} \mu(\alpha(i)) - \sum_{i=k+1}^q \mu(\alpha(i))\nonumber\\
	&\ge  (q-k+1)\mumin - (q-k)\mumax >0, \label{eq:mad2}
	\end{align}
	where the last inequality follows, as in (\ref{eq:mad1}), from the fact that $q-k+1\le \left\lfloor {q(\mbH) +1 \over 2}\right\rfloor$. 
	Gathering (\ref{eq:mad0}) with (\ref{eq:mad1}-\ref{eq:mad2}) concludes the proof of (\ref{eq:monotonemu}) in all cases. 
	
	Now fix $V_2$ and $V_1$ such that $|H\cap V_2|\geq |H\cap V_1|$ for any $H\in\maH$, and $|V_2|<|V_1|$ which from (\ref{eq:monotonemu}), implies that 
	$\mu(V_2) < \mu(V_1)$. Then, 
	applying again (\ref{eq:base}) to $V_2$ and $V_1$ we get that 
	\begin{align*}
	\label{eq:compareXY}
	X_n(V_2)+X_n(V_1)
	&\ge A_n(V_2)+A_n(V_1)-2\sum\limits_{H\in\maH}\left|H \cap V_2\right|M_n(H)\\
	&\ge A_n(V_2)+A_n(V_1)-2A_n(V_2),
	\end{align*}
	thus, from the usual argument, the model cannot be stable unless $\mu(V_2)\geq \mu(V_1)$, a contradiction.
\end{proof}



\subsection{Cycles}
\label{subsec:cycles}


\begin{proposition}
	\label{prop:cycles}
	Any $r$-uniform $\ell$-cycle of order $q$ such that $r$ divides $q$, is non-stabilizable. 
\end{proposition}
\begin{proof}
	The partition $V_1,V_2,\cdots,V_r$ of $\maV$ defined by 
	$$V_i=\left\lbrace v_{i+(j-1)r}\;;\,j\in \llbracket 1,q/r\rrbracket \right\rbrace,$$
	{satisfies Proposition \ref{prop:Hr-partiteunstable}.} 
\end{proof}

\noindent Figure \ref{fig:cycle} shows a $3$-uniform $2$-cycle of order $12$ and $3$-uniform $2$-cycle of order $6$.


\section{Stable systems}
\label{sec:stable}
We show hereafter that stable matching models on hypergraphs exist. With a view to showing how stability can be shown in concrete examples, 
we provide hereafter two case studies of simple hypergraphs, on 
which a stable stochastic matching model can be defined: complete $3$-uniform hypergraphs, and subhypergraphs of the latter where several hyperedges are erased. 

\subsection{Complete $3$-uniform hypergraphs}
\label{subsec:complete}
We first consider the case of a complete $3$-uniform hypergraph $\mbH$, an example of which for $q(\mbH)=4$ is represented in Figure \ref{fig:completeHyper} (left). 
We show that, in this case, the necessary condition given in Proposition 
\ref{prop:Ncond3} is also sufficient, 

\begin{theorem}
	\label{thm:stable3uniform}
	Let $\mbH=(\maV,\maH)$ be a complete $3$-uniform hypergraph of order $q(\mbH)\ge 4$. 
	Then, for any admissible policy $\Phi$ we have,
	\[\textsc{Stab}(\mbH,\Phi) = \mathscr N_3^{\tiny{-}}(\mbH),\]
	that is, the model $(\mbH,\Phi,\mu)$ is stable if and only if $\mu(i) < 1/3$ for any $i\in \maV$. 
\end{theorem}


\begin{proof}
	Necessity of the condition being shown in Proposition \ref{prop:Ncond3}, only the sufficiency remains to be proven. 
	Suppose that $\mu(i)<1/3$ for any $i\in \maV$, and fix $\alpha$ such that $\max_{i\in \maV} \mu(i) < \alpha <1/3$. 
	Define the planar Markov chain $\suite{U^{\alpha}_n}$ having the following transitions on $\N^2$, 
	\[\left\{\begin{array}{llll}
	\mbox{First axis:} \quad  &P^\alpha_{(x,0),(x+1,0)} &= \alpha,&\,x\in\N^+,\\
	&P^\alpha_{(x,0),(x,1)} &= 1-\alpha,&\,x\in\N^+,\\
	\mbox{Second axis:}  \quad  &P^\alpha_{(0,y),(0,y+1)} &= \alpha,&\,y\in\N^+,\\
	&P^\alpha_{(0,y),(1,y)} &= 1-\alpha,&\,y\in\N^+,\\
	\mbox{Interior:}  \quad & P^\alpha_{(x,y),(x+1,y)} &= \alpha,&\,x,y\in\N^+,\\
	&P^\alpha_{(x,y),(x,y+1)} &= \alpha,&\,x,y\in\N^+,\\
	&P^\alpha_{(x,y),(x-1,y-1)} &= 1- 2\alpha,&\,x,y\in\N^+,
	\end{array}\right.\]
	and arbitrary transitions from $(0,0)$ to any element of $\N^2$. (These transitions are represented in Figure \ref{Fig:transU} below). 
	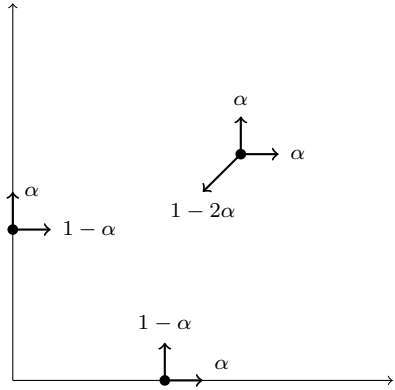
\begin{figure}[htb]
		\begin{center}
			\begin{tikzpicture}
			\draw[->] (0,0) -- (5,0) ;
			\draw[->] (0,0) -- (0,5);
			\fill (2,0) circle (2pt);
			\draw[->, thick] (2,0) -- (2.5,0) node [above right]{\scriptsize{$\alpha$}};
			\draw[->, thick] (2,0) -- (2,0.5) node [above]{\scriptsize{$1-\alpha$}};
			%
			\fill (3,3) circle (2pt);
			\draw[->, thick] (3,3) -- (3.5,3) node [right]{\scriptsize{$\alpha$}};
			\draw[->, thick] (3,3) -- (3,3.5) node [above]{\scriptsize{$\alpha$}};
			\draw[->, thick] (3,3) -- (2.5,2.5) node [below]{\scriptsize{$1-2\alpha$}};
			%
			\fill (0,2) circle (2pt);
			\draw[->, thick] (0,2) -- (0.5,2) node [right]{\scriptsize{$1-\alpha$}};
			\draw[->, thick] (0,2) -- (0,2.5) node [right]{\scriptsize{$\alpha$}};
			%
			%
			%
			
			\end{tikzpicture}
			\caption{\label{Fig:transU} Auxiliary Markov chain of the complete 3-uniform hypergraph.}
		\end{center}
		
	\end{figure}
	
	Denote by $\Delta=(\Delta_x,\Delta_y)$, $\Delta'=(\Delta'_x,\Delta'_y)$ and $\Delta''=(\Delta''_x,\Delta''_y)$, the mean (horizontal and vertical) 
	drifts of the chain $\suiten{U^\alpha_n}$, respectively on the interior, on the first and on the second axis, in a way that 
	\[\left\{\begin{array}{lll}
	\mbox{First axis:} \quad  &\Delta'_x=\alpha, \quad &\Delta'_y=1-\alpha;\\
	\mbox{Second axis:}  \quad  &\Delta''_x=1-\alpha, \quad &\Delta''_y=\alpha;\\
	\mbox{Interior:}  \quad & \Delta_x=3\alpha-1, \quad &\Delta_y=3\alpha-1.
	\end{array}\right.\]
	Thus, $\Delta_x <0$ and $\Delta_y<0$. Also, we have that 
	\[\Delta_x\Delta'_y - \Delta_y\Delta'_x = \Delta_x\Delta''_y - \Delta_y\Delta''_x = (3\alpha-1)(1-2\alpha) <0,\]
	so we can apply Theorem \ref{theo:TheoremFayolMalMain}, part (a), to claim that the Markov chain $\{U^\alpha_n\}$ is positive recurrent. 
	Specifically, it can be checked that, setting $u={1-3\alpha\over 2}>0$, for any $w$ such that ${{3\alpha-1}} < w <{(3\alpha - 1)\alpha \over 1-\alpha}<0$ we have that 
	\begin{equation}
	\label{eq:deltas}
	\begin{cases}
	2u \Delta_x + w \Delta_y &< 0,\\
	2u \Delta_y + w \Delta_x &< 0,\\
	2u \Delta'_x + w \Delta'_y &< 0,\\
	2u \Delta''_y + w \Delta_x &< 0.
	\end{cases}
	\end{equation}
	Second, as $4u^2> w^2$ the quadratic form $Q:(x,y) \mapsto ux^2 +uy^2+wxy$ is positive definite. 
	Then, in view of Lemma \ref{lem:UsInProofTheorem3.3.1}, it follows from (\ref{eq:deltas}) that, defining the mapping 
	\[L^\alpha : 
	\begin{cases}
	\N^2 &\longrightarrow \R_+ \\
	(x,y) &\longmapsto  \sqrt{Q(x,y)}=\sqrt{ux^2 +uy^2+wxy}, 
	\end{cases}\]
	we have that for some compact set $\mathcal K^\alpha\subset \N^2$, for any $(x,y) \in \overline{\mathcal K^\alpha}$, 
	\begin{equation}
	\label{eq:lyapua}
	\esp{L^\alpha\left(U^\alpha_{n+1}\right) - L^\alpha(U^\alpha_n) \mid U^\alpha_n = (x,y)} <0. 
	\end{equation}
	
	Now, as $\mbH$ is complete $3$-uniform, the states of the Markov chain $\suite{X_n}$ have at most two non-zero coordinates, 
	in other words, its state space is 
	\[\mathcal E = \Bigl\{\grx=(x_1,...,x_q) \in\N^q \,:\, x_ix_jx_k=0 \mbox{ for any distinct }i,j,k \in \llbracket 1,q \rrbracket\Bigl\}.\]
	Define the mapping 
	\[L: 
	\left\{\begin{array}{ll}
	\mathcal E &\longrightarrow \R_+\\
	\grx &\longmapsto \left\{\begin{array}{ll}
	0      &\mbox{ if } \grx=\mathbf 0,\\ 
	L^\alpha((x,0))    &\mbox{ if $\grx=x.\gre_i$, for some $x>0$, $i\in \maV$},\\
	L^\alpha((x,y))  &\mbox{ if $\grx=x.\gre_i+y.\gre_j$, for some $x,y>0$, $i\ne j$,}  
	\end{array}\right.
	\end{array}\right.\]    
	where the above definition is unambiguous due to the fact that $L^\alpha$ is a symmetric form on $\N^2$. 
	Also define the compact set 
	\[\mathcal K = \left\{\grx :=x.\gre_i+y.\gre_j\in \mathcal E\,:\, (x,y) \in\mathcal K^\alpha\right\}.\]
	Then, first, if $\grx\in\bar{\mathcal K}\cap \mathcal E$ is such that $\grx=x.\gre_i+ y.\gre_j$ for some $x,y>0$ and $i,j \in \maV$, $i\ne j$, we get that 
	\begin{multline*}
	\esp{L\left(X_{n+1}\right) - L(X_n) \mid X_n = \grx} \\
	\shoveleft{=(1-\mu(i)-\mu(j))\left(L\left(\grx-\gre_i-\gre_j\right)-L(\grx)\right)}\\
	\shoveright{+ \mu(i) \left(L\left(\grx+\gre_i\right)-L(\grx)\right) +  \mu(j) \left(L\left(\grx+\gre_j\right) - L(\grx)\right)}\\
	\shoveleft{=(1-\mu(i)-\mu(j))\left(L^\alpha\left(x-1,y-1\right)-L^\alpha(x,y)\right) }\\
	\shoveright{+ \mu(i) \left(L^\alpha\left(x+1,y\right)-L^\alpha(x,y)\right)+  \mu(j) \left(L^\alpha\left(x,y+1\right) - L^\alpha(x,y)\right)}\\
	\shoveleft{{< (1-2\alpha)}\left(L^\alpha\left(x-1,y-1\right)-L^\alpha(x,y)\right) }\\
	\shoveright{+ \alpha \left(L^\alpha\left(x+1,y\right)-L^\alpha(x,y)\right)+  \alpha \left(L^\alpha\left(x,y+1\right) - L^\alpha(x,y)\right)}\\
	= \esp{L^\alpha\left(U^\alpha_{n+1}\right) - L^\alpha(U^\alpha_n) \mid U^\alpha_n = (x,y)},
	\end{multline*}
	where, in the inequality above, we used the facts that $L^\alpha$ is non-decreasing in its first and second variables, and such that
	$L^\alpha\left(x-1,y-1\right){<}\, L^\alpha(x,y)$. 
	Likewise, if $\grx\in \bar{\mathcal K} \cap \mathcal E$ is such that $\grx=x.\gre_i$ for some $x>0$ and $i\in \maV$, we have that 
	\begin{multline*}
	\esp{L\left(X_{n+1}\right) - L(X_n) \mid X_n = \grx} \\
	\begin{aligned}
	&=\sum_{j\ne i}\mu(j)\left(L\left(\grx+\gre_j\right)-L(\grx)\right)+ \mu(i) \left(L\left(\grx+\gre_i\right)-L(\grx)\right)\\
	&=(1-\mu(i))\left(L^\alpha\left(x,1\right)-L^\alpha(x,0)\right)+ \mu(i) \left(L^\alpha\left(x+1,0\right)-L^\alpha(x,0)\right)\\
	&{< (1-\alpha)}\left(L^\alpha\left(x,1\right)-L^\alpha(x,0)\right)+ \alpha \left(L^\alpha\left(x+1,0\right)-L^\alpha(x,0)\right)\\
	&= \esp{ L^\alpha\left(U^\alpha_{n+1}\right) - L^\alpha(U^\alpha_n) \mid U^\alpha_n = (x,0)},
	\end{aligned}
	\end{multline*}
	remarking that $L^\alpha(x,1) { < }\, L^\alpha(x,0)$. Recalling that $X_n=[W_n]$ for all $n$, 
	using (\ref{eq:lyapua}) in both cases, we conclude using the Lyapunov-Foster Theorem \ref{the:LyapunovFosterTheorem} that the chain $\suite{W_n}$ is positive recurrent. 
\end{proof}


The complete 3-uniform $k$-partite hypergraphs generalize the complete $k$-partite graphs introduced in p.4 of \cite{MBM17}, also called {\em separable} graphs in 
\cite{MaiMoy16} and \cite{MoyPer17} - or {\em blow-ups} of the complete graph of order $k$ in some other references. Roughly speaking, a complete $3$-uniform $k$-partite hypergraph is a version of the complete $3$-uniform  hypergraph of order $k$, in which the $k$ nodes are replicated into several replicas, each of the $k$ sets of replicas forming an independent set $I_i$, such that all replicas of the same set do not share any 
hyperedge with each other, but all share hyperedges of size 3 with all other pairs of replicas belonging to two different other sets of replicas. 
Observe that in the particular case where all the sets $I_1,...,I_k$ are of cardinality 1 (i.e., there are no replica), the complete 3-uniform $k$-partite hypergraph is just the complete 3-uniform hypergraph of order $k$. We can then easily generalize the latter result, 

\begin{corollary}
	\label{cor:stable3unif}
	For $k\ge 4$, let $\tilde \mbH$ be a complete $3$-uniform $k$-partite hypergraph, and let $I_1,...,I_k$ be the corresponding partition into independent sets. 
	Then, for any admissible policy $\Phi$, 
	the model $(\tilde\mbH,\tilde\Phi,\tilde\mu)$ is stable if and only if $\tilde\mu(I_i) < 1/3$ for any $i\in \llbracket 1,k \rrbracket$. 
\end{corollary}
\begin{proof}
	The system has macroscopically (i.e., if we do not distinguish between items of classes that belong to the same independent set of the partition $I_1,...,I_k$) the same 
	behavior as the complete $3$-uniform hypergraph. Specifically, let $q$ be the order of the hypergraph $\tilde\mbH$, and define the mapping 
	\[\Psi:\left\{\begin{array}{ll}
	\N^q &\longrightarrow \N^k\\
	\tilde{\grx}=(\tilde x_1,...,\tilde x_q) &\longmapsto \grx=(x_1,...,x_k):\,\forall i\in\llbracket 1,k \rrbracket, x_i = \sum_{j\in \llbracket 1,q \rrbracket;\,j\in I_i} \tilde x_i.
	\end{array}\right.\]
	In words, $\Psi$ maps the detailed class content of the model, onto a class content where one puts altogether all the elements of classes belonging to the same 
	independent set of the partition $I_1,...,I_k$. Take $L$ as the Lyapunov function introduced in the previous proof. Fix an admissible policy $\tilde\Phi$ and a probability measure 
	$\tilde\mu\in\mathscr M(\tilde \maV)$, and let $\suite{\tilde X_n}$ be the class-content process of the model $(\tilde\mbH,\tilde\Phi,\tilde\mu)$. On another hand, let $\suite{X_n}$ 
	be the class-content process of the model $(\mbH,\Phi,\mu)$ defined on $\mbH=(\maV,\maH)$ the complete 3-uniform hypergraph 
	of order $k$, for an arbitrary matching policy $\Phi$ and a probability measure $\mu\in\mathscr M(\maV)$ such that $\mu(i)=\tilde\mu(I_i)$ for any $i\in\llbracket 1,k \rrbracket$. 
	Then, it is easily seen that $\suite{\tilde X_n}$ and $\suite{X_n}$ are connected by the following relation: for all $n$ and all $\tilde\grx\in\N^q$, 
	\[\esp{L\circ\Psi\left(\tilde X_{n+1}\right) - L\circ\Psi(\tilde X_n) \mid \tilde X_n = \tilde \grx} = \esp{L\left(X_{n+1}\right) - L(X_n) \mid X_n = \Psi(\tilde \grx)},\]
	and the argument in the proof of Theorem \ref{thm:stable3uniform} shows that the Markov chain $\suite{\tilde X_n}$ is positive recurrent 
	whenever $\tilde\mu(i)<1/3$, that is, $\mu(I_i)<1/3$, for all 
	$i\in\llbracket 1,k \rrbracket$. This concludes the proof.
\end{proof}


\subsection{{Incomplete $3$-uniform hypergraphs}}
As is shown in  Theorem \ref{thm:stable3uniform} and Corollary \ref{cor:stable3unif}, complete $3$-uniform hypergraphs and complete $3$-uniform $k$-partite hypergraphs are stabilizable for all matching policy $\Phi,$ for a large class of measures. 
We show hereafter that incomplete hypergraphs (in the sense defined hereafter) can also be stabilizable for a matching policy {\sc ml}, 

\begin{theorem}
	\label{thm:suff3unifincomplet} 
	Let $\mathbb H=(\maV,\maH)$ be a complete $3$-uniform hypergraph of {order} $q \ge 5$, and let 
	$\mbH'=(\maV,\maH')$ be 
	the $(3\textrm{-uniform})$ subhypergraph of $\mbH$ obtained by setting $\maH'=\maH\backslash \maJ$, where  $\maJ$ is a subset of $\maH$ containing disjoint hyperedges. Let $J$ be the union of the elements of $\maJ$.
	Then the model $(\mbH',\textsc{ml},\mu)$ is stable for any 
	$\mu$ in 
	\[\mathscr S(\mbH')=\biggl\{\mu\in\mathscr M(\maV):\left(\max\limits_{i\in J} \lambda_i(\mu)\vee \max\limits_{i\in \bar J} \nu_i(\mu)\right)<0\biggl\}\,\cap\,\mathscr N_2(\mathbb H')\, \cap\mathscr N^{\scriptsize{-}}_3(\mathbb H'),\]
	{where the $\lambda_i(\mu): i\in J$ and $\nu_{i}(\mu): i\in \bar J$ are defined respectively by (\ref{eq:Gdeflambdai}) and (\ref{eq:InCnu_i})}. 
\end{theorem}

\begin{proof}
	Fix $\Phi=\textsc{ml}$, and let $\mu\in {\mathscr S(\mbH')}$. For such $\mbH'$ the study of $\suite{Y_n}$ does not boil down to that of a planar Markov chain. 
	Instead, we study the embedded chain $\suiten{Y_n}=\suiten{X_{4n}}$, 
	and consider the following quadratic Lyapunov function,  
	\[Q:\left\{\begin{array}{ll}
	\N^q &\longrightarrow \R_+\\
	\grx &\longmapsto \sum_{i=1}^{q} (x_i)^2. 
	\end{array}\right.\]
	Fix $n\in\N$. We have the following alternatives given the value of the embedded chain $\suiten{Y_n}$ at time $n$,
	
	(i) First, for any $i\in \overline{J}$, and any integer $x_i\geq{2}$.  It follows that 
	for any $j\neq k\neq \ell\ne m \ne i\in\maV,$ the chain $\suite{Y_n}$ makes the transitions that we will present in Appendix (\ref{eq:GtransY1}) from state $Y_n=x_i.\gre_i$, then we deduce that 
	\begin{multline}\label{eq:GLyapY1}
	\Delta_i:=\esp{Q\left(Y_{n+1}\right) - Q\left(Y_n\right) | Y_n = x_i.\gre_i}\\
	\begin{aligned}
	= & \,(8x_i+16)\mu(i)^4 + 4(6x_i+10)\mu(i)^3\sum_{j} \mu(j) + 6(4x_i+8)\mu(i)^2\sum_{j} \mu(j)^2 \\
	& + 4(2x_i+10)\mu(i)\sum_{j} \mu(j)^3 + 16\sum_{j} \mu(j)^4 + 12(2x_i+1) \mu(i)^2 \sum_{\substack{j,k}} \mu(j)\mu(k)\\
	& + 12 \mu(i) \sum_{\substack{j,k}} \mu(j)^2\mu(k) -10(4x_i+4) \sum_{\substack{j,k,\ell}} \mu(j)^2\mu(k)\mu(\ell)\\
	& - 2(2x_i+5) \sum_{\substack{j,k,\ell}} \mu(j)^2\mu(k)\mu(\ell) -6(4x_i+4) \sum_{\substack{j,k}} \mu(j)^2\mu(k)^2\\
	& - 4(2x_i+5) \sum_{\substack{j,k}} \mu(j)^3\mu(k) + 24\mu(i)\sum_{\substack{j,k}}\mu(j)\mu(k)\mu(\ell)\\
	& -24(4x_i+4) \sum_{\substack{j,k,\ell,m}} \mu(j)\mu(k)\mu(\ell)\mu(m)=\lambda_i(\mu)x_i+\beta_i(\mu), 
	\end{aligned}
	\end{multline}
	for some bounded $\beta_i(\mu)$, and for
	\begin{multline}
	\label{eq:Gdeflambdai}
	\lambda_i(\mu)=\;8\mu^4(i) + 24\mu^3(i)\sum_{j\ne i} \mu(j) + 24\mu^2(i)\sum_{j\ne i} \mu^2(j)\\ + 8\mu(i)\sum_{j\ne i} \mu^3(j)
	+24\mu^2(i) \sum_{\substack{j,k\ne i}} \mu(j)\mu(k)
	-44\sum_{\substack{j,k,\ell\ne i}} \mu^2(j)\mu(k)\mu(\ell)\\-24\sum_{\substack{j,k\ne i}} \mu^2(j)\mu^2(k)
	-8\sum_{\substack{j,k\ne i}}\mu(j)\mu^3(k)-96\sum_{\substack{j,k,\ell,m\ne i}} \mu(j)\mu(k)\mu(\ell)\mu(m). 
	\end{multline}
	Consequently, as the above is negative, there exists $a_1^{*}$ such that $\Delta_i <0$ whenever $x_i \ge a_1^{*}$. 
	
	(ii) For any $i\in J$, and any integer $x_{i}\geq{2}$, the transitions of $\suite{Y_n}$ from the state $x_{i}.\gre_i$ can be retrieved in a similar fashion to	(\ref{eq:GtransY1}). It follows that 
			  for any $j\neq k\neq \ell\ne m \ne p \ne s \ne i\in\maV$. Set $H=\{i,j,k\}\subset J$,  the transitions that will present in the Appendix (\ref{eq:InCtransY1}), then we deduce that $\Delta_i'=\nu_{i}(\mu)x_{i}+\beta'_{i}(\mu)$ 	for some bounded $\beta'_{i}(\mu)$ (see Appendix (\ref{eq:InCLyapY1})), and setting $H=\{i,j,k\}$ as the only element of $\maJ$ such that $i\in H$, we obtain that 
		
	

		\begin{multline}
		\label{eq:InCnu_i}
		\nu_i(\mu)=8\mu^4(i) + 24\mu^3(i)\sum_{\ell\in\maV\backslash\{i\}} \mu(\ell) + 24\mu^2(i)\sum_{\ell\in\maV\backslash \{i\}} \mu^2(\ell) + 8\mu(i)\sum_{\ell\in\maV\backslash\{i\}} \mu^3(\ell)\\
		-8 \sum_{\substack{\ell\in \overline{H}}} \mu(j) \mu^3(\ell)
		-4 \sum_{\substack{\ell\in\overline{H}:\\\textrm{ends with }kk}}  \mu(j)\mu^2(k)\mu(\ell)
		-20\sum_{\substack{\ell\in\overline{H}:\\\textrm{otherwise}}} \mu(j)\mu^2(k)  \mu(\ell)\qquad\qquad\qquad\;\;\\
		-48 \mu(j)\mu(k) \sum_{\substack{\ell\in\overline{H}}} \mu^2(\ell)
		-4 \sum_{\substack{\ell\in\overline{H}:\\\textrm{ends with }\ell\ell}} \mu(j)\mu^2(\ell)\mu(m)
		-40 \sum_{\substack{\ell\in\overline{H}:\\\textrm{otherwise}}} \mu(j)\mu^2(\ell)\mu(m)\qquad\quad\;\;\\
		+ 48 \mu^2(i)\mu(j)\mu(k)-8 \mu(j)\mu(k)\sum_{\substack{\ell,m\in\overline{H}:\\\textrm{ends with }jk}} \mu(\ell)\mu(m)\
		-80 \mu(j)\mu(k)\sum_{\substack{\ell,m\in\overline{H}:\\\textrm{otherwise}}} \mu(\ell)\mu(m)\qquad\;\\
		-96 \sum_{\substack{\ell,m\in\overline{H}}} \mu(j)\mu(\ell)\mu(m)\mu(p)
		+24 \mu^2(i)\sum_{\substack{\ell\in\overline{H}}} \mu(j)\mu(\ell)
		+24 \mu^2(i)\sum_{\substack{\ell,m\in\overline{H}}} \mu(\ell)\mu(m)\qquad\qquad\;\;\;\\
		-24\sum_{\substack{\ell\in\overline{H}}} \mu^2(j)\mu^2(\ell)
		-4\sum_{\substack{\ell\in\overline{H}:\\\textrm{ends with }jj}} \mu^2(j)\mu(\ell)\mu(m)
		-40\sum_{\substack{\ell\in\overline{H}:\\\textrm{otherwise}}} \mu^2(j)\mu(\ell)\mu(m)\qquad\qquad\;\;\;\;\\
		-8\sum_{\substack{\ell\in\overline{H}}} \mu^3(j)\mu(\ell)
		+24\mu(i)\sum_{\substack{j,k\in H}}\mu(j)\mu^2(k)
		-8\sum_{\substack{\ell,m\in\overline{H}}} \mu^3(\ell)\mu(m)
		-24\sum_{\substack{\ell,m\in\overline{H}}} \mu^2(\ell)\mu^2(m)\;\;\;\\
		-4\sum_{\substack{\ell,m,p\in\overline{H}:\\ \textrm{ends with }\ell\ell}} \mu^2(\ell)\mu(m)\mu(p)-40\sum_{\substack{\ell,m,p\in\overline{H}:\\\textrm{otherwise}}} \mu^2(\ell)\mu(m)\mu(p)
		-96\sum_{\substack{\ell,m,p,s\in\overline{H}}} \mu(\ell)\mu(m)\mu(p)\mu(s).
		\end{multline}
	\normalsize{Thus, there exists $a_{2}^{*}$ such that $\Delta'_{i} <0$ whenever $x_{i} \ge a_{2}^{*}$.}
	
	\medskip
	(iii) For any $i\ne j$ such that $\{i,j\}$ is not included in a hyperedge of the family $\mathcal J$, 
	for any integers $x_{i},x_{j}>0$, we obtain that 
	\begin{equation*}
	\Delta_{ij}:=\esp{Q\left(X_{n+1}\right) - Q\left(X_n\right) | X_n = x_{i}.\gre_i+x_{j}.\gre_j}=\lambda_{ij}(\mu)x_{i}+\lambda_{ji}(\mu)x_{j}+\beta_{ij}(\mu), 
	\end{equation*}
	for a bounded $\beta_{ij}(\mu)$, and for 
	\begin{equation}
	\label{eq:Gdeflambdaij,ji}
	\lambda_{ij}(\mu)=2\Big(\mu(i)-\sum\limits_{\ell \in \maV\backslash \{i,j\}}\mu(\ell)\Big)\textrm{ and }\lambda_{ji}(\mu)=2\Big(\mu(j)-\sum\limits_{\ell \in \maV\backslash \{i,j\}}\mu(\ell)\Big).
	\end{equation}
	Now observe that $\maV\backslash \{i,j\}\in\mathcal{T}(\mbH)$, so $\sum\limits_{\ell \in \maV\backslash \{i,j\}}\mu(\ell)>\displaystyle\frac{1}{3},$ then $\lambda_{ij}(\mu)<0$ and $\lambda_{ji}(\mu)<0$.
	Thus there exists $a_{3}^{*}$ such that $\Delta_{ij}<0$ whenever ${x_{i} \vee x_{j}} \ge a_{3}^{*}$. 
	
	\medskip
	
	(iv) For any $i,j$ such that $i\ne j$ and $\{i,j\}\subset H$ for some $H\in \mathcal J$, for any integers $x_{i},x_{j}>0,$ we obtain that
	\begin{equation*}
	\Delta'_{ij} :=\esp{Q\left(X_{n+1}\right) - Q\left(X_n\right) | X_n = x_{i}.\gre_{i}+x_{j}.\gre_{j}}
	= \nu_{ij}(\mu)x_i + \nu_{ji}(\mu)x_j +\beta'_{ij}(\mu)
	\end{equation*}
	for a bounded $\beta'_{ij}(\mu)$ and 
	\begin{equation}
	\label{eq:Gdefnuij,ji}
	\nu_{ij}(\mu)=2\Big(\mu(i)-\sum\limits_{\ell \in \overline{H}}\mu(\ell)\Big)\textrm{ and }\nu_{ji}(\mu)=2\Big(\mu(j)-\sum\limits_{\ell \in \overline{H}}\mu(\ell)\Big).
	\end{equation}
	
	As $\overline{H}\in\mathcal{T}(\mbH)$, so $\sum\limits_{\ell \in\overline{H}}\mu(\ell)>\displaystyle\frac{1}{3},$ then $\nu_{ij}(\mu)<0$ and $\nu_{ji}(\mu)<0$.\\
	\normalsize{Again, there exists $a_{4}^{*}$ such that $\Delta'_{ij}<0$ whenever ${x_{i} \vee x_{j}} \ge a_{4}^{*}$.}

	\medskip
	
	(v) We finally consider the case where $X_n=x_i.\gre_{i}+x_j.\gre_{j}+x_k.\gre_{k}$ for $H=\{i,j,k\}$, for some $H\in\maJ$, 
	and integers $x_i,x_j$ and $x_k$ such that $x_i,x_j \geq x_k>0$. 
	\begin{align*}
	\Delta_H &:=\esp{Q\left(X_{n+1}\right) - Q\left(X_n\right) | X_n = x_{i}.\gre_i+x_{j}.\gre_j+x_{k}.\gre_k}\\
	& =\alpha_{i}(\mu)x_{i}+\alpha_{j}(\mu)x_{j}+\alpha_{k}(\mu)x_k+\beta_H(\mu), 
	\end{align*}
	for a bounded $\beta_H(\mu)$, and for 
	\begin{equation}
	\label{eq:Gdeflaphai,j,k}
	\alpha_{i}(\mu)=2\Big(\mu(i)-\sum\limits_{\ell \in \overline{H}}\mu(\ell)\Big),\;\alpha_{j}(\mu)=2\Big(\mu(j)-\sum\limits_{\ell \in \overline{H}}\mu(\ell)\Big)\textrm{ and }
	\alpha_{k}(\mu)=2\mu(k).
	\end{equation}
	As  $\overline{H}\in\mathcal{T}(\mbH)$, so $\alpha_{i}(\mu)<0$ and $\alpha_{j}(\mu)<0$.  
	From this, we deduce as above the existence of an integer $a_5^{*}$ such that $\Delta_H<0$ whenever 
	$x_i\vee x_j \ge a_5^{*}$.

	%
	\medskip
	
	To conclude, if we let $\mathcal{K}$ be the finite set
	\[{\mathcal K=\left\lbrace \grx\in\maE\;:\;x_i\le \max\Big(a_1^*,...,a_5^*,2\Big);\; i\in \maV \right\rbrace,}\]
	then if follows from the above arguments that for any $\grx\in\maE\cap\bar{\mathcal K}$ 
	and any $n\in\N$, 
	$$\esp{Q\left(Y_{n+1}\right) - Q\left(Y_{n}\right) | Y_n=\grx}<0.$$
	We deduce from Lyapunov-Foster Theorem \ref{the:LyapunovFosterTheorem} that the chain $\suiten{Y_n}$ is positive recurrent. 
	This is the case in turn for the chain $\suiten{X_n}$.
\end{proof}

\begin{remark}
	\rm 
	Observe that the only incomplete (in the sense of Theorem \ref{thm:suff3unifincomplet}) $3$-uniform hypergraph of order 
	4 would be obtained from the complete one by deleting only one {hyperedge}. However, as easily seen the transversal number of the resulting hypergraph is 1, so the latter is non-stabilizable from Proposition \ref{prop:superstar}.
\end{remark}


In the following examples we show how the stability can be shown for various incomplete $3$-uniform hypergraphs using 
Theorem \ref{thm:suff3unifincomplet}, 


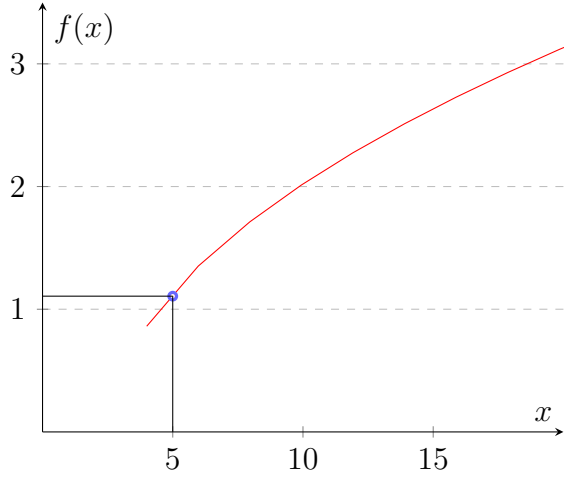
\begin{figure}
	\centering
	\begin{tikzpicture}
	\begin{axis}[
	axis lines =center,
	xlabel = $x$,
	ylabel = {$f(x)$},
	xmin=0, xmax=20,
	ymin=0, ymax=3.5,
	xtick={0,5,10,15},
	ytick={1,2,3},
	ymajorgrids=true,
	grid style=dashed,
	]
	\addplot [
	domain=4:200, 
	samples=100, 
	color=red,
	]
	{((2*x^4-9*x^3+12*x^2-13*x+12)/(6*x^2+10*x+24))^(0.25)};	
	\end{axis}
	\filldraw[color=blue!60, fill=red!5, very thick](1.713,1.8) circle (0.05);
	\draw[-] (1.713,0) -- (1.713,1.8);
	\draw[-] (0,1.8) -- (1.713,1.8);
	\end{tikzpicture}
	\caption{The curve of the function  $f(x)=\displaystyle\left(\frac{2x^4-9x^3+12x^2-13x+12}{6x^2+10x+24}\right)^\frac{1}{4}.$} 
	\label{fig:MinMaxq>=5} 
\end{figure}

\begin{corollary}
	Consider an incomplete $3$-uniform hypergraph $\mathbb{H}'$ satisfying the assumptions of Theorem \ref{thm:suff3unifincomplet}. 
	Recall (\ref{eq:defmuminmax}), and define the sets 
	\[\mathscr A(\mbH'):=\biggl\{\mu\in\mathscr M(\maV)\,:\, \displaystyle\frac{\mumaxi}{\mumini}<\left(\frac{2q^4-9q^3+12q^2-13q+12}{6q^2+10q+24}\right)^{1/4}\biggl\}\,\]
	and 
	\begin{equation*}
	\mathscr S_1(\mbH'):=\mathscr A(\mbH')\,\cap\,\mathscr N_2(\mathbb H')\, \cap\mathscr N^{\scriptsize{-}}_3(\mathbb H').	\end{equation*}
	Then the model 
	$(\mathbb{H}',\textsc{ml},\mu)$ is stable for any  $\mu\in\mathscr S_1(\mbH').$	
\end{corollary}
\begin{proof}
	Recalling (\ref{eq:Gdeflambdai}) and (\ref{eq:InCnu_i}), a simple algebra shows that 
	\[\mathscr A(\mbH') \subset \biggl\{\mu\in\mathscr M(\maV):\left(\max\limits_{i\in J} \lambda_i(\mu)\vee \max\limits_{i\in \bar J} \nu_i(\mu)\right)<0\biggl\},\]
	thus $\mathscr S_1(\mbH')\subset \mathscr S(\mbH')$. 
\end{proof}

\begin{ex}
	\label{ex:UnifDisIncomplet}
	\rm
	Observe that for any such $\mbH'=(\maV,\maH')$ satisfying the assumptions of Theorem \ref{thm:suff3unifincomplet}, 
	the model $(\mathbb{H}',\textsc{ml},\mu_{\textsc{u}})$ is stable for $\mu_{\textsc{u}}$ the uniform distribution on $\maV$. 
	Indeed, we have $\mu_{\textsc{u}}\in\mathscr S_1(\mbH')$.  
	To see this, first observe that $\forall q\geq 5,\;\frac{2q^4-9q^3+12q^2-13q+12}{6q^2+10q+24}>1,$ see Figure \ref{fig:MinMaxq>=5}. Moreover, it is immediate that 
	$\mu_{\textsc{u}}\in\mathscr N_3^-(\mbH')$. 
	It remains to show that $\mu_{\textsc{u}}\in\mathscr N_2(\mbH')$. 
	We proceed in three steps. First, for $q=5$ the only incomplete $3$-uniform hypergraph in the sense of Theorem \ref{thm:suff3unifincomplet} is the complete hypergraph on $\llbracket 1,5 \rrbracket$ minus one vertex, say $\{1,2,3\}$. 
	It is then easily seen that $\{4,5\}$ is the only minimal transversal of $\mbH'$. So $\tau(\mbH')=2$, in a way that for 
	all $T'\in\mathcal T(\mbH')$, $\mu_{\textsc{u}}(T')\ge 2/5> 1/3,$ 
	showing that $\mu_{\textsc{u}} \in \mathscr N_2(\mbH')$. 
	
	Now, if $q=6$ there are two incomplete $3$-uniform hypergraph in the sense of Theorem \ref{thm:suff3unifincomplet}: 
	the complete $3$-uniform hypergraph on $\llbracket 1,6 \rrbracket$ minus one hyperedge, say $\{1,2,3\}$; and the 
	complete $3$-uniform hypergraph on $\llbracket 1,6 \rrbracket$ minus two disjoint hyperedges, say $\{1,2,3\}$ and $\{4,5,6\}$. In both cases, $\{4,5,6\}$ is a minimal transversal of $\mbH'$, thus $\tau(\mbH')=3$, and so 
	$\mu_{\textsc{u}}(T')\ge 3/6 > 1/3,$ for all $T'\in\mathcal T(\mbH)$, 
	proving again that $\mu_{\textsc{u}} \in \mathscr N_2(\mbH')$. 
	
	We now address the case where $q>6$. First observe that 
	\begin{equation}
	\label{eq:trampo1}
	q-2-\left\lfloor {q\over 3}\right\rfloor > {q\over 3}.
	\end{equation}
	Then, let $p=|\mathcal J|$ (using the notation of Theorem \ref{thm:suff3unifincomplet}), 
	and denote 
	$\mathcal J=\{H_1,...,H_p\}$. It is easily seen that a transversal of $\mbH$ can be constructed 
	from any minimal transversal of $\mbH'$, by induction, as follows: 
	\begin{itemize}
		\item Take a minimal transversal $T'$ of $\mbH'$, 
		and set $\maH_0:=\maH'$ and $T_0:=T'$. 
		\item For any  $i=1,...,p$, set $\maH_i = \maH_{i-1} \cup \{H_i\}$ and set $T_i$, a 
		transversal of $(\maV,\maH_i)$ of minimal size among those including 
		$T_{i-1}$.  (${T_{i}}$ necessarily exists since $T_{i-1}\cup \{H_i\}$ is a transversal of $(\maV,\maH_{i})$, as easily seen by induction.) 
		\item We obtain $\maH=\maH_p$ by construction, and $T:=T_p$ is a transversal of $\mbH$. 
	\end{itemize}
	We claim that 
	\begin{equation}
	\label{eq:trampo2}
	|T| \le |T'|+p.
	\end{equation}
	To see this, observe that for any $i=1,...,p$ we have the following alternative: either $H_i \cap T_{i-1} = \emptyset$, in which case we can take $T_i $ of the form $T_{i-1}\cup \{k\}$ for any $k\in H_i$, or 
	$H_i \cap T_{i-1} \ne  \emptyset$, in which case $T_i = T_{i-1}$. In all cases we have that $|T_i| \le |T_{i-1}| + 1$, and (\ref{eq:trampo2}) follows by induction. 
	Observing that $|T| \ge \tau(H)=q-2$, that, as the $H_i$'s are disjoint, 
	$p\le \lfloor {q\over 3}\rfloor$,  and using (\ref{eq:trampo1}) and (\ref{eq:trampo2}), we finally obtain that 
	\[\mu_{\textsc{u}}(T') = {|T'| \over q} \ge {|T|-p\over q} >{1\over 3},\]
	hence, once again $\mu_{\textsc{u}} \in \mathscr N_2(\mbH')$. 
	
	To conclude, $\mu_{\textsc{u}}$ is in all cases, an element of $\mathscr S_1(\mbH')$, implying that the model $(\mathbb{H}',\textsc{ml},\mu_{\textsc{u}})$ is stable for all such $\mbH'$. 
\end{ex}

	\section{Discussion of results and conclusion} \label{discussion of results1}
In this chapter, we have studied a generalization of stochastic matching models on graphs, by allowing the matching structure to be a hypergraph. 
This class of models appears to have a wide range of applications in operations management, healthcare, and assemble-to-order systems. 
After formally introducing the model, we have proposed a simple Markovian representation, under IID assumptions. 
We have then addressed the general question of stochastic stability, viewed as the positive recurrence of the underlying Markov chain. For this class of systems, solving this elementary and central question turns out to be an intricate problem. 
As the results of Sections \ref{sec:Ncond} and \ref{sec:unstable} demonstrate, stochastic matching models on hypergraphs are in general, difficult to stabilize. Unlike the GM on graphs, the non-emptiness of the stability region with matching models on hypergraphs depends on a collection of conditions in the geometry of the compatibility hypergraph: rank, anti-rank, degree, size of the transversals, existence of cycles, and so on. \\
	
Nevertheless, we show in Section \ref{sec:stable} that the ``house'' of stable systems is not empty, but shelters models on various uniform hypergraphs that are complete, or complete up to a partition of their nodes (which is a reasonable assumption regarding kidney exchange programs with $3$-cycles, in which case, according to the compatibility of blood types and immunological characteristics, most but not all hyperedges of size $3$ appear in the compatibility graph). We provide the exact stability region of the system in the first case, and a lower bound in the second. For this, we resort to ad-hoc multi-dimensional Lyapunov techniques. 
	

	\clearpage

\pagestyle{empty}
\chapter{Multigraphs}
\label{chap5:Multigraph}
\pagestyle{fancy}
\def\I{{\mathbb I}}
\def\maV{{\mathcal V}}
\def\maE{{\mathcal E}}
\def\maI{{\mathcal I}}
\def\maM{{\mathcal M}}
\section*{Introduction}\label{chap5: intro}

In this chapter we provide a further extension of the previous chapter, we introduced a new stochastic matching model on hypergraph, that generalizes the GM model, then we provide the necessary conditions of stability for the present model, and we precise the stability region in particular cases of GM model. Such as among dating websites and peer-to-peer interfaces it is possible to assume that the items of the same class can be matched together. Hence the need to consider matching models on compatibility matching structures that are \textit{multigraph} rather than just a graph or hypergraph, i.e., an architecture of graph admitting self-loops that is, edges with permission to connect to themself. {\em En route}, by showing results for stochastic matching models on multigraphs, we show various results that have their inner interest for GM models on graphs - see in particular Propositions \ref{prop:ncond} and  \ref{prop:extppartite}. 
This stochastic matching model addressed in this chapter is formally defined as follows: items enter the system by single arrivals, and get matched by groups of 2 or possibly matched to itself in case that there exists  \textit{self-loops}, following compatibilities that are represented by a given multigraph. 
A matching policy determines the matchings to be executed in the case of a multiple-choice, and the unmatched items are stored in a buffer, waiting for a future match.\\

This chapter is organized as follows: In Section \ref{sec:results}, we present our main results for GM models on multigraphs, among which, the maximality and the explicit product form of the stationary probability for the {\sc fcfm} policy, and the maximality of Max-Weight policies. To illustrate these results, several examples
are presented in Section \ref{sec:examples}. The proofs of our main results are then presented in
Sections \ref{sec:ncond}, \ref{sec:FCFM} and \ref{sec:otherproofs}.


\section{Main results}
\label{sec:results}
We now state the main results of this chapter. 
Similarly to \cite{MaiMoy16}, we will be led to consider the set   
\begin{equation}
\label{eq:Ncond}
\textsc{Ncond}(G)= \left\{\mu \in \mathscr M\left(\maV\right):\forall\, I \in\, \I(G),\; \mu\left( I\right) < \mu\left(\maE\left( I\right)\right)\right\}. 
\end{equation}
Let us immediately observe that
\begin{lemma}
	\label{lemma:Ncond}
	For any connected multigraph $G$, we have that 
	\[\textsc{Ncond}\left(\check{G}\right) \subseteq\textsc{Ncond}(G).\]
	Where $\check{G}=(\maV,\check{\maE})$ is the maximal subgraph of $G$ (see Definition \ref{def:restricted}).
\end{lemma}
\begin{proof}
	The result simply follows from the obvious facts that $\mathbb I(G)\subset \mathbb I\left(\check G\right)$ and that, for any $ I\in\mathbb I(G)$, 
	$\maE( I)=\check{\maE}( I)$, since $ I\subset\maV_2$. 
\end{proof}
It is stated in Theorem 1 of \cite{MaiMoy16} that, if $G$ is a graph, the set {\sc Ncond}$(G)$ is non-empty if and only if $G$ is not a bipartite graph. 
This result can be generalized to multigraphs: 
\begin{proposition}
	\label{prop:ncond} 
	For any connected multigraph $G$, we have that 
	\[\textsc{Ncond}(G) =\emptyset \Longleftrightarrow \mbox{ $G$ is a bipartite graph.}\] 
\end{proposition}
\noindent Proposition \ref{prop:ncond} is proven in Section \ref{sec:ncond}. 
From Proposition 2 in \cite{MaiMoy16}, whenever $G$ is a graph (i.e., $\maV_1=\emptyset$), 
the set $\textsc{Stab}(G,\Phi)$ is included in $\textsc{Ncond}(G)$ for any admissible policy $\Phi$. In other words, for any measure $\mu$, 
belonging to $\textsc{Ncond}(G)$ is necessary for the stability of the system $(G,\Phi,\mu)$, for any $\Phi$. 
A similar result holds for any multigraph $G$:

\begin{proposition}
	\label{thm:mainmono} 
	For any connected multigraph $G=(\maV,\maE)$ and any admissible matching policy
	$\Phi$, we have that 
	\[\textsc{Stab}(G,\Phi) \subset \textsc{Ncond}(G).\]
\end{proposition}

\begin{proof}
	The proof is analog to that of Proposition 2 in \cite{MaiMoy16}. 
\end{proof}

\noindent Hence, the notion of maximality of a matching policy is:

\begin{definition}
	For any connected multigraph $G$ that is not a bipartite graph, a matching policy $\Phi$ is said maximal if the sets 
	$\textsc{Stab}(G,\Phi)$ and  $\textsc{Ncond}(G)$ coincide.
\end{definition}

Whenever $G$ is a graph, Theorem 1 of \cite{MBM17} shows, first, that the policy `First Come, First Matched' ({\sc fcfm}) is maximal, and second, that  the stationary probability of the chain $\suite{W_n}$ can be expressed in a remarkable product form. 
We generalize this result to multigraphs: 
\begin{theorem}
	\label{thm:FCFM}
	The matching policy `First Come, First Matched' is maximal: for any connected 
	multigraph $G$ that is not a bipartite graph, we have that {\sc Stab}$(G,\textsc{fcfm})=\textsc{Ncond}(G)$.  
	Moreover, for any $\mu\in\textsc{Ncond}(G)$ the unique stationary probability $\Pi_W$ of the chain $\left(W_n\right)_{n\in\mathbb{N}}$ is defined by 
	\begin{equation*}
	\left\{\begin{array}{ll}
	\Pi_W(\varepsilon) &=\alpha;\\
	\Pi_W(w)&=\alpha\displaystyle\prod\limits_{l=1}^q \frac{\mu(w_l)}{\mu(\mathcal{E}(\{w_1,\dots,w_l \}))},
	\mbox{ for all }w=w_1\dots w_q\in\mathbb{W}\setminus\{\varepsilon\}, 
	\end{array}\right.\end{equation*}
	where
	\begin{equation}
	\label{eq:defalpha}
	\alpha^{-1}=\!1\!+\!\sum\limits_{ I\in\mathbb{I}\left(\check G\right)}\sum\limits_{\sigma\in\mathfrak{S}_{| I|}}\prod\limits_{i=1}^{| I|} \frac{\mu\left(e_{\sigma(i)}\right)}{\mu(\mathcal{E}(\{e_{\sigma(1)},\dots,e_{\sigma(i)}\}))-\mu(\{e_{\sigma(1)},\dots,e_{\sigma(i)}\}\cap\maV_2)},
	\end{equation}
	and where we denote  $ I=\{e_1,\dots,e_{| I|}\}$ for any $ I\in\mathbb{I}\left(\check G\right)$. 
\end{theorem}
\noindent Theorem \ref{thm:FCFM} is proven in section \ref{sec:FCFM}. 
\begin{remark}
	\label{rem:fini}
	If the model is finite, i.e., all nodes of $G$ have self-loops or in other words, $\maV_2=\emptyset$, then it readily follows from 
	Theorem \ref{thm:FCFM} that the unique stationary probability on the finite state space $\mathbb W$, is given by 
	\begin{equation*}
	\left\{\begin{array}{ll}
	\Pi_W(\varepsilon) &=\alpha;\\
	\Pi_W\left(e_{\sigma(1)} \cdots e_{\sigma(| I|)}\right)&=\alpha\displaystyle\prod\limits_{i=1}^{| I|} \frac{\mu\left(e_{\sigma(i)}\right)}{\mu(\mathcal{E}(\{e_{\sigma(1)},\dots,e_{\sigma(i)}\}))},\mbox{ for all } I\in \mathbb I(\check G),\,\sigma \in \mathfrak{S}_{| I|},
	\end{array}\right.\end{equation*}
	with the normalizing constant 
	\begin{equation*}
	\alpha=\left[\!1\!+\!\sum\limits_{ I\in\mathbb{I}\left(\check G\right)}\sum\limits_{\sigma\in\mathfrak{S}_{| I|}}\prod\limits_{i=1}^{| I|} \frac{\mu\left(e_{\sigma(i)}\right)}{\mu(\mathcal{E}(\{e_{\sigma(1)},\dots,e_{\sigma(i)}\}))}\right]^{-1}\cdot 
	\end{equation*}
\end{remark}

On another hand, as is shown in Theorem 5.3 of \cite{JMRS20}, all Max-Weight matching policies such that $\beta>0$ are maximal 
whenever $G$ is a graph. (In particular, the maximality of `Match the Longest' ({\sc ml}) for GM models was first proven in Theorem 2 of \cite{MaiMoy16}, as a consequence of the corresponding result for EBM models, (see Theorem 7.1 of \cite{BGM13}).  
This result can also be generalized to multigraphs: 
\begin{theorem}
	\label{thm:ML}
	Any Max-Weight policy $\Phi$ such that $\beta>0$ is maximal: for any multigraph $G$ that is not a bipartite graph, we have that 
	{\sc Stab}$(G,\Phi)=\textsc{Ncond}(G)$. 
	
\end{theorem}

\noindent Aside from {\sc fcfm} and Max-Weight policies, we can determine, or lower-bound, the stability region of the model for particular classes 
of multigraphs. 
\medskip
Given a multigraph $G=(\maV_1\cup\maV_2,\maE),$ let us define an important class of matching policy that is called \textbf{$\maV_2$-favorable} policies and we will be taken in this chapter.


\begin{definition}
	We say that an admissible matching policy $\Phi$ on $G$ is $\maV_2$-favorable if any incoming item always prioritizes a match with a compatible 
	item of class in $\maV_2$ over a compatible item of class in $\maV_1$, whenever it has the choice. Formally, if the class detail is given by $x\in\mathbb X$ and the arrival is of class $v$, it never occurs that the incoming $v$-item is matched with a $j$-item, for some $j\in\maV_1$, while $\maP(x,v) \cap \maV_2 \ne \emptyset$. 
\end{definition}
\begin{definition} 
	Let $G$ be a connected multigraph. 
	We say that $G$ is complete $p$-partite, $p\ge 2$, if its maximal subgraph $\check G$ is complete $p$-partite. 
	Then, the minimal blow-up graph $\hat G$ (see Definition \ref{def:extended}) itself is called an {\em extended} complete $p$-partite graph. 
\end{definition}

Observe that an extended complete $p$-partite graph is {\em not} complete $p$-partite whenever the construction above is non-trivial, i.e., the multigraph in the above definition is not a graph, see an example in Figure \ref{fig:p-partiteCompletToBlow-upgraph}.
\begin{figure}[htb]
	\begin{center}
		\begin{tikzpicture}

		\fill (0,2) circle (2pt)node[above]{\scriptsize{1}};
		\fill (-1,0) circle (2pt)node[left]{\scriptsize{4}};
		\fill (1,0) circle (2pt)node[right]{\scriptsize{5}};
		\fill (-0.5,0.5) circle (2pt)node[above]{\scriptsize{$\;\;\;\;\;2$}};
		\fill (1.5,0.5) circle (2pt)node[right]{\scriptsize{3}};
		
		\draw[-] (1.5,0.5) -- (0,2);
		\draw[-] (-0.5,0.5) -- (0,2);
		\draw[-] (1,0) -- (0,2);
		\draw[-] (-1,0) -- (0,2);
		\draw[-] (-1,0) -- (1,0);
		\draw[-] (-1,0) -- (1.5,0.5);
		\draw[-] (-0.5,0.5) -- (1,0);
		\draw[-] (-0.5,0.5) -- (1.5,0.5);
		
		\draw[<-] (1.8,1) -- (2.2,1);
		
		
		\fill (4,2) circle (2pt)node[above]{\scriptsize{1}};
		\fill (3,0) circle (2pt)node[left]{\scriptsize{4}};
		\fill (5,0) circle (2pt)node[right]{\scriptsize{$\;5$}};
		\fill (3.5,0.5) circle (2pt)node[above]{\scriptsize{$\;\;\;\;\;2$}};
		\fill (5.5,0.5) circle (2pt)node[right]{\scriptsize{3}};
		
		\draw[-] (5.5,0.5) -- (4,2);
		\draw[-] (3.5,0.5) -- (4,2);
		\draw[-] (5,0) -- (4,2);
		\draw[-] (3,0) -- (4,2);
		\draw[-] (3,0) -- (5,0);
		\draw[-] (3,0) -- (5.5,0.5);
		\draw[-] (3.5,0.5) -- (5,0);
		\draw[-] (3.5,0.5) -- (5.5,0.5);
		
		\draw[thick,-] (5,0) to [out=-50,in=40,distance=11mm] (5,0);
		
		\draw[->] (5.8,1) -- (6.2,1);
		

		\fill (8,2) circle (2pt)node[above]{\scriptsize{1}};
		\fill (7,0) circle (2pt)node[left]{\scriptsize{4}};
		\fill (7.5,0.5) circle (2pt)node[above]{\scriptsize{$\;\;\;\;\;2$}};
		\fill (9,0) circle (2pt)node[right]{\scriptsize{5}};
		\fill (9.5,0.5) circle (2pt)node[right]{\scriptsize{3}};
		\fill (8,-0.5) circle (2pt)node[right]{\scriptsize{$\;\,\cop{5}$}};
		
		\draw[-] (9.5,0.5) -- (8,2);
		\draw[-] (7.5,0.5) -- (8,2);
		\draw[-] (9,0) -- (8,2);
		\draw[-] (7,0) -- (8,2);
		\draw[-] (7,0) -- (9,0);
		\draw[-] (7,0) -- (9.5,0.5);
		\draw[-] (7.5,0.5) -- (9,0);
		\draw[-] (7.5,0.5) -- (9.5,0.5);
		\draw[-] (8,-0.5) -- (9,0);
		\draw[-] (8,-0.5) -- (7.5,0.5);
		\draw[-] (8,-0.5) -- (7,0);
		\draw[-] (8,-0.5) -- (8,2);
		

		
		
		\end{tikzpicture}
		\caption{Middle: A multigraph $G.$ Left: Its maximal complete $3$-partite subgraph $\check G.$ Right: Extended complete $3$-partite graph $\hat{G}.$}
		\label{fig:p-partiteCompletToBlow-upgraph}
	\end{center}
\end{figure}
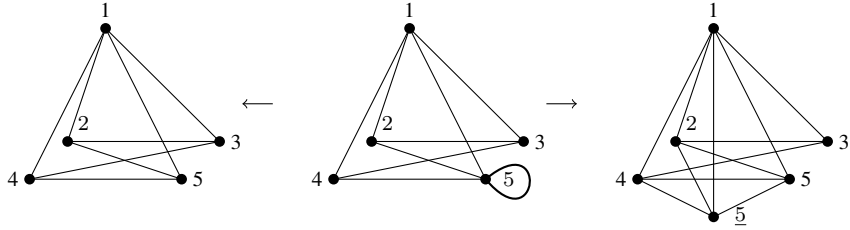

\begin{theorem}
	\label{thm:ppartite}
	Let $G$ be a complete $p$-partite multigraph, $p\ge 2$. Then, 
	\begin{enumerate}
		\item[(i)] If $p\ge 3$ or $\maV_1\neq \emptyset$, then any $\maV_2$-favorable matching policy $\Phi$ is maximal, that is, $\textsc{Stab}(G,\Phi)=\textsc{Ncond}(G).$
		\item[(ii)] If $p\ge 3$, then $\textsc{Ncond}\left(\check G\right)\subset \textsc{Stab}(G,\Phi)$, for any admissible matching policy $\Phi$. 
	\end{enumerate}  
\end{theorem}

With the above results in hands, we have the following panorama regarding the stability region of a matching model on a connected multigraph $G$: 
\begin{enumerate}
	\item[(i)] Any measure $\mu$ that does not belong to the set {\sc Ncond}($G$) makes the system unstable; 
	\item[(ii)] If $G$ is a bipartite graph, then the model cannot be stable;
	\item[(iii)] Otherwise, the region {\sc Ncond}($G$) is necessarily non-empty, and the models $(G,\textsc{fcfm},\mu)$ and $(G,\Phi,\mu)$ for any 
	Max-Weight policy $\Phi$, are stable for any $\mu\in\textsc{Ncond}(G)$;
	\item[(iv)] For any complete $p$-partite multigraph ($p\ge 2$) that is not a bipartite graph and any  $\mu\in\textsc{Ncond}\left({G}\right)$, 
	any model $(G,\Phi,\mu)$ such that $\mu\in\textsc{Ncond}\left(\check{G}\right)$ or $\Phi$ is $\maV_2$-favorable, is stable. 
\end{enumerate} 

As a by-product of Theorem \ref{thm:ppartite} we can determine, or lower-bound, the stability region of GM models on extended complete $p$-partite graphs.

\begin{definition}
	\label{def:extendsmeas}
	For any measures $\mu\in\mathscr M(\maV)$ and $\hat\mu\in\mathscr M(\hat\maV)$, we say that $\hat \mu$ extends $\mu$ on $\hat G$, and that 
	$\mu$ reduces $\hat \mu$ on $G$, if 
	\begin{equation}	
	\label{eq:MuHatMu}
	\left\{\begin{array}{ll}
	\displaystyle\hat{\mu}(i)&=\mu(i),\quad \text{for all} \; i\in\maV_2\,;\\
	\displaystyle\hat{\mu}(i) + \hat{\mu}(\underline{i}) &=\mu(i),\quad \text{for all} \; i\in\maV_1.
	\end{array}\right.
	\end{equation}
\end{definition}


\begin{definition}
	\label{def:extendspol}
	Let $\Phi$ and $\hat\Phi$ be two admissible matching policies, respectively on $G$ and $\hat G$. 
	We say that $\hat\Phi$ extends $\Phi$ on $\hat G$ if, for any $\mu\in\mathscr M({\maV})$ and $\hat\mu\in\mathscr M(\hat\maV)$, whenever both systems $(G,\Phi,\mu)$ and $(\hat G,\hat\Phi,\hat\mu)$ are in the same state $w\in \mathbb W$ and welcome the same arrival, $\Phi$ and $\hat\Phi$ induce the same choice of match, if any. 
\end{definition}

\begin{proposition}
	\label{prop:extppartite}
	Let $\hat G$ be an extended complete $p$-partite graph, $p\ge 2$, and $\check G$ be its reduced graph. 
	\begin{enumerate}
		\item[(i)] If $p\ge 3$ or $\maV_1\neq \emptyset$, then for any matching policy $\hat \Phi$ on $\hat G$ that extends a $\maV_2$-favorable policy on $G$, 
		$\textsc{Stab}(\hat G,\hat \Phi)=\textsc{Ncond}(\hat G).$
		\item[(ii)] If $p\ge 3$, then for any measure $\hat\mu$ on $\hat G$ whose reduced measure $\mu$ is an element of $\textsc{Ncond}(\check G)$,
		and any matching policy $\hat{\Phi}$ on $\hat G$, the model $(\hat G,\hat \Phi,\hat \mu)$ is stable. 
	\end{enumerate}  
\end{proposition}

\noindent The proofs of Theorem \ref{thm:ML}, Theorem \ref{thm:ppartite} and Proposition \ref{prop:extppartite} are given in section \ref{sec:otherproofs}. 

\section{A few examples}
\label{sec:examples}

In this section, we illustrate our main results by different examples.

\begin{ex}\label{ExampleSquare}\rm Consider the multigraph $G$ of Figure \ref{fig:ExampleSquare}, made of four nodes arranged in a square, with a self-loop at each node.
	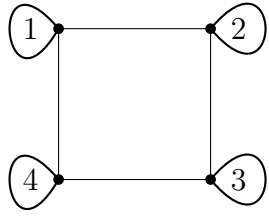
\begin{figure}[htb]
		\begin{center}	
			\begin{tikzpicture}
			\draw (-1,1)--(1,1) ;
			\draw (1,1)--(1,-1) ;
			\draw (1,-1)--(-1,-1) ;
			\draw (-1,-1)--(-1,1) ;
			
			\draw[thick,-] (-1,1) to [out=130,in=240,distance=15mm] (-1,1);
			\draw [thick,-] (1,1) to [out=50,in=-50,distance=15mm] (1,1);
			\draw[thick,-] (-1,-1) to [out=130,in=240,distance=15mm] (-1,-1);
			\draw [thick,-] (1,-1) to [out=50,in=-50,distance=15mm] (1,-1);
			
			\fill (-1,1) circle (2pt)node[left]{$1\;$} ;
			\fill (1,1) circle (2pt)node[right]{$\;2$} ;
			\fill (1,-1) circle (2pt)node[right]{$\;3$} ;
			\fill (-1,-1) circle (2pt)node[left]{$4\;$} ;
			\end{tikzpicture}
			\caption{Multigraph $G$ of Example~\ref{ExampleSquare}.}\label{fig:ExampleSquare}
		\end{center}
	\end{figure}
	Since all nodes have a self-loop, it follows from Remark \ref{rq:MultigraphSelfLoop} that any matching model on $G$ is necessarily stable, that is, for any admissible $\Phi$, we have that $\textsc{Stab}(G,\Phi)=\mathscr M(\maV).$ Let us focus on the $\textsc{fcfm}$ policy. 
	The set of admissible queue details is given by $\mathbb W = \{\varepsilon,1,2,3,4,13,24,31,42\},$ and as a consequence of Remark~\ref{rem:fini}, we can compute explicitly $\Pi_W$, obtaining the following values:
	\[\left\{\begin{array}{llll}
	& \Pi_W(\varepsilon)=\alpha && \\
	& \Pi_W(1)=\alpha\frac{\mu(1)}{1-\mu(3)} &&
	\Pi_W(2)=\alpha\frac{\mu(2)}{1-\mu(4)} \\
	& \Pi_W(3)=\alpha\frac{\mu(3)}{1-\mu(1)} &&
	\Pi_W(4)=\alpha\frac{\mu(4)}{1-\mu(2)} \\
	& \Pi_W(13)=\alpha\frac{\mu(1)}{1-\mu(3)}\mu(3) &&
	\Pi_W(24)=\alpha\frac{\mu(2)}{1-\mu(4)}\mu(4) \\
	& \Pi_W(31)=\alpha\frac{\mu(3)}{1-\mu(1)}\mu(1) &&
	\Pi_W(42)=\alpha\frac{\mu(4)}{1-\mu(2)}\mu(2),
	\end{array}\right.\]
	with
	$$
	\alpha=\left[1 + \mu(1)\frac{1+\mu(3)}{1-\mu(3)}
	+ \mu(2)\frac{1+\mu(4)}{1-\mu(4)}
	+ \mu(3)\frac{1+\mu(1)}{1-\mu(1)}
	+ \mu(4)\frac{1+\mu(2)}{1-\mu(2)} \right]^{-1},
	$$
	using the fact that $\mathbb{I}\left(\check G\right)=\{\{1\},\{2\},\{3\},\{4\},\{1,3\},\{2,4\}\}$.
\end{ex}

\begin{ex}\rm
	Consider the multigraph $G$ (at the middle) of Figure \ref{fig:GgraphGLmultigZAndGTilde}. 
	From Theorems \ref{thm:FCFM} and \ref{thm:ML}, both the stability region $\textsc{Stab}(G,\textsc{fcfm})$ under First Come, First Matched, and the stability region $\textsc{Stab}(G,\textsc{mw})$ 
	under any Max-Weight policy, coincide with the set 
	\begin{align*}
	\textsc{Ncond}(G) &=\left\{\mu\in\mathscr M(\maV):\mu(1) < \mu(2),\,\mu(\{1,3\})\vee\mu(\{1,4\})<{1\over 2}\right\}.
	\end{align*}
	Second, recall that the maximal subgraph $\check{G}$ is the pendant graph
	that studied in Lemma 3 of \cite{MaiMoy16}. For $\eta >0,$ there exists a linear Lyapunov function $L_{\eta}$ such that for any $w\in\mathbb{W},$ the drift  $\check\Delta^{\check{\Phi}}_{\mu}L_{\eta}(w)<0$ then from  Proposition \ref{prop:EspGmultiLeqEspGgrapg}, we have $\hat\Delta^{\hat{\Phi}}_{\hat\mu}L_{\eta}(w)<0$ and $\Delta^{\Phi}_{\mu}L_{\eta}(w)<0.$ Then the model $(G,\Phi,\mu)$ and $\left(\hat{G},\hat{\Phi},\hat\mu\right)$ are stable. 
\end{ex}

\begin{ex}\rm
	Consider now a multigraph $G$ whose maximal subgraph is a complete $2$-partite graph of order 3 (i.e., a string of 3 nodes), as represented in 
	Figure \ref{fig:ExGLandGTilde},
	\begin{figure}[htb]
		\begin{center}
			\begin{tikzpicture}
			\fill (0,2) circle (2pt)node[above]{\scriptsize{1}};
			\fill (0,1) circle (2pt)node[right]{\scriptsize{2}};
			\fill (-1,0) circle (2pt)node[left]{\scriptsize{3}};
			\draw[-] (0,1) -- (-1,0);
			\draw[-] (0,1) -- (0,2);
			\draw[<-] (1.8,1) -- (2.2,1);
			\fill (4,2) circle (2pt)node[above]{\scriptsize{1}};
			\fill (4,1) circle (2pt)node[right]{\scriptsize{2}};
			\fill (3,0) circle (2pt)node[left]{\scriptsize{3}};
			\draw[-] (4,1) -- (3,0);
			\draw[-] (4,1) -- (4,2);
			\draw[->] (5.8,1) -- (6.2,1);
			\fill (8,2) circle (2pt)node[above]{\scriptsize{1}};
			\fill (8,1) circle (2pt)node[right]{\scriptsize{2}};
			\fill (7,0) circle (2pt)node[left]{\scriptsize{3}};
			\fill (9,0) circle (2pt)node[right]{\scriptsize{$\cop{3}$}};
			\draw[-] (8,1) -- (9,0);
			\draw[-] (8,1) -- (7,0);
			\draw[-] (8,1) -- (8,2);
			\draw[-] (7,0) -- (9,0);
			\draw[thick,-] (3,0) to [out=50,in=140,distance=10mm] (3,0);
			\end{tikzpicture}
			\caption{Left: A multigraph whose maximal subgraph is a complete 2-partite graph of order 3. Right: Its minimal blow-up graph.}
			\label{fig:ExGLandGTilde}
		\end{center}
	\end{figure}
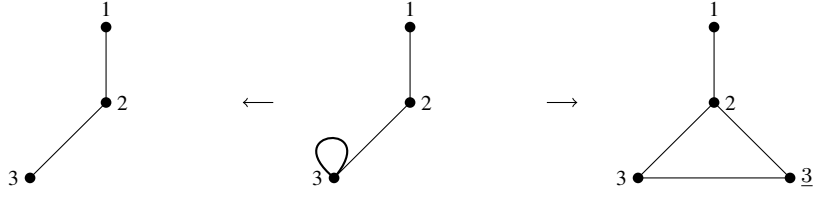
	
	\noindent We easily obtain that 
	\begin{align*}
	\textsc{Ncond}(G)&=\left\{\mu\in \mathscr M(\maV)\,:\, \mu(1) <\mu(2)<{1 \over 2}\right\},\\
	\textsc{Ncond}(\hat G) &= \left\{\hat\mu\in\mathscr M(\hat\maV):\hat\mu(1) < \hat\mu(2),\,\hat\mu(2)\vee\hat\mu(\{1,3\})\vee\hat\mu(\{1,\cop{3}\})<{1\over 2}\right\}.
	\end{align*}
	In view of Theorems \ref{thm:FCFM} and \ref{thm:ML}, the respective stability regions $\textsc{Stab}(G,\textsc{fcfm})$ and $\textsc{Stab}(G,\textsc{mw})$ under First Come, First Matched, 
	or any Max-Weight policy, coincide with $\textsc{Ncond}(G)$. 
	
	Let us first focus on the \textsc{fcfm} policy. The set of admissible queue details is given by:
	$$\mathbb{W}=\{\varepsilon\}\cup\left\{1^k : k\geq 1\right\}\cup\left\{2^k : k\geq 1 \right\}\cup\left\{1^r 3 1^{k-r} : k\geq 0, \; 0\leq r\leq k \right\}.$$
	By Theorem~\ref{thm:FCFM}, we have
	\[\left\{\begin{array}{ll}
	&\Pi_W(\varepsilon)=\alpha \\
	&\Pi_W(1^k)=\alpha\left(\frac{\mu(1)}{\mu(2)}\right)^k \qquad \Pi_W(2^k)=\alpha\left(\frac{\mu(2)}{1-\mu(2)}\right)^k \\
	&\Pi_W\left(1^r 3 1^{k-r}\right)= \alpha\left(\frac{\mu(1)}{\mu(2)}\right)^r \times \frac{\mu(3)}{1-\mu(1)} \times \left(\frac{\mu(1)}{1-\mu(1)}\right)^{k-r},
	\end{array}\right.\]
	and since $\mathbb{I}\left(\check G\right)=\{\{1\},\{2\},\{3\},\{1,3\}\}$, we can express $\alpha$ as follows,
	\begin{multline*}
	\alpha=\biggl[1 + \frac{\mu(1)}{\mu(2)-\mu(1)} + \frac{\mu(2)}{1-2\mu(2)} + \frac{\mu(3)}{1-\mu(1)} \\
	+ \frac{\mu(1)}{\mu(2)-\mu(1)}\frac{\mu(3)}{1-2\mu(1)}+\frac{\mu(3)}{1-\mu(1)}\frac{\mu(1)}{1-2\mu(1)} \biggl]^{-1}.\end{multline*}
	 
	Second, consider a matching policy $\hat\Phi$ such that a $2$-item always prioritizes a $1$-item over a $3$ or a $\cop{3}$-item. Then $\hat\Phi$ extends a $\maV_2$-favorable policy 
	$\Phi$ on $\hat G$. Thus, from Proposition \ref{prop:extppartite} (i), the stability region of the system is $\textsc{Ncond}(\hat G)$, 
	in other words $\hat\Phi$ is maximal on $\hat G$. We thereby generalize with a very simple proof, the result of Lemma 3 of \cite{MaiMoy16} to the case where $\mu(3) \ne \mu(\cop{3})$. 
	Last , in view of Theorem \ref{thm:ppartite} (i), any $\maV_2$-favorable matching policy on $G$ (i.e., such that $2$ prioritizes $1$ over $3$) is maximal, that is, has the stability region $\textsc{Ncond}(G)$. 
\end{ex}

\begin{ex}\rm
	Last, consider the multigraph $G$ represented in (the middle figure of) Figure 
	\ref{fig:p-partiteCompletToBlow-upgraph}. The maximal subgraph is complete $3$-partite, and we readily obtain that 
	\begin{align*}
	\textsc{Ncond}(\check G) &=\left\{\mu\in \mathscr M(\maV)\,:\, \mu(1) \vee \mu(\{2,4\}) \vee \mu(\{3,5\}) <{1\over 2}\right\};\\
	\textsc{Ncond}(G) &=\left\{\mu\in \mathscr M(\maV)\,:\, \mu(1) \vee \mu(\{2,4\}) <{1\over 2},\;\mu(3)<\mu(\{1,2,4\})\right\};\\
	\textsc{Ncond}(\hat G) &=\left\{\hat\mu\in \mathscr M(\hat{\maV})\,:\, \hat\mu(1) \vee \hat\mu(\{2,4\}) \vee \hat\mu(\{3,5\}) \vee \hat\mu(\{3,\cop{5}\}) <{1\over 2}\right\}.
	\end{align*}
	Then, from Theorems \ref{thm:FCFM} and \ref{thm:ML}, the respective stability regions $\textsc{Stab}(G,\textsc{fcfm})$ and $\textsc{Stab}(G,\textsc{mw})$ under First Come, First Matched, 
	or any Max-Weight policy coincide with the set $\textsc{Ncond}(G)$. From Theorem \ref{thm:ppartite} (i), for any policy $\Phi$ on $G$ according to which all items prioritize $3$-items over $5$ items is maximal, i.e., 
	$\textsc{Stab}(G,\Phi)=\textsc{Ncond}(G)$. From Theorem \ref{thm:ppartite} (ii), any policy $\Phi$ on $G$ is such that $\textsc{Ncond}(\check G)\subset \textsc{Stab}(G,\Phi)$. 
	Last, from Proposition \ref{prop:extppartite}, any policy $\hat\Phi$ on $\hat G$ giving priority to $3$-items over $5$ and $\cop{5}$-items is maximal, whereas for any matching policy 
	$\hat\Phi$ and any measure $\hat\mu$ on $\hat\maV$ extending a measure of $\textsc{Ncond}(\check G)$, the model $(\hat G,\hat\Phi,\hat\mu)$ is stable. 
\end{ex}

\section{Proof of Proposition 5.1}
\label{sec:ncond}
Throughout this section, fix a connected multigraph $G=(\maV,\maE)$, where $\maV=\maV_1 \cup \maV_2$, and denote its minimal blow-up graph by 
$\hat G=(\hat \maV, \hat \maE)$, where $\hat \maV = \maV \cup \cop{\maV_1}$.
We first have the following, 
\begin{lemma}
	\label{lemma:equivalenceNcond}
	For any $\mu\in\mathscr M(\maV)$ we have that 
	\begin{equation*}
	\mu\in\textsc{Ncond}(G)\iff\hat{\mu}_{\tiny{1/2}}\in\textsc{Ncond}(\hat{G}), 
	\end{equation*}
	where $\hat{\mu}_{\tiny{1/2}}$ is the extended measure of $\mu$ such that 
	\begin{equation}
    \label{eq:MuHat1/2=Mu1/2}
	\displaystyle\hat{\mu}_{\tiny{1/2}}(i) = \hat{\mu}_{\scriptsize{1/2}}(\underline{i}) ={1 \over 2} \mu(i)\quad \text{for all} \; i\in\maV_1.
	\end{equation}
	
	\begin{proof}
		For any $ I \in \I(G)$, set $\maE_1( I)=\maE( I)\cap \maV_1$ and $\maE_2( I)=\maE( I)\cap \maV_2$. 
		
		Let us prove first the implication \noindent $\Longleftarrow$: Let $\hat{\mu}_{\tiny{1/2}}\in\textsc{Ncond}(\hat{G})$ and $ I\in\mathbb{I}(G)$. As $ I \subset \maV_2$, we get that 
		\begin{align*}
		\mu( I)=\hat{\mu}_{\tiny{1/2}}( I)<\hat{\mu}_{\tiny{1/2}}(\hat{\maE}( I))
		&=\hat{\mu}_{\tiny{1/2}}\left(\maE( I)\cup\cop{\maE_1( I)}\right)\\
		&=\hat{\mu}_{\tiny{1/2}}(\maE_{1}( I))+\hat{\mu}_{\tiny{1/2}}(\maE_{2}( I))+\hat{\mu}_{\tiny{1/2}}\left(\cop{\maE_1( I)}\right)\\
		&=\displaystyle\mu(\maE_{1}( I))+\mu(\maE_{2}( I))\\
		&=\mu(\maE( I)),
		\end{align*}
		where the second equality is due to the fact that $\hat \maE=\maE\cup\cop{\maE_1}.$ The third follows from the fact that $\maE=\maE_1\cup\maE_2.$ The fouth follows from the fact that the equations (\ref{eq:MuHatMu}) and (\ref{eq:MuHat1/2=Mu1/2}) holds true.\\
		
		Let us prove now the opposite implication \noindent $\Longrightarrow$: Let us now fix $\mu\in\textsc{Ncond}(G)$ and $\hat{ I} \in \I(\hat G)$. 
		Clearly, $\hat I$ can be written as the union $\hat I= J _2\cup  J _1 \cup \cop{ I_1}$, where $ J _1 \cup  I_1 \subset \maV_1$, $ J _2 \subset \maV_2$ and $ J _2 \cup  J _1 \cup  I_1\in\I\left(\check G\right)$.	
		First observe that $ J _1$ and $ I_1$ are necessarily disjoint, as any element $i \in  J _1\cap  I_1$ would be such that $i\v \underline{i}$ in $\hat G$, 
		a contradiction to the fact that $\hat{ I} \in \I(\hat G)$. Thus, 
		\begin{align}
		\hat{\mu}_{\tiny{1/2}}(\hat{ I})
		=\hat{\mu}_{\tiny{1/2}}( J _2\cup J _1)+\hat{\mu}_{\tiny{1/2}}(\cop{ I_1})
		&=\hat{\mu}_{\tiny{1/2}}( J _2\cup J _1)+\hat{\mu}_{\tiny{1/2}}( I_1)\label{eq:IndSet1}\\
		&=\hat{\mu}_{\tiny{1/2}}( J _2\cup J _1\cup I_1)
		\leq\mu( J _2\cup J _1\cup I_1).\nonumber
		\end{align}
		Now, observe that 
		\begin{align*}
		\mu\left(\maE\left( J _2 \cup  J _1 \cup  I_1\right)\right) - \mu\left( J _2 \cup  J _1 \cup  I_1\right) 
		&= \mu\left(\maE\left( J _2 \cup  J _1 \cup  I_1\right)\right) - \mu\left( J _1 \cup  I_1\right)  -\mu\left( J _2\right)\\
		&= \mu\left(\maE\left( J _2 \cup  J _1 \cup  I_1\right)\cap \left( J _1 \cup  I_1\right)^c\right)  -\mu\left( J _2\right)\\
		&\ge \mu\left(\maE\left( J _2 \right)\right)  -\mu\left( J _2\right)>0,
		\end{align*}
		where the second equality is due to the fact that $ J _1 \cup  I_1 \subset \maE\left( J _2 \cup  J _1 \cup  I_1\right)$, because $ J _1\cup  I_1 \subset \maV_1$. 
		The weak inequality follows from the fact that $\maE\left( J _2 \right)$ is disjoint from $ J _1 \cup  I_1$ (because $\hat{ I}$ is an independent set of $\hat G$), and thereby, is included in 
		$\maE\left( J _2 \cup  J _1 \cup  I_1\right)\cap \left( J _1 \cup  I_1\right)^c$. The last strict inequality follows from the fact that 
		$ I_2$ is an independent set of $G$ (and of $\hat G$). This, together with (\ref{eq:IndSet1}), implies that 
		\begin{align*}
		\hat{\mu}_{\tiny{1/2}}(\hat{ I}) &<\mu(\maE( J _2\cup J _1\cup I_1))\\
		&=\mu(\maE_1( J _2\cup J _1\cup I_1))+\mu(\maE_2( J _2\cup J _1\cup I_1))\\
		&=\hat{\mu}_{\tiny{1/2}}(\maE_1( J _2\cup J _1\cup I_1))+\hat{\mu}_{\tiny{1/2}}\left(\cop{\maE_1( J _2\cup J _1\cup I_1)}\right)+\hat{\mu}_{\tiny{1/2}}(\maE_2( J _2\cup J _1\cup I_1))\\
		&=\hat{\mu}_{\tiny{1/2}}\left(\maE_1( J _2\cup J _1\cup I_1)\cup\cop{\maE_1( J _2\cup J _1\cup I_1)}\cup\maE_2( J _2\cup J _1\cup I_1)\right)\\
		&=\hat{\mu}_{\tiny{1/2}}\left(\maE( J _2\cup J _1\cup I_1)\cup\cop{\maE_1( J _2\cup J _1\cup I_1)}\right)\\
		&=\hat{\mu}_{\tiny{1/2}}\left(\hat{\maE}( J _2\cup J _1\cup I_1)\right)
		=\hat{\mu}_{\tiny{1/2}}\left(\hat{\maE}\left( J _2\cup J _1\cup\cop{ I_1}\right)\right)
		=\hat{\mu}_{\tiny{1/2}}(\hat{\maE}(\hat{ I})),
		\end{align*}
		where the first equality is due to the fact that $\maE=\maE_1\cup\maE_2.$ The second equality follows from the fact that the equations (\ref{eq:MuHatMu}) and (\ref{eq:MuHat1/2=Mu1/2}) holds true. The third follows from the disjoint sets. The fourth equality follows from the fact that $\maE=\maE_1\cup\maE_2.$ The fifth equality follows from the fact that $\hat\maE=\maE\cup\cop{\maE_1}.$ \\
		
		Which completes the proof. 
	\end{proof}
\end{lemma}
We can now turn to the proof of Proposition \ref{prop:ncond}:

\begin{proof}[Proof of Proposition \ref{prop:ncond}] 
	If $G$ is a bipartite graph, then it follows from Theorem 1 in \cite{MaiMoy16} that \textsc{Ncond}$(G)$ is empty. 
	Regarding the converse, suppose that $G$ is not a bipartite graph. Then, $\hat G$ cannot be a bipartite graph neither. Indeed, there are two cases: 
	\begin{itemize}
		\item If $G$ is a graph (i.e., $\maV_1=\emptyset$), then $\hat G = G$ and so is not a bipartite graph. 
		\item If $G$ is not a graph (i.e., $\maV_1 \ne \emptyset$), then, for any $i\in\maV_1$ and any $j\in \maE(i)\setminus\{i\}$, 
		$\hat G$ includes the triangle 
		$i\v \underline{i} \v j \v i$.  {(Observe that $j$ necessarily exists since $G$ is supposed connected with at least two nodes.)} 
		In particular, $\hat G$ is not a bipartite graph. 
	\end{itemize}
	As a consequence, again from Theorem 1 in \cite{MaiMoy16}, \textsc{Ncond}($\hat G$) is non-empty and thus, from Lemma \ref{lemma:equivalenceNcond}, 
	the set \textsc{Ncond}($G$) is also non-empty. 
	%
\end{proof}


\section{Proof of Theorem 5.1}
\label{sec:FCFM}

Let us recall that the multigraph $G=(\mathcal{V},\mathcal{E})$ is connected but is not a bipartite graph, with $|\mathcal{V}|\geq 2$. Then, in particular, $\textsc{Ncond}(G)\neq \emptyset$ (cf. Proposition \ref{prop:ncond}). Our product form result, Theorem~\ref{thm:FCFM}, follows from a reversibility scheme that generalizes to the case of multigraphs, the one constructed in \cite{MBM17}. In fact, we propose a proof that is simpler, at some points, than the one in \cite{MBM17}. We reproduce hereafter the main steps of this construction for easy reference, and only develop exhaustively the points that are specific to the present context, or based on different arguments. 

Hereafter, we denote by $\P_W$, the transition operator of the buffer-content Markov chain, that is, for all $w,w' \in\mathbb W,$ 
we write $P_W(w,w')=\pr{W_{n+1}=w' \mid W_n =w}, $ for any $n\in \N$. 


\subsection{Two auxiliary chains}
\label{subsec:aux}
As in section 3.2 of \cite{MBM17}, we first need to define two auxiliary Markov chains. For this, let us denote by $\oV$ an independent copy of $\V$, i.e., a set with the same cardinal formed with copies of elements of $\V$. We set $\bV= \V\cup\oV,$ and we define, for $\mathbf{w}\in\bV^*$,
\begin{align*}
&\V(\w )= \{a\in\V : |\w |_a>0\},\\
&\oV(\w )=\{\overline{a}\in\oV : |\w |_{\overline{a}}>0\}.
\end{align*}
For $a\in\bV,$ we will use the notation $\overline{\overline{a}}=a.$ 
\begin{definition} We define the \emph{backward detailed chain} as the process $(B_n)_{n\in\N}$ with values in $\bV^*$ given by 
	$B_0=\varepsilon$ and, for any $n\geq 1$,
	\begin{itemize}
		\item if $W_n=\varepsilon$ (i.e., all the items arrived up to time $n$ are matched at time $n$), then $B_n=\varepsilon$,
		\item otherwise, let $i(n)\in[\![1,n]\!]$ be the arrival time of the oldest item still in the buffer, then, the word $B_n$ is the word of length $n-i(n)+1$, defined, for any $\ell \in \ll 1,n-i(n)+1 \rr$, by
		\[(B_n)_\ell=\left\{\begin{array}{ll}
		V_{i(n)+\ell-1} \,\, &\mbox{if $V_{i(n)+\ell-1}$ has not been matched up to time $n$};\\
		\td{V_{k}}\,\, &\mbox{if $V_{i(n)+\ell-1}$ is matched at or before time $n$, with item $V_k$}\\
		&\mbox{(where $1 \leq k \le n$)}.\\
		\end{array}\right.\] 
		
	\end{itemize}
	
\end{definition}

In other words, 
the word $B_n$ gathers the class indexes of all unmatched items entered up to $n$, at the places corresponding to their arrival times, 
and the copies of the class  
indexes of the items matched before $n$, but after the arrival of the oldest unmatched item at $n$, at the place corresponding to the arrival time of 
their respective match. 


Observe that by the construction of $(B_n)_{n\in\N}$, for all $n\in\N$, the word $B_n$ necessarily contains all the letters of $W_n$. More precisely, for any $n\in\mathbb{N}$, $W_n$ is the restriction of the word $B_n$ to its letters in $\V$. 
Furthermore, $(B_n)_{n\in\N}$ is also a Markov chain since for any $n\geq 0$, the value of $B_{n+1}$ can be deduced from that of $B_n$ and from the class $V_{n+1}$ of the item entered at time $n+1$.  

A state $\w \in \bV^*$ is said to be \emph{admissible for $(B_n)_{n\in\N}$} if it can be reached by the chain $(B_n)_{n\in\N}$, under the $\textsc{fcfm}$ policy. We set
$$\mathbb{B}= \{\w \in\bV^* : \w  \; \text{is admissible for} \; (B_n)_{n\in\N}\}.$$
The following result can be proven exactly as Lemma 1 in \cite{MBM17}.
\begin{lemma}\label{lem:ADM_B}
	Let $\w =\w _1\dots \w _q \in \bV^*.$ Then, $\w\in\mathbb{B}$ if and only if $\w_1\in\V$ and for $1\leq i<j\leq q$,
	\begin{itemize}
		\item if $(\w_i,\w_j)\in \V^2$, then $\w_i\pv\w_j$,
		\item if $(\w_i,\w_j)\in \V\times\oV$, then $\w_i\pv{\overline \w_j}$.
	\end{itemize}
\end{lemma}

As a consequence of Lemma~\ref{lem:ADM_B}, any word $\w \in \mathbb{B}$ can be written as
$$
\w=b_1\overline{a_{11}}\overline{a_{12}}\dots\overline{a_{1k_1}}b_2\overline{a_{21}}\overline{a_{22}}\dots \overline{a_{2k_2}} b_3 \dots b_q \overline{a_{q1}} \dots \overline{a_{qk_q}},$$
where $q, k_1,\ldots, k_q\in\N,$ $b_1,\ldots,b_q\in\V$, $a_{ij}\in\V$ for $1\leq i\leq q, 1\leq j\leq k_i$, and
\[\left\{\begin{array}{ll}
\{ b_1, \dots, b_q \}=\V(\w )\in\mathbb{I}\left(\check G\right),\\
\forall i\in\ll 1,q\rr  , \; b_i\in\maV_1 \; \Rightarrow \left[\forall j\neq i, \; b_i\neq b_j\right], \\
\forall i \in \ll 1,q\rr  , \; \forall j\in \ll 1,k_i\rr  , \; a_{ij}\in\E(\{b_1,\dots,b_i\})^c.
\end{array}\right.\]
The transition operator of the chain $\suite{B_n}$ is denoted by $\P_B$, that is, for all $\bw,\bw' \in \mathbb B$, we write 
$\P_B(\bw,\bw')=\pr{B_{n+1}=\bw' \mid B_n=\bw}$, for all $n\in \N$. 

\begin{definition}
	We define the \emph{forward detailed chain} as the process $(F_n)_{n\in\N}$ with values in $\bV^*$ given by $F_0=\varepsilon$ (the empty word) and, for any $n\geq 1$,
	\begin{itemize}
		\item if $W_n=\varepsilon$ (i.e., all the items arrived up to time $n$ are matched at time $n$), then $F_n=\varepsilon$,
		\item otherwise, let $\mathscr U_n$ be the set of items arrived before time $n$ that are not matched at time $n$ (note that $\mathscr U_n$ is non-empty since $W_n \neq \varepsilon$). Also, set $$j(n)=\sup\left\{ m \geq n+1 : V_m \mbox{ is matched with an element of } \mathscr U_n\right\}.$$
		Observe that $j(n)$ is possibly infinite. Then, 
		if $j(n)$ is finite, $F_n$ is the word of $\bV^*$ of length $j(n)-n$ (respectively of $A^{\N}$ of length $+\infty$, if $j(n)=+\infty)$, 
		such that for any $\ell \in \ll 1,j(n)-n \rr$ (respectively $\ell \in \N_+)$, 
		\[(F_n)_\ell=\left\{\begin{array}{ll}
		V_{n+\ell} \,\, &\mbox{if $V_{n+\ell}$ is not matched with an item arrived up to $n$};\\
		\td{V_{k}}\,\, &\mbox{if $V_{n+\ell}$ is matched with item $V_k$, where $1 \leq k \le n$}.
		\end{array}\right.\] 
	\end{itemize}
\end{definition}


In other words, the word $F_n$ contains the copies of all the class indexes of the items entered up to time $n$ and matched after $n$, at the place corresponding to the arrival time of their respective match, 
together with the class indexes of all items entered after $n$ and before the last item matched with an item entered up to $n$, and not matched with an element entered before $n$, if any, at the place corresponding to their arrival time.  
Similarly to \cite{MBM17}, we make the three following simple observations: 
\begin{itemize}
	\item If $F_n \in \bV^*$ is finite, then $(F_n)_{j(n)-n} \in \td\maV$;
	\item $\suite{F_n}$ is a Markov chain; 
	\item If $F_n$ is a.s. an element of $\mathbf V^*$ for all $n\in\N$.
\end{itemize} 


As for the backward chain, we say that a state $\w \in \bV^*$ is \emph{admissible for $(F_n)_{n\in\N}$} if it can be reached by the chain $(F_n)_{n\in\N}$, under the $ \textsc{fcfm}$ policy. Then, we set 
$$\mathbb{F}=\{\w \in\bV^* : \w  \; \text{is admissible for} \; (F_n)_{n\in\N}\}$$
and we denote by $\P_F$ the transition operator of the chain $\suite{F_n}$ on $\mathbb F$.  
For any word $\w =\w _1\dots\w _n\in\bV^*$, let us define its reversed-copy by $\overleftarrow{\overline{\w }}= \overline{\w _n}\dots \overline{\w _1}\in\bV^*.$ Note that the map 
$\Psi: \bV^* \to \bV^*, \w \mapsto \overleftarrow{\overline{\w}}$ satisfies $\Psi\circ\Psi = Id_{\bV^*}$. Thus, $\Psi$ is a bijection and its inverse function is $\Psi^{-1}=\Psi$. 

\begin{lemma}\label{lem:ADMB_ADMF_iso}
	The map 
	\[\Phi:\begin{cases} {\mathbb B} \longrightarrow {\mathbb F}\\
	\w \longmapsto \Psi(\w)=\overleftarrow{\overline{\w}}\end{cases}\]
	is well-defined and bijective.
\end{lemma}

\begin{proof}
	Exactly as in Lemma 2 in \cite{MBM17}, it can be proven that $\bw\in\mathbb F$ if and only if $\Psi(\bw)\in \mathbb B$. 
	This guarantees that the mapping $\Phi$ is well-defined and surjective. It is injective because $\Psi$ clearly is so. 
\end{proof}

Let us define a measure $\nu$ on $\bV^*$ by $\nu(\varepsilon)= 1$ and
\begin{equation}
\label{eq:defnu}
\forall \w \in \bV^*\setminus\{\varepsilon\}, \; \nu(\w ) = \prod\limits_{i=1}^{|\V|} \mu(i)^{|\w |_i + |\overline{\w}|_i}.
\end{equation}

We can use the measure $\nu$ defined above to establish the following link between the dynamics of the chains $(B_n)_{n\in\N}$ and $(F_n)_{n\in\N}$. 
The following result can be established exactly as Lemma 3 in \cite{MBM17},


\begin{proposition}\label{prop:lienB_nF_n} For any $(\w ,\mathbf{w'})\in \mathbb{B}^2$, 
	we have that 
	\begin{equation*}\label{eq:lienB_nF_n}
	\nu(\w )\P_B(\w,\w') = \nu \left(\overleftarrow{\overline{\mathbf{w'}}}\right) \P_F\left(\overleftarrow{\overline{\w' }},\overleftarrow{\overline{\mathbf{w}}}\right).
	\end{equation*}
\end{proposition}


\subsection{Positive recurrence of $(B_n)_{n\in\N}$ and $(F_n)_{n\in\N}$.}
We will exploit the local balance equations of Proposition \ref{prop:lienB_nF_n} to derive stationary distributions of these two Markov chains. 
To this end, the following technical lemma will simplify the proofs.

\begin{lemma}\label{lem:nu(A*)}
	The measure $\nu$ defined by (\ref{eq:defnu}) satisfies the following properties:
	\begin{enumerate}
		\item For any $\A \subset \bV = \V\cup\oV$, we have $\nu(\A )=\mu(\V(\A )) + \mu\left(\overline{\oV(\A )}\right)$. 
		\item For any $\A _1,\ldots, \A _n\subset \bV, \; \nu(\A _1\dots\A _n)=\nu(\A _1)\dots\nu(\A _n)$. 
		In particular, $\nu(\A ^k)=\nu(\A )^k$.
		\item If $\A \subset \bV$ is such that $\nu(\A )<1$, then $\nu(\A ^*)={1\over 1-\nu(\A )}$.
	\end{enumerate}
\end{lemma}

\begin{proof} The first point follows from the definition of $\nu$ and the second point is a direct consequence of its multiplicative structure. 
	Regarding the third point, observe that $\A^*=\cup_{k\in\N} \A^k$, so that 
	\[\nu(\A^*)=\sum_{k\in\N}\nu\left(\mathcal{A}^k\right)=\sum_{k\in\N}\nu(\mathcal{A})^k.\]
\end{proof}
\noindent We can now state the following result,
\begin{proposition}\label{prop:B_nF_n_rec_pos}
	Suppose that $\mu\in\textsc{Ncond}(G)$. Then, 
	the chains $(B_n)_{n\in\N}$ and $(F_n)_{n\in\N}$ are positively recurrent and admit respectively the restrictions on $\mathbb B$ and on $\mathbb F$ 
	of $\nu$ (that is, $\nu_B(\w ) = \nu_F(\Phi(\w )) =\nu(\w)$, for any $\w\in\mathbb B)$  as unique stationary measure (up to a multiplicative constant), respectively on $\mathbb B$ and $\mathbb F$. 
\end{proposition}

\begin{proof}
	Let $\mu\in\textsc{Ncond}(G)$.\\
	\underline{Step 1}: we first prove that $\nu_B$ is a stationary measure for the chain $(B_n)_{n\in\N}.$\\
	For this, let us fix $\mathbf{w'}\in \mathbb{B}$. Then we have that 
	\begin{align*}
	\displaystyle\sum\limits_{\w \in \mathbb{B}} \frac{\P_B(\w,\w')\nu_B(\w )}{\nu_B(\mathbf{w'})}
	&= \sum\limits_{\w \in \mathbb{B}}
	\frac{\P_F\left(\overleftarrow{\overline{\w'}},\overleftarrow{\overline{\mathbf{w}}}\right)\nu_B \left(\overleftarrow{\overline{\mathbf{w'}}}\right)}{\nu_B(\mathbf{w'})}\\
	&= \sum\limits_{\w \in \mathbb{B}}
	\P_F\left(\overleftarrow{\overline{\w'}},\overleftarrow{\overline{\mathbf{w}}}\right)= 1,
	\end{align*}
	where the first equality follows from Proposition ~\ref{prop:lienB_nF_n}, 
	the second from the fact that 
	$\nu_B \left(\overleftarrow{\overline{\mathbf{w'}}}\right)=\nu_B(\mathbf{w'})$ and the last, from Lemma~\ref{lem:ADMB_ADMF_iso}.  
	Thus, for all $\bw'\in\mathbb B$, we have that \[\nu_B(\mathbf{w'})=\sum_{\w \in \mathbb{B}} \P_B (\mathbf{w},\w')\nu_B(\w ),\]
	which means exactly that $\nu_B$ is a stationary measure for the chain $(B_n)_{n\in\N}$.
	
	\bigskip
	
	\noindent
	\underline{Step 2}: we now prove that $\nu_B(\mathbb{B})<\infty$.
	
	By Lemma \ref{lem:ADM_B}, we know that 
	$\w\in\mathbb B\setminus\{\varepsilon\}$ if and only if $\w$ belongs to a set 
	\begin{multline}
	b_1\,\A^*_1\,b_2\,\A^*_2\dots b_q\,\A^*_q\\
	= \left\{w\in \mathbf V^* : w=b_1w^1b_2w^2\dots b_qw^q;\,w^i\in \A^*_i\mbox{, for all } i \in \llbracket 1,q \rrbracket \right\},\quad q\ge 1,
	\label{eq:defwords}
	\end{multline}
	where $b_1,\ldots, b_q$ are elements of 
	$\V$ such that $\{ b_1, \dots, b_q \}\in\mathbb{I}\left(\check G\right)$ and such that for all distinct $i, j$ in $\llbracket 1,q \rrbracket$, 
	$b_i\in\maV_1$ implies that $ b_i\neq b_j$, and where we denote $$\A_i=\overline{\E(\{b_1,\dots,b_i\})^c},\quad i\in\llbracket 1,q \rrbracket.$$
	
	Equivalently, by highlighting only the first occurrence of each letter of $\V$ appearing in $\w$ and employing a similar notation to (\ref{eq:defwords}) 
	we obtain that $\w\in\mathbb B\setminus\{\varepsilon\}$ if and only if $\w$ belongs to some set of the form 
	$$\C_{ I,\sigma}= e_{\sigma(1)}\,\B^*_{\sigma(1)}\,e_{\sigma(2)}\,\B^*_{\sigma(2)}\dots e_{\sigma(| I|)}\,\B^*_{\sigma(| I|)},$$ 
	where $ I = \left\{e_1,...,e_{| I|}\right\}\in\mathbb{I}\left(\check G\right)$, 
	$\sigma \in\mathfrak{S}_{| I|}$, 
	and where we denote \[\B_{\sigma(i)}=\overline{\E(\{e_{\sigma(1)},\dots,e_{\sigma(i)}\})^c}\,\cup\, (\{e_{\sigma(1)},\ldots,e_{\sigma(i)}\}\cap \mathcal{V}_2),\quad i\in\llbracket 1,| I| \rrbracket.\]
	In view of assertion (1) of Lemma~\ref{lem:nu(A*)}, we have that for all $i\in\llbracket 1,k \rrbracket$, 
	\begin{align*}
	\nu_B(\B_{\sigma(i)})&=\mu(\E(\{e_{\sigma(1)},\dots,e_{\sigma(i)}\}^c)+\mu(\{e_{\sigma(1)},\dots,e_{\sigma(i)}\}\cap \mathcal{V}_2)\\
	&=1-\mu(\E(\{e_{\sigma(1)},\dots,e_{\sigma(i)}\}))+\mu(\{e_{\sigma(1)},\dots,e_{\sigma(i)}\}\cap \mathcal{V}_2).
	\end{align*}
	Since $\{e_{\sigma(1)},\dots,e_{\sigma(i)}\}\in\mathbb{I}(\check G)$, we have, by definition, that 
	$\{e_{\sigma(1)},\dots,e_{\sigma(i)}\}\cap\mathcal{V}_2\in\mathbb{I}(G)$ and since the measure $\mu$ satisfies $\textsc{Ncond}(G)$, it follows that 
	\begin{align*}
	\mu\left(\{e_{\sigma(1)},\dots,e_{\sigma(i)}\}\cap\mathcal{V}_2\right) &< \mu\left(\maE\left(\{e_{\sigma(1)},\dots,e_{\sigma(i)}\}\cap\mathcal{V}_2\right)\right)\\
	&\leq \mu\left(\maE\left(\{e_{\sigma(1)},\dots,e_{\sigma(i)}\}\right)\right)\end{align*}
	and thereby, that 
	$\nu_B(\mathcal{B}_i)<1$. As a conclusion, applying successively all assertions {of} Lemma~\ref{lem:nu(A*)}, we obtain that for all such $ I$ and $\sigma$, 
	\begin{equation*}
	\nu_B(\C_{ I,\sigma})=\prod\limits_{i=1}^{| I|} \frac{\mu(e_{\sigma(i)})}{\mu(\mathcal{E}(\{e_{\sigma(1)},\dots,e_{\sigma(i)}\}))-\mu(\{e_{\sigma(1)},\dots,e_{\sigma(i)}\}\cap\maV_2)}.
	\end{equation*}
	
	The set $\mathbb B$ is the disjoint union of the sets $\C_{ I,\sigma}$, for $ I$ in the finite set $\mathbb I\left(\check G\right)$, and $\sigma$ 
	in the finite set $\mathfrak{S}_{| I|}$.  
	It follows that $\nu_B(\mathbb{B})$ is finite, and given by 
	\begin{align}
	\nu_B(\mathbb B)&=\nu_B(\varepsilon) + \sum\limits_{ I\in\mathbb{I}\left(\check G\right)}\sum\limits_{\sigma\in\mathfrak{S}_{| I|}} 
	\nu_B(\C_{ I,\sigma})\nonumber\\
	&=1+ \sum\limits_{ I\in\mathbb{I}\left(\check G\right)}\sum\limits_{\sigma\in\mathfrak{S}_{| I|}}\prod\limits_{i=1}^{| I|} \frac{\mu(e_{\sigma(i)})}{\mu(\mathcal{E}(\{e_{\sigma(1)},\dots,e_{\sigma(i)}\}))-\mu(\{e_{\sigma(1)},\dots,e_{\sigma(i)}\}\cap\maV_2)}.\label{eq:defalpha0}
	\end{align}
	
	\bigskip
	
	\noindent
	\underline{Step 3}: we conclude with the positive recurrence of the two chains. 
	
	By the results above, the chain $(B_n)_{n\in\N}$ has a stationary probability distribution on $\mathbb{B}$, which is given by the measure $\nu_B$ normalized by $\nu_B(\mathbb{B})$. 
	
	Observe that the chain is irreducible on $\mathbb{B}$. To see this, let $\bw\in\mathbb B$ and first observe that the empty word $\varepsilon$ leads to $\bw$ with positive probability for the transitions of $\suite{B_n}$ 
	(this is the constructive argument proving Lemma \ref{lem:ADM_B} - see the proof of Lemma 1 in \cite{MBM17}). Conversely, denoting by $b_1,\dots,b_q$ the elements of $\maV(\bw)$, it is easy to see that the word $\bw$ leads to the empty word with positive probability for the transitions of $\suite{B_n}$ : 
	indeed, by the definition of the policy {\sc fcfm}, if the chain is in the state $\bw$, then it will reach the empty state after exactly $q$ steps, 
	by seeing the successive arrivals of $q$ elements of respective classes in $\maE(b_1)$, $\maE(b_2), \dots , \maE(b_q)$, which concludes the proof 
	of irreducibility. 
	
	It then follows that the chain $\suite{B_n}$ is positively recurrent on $\mathbb{B}$ and that its stationary probability distribution is unique. Consequently, $\nu_B$ is the unique stationary measure (up to a multiplicative constant) of the chain $(B_n)_{n\in\N}$.
	
	Now, 
	as in step 1, we obtain that for all $\w'\in \mathbb B$, 
	\begin{equation*}
	\nu_F(\Phi(\mathbf{w'}))=\sum\limits_{\w \in \mathbb{B}} \P_F (\Phi(\mathbf{w}),\Phi(\w' ))\nu_F(\Phi(\w )).
	\end{equation*}
	Using Lemma~\ref{lem:ADMB_ADMF_iso}, we deduce that $\nu_F$ is a stationary measure for the chain $(F_n)_{n\in\N}$. 
	Then, step 2 shows equivalently that $\nu_F(\mathbb{F}) <\infty$. 
	So, the chain $(F_n)_{n\in\N}$ has a stationary probability distribution on $\mathbb{F}$, which is given by the measure $\nu_F$ normalized by $\nu_F(\mathbb{F})$. 
	
	Similarly as above, we can check that the chain $\suite{F_n}$ is irreducible on $\mathbb{F}$. First, the empty word leads with positive probability 
	to any element $\bw\in\mathbb F$, as can be checked using the same constructive argument as in the proof of Lemma \ref{lem:ADM_B}. 
	Conversely, suppose that the chain $\suite{F_n}$ is at time $n$ in a state 
	\[\bw = a_{qk_q} \dots a_{q1}\td{a_q}\dots\td{a_3} a_{2k_2}\dots a_{21}\td{a_2}a_{1k_1}\dots a_{11} \td{a_1}\in\mathbb F\]
	and let $r= q+\sum_{i=1}^q k_i$ be the length of $\bw$. Then, going forward in time, perform the {\sc fcfm} matching of 
	the `unmatched' elements of respective classes in $\maV(\bw)$. Say there remains in the system, at time $n+r$, $\ell$ unmatched elements denoted $c_1,c_2,\dots,c_\ell$ in their order of arrivals. Then, the chain can return to the empty state in particular if the first $\ell$ arrivals 
	after time $n+r$ (excluded) are of respective classes in $\maE(c_1)$, $\maE(c_2),\dots ,\maE(c_\ell)$. This concludes the proof of irreducibility. 
	
	As a consequence, the chain $\suite{F_n}$  is positively recurrent on $\mathbb{F}$ and its stationary probability distribution is unique. 
	Consequently, $\nu_F$ is the unique stationary measure (up to a multiplicative constant) of the chain $(F_n)_{n\in\N}$, which concludes the proof. 
\end{proof}

\subsection{Positive recurrence of $(W_n)_{n\in\N}$} 
The Markov chain $(W_n)_{n\in\N}$ can be seen as the projection of the chain $(B_n)_{n\in\N}$ on $\V^*$. In order to obtain the stationary probability distribution of $(W_n)_{n\in\N}$ from the one of $(B_n)_{n\in\N}$, we will use the following lemma:
\begin{lemma}\label{lem:projCM} Let $\P_Y$ and $\P_{Y'}$ be the transition matrices of two homogeneous 
	Markov chains $\suite{Y_n}$ and $\suite{Y'_n}$ with values in some countable sets $S$ and $S'$ respectively, and consider a map 
	$p:S \to S'$ satisfying 
	\begin{equation*}
	\forall a',b'\in S', \; \forall a\in p^{-1}(\{a'\}), \P_Y(a,p^{-1}(\{b'\})) = \P_{Y'}(a',b').
	\end{equation*} 
	Then, if a measure $\mu$ is invariant for $\P_Y$, the measure $\mu'$ defined by $\mu'(a')= \mu(p^{-1}(\{a'\}))$ for all $a'\in S'$, 
	is an invariant measure for $\P_{Y'}$ on $S'$.
\end{lemma}

\begin{proof} Let $\mu$ be an invariant measure for $\P_Y$, and let $b'\in S'$. We have
	\begin{align*}
	\sum_{a'\in S'} \mu'(a')\P_{Y'}(a',b')
	&= \sum_{a'\in S'} \left(\sum_{a\in p^{-1}(\{a'\})} \mu(a)\right)\P_{Y'}(a',b') \\
	&= \sum_{a'\in S'} \sum_{a\in p^{-1}(\{a'\})} \mu(a) \P_Y(a,p^{-1}(\{b'\}))\\
	&= \sum_{s\in S} \mu(s) \P_Y(s,p^{-1}(\{b'\}))=\mu(p^{-1} (\{b'\}))=\mu'(b'),
	\end{align*}
	meaning that $\mu'$ is invariant for $\P_{Y'}$.
\end{proof}

For $\mu\in\textsc{Ncond}(G)$, let us denote by $\Pi_B$ the unique stationary probability law associed to the chain $(B_n)_{n\in\N}$ (cf. Prop.~\ref{prop:B_nF_n_rec_pos}). It is defined by
\begin{equation*}
\forall \w \in \mathbb{B}, \; \Pi_B(\w )= \alpha \nu_B(\w), 
\end{equation*}
where $\alpha=(\nu_B(\mathbb B))^{-1}$ is given by (\ref{eq:defalpha}) in view of (\ref{eq:defalpha0}). 
Let us now introduce the projection 
\[p:\begin{cases}
\mathbb{B} \longrightarrow \mathbb{W}\\
\bw \longmapsto \bw|_{\V},
\end{cases}
\]
which is well-defined from Lemma~\ref{lem:ADM_B}. We have the following result: 

\begin{proposition}\label{prop:W_nRec}
	Let $\mu\in\textsc{Ncond}(G)$. Then, the Markov chain $(W_n)_{n\in\N}$ is positively recurrent, and its unique stationary probability distribution is the measure $\Pi_W$ defined on $\mathbb W$ by:
	\begin{equation*}
	\forall w\in\mathbb{W}, \; \Pi_W(w)= \Pi_B(p^{-1}(w))=\sum_{\w \in \mathbb{B} \; : \; \w |_{\V}=w} \Pi_B(\w). 
	\end{equation*}
\end{proposition}

\begin{proof} Let $\mu\in\textsc{Ncond}(G)$. We can apply Lemma~\ref{lem:projCM} to $\P_B$ and $\P_W$ to prove that $\Pi_W$ is a stationary distribution for $(W_n)_{n\in\N}$. 
	Indeed, using the fact that for any $n\in\mathbb{N}$, $W_n$ is the restriction of the word $B_n$ to its letters in $\V$, we have that 
	\begin{equation*}
	\forall w,w'\in \mathbb{W}, \; \forall \w\in p^{-1}(\{w\}), \; \P_B(\w,p^{-1}(\{w'\})) = \P_W(w,w').
	\end{equation*} 
	The measure $\Pi_W$ is a probability distribution on $\mathbb{W}$, since $\Pi_W(\mathbb{W})=\Pi_B(p^{-1}(\mathbb{W}))=\Pi_B(\mathbb{B})=1$.
	The chain $(W_n)_{n\in\N}$ being irreducible on $\mathbb{W}$, it follows that $\Pi_W$ is its unique stationary probability distribution.
\end{proof}

\subsection{Concluding the proof}
We first show that $\textsc{Stab}(G,\fcfm)=\textsc{Ncond} (G)$. 
First, we know from Proposition \ref{thm:mainmono} that $\textsc{Stab}(G,\fcfm) \subset \textsc{Ncond}(G).$ Also, 
from Proposition \ref{prop:ncond}, $ \textsc{Ncond}(G)\neq \emptyset$, since $G$ is not a bipartite graph. 
Then, for all $\mu \in \textsc{Ncond}(G)$, by Proposition~\ref{prop:W_nRec}, the chain $(W_n)_{n\in\N}$ is positively recurrent on $\mathbb{W}$. 
So, $\textsc{Ncond}(G)\subset \textsc{Stab}(G,\fcfm)$, and therefore $\textsc{Stab}(G,\fcfm)=\textsc{Ncond} (G)$. 

\bigskip

We now fix $\mu\in \textsc{Ncond}(G)$, and compute explicitly the unique stationary probability distribution $\Pi_W$ of the chain 
$(W_n)_{n\in\N}$. 
First, if $w=\varepsilon$, then $p^{-1}(\{w\})=\varepsilon$ and $\Pi_W(\varepsilon)=\alpha$, given by (\ref{eq:defalpha}). 
Now, fix $w\neq\varepsilon$ in $\mathbb W$. 
By (\ref{eq:defwords}), we know that if $w=w_1\dots w_q\in \mathbb{W}$, $q\ge 1$, then
$p^{-1}(\{w\})=w_1\,\A^*_1\,w_2\,\A^*_2\dots w_q\,\A^*_q,$ with $\A_i=\overline{\E(\{w_1,\dots,w_i\})^c}$, for all $i\in\llbracket 1,q \rrbracket$. 
Applying Lemma \ref{lem:nu(A*)} and observing that 
for all $i$, $\mu\left(\overline{\A_i}\right)<1$ since $\A_i\varsubsetneq\oV$, it follows that 
\begin{align*}
\Pi_W(w)&=\Pi_B(w_1\,\A^*_1\,w_2\,\A^*_2\dots w_q\,\A^*_q)\\
&=\alpha \nu_B(w_1\,\A^*_1\,w_2\,\A^*_2\dots w_q\,\A^*_q)\\
&=\alpha \prod_{i=1}^q{\mu(w_i)\over 1-\mu\left(\overline{\A_i}\right)} 
=\alpha \prod_{i=1}^q{\mu(w_i)\over \mu(\E(\{w_1,\dots,w_i\}))}\cdot
\end{align*}
The proof is complete. 


\section{Remaining proofs}\label{sec:otherproofs}

Throughout the section $G$ is a connected multigraph, $\check G$ is its maximal subgraph and $\hat G$ denotes its minimal blow-up graph. 
To simply compare a $(G,\Phi,\mu)$ system with the two corresponding matching models on graphs 
$(\hat G,\hat \Phi,\hat\mu)$ and $(\check G,\check \Phi,\mu)$, let us add a ``hat'' (resp. a ``check'') 
to all characteristics of the second (resp. the third) system: in particular, we denote, for all $n$, by $\hat V_n$ (resp. $\check V_n$), the class of the item entering in the 
$(\hat G,\hat\Phi,\hat\mu)$ \big(resp. $(\check G,\check\Phi,\mu)$\big) system at time $n$. 
The natural Markov chain of the system is then denoted by $(\hat W_n)_{n\in\mathbb{N}}$ \big(resp. $(\check W_n)_{n\in\mathbb{N}}$\big) 
and its state space, by $\hat{\mbW}$ \big(resp. $\check{\mbW}$\big). Specifically, 
\begin{align*}
\hat{\mathbb W} &=\Bigl\{ w\in \left(\maV\cup \underline{\maV_1}\right)^*\; : \; \forall  i\neq j \; \text{s.t.} \; (i,j) \in \hat\maE, \; |w|_i|w|_j=0 \Bigr\};\\
\check{\mathbb W} &=\Bigl\{ w\in \maV^*\; : \; \forall  i\neq j \; \text{s.t.} \; (i,j) \in \check\maE, \; |w|_i|w|_j=0 \Bigr\}.
\end{align*} 
Observe that we have $\mathbb W \subset \check{\mathbb W} \subset \hat{\mathbb W}$. 

For any measurable mapping $F:\mbW \to \R$ (resp. $\check{\mbW} \to \R$, $\hat{\mbW} \to \R$) and any given $w\in\mbW$ (resp. $\hat{w}\in\hat{\mbW}$, $\check{w}\in\check{\mbW}$), we denote by 
$\Delta F^\Phi_\mu(w)$ (resp. $\hat{\Delta} F^{\hat\Phi}_{\hat\mu}(\hat w)$, $\check{\Delta} F^{\check\Phi}_{\check\mu}(\check w)$) the drift of the chain $\suite{W_n}$ (resp. $\suite{\hat{W}_n}$, 
$\suite{\check{W}_n}$) starting from $w$ (resp. $\hat{w}$, $\check{w}$) for a $(G,\Phi,\mu)$  (resp. $(\hat G,\hat \Phi,\hat \mu)$, $(\check G,\check\Phi,\check\mu)$) system. 
In other words, for any $n\in\N$ we denote 
\begin{align*}
\Delta^\Phi_\mu F(w)&=\esp{F(W_{n+1})-F(W_n)\,\big|\,W_n=w} ;\\
\hat{\Delta}^{\hat\Phi}_{\hat\mu}F(\hat w)&=\esp{F(\hat W_{n+1})-F(\hat W_{n})\,\big|\,\hat W_n=\hat w};\\ 
\check{\Delta}^{\check\Phi}_{\check\mu}F(\check w)&=\esp{F(\check W_{n+1})-F(\check W_{n})\,\big|\,\check W_n=\check w}. 
\end{align*}

\subsection{Drift inequalities}
\label{sec:drift}


Consider the following  mappings, 

\begin{equation}
\label{eq:QuadraticFunction}
Q:\begin{cases}
\hat{\mathbb W} &\longrightarrow\R^+\\
\hat w &\longmapsto \sum\limits_{i=1}^{|\maV|}\left(|\hat w|_i\right)^2 + \sum\limits_{i=1}^{|\maV_1|}\left(|\hat w|_{\underline i}\right)^2;
\end{cases}\end{equation}
\begin{equation}
\label{eq:LinearFunction}
L: \begin{cases}
\hat{\mathbb W} &\longrightarrow \R^+\\
\hat w &\longmapsto \sum\limits_{i=1}^{|\maV|} \alpha_i|\hat w|_i + \sum\limits_{i=1}^{|\maV_1|}\alpha_{\underline{i}}|\hat w|_{\underline i},\;\;
\end{cases}
\end{equation}
\qquad \qquad \qquad\qquad\qquad\qquad \qquad\qquad\qquad with $\alpha_i,\alpha_{\underline{i}}\in\R^+$ and  $\alpha_i=\alpha_{\underline{i}}$ for all $i\in\maV_1.$\\

Where it follows from the observation above that $Q$ and $L$ are well defined also on $\mathbb W$ and $\check{\mathbb W}$.

\medskip

\begin{definition}
	\label{def:extendspol}
	Let $\Phi$ and $\hat\Phi$ be two admissible matching policies, respectively on $G$ and $\hat G$. 
	We say that $\hat\Phi$ extends $\Phi$ on $\hat G$ if, for any $\mu\in\mathscr M({\maV})$ and $\hat\mu\in\mathscr M(\hat\maV)$, whenever both systems $(G,\Phi,\mu)$ and $(\hat G,\hat\Phi,\hat\mu)$ are in the same state $w\in \mathbb W$ and welcome the same arrival, $\Phi$ and $\hat\Phi$ induce the same choice of match, if any. 
\end{definition}

\noindent We have the following result,

\begin{proposition}
	\label{prop:extquad}
	Let $\Phi$ be an admissible policy on $G$ and $\mu\in\mathscr M(\maV)$. 
	Let $\hat{\Phi}$ be a matching policy extending $\Phi$ on $\hat G$ and $\hat \mu$ be a measure extending $\mu$ on $\hat\maV$. 
	Then, 
	for all $w \in\mbW$ we have that 
	$\Delta^\Phi_\mu Q(w)
	\leq 
	\hat\Delta^{\hat\Phi}_{\hat\mu} Q(w). $
\end{proposition}

\begin{proof}
	Fix $w \in\mbW$ throughout the proof. Recall that for all $i\in\maV_1$ (if any), we have that $|w|_i\in\{0,1\}$, and let us set  
	\begin{align*}
	\mathcal O_{w} &=\{i\in \maV_1:|w|_i = 1\},\\
	\mathcal Z_{w} &=\{i\in \maV_1: |w|_i = 0\mbox{ and } |w|_j = 0\mbox{, for any }j \in \maE(i)\}.
	\end{align*}
	
	
	\medskip 
	
	First, for any $i \in\maV_2$, an incoming item of class $i$ finding the system $(G,\Phi,\mu)$ in a state $w$ finds the same possible matches (if any) as an incoming item of class $i$ finding the system 
	$(\hat G,\hat \Phi,\hat \mu)$ in a state $w$. 
	As $\hat\Phi$ extends $\Phi$, the choice of the match (if any) of the incoming item of class $i$ is then the same, or follows the same distribution in case of a draw, in both systems. 
	Thus, for all $n\in\mathbb{N}$, as $|\hat W_{n+1}|_{\underline{\maV_1}}=0$ the conditional distribution of $W_{n+1}\ind_{\{V_{n+1}=i\}}$ given $\{W_n= w\}$ equals the conditional distribution of 
	$\hat W_{n+1}\ind_{\left\{\hat V_{n+1}= i\right\}}$ given $\left\{\hat W_n= w\right\}$. 
	Therefore, we obtain that for all $n\in\mathbb{N}$,  
	\begin{multline}\label{eq:compML1}
	\esp{\left(Q(W_{n+1})-Q(W_n)\right)\ind_{\{V_{n+1} \in \maV_2\}}\,\big|\,W_n= w}\\
	\begin{aligned}
	= &\sum_{i\in\maV_2}\esp{Q(W_{n+1})\ind_{\{V_{n+1} =i\}}\,\big|\,W_n= w}-\sum_{i\in\maV_2}\esp{Q(W_{n})\ind_{\{V_{n+1} =i\}}\,\big|\,W_n= w}\\
	= &\sum_{i\in\maV_2}\esp{Q(W_{n+1})\ind_{\{V_{n+1} =i\}}\,\big|\,W_n= w} - \mu(\maV_2)Q(w)\\
	= &\sum_{i\in\maV_2}\esp{Q\left(\hat W_{n+1}\right)\ind_{\left\{\hat V_{n+1}=i\right\}}\,\big|\,\hat W_n= w} - \hat\mu(\maV_2)Q(w)\\
	= &\;\esp{\left(Q\left(\hat W_{n+1}\right)-Q\left(\hat W_{n}\right)\right)\ind_{\left\{\hat V_{n+1} \in \maV_2\right\}}\,\big|\,\hat W_n=w},\end{aligned}
	\end{multline}
	where the third equality due to the equality of the conditional distribution of $W_{n+1}\ind_{\{V_{n+1}=i\}}$ given $\{W_n= w\}$ and $\hat W_{n+1}\ind_{\left\{\hat V_{n+1}= i\right\}}$ given $\left\{\hat W_n= w\right\}.$ 
	\medskip
	
	Likewise, if a system $(G,\Phi,\mu)$ is in state $w$, then, for any $i \in \mathcal V_1\cap (\mathcal O_{w})^c \cap (\mathcal Z_{w})^c$, an incoming of class $i$ 
	finds the same possible matches (if any) as an incoming item of class $i$ or of class $\underline{i}$ finding the system $(\hat G,\hat \Phi,\hat \mu)$ in the state $w$. 
	Again, the choice of the match of the latter is the same in both systems, or follows the same distribution in case of a draw. Like in (\ref{eq:compML1}), we obtain that, for all $n\in\mathbb{N}$, 
	\begin{multline}\label{eq:compML2}
	\esp{\left(Q(W_{n+1})-Q(W_n)\right)\ind_{\{V_{n+1} \in \mathcal V_1\cap (\mathcal O_{w})^c \cap (\mathcal Z_{w})^c\}}\,\big|\,W_n= w}\\
	= \mathbb E\Biggl[\!\left(\!Q\left(\hat W_{n+1}\right)-Q\left(\hat W_{n}\right)\!\right)\!\ind_{\left\{\hat V_{n+1} \in \left(\mathcal V_1\cap (\mathcal O_{w})^c \cap (\mathcal Z_{w})^c\right) \cup \left(\underline{\mathcal V_1\cap (\mathcal O_{w})^c \cap (\mathcal Z_{w})^c}\right)\right\}}\big|\,\hat W_n=w\!\Biggl],
	\end{multline}
	where we also use the fact that 
	\[\hat{\mu}\left(\left(\mathcal V_1\cap (\mathcal O_{w})^c \cap (\mathcal Z_{w})^c\right) \cup \left(\underline{\mathcal V_1\cap (\mathcal O_{w})^c \cap (\mathcal Z_{w})^c}\right)\right) = \mu(\mathcal V_1\cap (\mathcal O_{w})^c \cap (\mathcal Z_{w})^c).\]
	Now, if a system $(G,\Phi,\mu)$ is in state $w$, then, for any $i\in \mathcal Z_{w}$, and incoming item of class $i$ finds no possible match. 
	So, it is stored in line and the coordinate $i$ of the chain increases from 0 to 1. Consequently, for any $n\in\mathbb{N}$, conditional on $\{W_n= w\}$ and for any such $i$, we get that 
	\begin{equation}
	\label{eq:compML3bis}
	\left(Q(W_{n+1})-Q(W_n)\right)\ind_{\{V_{n+1} =i\}}
	= \left(Q(wi)-Q(w)\right)\ind_{\{V_{n+1} =i\}}= \ind_{\{V_{n+1} =i\}}.
	\end{equation}
	Similarly, if the system $(\hat G,\hat\Phi,\hat\mu)$ is in the state $w$ and the entering item is of class $i\in \mathcal Z_{w}$ or of class $\underline i \in \underline{\mathcal Z_{w}}$, then, in both cases, the entering item does not find any possible match in $(\hat G,\hat\Phi,\hat\mu)$ and so the coordinate $i$ or $\underline i$ of the Markov chain increases from 0 to 1. Thus, given that $\left\{\hat W_n= w\right\}$, we get that 
	\begin{multline}
	 \label{eq:compML3ter}
	\left(Q\left(\hat W_{n+1}\right)-Q\left(\hat W_{n}\right)\right)\ind_{\left\{\hat V_{n+1} \in \{i,\underline i\}\right\}}\\
	= \left(Q(wi)-Q(w)\right)\ind_{\left\{\hat V_{n+1} =i\right\}}
	+  \left(Q\left(w\underline{i}\right)-Q(w)\right)\ind_{\left\{\hat V_{n+1} = \underline i\right\}}
	=\ind_{\left\{\hat V_{n+1} =i\right\}}+\ind_{\left\{\hat V_{n+1} = \underline i\right\}}.\end{multline}
	This, together with (\ref{eq:compML3bis}), entails that
	\begin{multline}\label{eq:compML3}
	\esp{\left(Q(W_{n+1})-Q(W_n)\right)\ind_{\{V_{n+1} \in\mathcal Z_{w}\}}\,|\,W_n= w}\\
	\begin{aligned}
	=& \sum_{i\in \mathcal Z_{w}} \esp{ \ind_{\{V_{n+1} =i\}}\,|\,W_n= w}\\
	=& \;\mu(\maZ_{w})\\
	=& \;\hat\mu\left(\maZ_{w}\right) + \hat\mu\left(\underline{\maZ_{w}}\right)\\
	=& \sum_{i\in \mathcal Z_{w}} \esp{ \ind_{\{\hat V_{n+1} =i\}}+\ind_{\{\hat V_{n+1} = \underline i\}}\,|\,\hat W_n= w}\\
	=& \;\esp{\left(Q(\hat W_{n+1})-Q(\hat W_{n})\right)\ind_{\{\hat V_{n+1} \in \mathcal Z_{w}\cup \underline{\maZ_{w}}\}}\,|\,\hat W_n=w},
	\end{aligned}\end{multline}
	where the first equality due to the equation (\ref{eq:compML3bis}). The third equality follows from the equation (\ref{eq:MuHatMu}) as $\maZ_{w}\subset\maV_1.$ The fourth equality follows from the equation (\ref{eq:compML3ter}).\\
	
	\medskip
	\noindent
	At last, if the system $\left(G,\Phi,\mu\right)$ is in the state $w$, then, the arrival of a class $i$-item, for $i\in \mathcal O_{w}$, leads to the matching of two items 
	of class $i$. Therefore, as $|w|_i=1$ we obtain  
	\begin{equation}
	\left(Q(W_{n+1})-Q(W_n)\right)\ind_{\{V_{n+1} \in \mathcal O_{w}\}}
	= -\ind_{\{V_{n+1} \in \mathcal O_{w}\}}.\label{eq:compML4bis}
	\end{equation}
	Now, suppose that the system $\left(\hat G,\hat\Phi,\hat\mu\right)$ is in the state $\hat W_n =w$. 
	Then, if an item of class $i\in \mathcal O_{w}$ enters in the system, the corresponding item is not matched and the number of $i$-items in the system increases from 1 to 2. 
	Therefore we get that 
	\begin{equation}
	\label{eq:compML4ter}
	\left(Q\left(\hat W_{n+1}\right)-Q\left(\hat W_{n}\right)\right)\ind_{\left\{\hat V_{n+1} \in \mathcal O_{w}\right\}}
	= 3\ind_{\left\{\hat V_{n+1} \in \mathcal O_{w}\right\}}.
	\end{equation}
	\noindent
	If on the other hand, an item of class $\underline i\in \underline{\mathcal O_{w}}$ enters in the same system $(\hat G,\hat\Phi,\hat\mu)$, 
	then, the corresponding item match with the stored class $i$-item and so the coordinate $i$ of the chain decreases to 0. Thus, 
	\begin{equation*}
	\left(Q\left(\hat W_{n+1}\right)-Q\left(\hat W_{n}\right)\right)\ind_{\left\{\hat V_{n+1} \in \underline{\mathcal O_{w}}\right\}}
	= -\ind_{\left\{\hat V_{n+1} \in \underline{\mathcal O_{w}}\right\}}.
	\end{equation*}
	Gathering this with (\ref{eq:compML4bis}) and (\ref{eq:compML4ter}) and then taking expectations, we obtain that 
	\begin{multline}\label{eq:compML4}
	\esp{\left(Q(W_{n+1})-Q(W_n)\right)\ind_{\{V_{n+1} \in\mathcal O_{w}\}}\,\big|\,W_n= w}\\
	= \esp{\left(Q\left(\hat W_{n+1}\right)-Q\left(\hat W_{n}\right)\right)\ind_{\left\{\hat V_{n+1} \in \mathcal O_{w}\cup \underline{\mathcal O_{w}}\right\}}\,\big|\,\hat W_n=w} -4\hat\mu\left(\mathcal O_{w}\right).
	\end{multline}
	Finally, (\ref{eq:compML1}) together with (\ref{eq:compML2}), (\ref{eq:compML3}) and (\ref{eq:compML4}) give that
	\begin{equation} \label{eq:compMLfinal}
	\Delta^\Phi_\mu Q(w)
	=\hat\Delta^{\hat\Phi}_{\hat\mu} Q(w)- 4\hat\mu\left(\mathcal O_{w}\right),
	\end{equation}
	which concludes the proof. 
\end{proof}

\begin{definition}
	\label{def:reducesspol}
	Let $G$ be a connected multigraph and $\Phi$ be an admissible matching policy on $G$. We say that $\check\Phi$ reduces $\Phi$ if, 
	for any ${\mu}\in\mathscr M({{\maV}})$, whenever the two systems  $(\check G,\check\Phi,\mu)$ and $(G,\Phi,\mu)$ are in the 
	same state ${w} \in {\mbW}$ and welcome the same arrival, then $\check\Phi$ and $\Phi$ induce the same choice of match, if any.
\end{definition}



\begin{proposition}
	\label{prop:EspGmultiLeqEspGgrapg}
	Let $G=(\maV,\maE)$ be a connected multigraph and $\Phi$ be a class admissible policy on $G$ and $\mu\in\mathscr M(\maV)$. 
	Let $\hat{\Phi}$ be a matching policy extending $\Phi$ on $\hat G$, $\hat \mu$ a measure extending $\mu$ on $\hat\maV$ 
	and $\check \Phi$ be a policy that reduces $\Phi$ on $\check G$. 
	Then the drift of the respective Markov chains are such that for all $w\in\mbW$, 
	\begin{equation}
	\label{eq:driftL}
	{\Delta}^{\Phi}_{\mu}L(w) \le \hat{\Delta}^{\hat\Phi}_{\hat\mu}L(w) \le \check{\Delta}^{\check\Phi}_{\mu}L(w).
	\end{equation}
\end{proposition}
\begin{proof}
	Fix $w\in\mbW$. 
	The only case in which the proof of the left inequality of (\ref{eq:driftL}) differs from that of Proposition \ref{prop:extquad} is when 
	an item of class $i\in \mathcal O_{w}$ enters the $(\hat G,\hat\Phi,\hat\mu)$ system in a state $w$ and we multiply each transition by it's convenient $\alpha_i$. 
	Then, we now get that for all $n$, 
	\begin{equation*}
	\left(L(\hat W_{n+1})-L(\hat W_{n})\right)\ind_{\{\hat V_{n+1} \in \mathcal O_{w}\}}
	=\sum_{i\in \mathcal O_{w}}\alpha_i\ind_{\{\hat V_{n+1} =i\}},
	\end{equation*}	
	\noindent which, taking expectations and reasoning as in (\ref{eq:compMLfinal}), leads to 
	\begin{equation*} 
	{\Delta}^{\Phi}_{\mu}L(w) 
	= \hat{\Delta}^{\hat\Phi}_{\hat\mu}L(w)  
	- \sum_{i\in\mathcal{O}_w}\alpha_i\hat\mu\left(i\right).
	\end{equation*} 
	\medskip
	We now turn to the proof of the right inequality of (\ref{eq:driftL}). Denote 
	\begin{equation*}
	\mathcal P_{w} =\{i\in \maV_1\,:\,|w|_i > 0\}. 
	\end{equation*}
	Fix also $n\in\N$, and denote by $\check V_{n}$, the class of the incoming item at time $n$ 
	in the $(\check G,\check\Phi,\mu)$ system. First, similarly to (\ref{eq:compML1}), (\ref{eq:compML2}) and (\ref{eq:compML3}) we clearly get that 
	\begin{multline}\label{eq:compsep1}
	\esp{\left(L(\check W_{n+1})-L(\check W_{n})\right)\ind_{\{\check V_{n+1} \in \maV_2 \cup \left(\mathcal V_1\cap (\mathcal P_{w})^c \right)\}}
		\,|\,\check W_n= w}\\
	=\esp{\left(L(\hat W_{n+1})-L(\hat W_{n})\right)\ind_{\left\{\hat V_{n+1} \in \maV_2\cup \left(\mathcal V_1\cap (\mathcal P_{w})^c \cup \left(\underline{\mathcal V_1\cap (\mathcal P_{w})^c }\right)\right)\right\}}|\hat W_n=w}.
	\end{multline}
	Now, we also clearly have that 
	\begin{align*}
	\esp{\left(L(\check W_{n+1})-L(\check W_{n})\right)\ind_{\{\check V_{n+1} \in \mathcal P_{w}\}}
		\,|\,\check W_n= w} &=\sum_{i\in\mathcal{P}_w}\alpha_i\mu\left(i\right);\\
	\esp{\left(L(\hat W_{n+1})-L(\hat W_{n})\right)\ind_{\{\hat V_{n+1} \in \mathcal P_{w}\}}
		\,|\,\hat W_n= w} &=  \sum_{i\in\mathcal{P}_w}\alpha_i\hat\mu\left(i\right);\\
	\esp{\left(L(\hat W_{n+1})-L(\hat W_{n})\right)\ind_{\{\hat V_{n+1} \in \underline{\mathcal P_{w}}\}}
		\,|\,\hat W_n= w} &= -\sum_{i\in\underline{\mathcal P_{w}}}\alpha_i\hat\mu\left(i\right),
	\end{align*}
	which, together with (\ref{eq:compsep1}), implies that 
	\begin{equation*} 
	\hat{\Delta}^{\hat\Phi}_{\hat\mu}L(w) 
	= \check{\Delta}^{\check\Phi}_{\mu}L(w)
	-2\sum_{i\in\underline{\mathcal P_{w}}}\alpha_i\hat\mu\left(i\right).
	\end{equation*}
\end{proof}


\subsection{Proofs of the remaining main results}
\label{subsec:proofs}

We are now in position to prove Theorem \ref{thm:ML}, Theorem \ref{thm:ppartite} and Proposition \ref{prop:extppartite}. 

\begin{proof}[Proof of Theorem \ref{thm:ML}]
	Let $\mu\in \textsc{Ncond}(G)$. From Lemma \ref{lemma:equivalenceNcond}, the measure $\hat\mu$ belongs to {\sc Ncond}$(\hat G)$. 
	Let $\Phi$ be a matching policy of the Max-Weight class on $G$, with $\beta >0$. Clearly, its extension $\hat\Phi$ is also 
	of the Max-Weight class on $\hat G$. Then, we know from Theorem 5.3 in \cite{JMRS20} that the model $(\hat G,\hat\Phi,\hat\mu)$ is stable. 
	In particular, 
	we see in the proof of Theorem 5.3 in \cite{JMRS20} that 
	the Lyapunov-Foster Theorem \ref{the:LyapunovFosterTheorem} can be applied to the chain $\left(\hat W_n\right)_{n\in\mathbb{N}}$ for the quadratic function $Q$. 
	Specifically, there exist $\eta >0$ and a finite set $\hat{\mathcal { K}} \subset \hat{\mbW}$ such that 
	$\hat\Delta^{\hat\Phi}_{\hat\mu} Q(\hat{w}) <-\eta$ for all $\hat{w} \not\in \hat{\mathcal K}$. 
	Thus, in view of Proposition \ref{prop:extquad}, we have that $\Delta^{\Phi}_{\mu} Q({w}) <-\eta$ 
	for any $w$ that lies outside the finite subset $\mathcal K=\hat{\mathcal K}\cap \mbW$. 
	We conclude by applying the Lyapunov-Foster Theorem to the mapping $Q$ and the compact set $\mathcal K$. 
\end{proof}

\begin{proof}[Proof of Theorem \ref{thm:ppartite}]
	(i)  Fix $\mu\in \textsc{Ncond}(G)$. 
	First, if $G$ is a graph ($\maV_1 =\emptyset$) and $G=\check G$ is complete $p$-partite for $p\ge 3$, then 
	the result follows from Theorem 2, Assertion (16) in \cite{MaiMoy16}: specifically, we have that for some $\eta>0$, 
	for any $w\in\mbW\setminus\{\varepsilon\},$ 
	\begin{equation}
	\label{eq:majoreG}
	{\Delta}^{\Phi}_{\mu}L(w)
	< -\eta,
	\end{equation}
	and the Lyapunov-Foster criterion applies. 
	Now, if $G$ is not a graph, i.e. $\maV_1\ne\emptyset$, then let 
	\[\delta=\min\left\{\mu(\maE( I))-\mu( I)\,:\, I\in\I(G)\right\},\]
	which is strictly positive since $\mu\in \textsc{Ncond}(G)$, and the mapping 
	\[L_\delta : \left\{\begin{array}{ll}
	\mbW &\longrightarrow \R^+\\
	w&\longmapsto \displaystyle\sum_{i\in\maV_1} {\delta \over 2\mu(\maV_1)} |w|_i + \sum_{i\in\maV_2} |w|_i.  
	\end{array}\right.
	\]
	Then, for any $w\in\mbW$ the set $ I^{w}=\{i\in\maV\,:\,|w|_i>0\}$ is an independent set of $\check G$, so by the very definition of a complete $p$-partite graph, there exists a unique maximal independent set 
	$\check{ I}$ of $\check G$ such that $ I^{w} \subset \check{ I}$. 
	Then, for all $w\in\mbW$ such that $ I^{w}\cap\maV_2\ne \emptyset$, for any $n$, if $\{W_n=w\}$ the Markov chain can make two types of moves upon the arrival of $V_{n+1}$: 
	\begin{itemize}
		\item either one coordinate of $W_n$ decreases from 1 if $V_{n+1}$ is of a class in $\check{ I}^c=\check{\maE}\left(\check{ I}\right)$, or of a class in $ I^{w}\cap \maV_1$;
		\item or one coordinate of $W_n$ increases from 1, if $V_{n+1}$ is of a class in $\check{ I}\cap\left(( I^{w})^c\cup \maV_2\right)$. 
	\end{itemize} 
	Therefore, for any $\maV_2$-favorable matching policy $\Phi$ we have that 
	\begin{multline} \label{eq:vinci1}
	{\Delta}^{\Phi}_{\mu}L_\delta(w)\\
	=-{\delta \over 2\mu(\maV_1)} \mu\left(\maV_1 \cap  I^{w}\right)\ind_{\{\maV_1 \cap  I^{w}\ne \emptyset\}}+ {\delta \over 2\mu(\maV_1)}\mu\left(\maV_1 \cap \check{ I}\cap( I^{w})^c\right)
	+\mu\left(\check{ I}\cap \maV_2\right) - \mu\left(\check{ I}^c\right).
	\end{multline} 
	Observe that $\check{ I}\cap \maV_2$ is an independent set of $G$, and that $\check{ I}^c=\maE\left(\check{ I}\cap \maV_2\right)$. Hence (\ref{eq:vinci1}) implies that 
	\begin{equation*} 
	{\Delta}^{\Phi}_{\mu}L_\delta(w)
	\le  {\delta \over 2\mu(\maV_1)}\mu\left(\maV_1 \cap \check{ I}\cap( I^{w})^c\right)
	+\mu\left(\check{ I}\cap \maV_2\right) - \mu\left(\maE\left(\check{ I}\cap \maV_2\right)\right)
	\le  {\delta \over 2} - \delta =-{\delta \over 2}.
	\end{equation*} 
	As this is true for any $w$ outside the finite set $\{w\in\mbW\,:\,  I^{w}\cap\maV_2 =\emptyset\}$, we conclude again using the Lyapunov-Foster Theorem that $\textsc{Stab}(G,\Phi)=\textsc{Ncond}(G).$

	(ii) Fix $\mu\in\textsc{Ncond}(\check G),$ and an admissible matching policy $\Phi$. Applying (\ref{eq:majoreG}) to $\check G$, we obtain that for any $\check\Phi$ that reduces $\Phi$, 
	for some $\eta>0$ we have 
	${\check\Delta}^{\check\Phi}_{\mu}L(\check{w})<-\eta$ 
	for all $\check{w}\in\check{\mbW}\setminus\{\varepsilon\}$. 
	Combining this with (\ref{eq:driftL}), and recalling that $\mbW\subset \check{\mbW}$ we obtain that ${\Delta}^{\Phi}_{\mu}L({w}) <-\eta$ for all $w\in\mbW\setminus\{\varepsilon\}$, 
	which concludes the proof. 
\end{proof}
\begin{proof}[Proof of Proposition \ref{prop:extppartite}]
	(i) Remark that for any $\hat w\in\hat\mbW$, the set $\{i\in\maV\,:\,|\hat w|_i>0\}$ is again an independent set of $\check G$. So we can apply, for any  
	$\hat{\mu}\in\textsc{Ncond}(\hat{G})$, the exact same argument as for assertion (i) in Theorem \ref{thm:ppartite}, by replacing $\maV_1$ 
	by $\maV_1\cup\cop{\maV_1}$. 
	
	(ii) 
	Let $\hat\mu$ be an element of $\mathscr M(\hat{\maV})$ whose reduced measure $\mu$ belongs to $\textsc{Ncond}(\check G)$. Let $\hat\Phi$ be an admissible policy on $\hat{\maV}$, 
	$\Phi$ be a policy on $\maV$ such that $\hat\Phi$ extends $\Phi$, and $\check\Phi$ be a policy reducing $\Phi$ on $\check G$. 
	
	First, as in (\ref{eq:majoreG}) there exists $\eta>0$ such that $\check\Delta^{\check\Phi}_{\mu} L(w)<-\eta$ for any $w\in \mbW\setminus\{\varepsilon\}$. 
	
	Fix $\hat w$ in $\hat{\mbW}\setminus\{\varepsilon\}$. 
	Then define the permutation $\gamma$ of $\hat{\maV}$ by
	\[\begin{cases}
	\gamma(i) &=\underline i \mbox{ and } \gamma(\underline i) =i \mbox{ if } |\hat w|_{\underline i}>0\mbox{ and } |\hat w|_{i}=0,\,i\in \maV_1,\\
	\gamma(j) &=j\,\mbox{ else, for any }j\in \hat{\maV}.
	\end{cases}\]
	Let us also denote by $\gamma(\hat w)$, the word obtained from $w$ by replacing the letters of $\hat w$ by their image through $\gamma$, in other words for all $i\in \llbracket 1,|\hat w|\rrbracket$, 
	$\gamma(\hat w)_i=\gamma(\hat w_i)$. Observe that $\gamma(\hat w)$ is clearly an element of $\mbW\setminus\{\varepsilon\}$, so in view of the above observation we have that 
	\begin{equation}
	\label{eq:finala}
	\check\Delta^{\check\Phi}_{\mu} L(\gamma(\hat w))<-\eta.
	\end{equation}
	Now, as $i$ and $\underline i$ have the same connectivity in $\hat G$ for any $i\in\maV_1$, for all $n$ the conditional distribution of $\hat W_{n+1}$ given $\{\hat W_n = \hat w\}$ in the $(\hat G,\hat\Phi,\hat\mu)$ 
	system equals that of  $\hat W_{n+1}$ given $\{\hat W_n = \gamma(\hat w)\}$ in the $(\hat G,\hat\Phi,\hat\mu\circ\gamma)$ system. In particular, we have that 
	\begin{equation}
	\label{eq:finalb}
\hat{\Delta}^{\hat\Phi}_{\hat\mu}(\hat w) = \hat{\Delta}^{\hat\Phi}_{\hat\mu\circ\gamma}(\gamma(\hat w)).
	\end{equation}
	On the other hand, as $\gamma(\hat w)$ is an element of $\mbW$ and 
	the measure $\hat\mu\circ\gamma \in \mathscr M(\hat{\maV})$ clearly extends the measure $\mu$, the right inequality of (\ref{eq:driftL}) implies that
	\begin{equation*}
	\hat{\Delta}^{\hat\Phi}_{\hat\mu\circ\gamma}L(\gamma(\hat w)) \le \check{\Delta}^{\check\Phi}_{\mu}L(\gamma(\hat w)),
	\end{equation*}
	and it follows from (\ref{eq:finala}-\ref{eq:finalb}) that $\hat{\Delta}^{\hat\Phi}_{\hat\mu}(\hat w)<-\eta$. As this is true for any $\hat w$ in $\hat{\mbW}\setminus\{\varepsilon\}$, the proof is complete. 
\end{proof}


\section{Discussion of results and conclusion} \label{discussion of results2}
In this chapter, we have studied a generalization of stochastic matching models on graphs by allowing the self-loops matching. Different applications appear to this class of models such as dating sites or collaborative sites, that is, individuals of the same class can be married. \\

For a given multigraph $G$, we build its maximal subgraph $\check{G}$ that is obtained by deleting all self-loops in $G,$ and its minimal blow-up graph $\hat{G}$  that is obtained by duplicating each node having a self-loop by two nodes having the same neighborhood and replacing each self-loop by an edge between the node and its copy. Taken into consideration the graphs $\check G$ and $\hat G$ which are built above, we can transmit and generalize different results to $G.$

The multigraph $G$ under the matching policies {\sc fcfm} and Max-Weight such that $\beta >0$ are maximal. Also, if $G$ is a complete $p$-partite multigraph, $(p\geq 2)$, then for $p\geq 3$ or $\maV_1\neq 0,$ any $\maV_2$-favorable matching policy is maximal. \\

  

\pagestyle{empty}
\chapter{Fluid limits techniques for stability}
\label{chap6:FluidLimits}
\pagestyle{fancy}

\section*{Introduction}\label{chap6: intro}


In the previous chapters, we studied the stability of stochastic matching models using Lyapunov techniques. In this chapter, we present a new approach that allows us to retrieve, and complete these results in continuous-time settings. This technique consists of speeding up time and scaling the process appropriately to obtain a deterministic and continuous approximation of the original process, which allows, among other features,  to study the ergodicity properties of the process at hand. In the limit, one gets a sort of caricature of the initial stochastic process which is defined as its \textbf{fluid limit}.\\

This chapter is organized as follows: In Section \ref{sec:NCStab} we start by providing necessary conditions of stability for graphical matching models in continuous time. In Section \ref{sec:FluideStability} we derive fluid approximations of matching models on multigraphs. In Section \ref{sec:CasesStudy} we present different case studies. Last, in Section \ref{sec:ExamplesFluidHypergraph} we elaborate the fluid limit technique to study matching models on complete $3$-uniform hypergraphs and complete $3$-uniform $k$-partite hypergraphs. We conclude and discuss this chapter in Section \ref{sec:discussion of results3}.


\section{Necessary conditions of stability}
\label{sec:NCStab}
The necessary condition of stability for discrete-time stochastic matching models on graphs was recalled in Section \ref{sec:matching model}. For all $\mathbb{G}=(\maV,\maE)$ 
	it reads as follows,
	\begin{equation}
	\label{eq:NcondDuGraph}
	\textsc{Ncond}(\mathbb{G}):=\big\lbrace \mu \textrm{ with support }\maV\;:\mu(I)<\mu(\maE(I))\;\textrm{ for all } I\in\mathbb{I}(\mathbb{G})\big\rbrace.
	\end{equation}
	The first question that arises is to find an analog of (\ref{eq:NcondDuGraph}) for continuous-time models, that is, in the context of matching queues, as defined in section \ref{sec:MatchingQueuContiounsTime} (recall the notation therein). 
	The following condition was introduced in \cite{MoyPer17}: for any matching graph $\mathbb{G},$ 
	\begin{equation}
	\label{eq:NcondCDuGraph}
	\textsc{Ncond}_C(\mathbb G):=\left\lbrace \lambda\in (\R^{++})^{|\maV|}\;:\;\bar{\lambda}_{I }<\bar{\lambda}_{\maE(I )}\textrm{ for all }I \in\mathbb{I}(G)\right\rbrace.
	\end{equation}

Fix a graph $\mbG=(\maV,\maE)$, a class-admissible matching policy $\Phi,$ and an arrival vector $\lambda := (\lambda_1, \cdots , \lambda_{|\maV|})$. Let 
for all $i\in\maV$, 
\[\mu_\lambda(i)={\lambda_i \over \bar \lambda}={\lambda_i \over\sum\limits_{i\in\maV}\lambda_i}\cdot\]
Then, it is easily seen that $\mu_\lambda$ defines a probability measure on 
$\maV$. Further, if we denote for all $n\geq 1$, by $A(n):=\{A_n(i),\;i\in\maV\}$ 
the vector tracking the number of items in the buffers of all nodes up to time 
$n$ in the discrete-time system associated to $\mathbb{G},$ $\Phi$ and $\mu_\lambda$. 
On another hand, let $N$ be the superposition of the arrival Poisson processes in the continuous-time matching system associated with $\mathbb{G}$, $\Phi$ and $\lambda$. Then, it is immediate that we have the identity in distribution $Q_t=A(N(t))$.  As there are finitely many Poisson processes in the continuous-time model, the sojourn time of the corresponding CTMC is of a rate that is bounded away from zero. This implies that $Q$ is positive recurrent if and only if $A$ is so, see \cite[Theorem 6.18]{Kulkarni17}. Consequently, relating (\ref{eq:NcondDuGraph}) to (\ref{eq:NcondCDuGraph}) we get that the stability region of the continuous-time model is included in $\textsc{Ncond}_C(\mathbb G)$.

We now use fluid-limit technique to derive precisely the stability region of continuous-time models. A throughout presentation of the following section can be found e.g. in \cite{MoyPer17}.

\section{Fluid Stability}
\label{sec:FluideStability}

Fix a multigraph $G=(\maV_1\cup\maV_2,\maE)$, a matching policy 
	$\Phi$ of the priority type and an arrival vector $\lambda$. 
	Observe that for all $t\geq 0$, for all $i\in\maV,\;Q_i(t)$ increases by 1 for each  class-$i$ arrivals such that $Q_j(t)=0$ for any $j\in\maE(i)$. On another hand, $Q_i(t)$ decreases by 1 (when it is positive) for any class-$j$ arrival  with $j\in\maE(i)$, such that $Q_k(t)$ are empty for any $k\in\maE(j)$ such that $j$ gives a higher priority to $k$ over $i.$

The state space in CTMC is then as follows:
\begin{multline}
\label{eq:StateSpaceConTime}
\mathbb{E}=\biggl\lbrace z\in(\Z^{+})^{|\maV|}\;:\;z_iz_j=0, \textrm{ for any }i\in\maV,\\
\textrm{ and } j\in\maE(i)\textrm{ with }|z_k|\leq 1\textrm{ for any } k\in\maV_1\biggl\rbrace.
\end{multline}


For each $i\in\maV,$ we introduce the following sets,
\begin{equation}
\label{eq:NOP}
\begin{array}{rrll}
&\mathcal{N}_i&:=&\{z\in\mathbb{E}\;:\; z_i>0\};\\	
&\mathcal{O}_i&:=&\{z\in\mathbb{E}\;:\;z_j=0\;\;\textrm{for all }j\in\mathcal{E}(i)\};\\
&\Phi_j(i)&:=&\left\lbrace k\in\mathcal{E}(j);\;j\textrm{ gives priority to }k\textrm{ over }i\textrm{ according to }\Phi \right\rbrace;\\
&\mathcal{P}_j(i)&:=&\{z\in\mathbb{E}\;:\;z_k=0\;\;\textrm{for all }k\in\Phi_j(i)\},\;j\in\mathcal{E}(i).
\end{array}
\end{equation}
\medskip
\textbf{\textit{Marginal process corresponding to a particular node}.} 
Fix a matching node $i_0$ of $G$. Let
\[\mathcal{R}:=\mathcal R^{i_0}=\maV\backslash\big({i_0}\cup\maE(i_0)\big)=\left\lbrace i_1,\cdots,i_{|\mathcal{R}|}\right\rbrace.\]
For $x,y\in\mathbb{E},$ denote by $\maA(x,y)$ the infinitesimal generator of the queue process $Q.$
For any $z\in \mathbb{E}$ such that $z_{i_0}>0,$ the only positive terms $\mathcal{A}(z,y),\;y\in \mathbb{E},$ are given by\\
\begin{equation}
\label{eq:InfiInitial}
\left\lbrace\begin{array}{llccc}
\mathcal{A}(z,z+\gre_{i_0})=\lambda_{i_0};\\
\\
\mathcal{A}(z,z-\gre_{i_0})=\sum_{j\in\maE(i_0)}\left(\lambda_{j}{\ind_\mathcal{P}}_j(i_0)(z)\right);\\
\\
\mathcal{A}(z,z+\gre_{i_{\ell}})=\lambda_{i_{\ell}}{\ind_\mathcal{O}}_{i_{\ell}}(z),\;\ell\in\llbracket 1,|\mathcal R|\rrbracket;\\
\\
\mathcal{A}(z,z-\gre_{i_{\ell}})={\ind_\mathcal{N}}_{i_{\ell}}(z)\sum\limits_{\substack{j\in\maE(i_0)\\
		i_0\notin\Phi_j(i_{\ell})}}\left(\lambda_{j}{\ind_\mathcal{P}}_j(i_0)(z)\right),\;\ell\in\llbracket 1,|\mathcal R|\rrbracket.
\end{array}\right.
\end{equation}




\medskip

Let $R:=R^{i_0} = \{R(t) \;:\; t \geq 0\}$ denote the restriction of the process $Q$ to the nodes of $\mathcal{R}$, i.e.,
\begin{equation}
\label{eq:Rrestriction}
R=(R_1,R_2,\cdots,R_{|\mathcal R|}):=\left(Q_{i_1},\cdots,Q_{i_{|\mathcal R|}}\right).
\end{equation}

Define the set
\begin{multline}
\label{eq:StateSpaceConTimeMarg}
\mathbb{E}^\mathcal R=\biggl\lbrace z\in(\Z^{+})^{|\mathcal R|}\;:\;z_{i_k}z_{i_l}=0, \textrm{ for any }i_k\in\mathcal R,\\
\textrm{ and } i_l\in\maE(i_k)\textrm{ with }|z_{i_m}|\leq 1\textrm{ for any } i_m\in\maV_1\biggl\rbrace.
\end{multline}

Conditionally on the ${i_0}$-th coordinate of $Q$ being positive, 
	the process $R$ clearly coincides in distribution with a Markov process 
	$\mathcal X$ on $\mathbb{E}^{\mathcal R}$. The latter will be termed {\em marginal Markov process} associated to node $i_0$. 
	The idea is as follows: using a stochastic averaging principle, as in 
	\cite{Kurtz92}, showing that the marginal process reaches its stationary state 
	immediately at fluid scale. The drifts of the fluid limit of $Q$ will then be a function of this stationary measure.

\textbf{
	\textit{The FWLLN}}. We consider the sequence of fluid-scaled processes $\{\bar{Q}_n\; :\; n \geq 1\}$, defined via
\begin{equation*}
\label{eq:QbarEqQOvern}
\bar{Q}^n(t)=\frac{Q^n(t)}{n}:=\frac{Q(nt)}{n},\;t\geq 0,\;\;n\geq 1.
\end{equation*}
We also denote by $\mathcal{X}^n$ the $n$-th marginal process corresponding to $i_0,$ defined by
\begin{equation*}
\label{eq:Xn=X(nt)}
\mathcal{X}^n(t)=\mathcal{X}(nt),\;t\geq 0,
\end{equation*}
and define 
\begin{equation*}
\label{eq:Xbarn=X(nt)Overn}
\bar{\mathcal{X}}^n(t)=\frac{\mathcal{X}(nt)}{n},\;t\geq 0,\;\;n\geq 1.
\end{equation*}
\medskip
For the fluid analysis, we make two assumptions below,  

\textit{ ASSUMPTION 1}. 
$Q^n(0)\in\mathbb{E},$ for any $n\geq 1,$ and $\bar{Q}^n(0)\Rightarrow\bar{Q}(0)$ as $n\longrightarrow\infty,$ where $\bar{Q}(0)$ is a deterministic element of $\R^{|\maV|},$ with $\bar{Q}_{i_0}(0)>0$ and $\bar{Q}_i(0)=0,\;i\in\maV\backslash\{i_0\}.$\\

\textit{ ASSUMPTION 2}. For all $n\geq 1,$ the $\mathbb{E}^{\mathcal{R}}-$valued process $\mathcal{X}^n$ is ergodic with stationary probalility $\pi^n.$

For $n\geq 1,$ let 
\begin{equation}
\label{eq:Rhon}
\rho^n:=\rho^n(Q^n(0)):=\inf\{t\geq 0\;:\;Q_{i_0}^n(t)=0\},\;\textrm{ with }\inf\emptyset:=\infty.
\end{equation}
\begin{sloppypar}
	In the case where $\mathbb G$ is a graph, the following result was given as Theorem 4 of \cite{MoyPer17}. Hereafter, for a sequence of random variables 
	$\{\rho^n\}$, we write $\rho^n\Rightarrow\rho$ whenever $\pr{\rho^n>M}\underset{n\to\infty}{\longrightarrow} 1$ for any $M\in\R$.
	\end{sloppypar}
\begin{theorem}
	\label{the:FWLLN}
	(FWLLN) Let $\mathbb G$ be a graph, and $(\mathbb{G},\Phi,\lambda)_C$ be a matching queue such that $\Phi$
	is class-admissible. If, for some node $i_0$
	\begin{equation}
	\label{eq:Lambdai0-SumPi}
	\lambda_{i_0}-\sum\limits_{j\in\maE(i_0)}\lambda_j\pi\left(\mathcal{P}_j^{\mathcal{R}}(i_0)\right)<0,
	\end{equation}
	for $\pi$ stationary probability 
	and $\mathcal{P}_j^{\mathcal{R}}(i_0),\;j\in\maE(i_0)$
	, then $\rho^n\Rightarrow\rho$ in $\R$ as $n\longrightarrow\infty,$ for $\rho^n$ in (\ref{eq:Rhon}), where 
	\begin{equation}
	\label{eq:Rho}
	\rho=\displaystyle\frac{\bar{Q}_{i_0}(0)}{\sum_{j\in\maE(i_0)}\lambda_j\pi\left(\mathcal{P}_j^{\mathcal{R}}(i_0)\right)-\lambda_{i_0}}.
	\end{equation}
	Otherwise, $\rho^n\Rightarrow\infty.$ In either case, $\bar{Q}^n\Rightarrow\bar{Q}$ in $\mathbb{D}^{|\maV|}[0,\rho)$ as $n\longrightarrow\infty,$ where
	\begin{equation}
	\label{eq:FormeGneralOfQ}
	\left\lbrace\begin{aligned}
	\bar{Q}_{i_0}(t)&=\bar{Q}_{i_0}(0)+\left(\lambda_{i_0}-\sum\limits_{j\in\maE(i_0)}\lambda_j\pi\left(\mathcal{P}_j^{\mathcal{R}}(i_0)\right)\right)t,\\
	\bar{Q}_i(t)&=0,\qquad i\in\maV\backslash\{i_0\}.
	\end{aligned}\right.
	\end{equation}
\end{theorem}

Notice that $\rho^n$ going to infinity in the sense specified above readily  
entails that the process $Q$ cannot be positive recurrent, as is shown 
in Lemma 1 of \cite{MoyPer17}, following Proposition 9.9 in \cite{R13}.
In other words,

\begin{corollary}
	\label{cor:unstabilityfluid}
	If $\rho^n \Rightarrow\infty$, for $\rho^n$ in (\ref{eq:Rhon}), then $(\mbG, \Phi,\lambda)_C$ is unstable.
\end{corollary} 

Now, if $G=(\maV_1\cup\maV_2,\maE)$ is a multigraph, 
for any $i\in\maV_1$, $Q_{i}(0)$ is always zero or one, so 
the $i$-th coordinate of the fluid limit $\bar{Q}$ is necessarily null at all times. 
Moreover, it is immediate to observe the following,

\begin{corollary}
	\label{cor:FwllnMulti} 
	For any class-admissible policy $\Phi$, the conclusions of Theorem \ref{the:FWLLN} and Corollary \ref{cor:unstabilityfluid} remain valid 
	for a continuous-time model $(\mathbb{G},\Phi,\lambda)_C$ on a multigraph 
	$\mathbb G$ if we assume that (\ref{eq:Lambdai0-SumPi}) holds 
	for some $i_0\in\maV_2$.
\end{corollary}
In the Section \ref{sec:CasesStudy} below, we will study examples of multigraphs for which the stationary probability of the marginal process can be explicitly computed, and so an explicit fluid limit and an explicit necessary condition of stability can be derived, respectively using Theorem \ref{the:FWLLN} and Corollary \ref{cor:FwllnMulti}. 
Moreover, using fluid stability arguments as in \cite{JD95} will also prove that the latter conditions are also sufficient, thereby providing the exact stability region of the models under consideration.

\section{Multigraphical cases study}
\label{sec:CasesStudy}
In this section we present different examples of multigraph $G=(\maV=\maV_1\cup\maV_2,\mathcal{E})$ and the corresponding matching queues $(G,\Phi,\lambda)_C,$ for the arrival-rate vector $\lambda:=(\lambda_1,\cdots,\lambda_{|\maV|})$ and a priority policy $\Phi$ that is \textit{depicted by the arrows} on each dedicated figure. We deduce the precise stability regions of the corresponding stochastic matching models using the above fluid limit results. 

\subsection{Pendant graphs with a self-loop}
In the subsection below, we present various multigraphs that consist of a pendant graph with a self-loop on a given vertex.

\subsubsection{Pendant graph with a self-loop on the vertex 2}
\medskip
Consider the multigraph $G$ depicted on Figure \ref{fig:PendantLoopNode2} such that $\maV_1=\{2\}$, $\maV_2=\{1,3,4\}$ and $\mathcal{E}=\{\{1,2\},\{2\},\{2,3\},\{2,4\},\{3,4\}\}$.

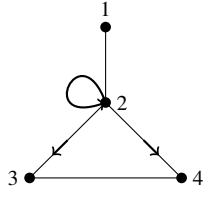
\begin{figure}[htb]
	\begin{center}
		\begin{tikzpicture}

		\fill (8,2) circle (2pt)node[above]{\scriptsize{1}};
		\fill (8,1) circle (2pt)node[right]{\scriptsize{2}};
		\fill (7,0) circle (2pt)node[left]{\scriptsize{3}};
		\fill (9,0) circle (2pt)node[right]{\scriptsize{4}};
		
		\draw[-] (8,1) -- (9,0);
		\draw[-] (8,1) -- (7,0);
		\draw[-] (8,1) -- (8,2);
		\draw[-] (7,0) -- (9,0);
		\draw[thick,->] (8.5,0.5) to (8.7,0.3);
		\draw[thick,->] (7.5,0.5) to (7.3,0.3); 
		\draw[thick,->] (8,1) to [out=110,in=200,distance=10mm] (8,1);
		\end{tikzpicture}
		\caption{Multigraph $G$ with a self-loop on the vertex 2.}
		\label{fig:PendantLoopNode2}
	\end{center}
\end{figure}

\begin{proposition}
	\label{prop:PendantLoopNode2}
	Let $G$ be the pendant graph with a self-loop on the vertex 2 and $\Phi$ the matching policy depicted on Figure \ref{fig:PendantLoopNode2}. Consider an arrival-rate vector $\lambda\in\textsc{Ncond}_C(G),$ i.e., 
	\begin{equation*}
	\lambda_1<\lambda_2,\qquad\lambda_1+\lambda_3<\lambda_2+\lambda_4\;\;\textrm{ and }\;\;\lambda_1+\lambda_4<\lambda_2+\lambda_3.
	\end{equation*}
	If $\bar{Q}^n(0) \Rightarrow x\gre_1$ in $\R^4$ for some $x\in\R^{++}$, then $\bar{Q}^n\Rightarrow\bar{Q}$ in $\D^4[0,\rho)$ as $n\to +\infty$, where
	\begin{equation*}
	\bar{Q}(t)=(x+(\lambda_1-\lambda_2\alpha)t,0,0,0),\qquad 0\leq t <\rho,
	\end{equation*}
	for $\rho:=x/(\lambda_2\alpha-\lambda_1)$ if $\alpha>\lambda_1/\lambda_2$ and $\rho:=\infty$ otherwise, and for 
	\begin{equation}
	\label{eq:alphaPendantVertex2}
	\alpha:=\left[1+\frac{\lambda_3}{\lambda_2+\lambda_4-\lambda_3}+\frac{\lambda_4}{\lambda_2+\lambda_3-\lambda_4}\right]^{-1}=\frac{(\lambda_2)^2-(\lambda_3-\lambda_4)^2}{\lambda_2(\lambda_2+\lambda_3+\lambda_4)}.
	\end{equation}
	\begin{proof}
		The result follows from Corollary \ref{cor:FwllnMulti} by considering $i_0=1$. In that case the marginal process $\mathcal{X}$ has the following infinitesimal generator $A^{\mathcal{R}}$, with $\mathcal{R}=\{3,4\},$ 
		
		$$\left\lbrace\begin{array}{llc}
		\mathcal{A}^\mathcal{R}((i,0),(i+1,0))&=\lambda_3;\\
		\mathcal{A}^\mathcal{R}((i,0),(i-1,0))&=(\lambda_2+\lambda_4)\ind_{[1,+\infty)}(i);\\
		\mathcal{A}^\mathcal{R}((i,0),(i,0))&=-\left(\lambda_3+(\lambda_2+\lambda_4)\ind_{[1,+\infty)}(i)\right);\\
		\mathcal{A}^\mathcal{R}((0,j),(0,j+1))&=\lambda_4;\\
		\mathcal{A}^\mathcal{R}((0,j),(0,j-1))&=(\lambda_2+\lambda_3)\ind_{[1,+\infty)}(j);\\
		\mathcal{A}^\mathcal{R}((0,j),(0,j))&=-\left(\lambda_4+(\lambda_2+\lambda_3)\ind_{[1,+\infty)}(j)\right);\\
		\end{array}\right.$$
		
		then we get,
		\begin{equation*}\left\lbrace\begin{array}{lll}
		\displaystyle-\lambda_3\pi(0,0)+(\lambda_2+\lambda_4)\pi(1,0)=0	\\
		\;\;\;\displaystyle\lambda_3\pi(0,0)-(\lambda_2+\lambda_3+\lambda_4)\pi(1,0)+(\lambda_2+\lambda_4)\pi(2,0)=0	\\
		\;\;\;\displaystyle\lambda_3\pi(1,0)-(\lambda_2+\lambda_3+\lambda_4)\pi(2,0)+(\lambda_2+\lambda_4)\pi(3,0)=0	\\
		\;\;\;\qquad\ldots\quad\qquad\qquad\qquad\qquad\ldots\quad\qquad\qquad\qquad\ldots=0
		\end{array}\right.
		\end{equation*}

		\begin{equation*}\Longleftrightarrow\left\lbrace\begin{array}{lll}
		\displaystyle\pi(1,0)=\left(\frac{\lambda_3}{\lambda_2+\lambda_4}\right)\pi(0,0)	\\
		\\
		\displaystyle\pi(2,0)=\left(\frac{\lambda_3}{\lambda_2+\lambda_4}\right)\pi(1,0)	\\
		\\
		\displaystyle\pi(3,0)=\left(\frac{\lambda_3}{\lambda_2+\lambda_4}\right)\pi(2,0)	\\
		\qquad\quad\;\vdots
		\end{array}\right.
		\qquad\Longleftrightarrow\begin{array}{lll}
		\displaystyle\pi(i,0)=\left(\frac{\lambda_3}{\lambda_2+\lambda_4}\right)\pi(i-1,0)&;&i\geq 1.	
		\end{array}
		\end{equation*}
		Set $\pi(0,0)=\alpha,$ so, for any state $x\in\mathbb{E}^{\mathcal{R}}$ we have,\\
		
		$$\pi(x)=\left\lbrace\begin{array}{cccc}
		\alpha\left(\displaystyle\frac{\lambda_3}{\lambda_2+\lambda_4}\right)^i&x=(i,0),\;i\ge1\\
		&\\
		\alpha\left(\displaystyle\frac{\lambda_4}{\lambda_2+\lambda_3}\right)^j&x=(0,j),\;j\ge1.\\
		\end{array}\right.$$
		
		Then we have,
		\begin{align*}
		1&=\sum\limits_{i\geq 1}\pi(i,0)+\sum\limits_{j\geq 1}\pi(0,j)+\pi(0,0)\\
		&=\pi(0,0)\left[\sum\limits_{i\geq 1}\pi(i,0)+\sum\limits_{j\geq 1}\pi(0,j)+1\right]\\
		&=\pi(0,0)\left[\frac{\lambda_3}{\lambda_2+\lambda_4-\lambda_3}+\frac{\lambda_4}{\lambda_2+\lambda_3-\lambda_4}+1\right].
		\end{align*}

		The stationary distribution $\pi$ of this reversible CTMC is unique, so Assumption 2 holds. The stated convergence of $Q^n$ to the fluid limits follows from Corollary \ref{cor:FwllnMulti}. 
	\end{proof}
\end{proposition}

\begin{proposition}
	The matching queue $(G,\Phi,\lambda)_C$ corresponding to the pendant graph $G$ with a self-loop on the vertex 2 and the priority type $\Phi$ depicted on Figure \ref{fig:PendantLoopNode2}, is stable if and only if $\textsc{Ncond}_C(G)$ holds together with 
	\begin{equation*}
	\lambda_1<\alpha\lambda_2\;\;\;\textrm{ for } \alpha \textrm{ in } (\ref{eq:alphaPendantVertex2}).
	\end{equation*} 
\end{proposition}
\begin{proof}
	It follows from Proposition \ref{prop:PendantLoopNode2} that, for any initial condition of the form $(x,0,0,0),\;x>0$, the fluid limit $\bar{Q}$ will hit the origin if and only if $\lambda_1<\alpha\lambda_2$.
	
	
Assume that $\bar{Q}_1(0)>0$. Then at most one of $\bar{Q}_3(0)$ or $\bar{Q}_4(0)>0$. Say $\bar{Q}_3(0)>0.$ According to the priority of 3 over 1, then
	\begin{equation*}
	\bar{Q}_3(t)=\bar{Q}_3(0)+(\lambda_3-\lambda_2-\lambda_4)t,\;\;0\le t\le\frac{\bar{Q}_3(0)}{\lambda_2+\lambda_4-\lambda_3}.
	\end{equation*}
	In particular, the fluid process $\bar{Q}_3$ will hit $0$ in finite time, so that $\bar{Q}_3$ will hit
	the origin in finite time by Proposition \ref{prop:PendantLoopNode2}. A similar argument applies to the case where $\bar{Q}_4(0)>0.$

	Assume now that $\bar{Q}_2(0)>0,$  there exists $t_0$ such that $\forall t>t_0,\;\bar{Q}_2(t)=0$ (self-loop on vertex 2.)
	Now, since the prelimit processes $Q_i,\; i = 2, 3, 4$ have drifts towards 0 whenever any of them is strictly positive, the fluid limit must remain in state 0 after hitting this state, and Proposition \ref{prop:PendantLoopNode2} shows that $\bar{Q}_1$ will also remain fixed at 0 after hitting that state. 
	Thus, the ergodicity of the system follows from 
	\cite[Theorem 4.2]{JD95}.
\end{proof}

\begin{proposition}
	\label{prop:IncStricteVerX2}
	We have the strict inclusion
	\begin{equation*}
	\left\lbrace\lambda_1<\alpha\lambda_2\right\rbrace\cap\textsc{Ncond}_C(G)\subsetneq\textsc{Ncond}_C(G).
	\end{equation*}	
	\begin{proof}
		
		Fix $\varepsilon\in(0,2/5]$ and set \\
		$$\left\lbrace\begin{array}{cl}
		\lambda_1=&\varepsilon/2;\\
		\lambda_2=&\varepsilon;\\
		\lambda_3=&\lambda_4=1/2-3\varepsilon/4.\\
		\end{array}\right.$$\\
		Clearly, $\lambda$ belong to the set $\textsc{Ncond}_C(G),$ but not to $\{\lambda_1<\alpha\lambda_2\}$, since
		\begin{equation*}
		\lambda_1-\alpha\lambda_2=\varepsilon/2-\frac{\varepsilon^2}{\varepsilon+1-3\varepsilon/2}=\varepsilon/2-\frac{2\varepsilon^2}{2-\varepsilon}\geq 0.
		\end{equation*} 
	\end{proof}
\end{proposition} 
\subsubsection{Pendant graph with a self-loop on the  vertex $3$}
\medskip
Consider the multigraph $G$ depicted on Figure \ref{fig:PendantLoopNode3} such that $\maV_1=\{3\}$, $\maV_2=\{1,2,4\}$ and $\mathcal{E}=\{\{1,2\},\{2,3\},\{2,4\},\{3\},\{3,4\}\}$.
\begin{figure}[htp]
	\begin{center}
		\begin{tikzpicture}

		\fill (8,2) circle (2pt)node[above]{\scriptsize{1}};
		\fill (8,1) circle (2pt)node[right]{\scriptsize{2}};
		\fill (7,0) circle (2pt)node[below]{\scriptsize{3}};
		\fill (9,0) circle (2pt)node[below]{\scriptsize{4}};
		
		\draw[-] (8,1) -- (9,0);
		\draw[-] (8,1) -- (7,0);
		\draw[-] (8,1) -- (8,2);
		\draw[-] (7,0) -- (9,0);
		\draw[thick,->] (8.5,0.5) to (8.7,0.3);
		\draw[thick,->] (7.5,0.5) to (7.3,0.3);
		\draw[thick,->] (7,0) to [out=110,in=200,distance=10mm] (7,0);
		\end{tikzpicture}
		\caption{Multigraph $G$ with a self-loop on the vertex 3.}
		\label{fig:PendantLoopNode3}
	\end{center}
\end{figure}
\begin{proposition}
	\label{prop:PendantLoopNode3}
	Let $G$ be the pendant graph with a self-loop on the vertex 3 and $\Phi$ the matching policy depicted on Figure \ref{fig:PendantLoopNode3} above. Consider an arrival-rate vector $\lambda\in\textsc{Ncond}_C(G),$ i.e.,  
	\begin{equation*}
	\lambda_1<\lambda_2<\lambda_1+\lambda_3+\lambda_4,\quad\textrm{ and }\quad\lambda_1+\lambda_4<\lambda_2+\lambda_3.
	\end{equation*}
	If $\bar{Q}^n(0) \Rightarrow x\gre_1$ in $\R^4$ for some $x\in\R^{++}$, then $\bar{Q}^n\Rightarrow\bar{Q}$ in $\D^4[0,\rho)$ as $n\to +\infty$, where
	\begin{equation*}
	\bar{Q}(t)=(x+(\lambda_1-\lambda_2\alpha)t,0,0,0),\qquad 0\leq t <\rho,
	\end{equation*}
	for $\rho:=x/(\lambda_2\alpha-\lambda_1)$ if $\alpha>\lambda_1/\lambda_2$ and $\rho:=\infty$ otherwise, and for 
	\begin{equation}
	\label{eq:alphaPendantVertex3}
	\alpha:=\left[1+\frac{\lambda_3}{\lambda_2+\lambda_3+\lambda_4}+\frac{\lambda_4}{\lambda_2+\lambda_3-\lambda_4}\right]^{-1}=\displaystyle\frac{(\lambda_2+\lambda_3)^2-\lambda_4^2}{\lambda_2^2+2\lambda_3^2+3\lambda_2\lambda_3+\lambda_2\lambda_4}.
	\end{equation}
	\begin{proof}
		We prove with the same argue for Proposition \ref{prop:PendantLoopNode2}.
		Set $\pi(0,0)=\alpha,$ we get that \\
		$$\displaystyle\pi(x)=\left\lbrace\begin{array}{lll}
		\displaystyle\frac{\alpha\lambda_3}{\lambda_2+\lambda_3+\lambda_4}&x=(1,0);\\
		&\\
		\alpha\left(\displaystyle\frac{\lambda_4}{\lambda_2+\lambda_3}\right)^j&x=(0,j),\;j\ge1.\\
		\end{array}\right.$$
		Again, the stationary distribution $\pi$ of this reversible CTMC is unique, so Assumption 2 holds. The stated convergence of $Q^n$ to the fluid limits follows from  Corollary \ref{cor:FwllnMulti}. 
	\end{proof}
\end{proposition}
\begin{proposition}
	The matching queue $(G,\Phi,\lambda)_C$ corresponding to the pendant graph $G$ with a self-loop on the vertex 3 and the priority type $\Phi$ depicted on Figure \ref{fig:PendantLoopNode3}, is stable if and only if $\textsc{Ncond}_C(G)$ holds together with 
	\begin{equation*}
	\lambda_1<\alpha\lambda_2\;\;\;\textrm{ for } \alpha \textrm{ in } (\ref{eq:alphaPendantVertex3}).
	\end{equation*} 
	\begin{proof}
		Assume that $\bar{Q}_1(0)>0.$ It follows from the Proposition \ref{prop:PendantLoopNode3} that, for any initial condition of the form $(x,0,0,0),\;x>0$, the fluid limit $\bar{Q}$ will hit the origin if and only if $\lambda_1<\alpha\lambda_2$.
		
		Assume now that $\bar{Q}_2(0)>0,$  then we have 
		\begin{equation*}
		\bar{Q}_2(t)=\bar{Q}_2(0)+(\lambda_2-\lambda_1-\lambda_3-\lambda_4)t,\;\;0\le t\le\frac{\bar{Q}_2(0)}{\lambda_1+\lambda_3+\lambda_4-\lambda_2}.
		\end{equation*}
		According to $\lambda\in\textsc{Ncond}_C(G)$
		the fluid queue hits the origin in finite time.	
		
		Assume that $\bar{Q}_1(0)>0,$ with at most one of $\bar{Q}_3(0)$ or $\bar{Q}_4(0)>0$. Set $\bar{Q}_3(0)>0$ then it is equal to 1. So there exists $t_0>0$ such that  
		\begin{equation*}
		\bar{Q}_3(t)=0,\;\;\forall t>t_0.
		\end{equation*}
		Now assume that $\bar{Q}_1(0)>0,$ and $\bar{Q}_4(0)>0,$ then according to the priority of 4 over 1, then
		\begin{equation*}
		\bar{Q}_4(t)=\bar{Q}_4(0)+(\lambda_4-\lambda_2-\lambda_3)t,\;\;0\le t\le\frac{\bar{Q}_4(0)}{\lambda_2+\lambda_3-\lambda_4}.
		\end{equation*}
		According to $\lambda\in\textsc{Ncond}_C(G),$ in particular, the fluid process $\bar{Q}_4$ will hit $0$ in finite time, so that $\bar{Q}_4$ will hit
		the origin in finite time by Proposition \ref{prop:PendantLoopNode3}.\\ 
		
		Now, since the prelimit processes $Q_i,\; i = 2, 3, 4$ have drifts towards 0 whenever any of them is strictly positive, and Proposition \ref{prop:PendantLoopNode3} shows that $\bar{Q}_1$ will also remain fixed at 0 after hitting that state. Thus, the ergodicity of the system follows again from \cite[Theorem 4.2]{JD95}.
	\end{proof}
\end{proposition}
\begin{proposition}
	\label{prop:IncStricte}
	We have the strict inclusion
	\begin{equation*}
	\left\lbrace\lambda_1<\alpha\lambda_2\right\rbrace\cap\textsc{Ncond}_C(G)\subsetneq\textsc{Ncond}_C(G).
	\end{equation*}	
	\begin{proof}
		
		Fix $\varepsilon\in(0,1/3]$ and set \\
		$$\left\lbrace\begin{array}{cl}
		\lambda_1=&\varepsilon/2;\\
		\lambda_2=&\varepsilon.\\
		\lambda_3=&\lambda_4=1/2-3\varepsilon/4;\\
		\end{array}\right.$$\\
		Clearly, $\lambda$ belongs to $\textsc{Ncond}_C(G),$ but 
		not to $\{\lambda_1<\alpha\lambda_2\}$, since
		\begin{equation*}
		\lambda_1-\alpha\lambda_2=\varepsilon/2-\frac{\left(\varepsilon+1/2-3\varepsilon/4\right)^2-\left(1/2-3\varepsilon/4\right)^2}{\varepsilon^2+2\left(1/2-3\varepsilon/4\right)^2+4\varepsilon\left(1/2-3\varepsilon/4\right)}\left(\varepsilon\right)=\varepsilon/2-\frac{8\varepsilon^2-4\varepsilon^3}{-7\varepsilon^2+4\varepsilon+4}\geq 0.
		\end{equation*} 
	\end{proof}
\end{proposition} 

\subsection{Complete bipartite graphs with a self-loop}
In this subsection, we present a different type of multigraphs of the form of complete bipartite graph with a self-loop on a such vertex.
\subsubsection{Complete bipartite graph of order 3 with a self-loop on the vertex 2}
Consider the multigraph $G$ depicted on Figure \ref{fig:BipQ3Verx2} such that $\maV_1=\{2\}$, $\maV_2=\{1,3\}$ and $\mathcal{E}=\{\{1,2\},\{2,3\},\{2\}\}$.
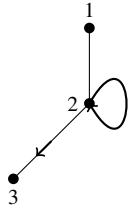
\begin{figure}[htp]
	\begin{center}
		\begin{tikzpicture}
		\fill (8,2) circle (2pt)node[above]{\scriptsize{1}};
		\fill (8,1) circle (2pt)node[left]{\scriptsize{2}};
		\fill (7,0) circle (2pt)node[below]{\scriptsize{3}};
		
		\draw[-] (8,1) -- (7,0);
		\draw[-] (8,1) -- (8,2);
		\draw[thick,->] (7.5,0.5) to (7.3,0.3);
		\draw[thick,->] (8,1) to [out=60,in=300,distance=13mm] (8,1);
		\end{tikzpicture}
		\caption{Complete bipartite graph  of order 2  with a self-loop on the vertex 2.}
		\label{fig:BipQ3Verx2}
	\end{center}
\end{figure}
\begin{proposition}
	\label{prop:BipQ3Verx2}
	Let $G$ be the bipartite graph with a self-loop on the vertex 2 and $\Phi$ the matching policy depicted on Figure \ref{fig:BipQ3Verx2} above. Consider an arrival-rate vector $$\lambda\in\textsc{Ncond}_C(G)=\{\lambda_1+\lambda_3<\lambda_2\}.$$ 
	If $\bar{Q}^n(0) \Rightarrow x\gre_1$ in $\R^3$ for some $x\in\R^{++}$, then $\bar{Q}^n\Rightarrow\bar{Q}$ in $\D^3[0,\rho)$ as $n\to +\infty$, where
	\begin{equation*}
	\bar{Q}(t)=(x+(\lambda_1-\alpha\lambda_2)t,0,0),\qquad 0\leq t <\rho,
	\end{equation*}
	for $\rho:=x/(\lambda_2\alpha-\lambda_1)$ if $\alpha>\lambda_1/\lambda_2$ and $\rho:=\infty$ otherwise, and for 
	
	\begin{equation}
	\label{eq:BipQ3Verx2}
	\alpha:=\left[1+\frac{\lambda_3}{\lambda_2-\lambda_3}\right]^{-1}=\displaystyle\frac{\lambda_2-\lambda_3}{\lambda_2}.
	\end{equation}
	\begin{proof}
		We argue as above. Setting $\pi(0)=\alpha,$ for any $i\geq 1,$ we get that 
		$$\pi(i)=\alpha\left(\displaystyle\frac{\lambda_3}{\lambda_2}\right)^i,$$
		and the result follows again from Corollary \ref{cor:FwllnMulti}.
	\end{proof}
\end{proposition}

\begin{proposition}
	The matching queue $(G,\Phi,\lambda)_C$ corresponding to the complete bipartite graph with a self-loop on the vertex 2 depicted on Figure \ref{fig:BipQ3Verx2} and for all matching policy $\Phi$, $(G, \textsc{fcfm},\mu)_C$ is stable if and only if $\mu$ belongs to the set $\textsc{Ncond}_C(G)$.
	\begin{proof}
		First, consider that the priority type policy that depicted on Figure \ref{fig:BipQ3Verx2}, we have:			
		

		Assume that that $\bar{Q}_1(0)>0,$ and $\bar{Q}_3(0)>0,$ then according to the priority of 3 over 1, then
		\begin{equation*}
		\bar{Q}_3(t)=\bar{Q}_3(0)+(\lambda_3-\lambda_2)t,\;\;0\le t\le\frac{\bar{Q}_3(0)}{\lambda_2-\lambda_3}.
		\end{equation*}
		According to $\lambda\in\textsc{Ncond}_C(G),$ in particular, the fluid process $\bar{Q}_3$ will hit $0$ in finite time, so that $\bar{Q}_3$ will hit
		the origin in finite time by Proposition \ref{prop:BipQ3Verx2}.

		Now assume that $\bar{Q}_1(0)>0,$ it follows from the proposition \ref{prop:BipQ3Verx2} that, for any initial condition of the form $(x,0,0),\;x>0$, the fluid limit $\bar{Q}$ will hit the origin if and only if $\lambda_1<\alpha\lambda_2.$
		Indeed, we have 	
		\begin{align*}
		\lambda_1&<\alpha\lambda_2\\
		\iff\lambda_1&<\left(\frac{\lambda_2-\lambda_3}{\lambda_2}\right)\lambda_2\\
		\iff\lambda_1&<\lambda_2-\lambda_3.
		\end{align*}
		For this, it suffices that $\lambda\in\textsc{Ncond}_C(G).$ 
		However, by symmetry between the vertices 1 and 3, we conclude that for all class-admissible matching policy $\Phi,$ we have 
		$\textsc{Stab}_C(G,\Phi)=\textsc{Ncond}_C(G).$ See Theorem \ref{thm:ppartite} (i).
	\end{proof}
\end{proposition}

\subsubsection{Complete bipartite graph of order $q=4$ with a self-loop on the vertex 1}

Consider the multigraph $G$ depicted on Figure \ref{fig:BipQ4Verx1} such that $\maV_1=\{1\}$, $\maV_2=\{1,2,3\}$ and $\mathcal{E}=\{\{1\},\{1,2\},\{1,4\},\{2,3\},\{3,4\}\}$.
\begin{figure}[htp]
	\begin{center}
		\begin{subfigure}[b!]{0.3\textwidth}
		\begin{tikzpicture}
		
		\fill (4,3) circle (2pt)node[above]{\scriptsize{2}};
		\fill (4,1) circle (2pt)node[above]{\scriptsize{4}};
		\fill (0,3) circle (2pt)node[above]{\scriptsize{1}};
		\fill (0,1) circle (2pt)node[above]{\scriptsize{3}};
		
		\draw[-] (0,1) -- (4,1);
		\draw[-] (0,1) -- (4,3);
		\draw[-] (0,3) -- (4,1);
		\draw[-] (0,3) -- (4,3);
		
		\draw[thick,->] (0.5,3) -- (1,3);
		\draw[thick,->] (0.5,1.25) -- (1,1.5);
		
		\draw[thick,->] (3.5,1.25) -- (3,1.5);
		\draw[thick,->] (3.5,3) -- (3,3);
		\draw[thick,->] (0,3) to [out=30,in=150,distance=17mm] (0,3);
			\end{tikzpicture}
			\caption{Non $\maV_2$-favorable policy.}
		\end{subfigure}
	$\qquad\qquad\qquad$
		\begin{subfigure}[b!]{0.3\textwidth}
		\begin{tikzpicture}
		
			\fill (15,3) circle (2pt)node[above]{\scriptsize{2}};
		\fill (15,1) circle (2pt)node[above]{\scriptsize{4}};
		\fill (11,3) circle (2pt)node[above]{\scriptsize{1}};
		\fill (11,1) circle (2pt)node[above]{\scriptsize{3}};
		
		\draw[-] (11,1) -- (15,1);
		\draw[-] (11,1) -- (15,3);
		\draw[-] (11,3) -- (15,1);
		\draw[-] (11,3) -- (15,3);
		
		\draw[thick,->] (11.5,3) -- (12,3);
		\draw[thick,->] (11.5,1.25) -- (12,1.5);
		
		\draw[thick,->] (14.5,1) -- (14,1);
		\draw[thick,->] (14.5,2.75) -- (14,2.5);
		\draw[thick,->] (11,3) to [out=30,in=150,distance=17mm] (11,3);
		\end{tikzpicture}
		\caption{$\maV_2$-favorable policy.}
	\end{subfigure}
		\caption{Complete bipartite graph  of order 4 with a self-loop on the vertex 1.}
		\label{fig:BipQ4Verx1}

	\end{center}
\end{figure}
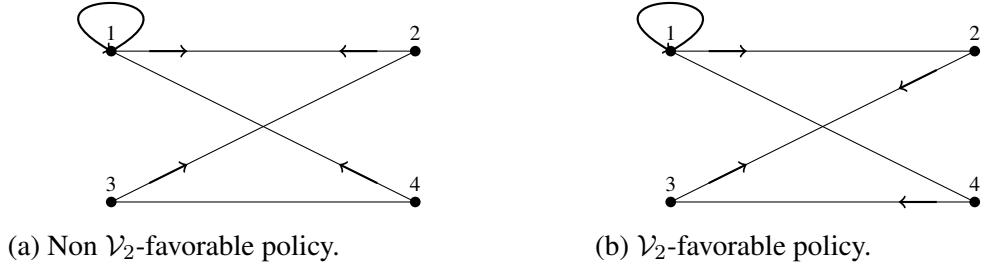

\begin{proposition}
	\label{prop:BipQ4Verx1}
	Let $G$ be the bipartite graph with a self-loop on the vertex 1 and $\Phi$ the matching policy depicted on Figure \ref{fig:BipQ4Verx1} (a). Consider an arrival-rate vector $$\lambda\in\textsc{Ncond}_C(G)=\{\lambda_2+\lambda_4<\lambda_1+\lambda_3,\;\textrm{ and }\;\lambda_3<\lambda_2+\lambda_4\}.$$
	If $\bar{Q}^n(0) \Rightarrow x\gre_3$ in $\R^4$ for some $x\in\R^{++}$, then $\bar{Q}^n\Rightarrow\bar{Q}$ in $\D^4[0,\rho)$ as $n\to +\infty$, where
	\begin{equation*}
	\bar{Q}(t)=(0,0,x+(\lambda_3-\alpha_1(\lambda_2+\lambda_4))t,0),\qquad 0\leq t <\rho,
	\end{equation*}
	for $\rho:=x/((\lambda_2+\lambda_4)\alpha_1-\lambda_3)$ if $\alpha_1>\lambda_3/(\lambda_2+\lambda_4)$ and $\rho:=\infty$ otherwise, and for 
	
	\begin{equation}
	\label{eq:BipQ4Verx1}
	\alpha_1=\left[1+\displaystyle\frac{\lambda_1}{\lambda_1+\lambda_2+\lambda_4}\right]^{-1}=\displaystyle\frac{\lambda_1+\lambda_2+\lambda_4}{2\lambda_1+\lambda_2+\lambda_4}.
	\end{equation}
	\begin{proof}
		We argue as above, by computing explicitly the unique stationary distribution of the reversible marginal process. 
	\end{proof}
\end{proposition}
\begin{proposition}
	The matching queue $(G,\Phi,\lambda)_C$ corresponding to the complete bipartite graph $G$ with a self-loop on the vertex 1 and the priority type $\Phi$ depicted on Figure \ref{fig:BipQ4Verx1} (a), is stable if and only if $\textsc{Ncond}_C(G)$ holds together with
	\begin{equation*}
	\lambda_3<(\lambda_2+\lambda_4)\alpha_1
	\end{equation*} 
	for $\alpha_1$ in (\ref{eq:BipQ4Verx1}).

	\begin{proof}
		
		
		
		
		Assume that $\bar{Q}_2(0)>0,$  then according to the priority of 2 over 4, then  \begin{equation*}
		\bar{Q}_2(t)=\bar{Q}_2(0)+(\lambda_2-(\lambda_1+\lambda_3))t\qquad 0\le t<\frac{\bar{Q}_2(0)}{\lambda_1+\lambda_3-\lambda_2}.
		\end{equation*}  
		
		Now assume that $\bar{Q}_4(0)>0,$ so we have 
		\begin{equation*}
		\bar{Q}_4(t)=\bar{Q}_4(0)+(\lambda_4-(\lambda_1+\lambda_3)\alpha_2)t\qquad 0\le t<\frac{\bar{Q}_4(0)}{(\lambda_1+\lambda_3)\alpha_2-\lambda_4},
		\end{equation*}  
		for a simple calculs we get, $\alpha_2=\displaystyle\frac{\lambda_1+\lambda_3-\lambda_2}{\lambda_1+\lambda_3}.$
		
		Observe that $\lambda_4-(\lambda_1+\lambda_3)\alpha_2=\lambda_2+\lambda_4-(\lambda_1+\lambda_3)$, which is negative according to  $\textsc{Ncond}_C(G).$
		
		Now assume that $\bar{Q}_3(0)>0,$ so we have
		\begin{equation*}
		\bar{Q}_3(t)=\bar{Q}_3(0)+(\lambda_3-(\lambda_2+\lambda_4)\alpha_1)t\qquad 0\le t<\frac{\bar{Q}_3(0)}{(\lambda_2+\lambda_4)\alpha_1-\lambda_3},
		\end{equation*}  
		it follows from the Proposition \ref{prop:BipQ4Verx1} that, for any initial condition of the form $(x,0,0,0),\;x>0$, the fluid limit $\bar{Q}$ will hit the origin if and only if $\lambda_3<\alpha_1(\lambda_2+\lambda_4).$
	\end{proof}
\end{proposition}

\begin{proposition}
\label{prop:IncStrict}
We have the strict inclusion
\begin{equation*}
\left\lbrace\lambda_3<(\lambda_2+\lambda_4)\alpha_1\right\rbrace\cap\textsc{Ncond}_C(G)\subsetneq\textsc{Ncond}_C(G).
\end{equation*}	
\begin{proof}
	Fix $\varepsilon\in(0,4/7]$ and set \\
	$$\left\lbrace\begin{array}{cl}
	\lambda_1=&1-5\varepsilon/4;\\
	\lambda_2=&\lambda_4=\varepsilon/2;\\
	\lambda_3=&3\varepsilon/4.\\
	\end{array}\right.$$\\
	Clearly, $\lambda$ belong to $\textsc{Ncond}_C(G),$ but not to the set $\{\lambda_3<(\lambda_2+\lambda_4)\alpha_1\}$, since
	\begin{equation*}
	\lambda_3-(\lambda_2+\lambda_4)\alpha_1=\frac{3\varepsilon}{4}-(\varepsilon)\frac{1-5\varepsilon/4+\varepsilon}{2-5\varepsilon/2+\varepsilon}=\frac{3\varepsilon}{4}-(\varepsilon)\frac{4-\varepsilon}{8-6\varepsilon}\geq 0.
	\end{equation*} 
\end{proof}
\end{proposition} 
\begin{remark}
\rm
Obseve that the multigraph $G=(\maV,\maE)$ mentioned above is formed by $\maV_1=\{1\}$ and $\maV_2=\{2,3,4\}.$ Then, if the matching policy $\Phi$ is $\maV_2$-favorable depicted on Figure \ref{fig:BipQ4Verx1} (b), we get that $$\bar{Q}_3(t)=\bar{Q}_3(0)+(\lambda_3-\lambda_2-\lambda_4)t\qquad 0\le t<\frac{\bar{Q}_3(0)}{\lambda_2+\lambda_4-\lambda_3},$$
which satisfies the Theorem \ref{thm:ppartite}. Thus for all $\Phi=\maV_2$-favorable $(G, \maV_2-\textrm{favorable},\mu)_C$ is stable if and only if $\mu$ belongs to the set $\textsc{Ncond}_C(G)$.  
\end{remark}

\section{Hypergraphical cases study}
\label{sec:ExamplesFluidHypergraph}
In this section, we develop an example of a hypergraph in which we apply the fluid limits technique presented in Section \ref{sec:FluideStability} with some adaptation for hypregraphs and we find the results proved in Theorem \ref{thm:stable3uniform} for complete $3$-uniform hypergraph of order 4.  



\subsection{Complete $3$-uniform hypergraph of order 4}
First of all, given a complete $3$-unifom hypergraph $\mbH=(\maV,\maH)$ of order $q\geq 4,$ then observe that  $\mathcal{N}^-_3(\mbH)=\left\lbrace\mu(i)<\frac{1}{3},\;\forall i\in\maV\right\rbrace=\left\lbrace\mu(i)<\frac{1}{2}\sum_{j\in\maV\backslash\{i\}}\mu(j),\;i\in\maV\right\rbrace.$
\begin{proposition}
\label{prop:Comp3uniHyperConti}
Let $\mbH=(\maV,\maH)$ be a complete $3$-uniform hypergraph of order 4 and  for all  matching policy $\Phi$. Consider an arrival-rate vector $\lambda\in\left(\mathscr{N}^{-}_3\right)_C(\mbH)$ i.e.\\
$$\left\lbrace\lambda_i<\frac{1}{2}\sum_{j\in\maV\backslash\{i\}}\lambda(j),\;i\in\maV\right\rbrace.$$

\begin{enumerate}
	
	\item If $\bar{Q}^n(0) \Rightarrow x\gre_1$ in $\R^4$ for some $x\in\R^{++}$, then $\bar{Q}^n\Rightarrow\bar{Q}$ in $\D^4[0,\rho)$ as $n\to +\infty$, where
	\begin{equation*}
	\bar{Q}(t)=(x+(\lambda_1-\alpha_2(\lambda_3+\lambda_4)-\alpha_3(\lambda_2+\lambda_4)-\alpha_4(\lambda_2+\lambda_3))t,0,0,0),\qquad 0\leq t <\rho,
	\end{equation*}
	for $\rho:=x/(\alpha_2(\lambda_3+\lambda_4)+\alpha_3(\lambda_2+\lambda_4)+\alpha_4(\lambda_2+\lambda_3)-\lambda_1),$	with
	\begin{equation}
	\left\lbrace\begin{array}{cccc}
	\alpha_2&=\displaystyle\frac{(\lambda_3-\lambda_4)^2+\lambda_2(\lambda_3+\lambda_4)}{\lambda_2^2-(\lambda_3-\lambda_4)^2};\\
	&\\
	\alpha_3&=\displaystyle\frac{(\lambda_2-\lambda_4)^2+\lambda_3(\lambda_2+\lambda_4)}{\lambda_3^2-(\lambda_2-\lambda_4)^2};\\
	&\\
	\alpha_4&=\displaystyle\frac{(\lambda_2-\lambda_3)^2+\lambda_4(\lambda_2+\lambda_3)}{\lambda_4^2-(\lambda_2-\lambda_3)^2}.
	\end{array}\right.
	\end{equation}
	
	\item If $\bar{Q}^n(0) \Rightarrow x\gre_{1}+y\gre_{2}$ in $\R^4$ for some $x,y\in\R^{++}$, then $\bar{Q}^n\Rightarrow\bar{Q}'$ in $\D^4[0,\rho)$ as $n\to +\infty$, where
	\begin{equation*}
	\bar{Q}'(t)=(x+(\lambda_{1}-\lambda_3-\lambda_4))t,y+(\lambda_{2}-\lambda_3-\lambda_4))t,0,0),\qquad 0\leq t <\rho,
	\end{equation*}
	for $\rho:=\min\left(x/(\lambda_3+\lambda_4-\lambda_1),y/(\lambda_3+\lambda_4-\lambda_2)\right).$
\end{enumerate}
\begin{proof}
	\begin{enumerate}
		\item The marginal process is $\mathcal{R}=\{2,3,4\}.$ Then, similarly to the above resolution, we deduce that
		$$\displaystyle\pi(x)=\left\lbrace\begin{array}{ccccc}
		\pi(0,0,0)\left(\displaystyle\frac{\lambda_2}{\lambda_3+\lambda_4}\right)^i,&x=(i,0,0)&i\geq 1;\\
		&\\
		\pi(0,0,0)\left(\displaystyle\frac{\lambda_3}{\lambda_2+\lambda_4}\right)^j,&x=(0,j,0)&j\geq 1;\\
		&\\
		\pi(0,0,0)\left(\displaystyle\frac{\lambda_4}{\lambda_2+\lambda_3}\right)^k,&x=(0,0,k)&k\geq 1.
		\end{array}\right.$$
		
		Set $\pi(0,0,0)=\pi(0),$ and for any $i,j,k\geq 1,$ we set the following, \[\pi(i,0,0)=\pi_2(i),\quad\pi(0,j,0)=\pi_3(j)\textrm{  and  }\pi(0,0,k)=\pi_3(k).\]
		\begin{align*}
		1&=\sum\limits_{i\geq 1}\pi_1(i)+	\sum\limits_{j\geq 1}\pi_2(j)+	\sum\limits_{k\geq 1}\pi_3(k)+\pi(0)\\
		&=\pi(0)\left[\sum\limits_{i\geq 1}\left(\frac{\lambda_2}{\lambda_3+\lambda_4}\right)^i+\sum\limits_{j\geq 1}\left(\frac{\lambda_3}{\lambda_2+\lambda_4}\right)^j+	\sum\limits_{k\geq 1}\left(\frac{\lambda_4}{\lambda_2+\lambda_3}\right)^k+1\right]\\
		&=\pi(0)\left[\frac{\lambda_2}{\lambda_3+\lambda_4-\lambda_2}+\frac{\lambda_3}{\lambda_2+\lambda_4-\lambda_3}+	\frac{\lambda_4}{\lambda_2+\lambda_3-\lambda_4}+1\right]\\
		&=\pi(0)\left[\displaystyle\frac{4\lambda_2\lambda_3\lambda_4}{(\lambda_3+\lambda_4-\lambda_2)(\lambda_2+\lambda_4-\lambda_3)(\lambda_2+\lambda_3-\lambda_4)}\right].
		\end{align*}
		Then we have,
		\begin{equation*}
		\pi(0)=\displaystyle\frac{(\lambda_3+\lambda_4-\lambda_2)(\lambda_2+\lambda_4-\lambda_3)(\lambda_2+\lambda_3-\lambda_4)}{4\lambda_2\lambda_3\lambda_4}.
		\end{equation*}
		Then the queue $\bar{Q}_1$ be as follow, 
		\begin{align*}
		\bar{Q}_1(t)&=\bar{Q}_1(0)+\left(\lambda_{1}-\lambda_2\left(\sum\limits_{j\geq 1}\pi_3(j)+\sum\limits_{k\geq 1}\pi_4(k)\right)\right.\\
		&\qquad\qquad\qquad\;\;\;\left.-\lambda_3\left(\sum\limits_{i\geq 1}\pi_2(i)+\sum\limits_{k\geq 1}\pi_4(k)\right)-\lambda_4\left(\sum\limits_{i\geq 1}\pi_2(i)+\sum\limits_{j\geq 1}\pi_3(j)\right)\right)t\\
		&=\bar{Q}_1(0)+\left(\lambda_1-\pi(0)\left[\lambda_2\left(\frac{\lambda_3}{\lambda_2+\lambda_4-\lambda_3}+\frac{\lambda_4}{\lambda_2+\lambda_3-\lambda_4}\right)\right.\right.\qquad\qquad\;\;\\
		&\qquad\qquad\qquad\qquad\qquad+\lambda_3\left(\frac{\lambda_2}{\lambda_3+\lambda_4-\lambda_2}+\frac{\lambda_4}{\lambda_2+\lambda_3-\lambda_4}\right)\qquad\\
		&\qquad\qquad\qquad\qquad\qquad+\left.\left.\lambda_4\left(\frac{\lambda_3}{\lambda_2+\lambda_4-\lambda_3}+\frac{\lambda_2}{\lambda_3+\lambda_4-\lambda_2}\right)\right]\right)t\\
		&=\bar{Q}_1(0)+\left[\lambda_1-\displaystyle\frac{(\lambda_3+\lambda_4-\lambda_2)(\lambda_2+\lambda_4-\lambda_3)(\lambda_2+\lambda_3-\lambda_4)}{4\lambda_2\lambda_3\lambda_4}\right.\qquad\quad\\
		&\qquad\qquad\qquad\qquad\qquad\left.\left(\frac{2\lambda_2\lambda_3\lambda_4(\lambda_2+\lambda_3+\lambda_4)}{(\lambda_3+\lambda_4-\lambda_2)(\lambda_2+\lambda_4-\lambda_3)(\lambda_2+\lambda_3-\lambda_4)}\right)\right]t\\
		&=\bar{Q}_1(0)+\left(\lambda_1-\displaystyle\frac{1}{2}\left(\lambda_2+\lambda_3+\lambda_4\right)\right)t,\quad \forall t>0.\qquad\qquad\qquad\qquad\quad\;
		\end{align*}
		Therefore, similarly we obtain that for any $i\in\maV,\;\exists t_0>0,\;\textrm{any }\bar{Q}_i(t)<0,\;\forall t>t_0.$
		
		\item Observe that, for any $i,j\in\maV$, we have $\maV\backslash\{i,j\}$ be a transversal of $\mathbb{H}$ and  
		$\sum_{k\in\maV\backslash\{i,j\}}\mu(k)>1/3,$ then $\sum_{k\in\maV\backslash\{i,j\}}\mu(k)>\mu(i)$ and $\sum_{k\in\maV\backslash\{i,j\}}\mu(k)>\mu(j).$ In other words, for any $i,j\in\maV,$ we have $\displaystyle\lambda_{\maV\backslash\{i,j\}}>\lambda_i$  and $\displaystyle\lambda_{\maV\backslash\{i,j\}}>\lambda_j.$
		
		Now, for all $t\ge 0$ we have, $\bar{Q}_1(t)=\bar{Q}_1(0)+\left(\lambda_1-\lambda_3-\lambda_4\right)t,\quad \forall t>0.$ 
		Then, as a result of the observation above, there exists $t_0>0$ such that 
		$\bar{Q}_1(t)<0$ for all $t \ge t_0$. Similarly, we deduce that there exists $t_1>0$ such that $\bar{Q}_2(t)<0$ for all $t \ge t_1$. We conclude that there exists $t\geq t_0\vee t_1$ the prelimit process $Q_i,\;i=1,2$ have drifts towards 0 whenever any them is striclty positive, then $\bar{Q}_1$ and $\bar{Q}_2$ will remain fixed at 0 after hitting that state. Thus the ergodicity of the sytem follows again from \cite[Theorem 4.2]{JD95}.

	\end{enumerate}
	
	
\end{proof}

\end{proposition}
Observe that, the complete 3-uniform $k$-partite hypergraph is just the complete 3-uniform hypergraph of order $k,$ where all the sets $I_1,...,I_k$ are of cardinality 1 (i.e., there are no replica). In particular case the complete 3-uniform $4$-partite hypergraph is the complete 3-uniform hypergraph of order $4.$

\section{Discussion of results and conclusion} \label{sec:discussion of results3}

In this chapter, we have generalized stochastic matching models to the continuous-time settings. Items of different classes arrive at the system according to an independent Poisson process of intensity $\lambda.$ We addressed a new technique that allows us to explicitly derive the stability region of the models at hand, by using fluid limits techniques, to prove the existence of a local steady-state, rather than applying Lyapunov techniques.  The scaling consists of speeding up time by the norm of the initial state. The behavior of such rescaled stochastic
processes is analyzed when the scaling parameter goes to infinity. In Corollary \ref{cor:FwllnMulti}, we generalized the \cite[Theorem 4]{MoyPer17} to multigraphs. The main difference of graphs, that is, for any $i\in\maV,\; Q_i(0)$ is always zero or one, so the $i$-th coordinate of the fluid limit $\bar{Q},$ that is, necessarily null at all times.\\

Moreover, we have studied several examples of multigraphs for which the stationary probability of the marginal process can be explicitly computed. We have studied some cases of pendant graph with a self-loop on a vertex and we distinguish that the stability region is strictly included in the necessary condition $\textsc{Ncond}_C.$ However, we retrieve the results of Theorem \ref{thm:ppartite} dedicated to the exact stability of the model by considering the complete bipartite graphs with a self-loop on a vertex with the $\maV_2$-favorable matching policy. 

In Section \ref{sec:ExamplesFluidHypergraph}, also by using fluid limit arguments, we provided the exact stability region to the model for complete 3-uniform hypergraphs of order 4.




%

\pagestyle{empty}
\chapter{Comparison of models for organ transplant applications}
\label{chap7:Apllication}
\pagestyle{fancy}

\section{Introduction}
\label{sec:IntroChap7}
A blood type (also called a blood group) is defined as the classification of blood-based on the presence or absence of inherited antigenic substances on the surface of red blood cells (RBCs). A series of related blood types constitute a blood group ABO system, see \cite{W66}. \\

We set the blood types as follows:
\begin{itemize}
	\item \textit{Blood Type A} - If the red blood cell contains only "A" molecules.
	\item \textit{Blood Type B} - If the red blood cell contains only "B" molecules.
	\item \textit{Blood Type AB} - If the red blood cell contains a mixture of molecules "A" and "B".
	\item \textit{Blood Type O} - If the red blood cell has neither "A" nor "B" molecules.
\end{itemize}
\textbf{Donating Blood by Compatible Type}



Blood groups are very important when a transfusion is required. During the transfusion, the patient must 
receive a blood that is compatible with his/her own. If the blood types do not match, the red blood cells will clump together and form clots that can block blood vessels and cause death.\\
If two different blood groups are mixed, the blood cells can clump together in the blood vessels, causing potentially fatal diseases. Therefore, it is important to match blood types before performing transfusions. In emergencies, blood group O may be given because all blood groups will accept it. However, there is always a risk. See \cite{HD19}.\\


The compatibilities of Blood Types Donors are described as follows:
\begin{center}
	\begin{tabular}{|c|c|c|}
		\hline 
		\textrm{Blood Type}	&\textrm{Donate Blood To}&\textrm{Receive Blood From}\\
		\hline
		A&A,\;AB&A,\;O\\
		\hline
		B&B,\;AB&B,\;O\\
		\hline
		AB&AB&A,\;B,\;AB,\;O\\
		\hline
		O&A,\;B,\;AB,\;O&O\\
		\hline
	\end{tabular}
\end{center}

\medskip

\textbf{Kidney Transplantation:}
Kidneys for transplantation come from two different sources: 
a \textit{living donor} or a \textit{deceased donor}. See \cite{HJS19}.

\textit{The Living Donor:}
In most cases, the donor is a family member. The donor must be in excellent health, well informed about transplantation, and able to give informed consent. Any healthy person can donate a kidney safely.

\textit{Deceased Donor:}
It is a person who has suffered brain death. The kidneys are removed and stored until a recipient has been selected.

Regardless of the type of kidney transplant-living donor or deceased donor-special blood tests are needed to find out what type of blood and tissue is present. These test results help to match a donor kidney to the recipient.

In every two above cases, special blood tests and tissue are needed to find which help to match a donor's kidney. To receive a kidney where the recipient's markers and the donor's markers all are the same is a "perfect match" kidney. Perfect match transplants have the best chance of working for many years. Most perfectly matched kidney transplants come from siblings.\\


\textit{Crossmatch:}
Throughout life, the body makes substances called \textit{antibodies} that act to destroy foreign materials. 
The crossmatch is done by mixing the recipient's blood with cells from the donor. If the crossmatch is positive, it means that there are antibodies against the donor. The recipient should not receive this particular kidney unless special treatment is done before transplantation to reduce the antibody levels. If the crossmatch is negative, it means the recipient does not have antibodies to the donor and that they are eligible to receive this kidney. See \cite{OYO18}.\\
Cross matches are performed several times during preparation for a living donor transplant, and a final crossmatch is performed within 48 hours before this type of transplant. The cross matches is a particular organization of kidney transplant with living donors authorized by the law of bioethics of July 7, 2011, and its decree of application published in September 2012. This donation is governed by three principles laid down by the law:
\begin{itemize}
	\item The information of the donor,
	\item anonymity between the two pairs,
	\item simultaneity of surgical interventions.
\end{itemize} 

This solution can be considered when the loved one who wishes to donate is not compatible with the patient.\\



\section{The models}
In all the aforementioned applications, elements (donors, receivers, or couples donor/receiver) arrive into the healthcare system at random times, and with various specificities.  {\em Compatibilities} between elements (donors with receiver, or couples donor/receiver with other couples) need to be taken into account when performing the matches between them.\\

Thereafter, we will address the particular case of {\em living donors}: we assume that the elements entering the system are {\em couples} of family members consisting of a giver and a receiver and that the giver and the receiver may not be compatible amongst them. This system is modeled as a stochastic matching model in which items (i.e.,  couples) are gathered into classes. Here, we say for instance that the class of a given item $i$ is $(A,B)$ if the giver of the couple $i$ has blood type $A$ and the receiver has blood type $B$. Items of the various classes enter the system following independent Poisson processes of designated intensities and using a simple homogenization argument, we focus on the embedded chain of the corresponding continuous-time system, namely, we work with a {\em discrete-time stochastic matching model}. Three types of matching can be taken into account: 
	\begin{enumerate}
	\item In the {\bf one by one} matching model, elements can be matched within couples, if possible. This corresponds equivalently, to an extended bipartite matching model (EBM) as defined in Chapter \ref{chap2: state of art}, in which items are single individuals rather than couples, and couples of items (giver/receiver) enter the system simultaneously. Then, the compatibility graph is bipartite, between givers and receivers. 
	\item The {\bf two by two} matching model corresponds to the cross-transplant system, namely, couples of items, says of classes $i$ and $j$, can be matched if and only if the 
	giver class of class $i$ is compatible with the receiver class of class $j$, and the giver class of class $j$ is compatible with the receiver class of class $i$. Then the two transplants are performed simultaneously. In our framework, this corresponds to a general stochastic matching model (GM) in which the compatibility graph is general, and represents the compatibilities between {\em couples}, in the sense specified above. A {\em matching} between two couples is made whenever the two matchings are performed together. 
	\item The {\bf three by three} matching model allows couples to be matched by groups of three. Then, a match between the three couples of respective classes $i$, $j$ and $k$ is possible if and only if the giver of $i$ is compatible with the receiver of $j$, the giver of $j$ is compatible with the receiver of $k$ and the giver of $k$ is compatible with the receiver of $i$. Then, the three transplants can be performed simultaneously. 
	These settings thus correspond to a general stochastic model on a hypergraph that is $3$-uniform, namely, matches are performed between groups of three items (i.e., couples) only. 
\end{enumerate}

\begin{figure}	[htp!]
	\centering
	\includegraphics[width=.55\linewidth]{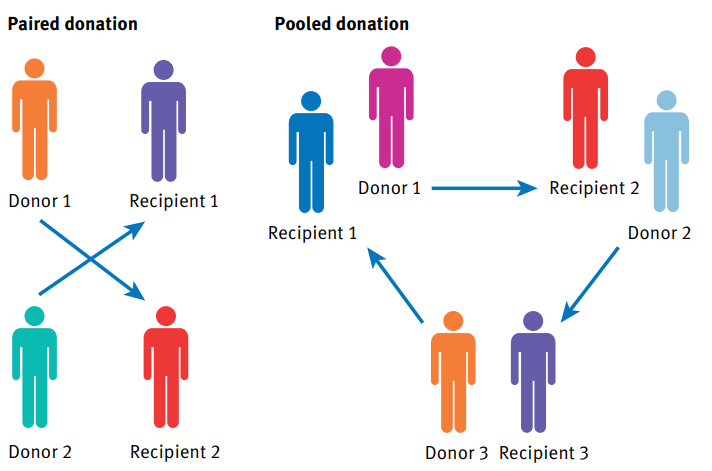}
	\caption{Paired and pooled organ donation - Cross matching Two-by-Two and Three-by-Three \cite{BEC16}.}
	\label{fig:2x2And3x3}
\end{figure}
To compare these three concurrent matching architectures, we consider a toy example in which only four couple classes are present: 
\begin{equation}
\label{eq:defclasses}
1:=(A,AB),\,2:=(O,AB),\,3:=(B,AB)\,\mbox{ and }\,4:=(A,B).
\end{equation}
This situation may occur if we address only a designated part of the whole transplant network, or if we consider an access control, for instance. 
We then set $\maV=\{1,2,3,4\}$, and consider the various matching structure $\mathbb{S}=(\maV,\mathcal S),$ on $\maV$ that correspond, respectively, to the various types of matchings introduced above.

\begin{remark} 
	\label{remark}\rm 
	As will appear clearly below, similar matching models are obtained if, instead of 
	(\ref{eq:defclasses}), we consider for instance arrivals of the following 
	couples,
	\begin{itemize}
		\item $1=(A,AB),$  $2=(O,AB),$  $3=(B,A)$ and $4=(A,B);$
		\item $1=(A,AB),$  $2=(O,A),$  $3=(B,A)$ and $4=(A,B);$
		\item $1=(B,AB),$  $2=(O,B),$  $3=(A,B)$ and $4=(B,A);$
		\item $1=(A,AB),$  $2=(B,AB),$  $3=(A,B)$ and $4=(O,A);$
		\item  And so on....
	\end{itemize}
\end{remark}

\section{Matching one by one}
In the above case, the matching one-by-one is preferable for items of classes $1$, $2$, and $3$, since the transplants between the giver and the receiver of each of these compatible couples can be made between family members. However, it is easily seen that, if incoming items of classes $1$, $2$ and $3$ are matched `with themselves' in a systematic way, then the elements of the couples of class $4$, which are not compatible, will never be matched since a $A$-giver cannot give to a $B$-receiver. Then the resulting system is unstable since 
class $4$-items will accumulate to infinity.

\section{Matching two by two}
\label{sec:Matching2X2}
Now, consider the case of the two by two matching. In this type of matching, 
the matching structure $\mathbb{S}$ is a graph such that 
$\mathbb{S}=\mathbb{G}=(\maV,\maE)$ and whose edges are defined as 
$$\maE=\left\lbrace\{1,2\},\;\{1,3\},\;\{2,3\},\;\{2,4\},\;\{3,4\}\right\rbrace,$$ 
as is easily seen. The graph is thus a complete $3$-partite graph of order 4, as depicted in  Figure \ref{fig:3-partite complete matching donate} below. In other words, $\mathbb{G}$ is analog to separable graph of order 3 and $I _1=\{2\},\;I _2=\{3\}$ and $I _3=\{1,4\}$ are the maximal independent sets partitionning $\maV$,

\begin{figure}[htb]
	\begin{center}
		\begin{tikzpicture}

		\fill (2,3) circle (2pt)node[above]{1};
		\fill (2,1.5) circle (2pt)node[right]{2};
		\fill (0.5,0) circle (2pt)node[left]{3};
		\fill (3.5,0) circle (2pt)node[right]{4};
		
		\draw[-] (0.5,0) -- (3.5,0);
		\draw[-] (0.5,0) -- (2,1.5);
		\draw[-] (3.5,0) -- (2,1.5);
		\draw[-] (2,1.5) -- (2,3);
		\draw[-] (2,3) -- (0.5,0);
		\end{tikzpicture}
		\caption{The compatible complete $3$-partite graph $\mathbb{G}$ of order 4.}
		\label{fig:3-partite complete matching donate}
	\end{center}
\end{figure}
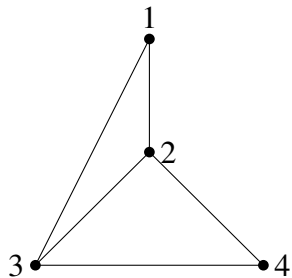


From \cite[Section 3]{MaiMoy16}, we get the following necessary condition of stability 
for any matching model on $\mathbb{G}$, 
\begin{equation}
\label{eq:NcondG}
\textsc{Ncond}(\mathbb{G})=\left\lbrace\begin{array}{lc}
\mu(2)&<1/2;\\
\mu(3)&<1/2;\\
\mu(1)+\mu(4)&<1/2.
\end{array}\right.
\end{equation}
\begin{proposition}
	\label{prop:Stab=NondG}
	Consider the graph $\mathbb{G}=(\maV, \maE).$ Then, for all matching policy $\Phi$ the sets $\textsc{Stab}(\mathbb{G},\Phi)$ and $\textsc{Ncond}(\mathbb{G})$ coincide, that is, the general stochastic matching model $(\mathbb{G},\Phi,\mu)$ is stable if and only if
	$\mu$ satisfies condition (\ref{eq:NcondG}).
	In other words we have 	
	\begin{equation}
	\label{eq:NcondGS}
	\textsc{Stab}(\mathbb{G},\Phi)=\left\lbrace\begin{array}{lc}
	\mu(2)&<1/2;\\
	\mu(3)&<1/2;\\
	\mu(1)+\mu(4)&<1/2.
	\end{array}\right.
	\end{equation}
\end{proposition}
\begin{proof}
	As $\mathbb{G}$ is a separable graph of order 3, then from \cite[Proposition 2]{MaiMoy16}, we get that 
	$$\forall \Phi,\;\;\mu\in\textsc{Stab}(\mathbb{G},\Phi)\iff\mu\in\textsc{Ncond}(\mathbb{G}).$$ 
\end{proof}	
\begin{proposition}
	\label{prop:StabStationary}
	Consider the graph $\mathbb{G}=(\maV, \maE),$ and any admissible matching policy $\Phi$. Then the stationary probability of the natural Markov chain $(X_n)_{n\in\N}$ reads as follows: for all $x\in\mathbb{X}$, 
	\begin{equation}
	\label{eq:Pi}
	\pi(x)=\left\lbrace\begin{array}{lllc}
	\alpha\left(\displaystyle\frac{\mu(2)}{1-\mu(2)}\right)^i&x=(0,i,0,0)&;\;i\geq 1\\
	\\
	\alpha\left(\displaystyle\frac{\mu(3)}{1-\mu(3)}\right)^j&x=(0,0,j,0)&;\;j\geq 1\\
	\\
	\alpha\left(\displaystyle\frac{\mu(1)+\mu(4)}{1-(\mu(1)+\mu(4))}\right)^{k+\ell}&x=(k,0,0,\ell)&;\;k\vee\ell\geq 1
	\end{array}\right.
	\end{equation}
	where the normalizing constant is given by 
	\begin{equation}
	\label{eq:alpha}
	\alpha=\pi(0,0,0,0)=\displaystyle\frac{(1-2\mu(2))(1-2\mu(3))(1-2(\mu(1)+\mu(4)))}{4\mu(2)\mu(3)(\mu(1)+\mu(4))}.
	\end{equation}
	
	\begin{proof}

		
		
		
		Whatever $\Phi$ is, the transition matrix $P$ of the class-content is defined as follows,

		$$\left\lbrace\begin{array}{llc}
		P((0,i,0,0),(0,i+1,0,0))&=\mu(2);\\
		P((0,i,0,0),(0,i-1,0,0))&=1-\mu(2);\\
		P((0,0,j,0),(0,0,j+1,0))&=\mu(3);\\
		P((0,0,j,0),(0,0,j-1,0))&=1-\mu(3);\\
		P((k,0,0,\ell),(k+1,0,0,\ell))&=\mu(1);\\
		P((k,0,0,\ell),(k,0,0,\ell+1))&=\mu(4);\\
		P((k,0,0,\ell),(k-1,0,0,\ell))&=\mu(2)+\mu(3)&\textrm{ choice } 1;\\
		P((k,0,0,\ell),(k,0,0,\ell-1))&=\mu(2)+\mu(3)&\textrm{ choice } 4.
		\end{array}\right.$$

		Then, we get 
		
		$$\pi(x)=\left\lbrace\begin{array}{lllc}
		\pi(0,0,0,0)\left(\displaystyle\frac{\mu(2)}{1-\mu(2)}\right)^i&x=(0,i,0,0)&;\;i\geq 1\\
		\\
		\pi(0,0,0,0)\left(\displaystyle\frac{\mu(3)}{1-\mu(3)}\right)^j&x=(0,0,j,0)&;\;j\geq 1\\
		\\
		\pi(0,0,0,0)\left(\displaystyle\frac{\mu(1)+\mu(4)}{1-(\mu(1)+\mu(4))}\right)^{k+\ell}&x=(k,0,0,\ell)&;\;k\vee\ell\geq 1.
		\end{array}\right.$$
		
		Set $\pi(0,0,0,0)=\pi(0),$ and for any $i,j,k\geq 1$ we set the following, \[\pi(0,i,0,0)=\pi_2(i),\quad\pi(0,0,j,0)=\pi_3(j)\textrm{  and  }\pi(k,0,0,\ell)=\pi_{1,4}(k+\ell).\]
		Then, we have 
		\begin{align*}
		1&=\sum\limits_{i\geq 1}\pi_2(i)+	\sum\limits_{j\geq 1}\pi_3(j)+	\sum\limits_{k\vee\ell\geq 1}\pi_{1,4}(k+\ell)+\pi(0)\\
		&=\pi(0)\left[\sum\limits_{i\geq 1}\left(\frac{\mu(2)}{1-\mu(2)}\right)^i+\sum\limits_{j\geq 1}\left(\frac{\mu(3)}{1-\mu(3)}\right)^j+\sum\limits_{k\vee\ell\geq 1}\left(\frac{\mu(1)+\mu(4)}{1-(\mu(1)+\mu(4))}\right)^{k+\ell}+1\right]\\
		&=\pi(0)\left[\frac{\mu(2)}{1-2\mu(2)}+\frac{\mu(3)}{1-2\mu(3)}+	\frac{\mu(1)+\mu(4)}{1-2(\mu(1)+\mu(4))}+1\right]\\
		&=\pi(0)\left[\displaystyle\frac{4\mu(2)\mu(3)(\mu(1)+\mu(4))}{(1-2\mu(2))(1-2\mu(3))(1-2(\mu(1)+\mu(4)))}\right], 
		\end{align*}
		and we conclude that 
		\begin{equation}
		\label{eq:ProStatiOfGraphOnzero}
		\pi(0,0,0,0)=\displaystyle\frac{(1-2\mu(2))(1-2\mu(3))(1-2(\mu(1)+\mu(4)))}{4\mu(2)\mu(3)(\mu(1)+\mu(4))}.
		\end{equation}
	\end{proof}
\end{proposition}

\begin{proposition}
\label{ex:ExForFCFMProbStatiGraph}
Consider the graph $\mathbb{G}.$ 
If $\Phi=${\sc fcfm}, then we retrieve the results of 
\cite[Theorem 1]{MBM17}.\end{proposition}
\begin{proof}
The set of admissible queue details $\mathbb{W}$ is given by, 
\begin{align*}
\mathbb{W}=&\{\varepsilon\}\cup\left\{2^k : k\geq 1\right\}\cup\left\{3^k : k\geq 1 \right\}\\
&\cup\left\{1^{r_1} 4^{\ell_1} 1^{r_2}4^{\ell_2}... : k,k'\geq 0, \; 0\leq r_i\leq k,\;0\leq \ell_j\leq k',\;i=\{1,2,...\},\;j=\{1,2,...\} \right\}.
\end{align*}
We compute explicitly $\Pi_W$, as the following values:
\[\left\{\begin{array}{llll}
\Pi_W(\varepsilon)&=\alpha  \\
\Pi_W(2^k)&=\alpha\left(\frac{\mu(2)}{1-\mu(2)}\right)^k \\
\Pi_W(3^k)&=\alpha\left(\frac{\mu(3)}{1-\mu(3)}\right)^k  \\
\Pi_W(1^{r_1} 4^{\ell_1} 1^{r_2}4^{\ell_2}...)&=\alpha\left(\frac{\mu(1)}{\mu(2)+\mu(3)}\right)^{r_1}\times\left(\frac{\mu(4)}{\mu(2)+\mu(3)}\right)^{\ell_1}\times\left(\frac{\mu(1)}{\mu(2)+\mu(3)}\right)^{r_2}\times\left(\frac{\mu(4)}{\mu(2)+\mu(3)}\right)^{\ell_2}...\\
&=\alpha\left(\frac{\mu(1)}{\mu(2)+\mu(3)}\right)^k\times\left(\frac{\mu(4)}{\mu(2)+\mu(3)}\right)^{k'}
\end{array}\right.\]
with
\begin{align*}
\alpha=\left[1 +\frac{\mu(1)}{\mu(2)+\mu(3)-\mu(1)}+\frac{\mu(2)}{1-2\mu(2)}+\frac{\mu(3)}{1-2\mu(3)}+\frac{\mu(4)}{\mu(2)+\mu(3)-\mu(4)}\right.
\\
+ \left(\frac{\mu(1)}{\mu(2)+\mu(3)-\mu(1)}\right)\left(\frac{\mu(4)}{\mu(2)+\mu(3)-\mu(1)-\mu(4)}\right)\qquad\\
+ \left.\left(\frac{\mu(4)}{\mu(2)+\mu(3)-\mu(4)}\right)\left(\frac{\mu(1)}{\mu(2)+\mu(3)-\mu(1)-\mu(4)}\right)\right]^{-1},
\end{align*}
where we use the fact that the set of independent sets reads 
$\mathbb{I}\left(\mathbb G\right)=\{\{1\},\{2\},\{3\},\{4\},\{1,4\}\}$.
\end{proof}


\section{Matching three by three}
\label{sec:Matching3X3}
We now consider the three by three matching procedure: for any triplet of items of respectives classes $(X_1,Y_1),\;(X_2,Y_2)$ and $(X_3,Y_3),$ then $X_1$ donates to $Y_2,$ $X_2$ donates to $Y_3$  and  $X_3$ donates to $Y_1.$ In this type of matching, the matching structure $\mathbb{S}$ is an hypergraph. Specifically, 
we set $\mathbb{S}=\mathbb{H}=(\maV,\maH)$, where the hyperedges are defined as $H_1=\{1,2,3\},\;  H_2=\{1,4,2\},\; H_3=\{4,3,1\},$ and $H_4=\{4,2,3\}$. In other words, the hypergraph is a complete $3$-uniform hypergraph as depicted in Figure \ref{fig:completeHyper}.

From Proposition \ref{prop:Ncond3} the necessary condition of stability for this matching model reads 
\begin{equation}
\label{eq:N3H}
\mathscr N^{-}_3(\mathbb H)=\left\lbrace\mu\in\mathscr{M}(\maV)\;:\;\mu(i)<1/3,\;\forall i\in\maV\right\rbrace.
\end{equation}
\begin{proposition}
\label{prop:Stab=NondH}
Consider the hypergraph $\mathbb{H}=(\maV, \maH)$ . Then, for all matching policy $\Phi$ the sets $\textsc{Stab}(\mathbb{H},\Phi)$ and $\mathscr N^-_3(\mathbb{H})$ coincide, in other words the general stochastic matching model $(\mathbb{H},\Phi,\mu)$ is stable if and only if
$\mu$ satisfies condition (\ref{eq:N3H}).
In other words we have 	
\begin{equation}
\label{eq:N3HS}
\textsc{Stab}(\mathbb{H},\Phi)=\left\lbrace
\mu\in\mathscr{M}(\maV)\;:\;\mu(i)<1/3,\;i\in\maV.\right\rbrace\end{equation}
\end{proposition}
\begin{proof}
As $\mathbb{H}$ is a complete 3-uniform hypergraph of order 4, the result follows from Theorem \ref{thm:stable3uniform}.
\end{proof}	

\section{Comparaison between $\textsc{Stab}(\mathbb G,\Phi)$ and $\textsc{Stab}(\mathbb H,\Phi)$}
\label{sec:ComparaisonConclusion}
To summarize, we obtain that the stability region of the two-by-two matching system 
is $\textsc{Stab}(\mathbb G,\Phi)$ given by (\ref{eq:NcondGS}), whereas the stability region of the three-by-three 
matching system is $\textsc{Stab}(\mathbb H,\Phi)$, given by (\ref{eq:N3HS}).

Let us observe that the two regions are not included in one another. Indeed, it is easily checked that 
the probability measure $\mu(1)=0.25,\;\mu(2)=0.35,\;\mu(3)=0.2$ and $\mu(4)=0.2$,
is an element $\textsc{Stab}(\mathbb G,\Phi)$ but not of 
$\textsc{Stab}(\mathbb H,\Phi).$ On the other hand, the probability measure 
$\mu(1)=0.3,$ $\mu(2)=0.2,$ $\mu(3)=0.2$ and	$\mu(4)=0.3$ clearly belongs to 
$\textsc{Stab}(\mathbb H,\Phi)$, but is not an element of  $\textsc{Stab}(\mathbb G,\Phi).$ \\

{\bf Conclusion 1:} Take the stability of the system as a primary performance criterion. 
Then, the one-by-one matching is never stable. Second, it is preferable to perform matchings two-by-two in some cases, and matchings three-by-three in other cases. \\

The intersection between the two stability regions is given by 
\begin{equation}
\label{eq:intersect}
\textsc{Stab}(\mathbb G,\Phi)\bigcap\textsc{Stab}(\mathbb H,\Phi):=\left\lbrace\begin{array}{lccc}
\mu(i)<1/3&i\in\maV;\\
\mu(1)+\mu(4)<1/2.\end{array}\right.
\end{equation}

The question arising now is the following: suppose that $\mu$ satisfies (\ref{eq:intersect}), that is, it belongs to both $\textsc{Stab}(\mathbb G,\Phi)$ and $\textsc{Stab}(\mathbb H,\Phi)$, implying that both systems are stable. Considering a secondary performance criterion, the frequency of construction points, that is, of visits to the zero state, {\bf what is the best matching procedure} between two by two and three by three matchings? 

We saw in Section \ref{sec:Matching2X2} that the probability of finding an empty system in steady-state can be given in closed form for the two-by-two matching system. However, such a result is, to date, not available for the three-by-three system. We then resort to simulations to compare the two procedures. 

We have simulated one thousand trajectories of the three-by-three system, each one starting from an empty system and consisting of one million arrivals. Then, 
\begin{itemize}
\item[\textbf{Step 1:}] We count the total number of construction points over the million arrivals, and we get the average over the thousand trajectories.
\item[\textbf{Step 2:}] We count the total number of empty buffers at the instance 999999 for each of one million arrivals and get the average over the thousand trajectories.
\end{itemize}

In Table \ref{tab:SimulationAppli} hereafter we present our results for different distributions $\mu$, applying for each distribution the two aforementioned steps. We denote the first average by `Trajectorial Average' and the average of the empty buffers at time 999999 by `Av. EB.' It is obvious that the three-by-three matching system is 3-periodic. Thus, starting from the empty state, it can be empty only at times that are multiples of three. So, it is only pertinent to compare the average of construction points for the whole trajectories (fourth column) to the third of the number of construction points seen at time 999999, a multiple of three (see the fifth column).

\begin{table}[htp]
	\centering
	\begin{tabular}{|c|c|c|c|c|l|l|}
		\hline
		\textbf{$\mu(1)$}&\textbf{$\mu(2)$}&\textbf{$\mu(3)$}&\textbf{$\mu(4)$}& \textbf{Trajectorial Average}&\textbf{Av. EB}&\textbf{$\pi(0,0,0,0)$}\\
		\hline
		\hline
		0.25&0.27&0.25&0.23&0.05137131
		&0.155
		
		&0.07098765
		\\
		\hline
		0.25&0.26&0.25&0.24& 0.05348423
		&0.160
		
		&0.03767661
		\\
		\hline
		0.25&0.264&0.25&0.236& 0.05279672
		&0.17
		
		&0.05150268
		\\
		\hline
		0.25&0.263&0.25&0.237&0.05294198
		&0.156
		
		&0.04811018
		\\
		\hline
		
		
		
		0.25&0.265&0.25&0.235&0.05260348
		&0.157
		
		&0.05485314
		\\
		\hline
		0.19&0.26&0.25&0.3&0.03445104
		&0.115
		
		&0.03767661
		\\
		\hline
		
		0.25&0.3&0.21&0.24&0.03740660
		&0.109
		
		&0.03757694
		\\
		\hline
		0.25&0.32&0.19&0.24&0.01779171
		&0.062
		
		&0.03745972
		\\
		\hline
		0.17&0.26&0.25&0.32&0.01639216
		&0.035
		
		&0.03767661
		\\
		\hline
		0.18&0.32&0.32&0.18&0.00542774
		&0.018
		
		&0.24609380
		\\
		\hline
		0.197&0.253&0.25&0.3&0.03558512
		&0.111
		
		&0.01178613
		\\
		\hline
	\end{tabular}
	\caption{The comparison of the stationary probability on $\mathbb{G}$ and the simulated results of the matching model on $\mathbb{H}.$}
	\label{tab:SimulationAppli}
\end{table}

%
	
	
	
%
	







{\bf Conclusion 2:} In Table \ref{tab:SimulationAppli} we emphasize, first, the speed of convergence to the steady-state (approximated by the final state - column 5), as the results of columns 4 and 5 tend to coincide. Second, when comparing the frequency of construction points for the three-by-three system to the stationary probability of an empty two-by-two system, we see that the first one seems to perform better in some cases, while the second performs better for some other values of $\mu$. 


\section{Discussion of results and conclusion}  
In this chapter, we have proposed simple modeling of kidney transplant systems using stochastic matching models. We have shown that models on various matching structures (bipartite models, general models on graphs, and models on hypergraphs) are suitable to various contexts. In the context of cross-transplants, a simple case study has shown that there is no clear hierarchy between cross-transplants by pairs of couples (matchings two-by-two) and cross-transplants by triplets of couples (matching three by three) when it comes to comparing the stability regions.
In some cases, one system is stable while the other is not. 
Second, we show that the same remark holds if we compare the frequency of construction points in simulated three-by-three systems to the (exact) value of the steady-state probability of the two-by-two system being empty: the first performs better for some values of $\mu$, while the second performs better otherwise. Moreover, we do not understand which is the best matching procedure.



\clearpage
\pagestyle{empty}

\addcontentsline{toc}{chapter}{Conclusion and Perspectives}
\clearpage
\chapter*{Conclusion and perspectives}
\pagestyle{fancy}
\fancyhf{}
\fancyfoot[CE,CO]{\thepage}
\fancyhead[RE,LO]{{\it CONCLUSIONS AND PERSPECTIVES}}

In this thesis, we have studied a generalization of stochastic matching models on the graph, by allowing the matching structure to be a hypergraph or multigraph. 
This class of models appears to have a wide range of applications in operations management, healthcare, and assemble-to-order systems. 
After formally introducing the model, we have proposed a simple Markovian representation, under IID assumptions. 
We have then addressed the general question of stochastic stability, viewed as the positive recurrence of the underlying Markov chain. For this class of systems, solving this elementary and central question turns out to be an intricate problem.

In this thesis, the stochastic matching model in discrete-time is formally defined as follows: items enter the system by a single. On other hand, in continuous-time items enter the system according to an independent Poisson process of intensity $\lambda>0.$ The arrivals  get matched by groups of 2 or more (hypergraphical cases), following compatibilities that are represented by a given hypergraph and by groups of 2 with possible compatibility with itself (multigraphical cases). A matching policy determines the matchings to be executed in the case of a multiple-choice, and the unmatched items are stored in a buffer, waiting for a future match.

\addcontentsline{toc}{section}{1. Conclusion}
\section*{1. Conclusion}
\label{sec:conclu}
Stochastic matching models on hypergraphs are in general, difficult to stabilize. Unlike the GM on graphs, the non-emptiness of the stability region on matching models on hypergraphs depends on a collection of conditions on the geometry of the compatibility hypergraph: rank, anti-rank, degree, size of the transversals, existence of cycles, and so on. 

Nevertheless, we showed that the `house' of stable systems is not empty, but shelters models on various uniform hypergraphs that are complete, or complete up to a partition of their nodes. 
 We have provided the \textit{exact} stability region of the system  $(\mathbb{H},\Phi,\mu)$ where $\mathbb{H}=(\maV,\maH)$ is a complete 3-uniform hypergraphs and for all admissible matching policy $\Phi,$ i.e. $\textsc{Stab}(\mathbb{H},\Phi)=\left\lbrace\mu\;;\;\mu(i)<1/3,\textrm{ for all }i\in\maV\right\rbrace.$ Also, we extended the exact stability region for complete 3-uniform $k$-partite hypergraph.  Moreover, we demonstrated a lower bound for the stability region for the incomplete 3-uniform hypergraphs for a matching policy {\sc ml}. For this, we resorted to ad-hoc multi-dimensional Lyapunov techniques in discrete-time (each step an item enters the system). \\

In this thesis, we have also studied a generalization of stochastic matching models on graphs by allowing the self-loops matching, that is, a stochastic matching model on multigraphs. Given a multigraph $G,$ its maximal subgraph $\check{G}$ obtained by deleting all self-loops in $G$ and the minimal blow-up graph $\hat{G}$ obtained from $G$ by duplicating each node having a self-loop by two nodes having the same neighborhood  and replacing each self-loop by an edge between the node and its copy.  From that, we transmit and generalize several results that are known for $\check{G}$ and $\hat{G}$ to their multigraphs. 


Also, in this context, we have provided the \textit{exact} stability region under {\sc fcfm} and {\sc MW} policy such that $\beta>0.$ Given a multigraph $G=(\maV_1\cup\maV_2,\maE),$ we introduced a new matching policy called $\maV_2$-favorable, that is, for any arrival items always prioritizes to match with an item in $\maV_2$ over an item in $\maV_1.$ 

In addition, we have proved that if $\check{G}$ is a complete $p$-partite graph, $(p\geq 2)$, then for $p\geq 3$ or $\maV_1\neq 0,$ any $\maV_2$-favorable matching policy is maximal. \\


In this thesis, we have also studied a new technique that allows us to find the stability of the model by using the fluid limits in continuous-time (items of the various classes enter the system following the independent Poisson process of designated intensities). This new method consists of speeding up the time and rescaling the process to get a sort of caricature of the initial stochastic process. We provide further results that are in agreement with the previous results. The advantage of the continuous-time setting is that powerful fluid-limit techniques can be employed, which greatly facilitate the stability analysis.

Indeed, we have also studied several cases of multigraphs such as pendant graphs with a self-loop on such vertex, we proved a lower bound of stability region. In other words, the model is stable if and only if the necessary condition for stability on continuous-time $\textsc{Ncond}_C$ holds together with certain conditions. Moreover, we retrieved the results for some specific complete bipartite graphs with a self-loop on such vertex for $\maV_2$-favorable. Further, using the fluid limits technique, we provided the exact stability region of the model on complete 3-uniform hypergraphs of order 4.\\

To illustrate the practical relevance of our results, we also have studied an application of the stochastic matching model that addressed particular cases of living donors in the context of cross-matching. In that case, the items enter the system by couples $(X,Y)$ of family members, the first component $X$ represents the `giver' and the second component $Y$ represents the `receiver'. The problem consists of studied a special case in three types of matching on various matching structures:
\begin{itemize}
	\item bipartite graphs (matching one-by-one);
	\item general graphs (matching two-by-two);
	\item complete 3-uniform hypergraphs (matching three-by-three).
\end{itemize}
We studied the performance criterion between the two-by-two matching system and three-by-three matching system of the frequency of visits to zero state, to distinguish what is the best matching procedure. Motivated by a simulation of one thousand trajectories of three-by-three, we see that the two-by-two seems to perform better in some cases, while three-by-three performs better for some other values of $\mu$. In this instance, we cannot conclude which is the best.


\addcontentsline{toc}{section}{2. Perspectives}
\section*{2. Perspectives}

All the results obtained in this thesis, are according to some hypotheses: stochastic matching model on hypergraphs and multigraphs, stability region, steady-state, and so on. 
There is still much to do regarding this class of systems. Let us give a few directions of research that we are currently following, or aim to follow in a near future:
\begin{itemize}
	\item Finding an explicit form of the stationary probability of the chain $\{W_n;\;n\in\N\}$ of the model $(\mathbb H,\textsc{fcfm},\Phi).$ 
	\item Comparing the necessary conditions of stability of the model $(\mathbb H,\textsc{fcfm},\Phi)$ between them.
	\item Applying the proposed matching model on hypergraphs and multigraphs for several domains of applications.
	\item Comparing the simple hypergraphs, of the optimal policy with those obtained by the “Greedy Primal-Dual” optimization algorithms (Nazari and Stolyar, 2017), a dynamic control strategy introduced to maximize the utility of queue networks waiting subject to stability, which turns out to be (asymptotically) optimal in this case.
	\item Comparing the matching policies for the models on the hypergraphs and the determination of an optimal matching policy: \textsc{mw}? In which way? 
	\item Determining the exact zones of stability of the random pairing models on particular cases of hypergraphs, in particular for small non-trivial hypergraphs such as the Fano plane (the projective plane of size 7) or the models representing the networks of cross kidney donation with loop.
	\item Generalizing the results of Theorem \ref{thm:stable3uniform}, to the complete $k$-uniform hypergraphs. In that case, we conjecture that the stability region is equal to $\mathscr{N}^{-}_3(\mathbb{H})=\{\mu,\;\mu(i)<1/k,\;i\in\maV\}.$
\end{itemize}

We believe that the present thesis represents a good starting point for a fruitful avenue for research on such systems.


\addcontentsline{toc}{chapter}{Appendix}
\clearpage
\vspace{-0.5cm}
\chapter*{Appendix}
\pagestyle{fancy}
\fancyhf{}
\fancyfoot[CE,CO]{\thepage}
\fancyhead[RE,LO]{{\it THE TRANSITIONS FOR THE PROOF OF THEROEM} \ref{thm:stable3uniform}}
\section*{The transitions for the proof of Theorem \ref{thm:stable3uniform}}
(i) First, for any $i\in\bar{J},$ and any integer $x_i\geq 2,$ the chain $\{Y_n;\;n\in\N\}$ makes the following transitions from state $Y_n=x_i.\gre_i,$	\begin{equation}
\label{eq:GtransY1}
Y_{n+1}=\left\{\begin{array}{ll}
Y_n+4\gre_i &\mbox{ w.p. }\mu(i)^4;\\
Y_n+3\gre_i+\gre_j &\mbox{ w.p. }4\mu(i)^3\mu(j);\\
Y_n+2\gre_i+2\gre_j &\mbox{ w.p. }6\mu(i)^2\mu(j)^2;\\
Y_n+\gre_i+3\gre_j &\mbox{ w.p. }4\mu(i)\mu(j)^3;\\
Y_n+4\gre_j &\mbox{ w.p. }\mu(j)^4;\\
Y_n+\gre_i &\mbox{ w.p. }12\mu(i)^2\mu(j)\mu(k);\\
Y_n+\gre_j &\mbox{ w.p. }12\mu(i)\mu(j)^2\mu(k);\\
Y_n-2\gre_i &\mbox{ w.p. }10\mu(j)^2\mu(k)\mu(\ell)\\
& \mbox{ (\scriptsize{the input has 2 $j$, 1 $k$ and 1 $\ell$, but does not end in $jj$});}\\
Y_n-\gre_i+2\gre_j &\mbox{ w.p. }2\mu(k)\mu(\ell)\mu(j)^2\\ 
&\mbox{ (\scriptsize{the input is of the form $ij jj$})};\\
Y_n-2\gre_i &\mbox{ w.p. }6\mu(j)^2\mu(k)^2;\\
Y_n-\gre_i+2\gre_j &\mbox{ w.p. }4\mu(j)^3\mu(k);\\
Y_n+\gre_j &\mbox{ w.p. }24\mu(i)\mu(k)\mu(\ell)\mu(j);\\
Y_n-2\gre_i &\mbox{ w.p. }24\mu(j)\mu(k)\mu(\ell)\mu(m).
\end{array}\right.
\end{equation}
(ii) For any $i\in J,$ and any integer $x_i\geq 2,$ the transitions of $\{Y_n;\;n\in\N\}$ from the state $x_i.\gre_i$ can be retrieved as a similar fashion to (\ref{eq:GtransY1}). 
 Set $H=\{i,j,k\}\subset J,$ we have the following transitions,
\begin{equation}
\label{eq:InCtransY1}
%
Y_{n+1}=\left\{\begin{array}{llll}
Y_n+4\gre_i &\mbox{ w.p. }\mu(i)^4;\\
Y_n+3\gre_i+\gre_{\ell}&\mbox{ w.p. }4\mu(i)^3\mu(\ell)&\ell\in\maV\backslash\{i\};\\
Y_n+2\gre_i+2\gre_{\ell} &\mbox{ w.p. }6\mu(i)^2\mu(\ell)^2&\ell\in\maV\backslash\{i\};\\
Y_n+\gre_i+3\gre_{\ell} &\mbox{ w.p. }4\mu(i)\mu(\ell)^3&\ell\in\maV\backslash\{i\};\\
Y_n+4\gre_{\ell} &\mbox{ w.p. }\mu(\ell)^4&\ell\in\maV\backslash\{i\};\\
%
Y_n+\gre_j &\mbox{ w.p. }12\mu(i)\mu(j)^2\mu(\ell)&\ell\in \overline{H};\\
Y_n+\gre_{\ell}&\mbox{ w.p. }12\mu(i)\mu(j)\mu(\ell)^2&\ell\in \overline{H};\\
Y_n+\gre_{\ell} &\mbox{ w.p. }12\mu(i)\mu(\ell)^2\mu(m)&\ell\ne m\in \overline{H};\\
Y_n+\gre_j&\mbox{ w.p. }6\mu(i)\mu(j)\mu(\ell)\mu(m)&\ell \ne m \in \overline{H}&\textrm{ends with }j;\\
Y_n+\gre_{\ell}&\mbox{ w.p. }18\mu(i)\mu(j)\mu(\ell)\mu(m)&\ell \ne m \in \overline{H}&\textrm{otherwise};\\
Y_n+\gre_{\ell}&\mbox{ w.p. }24\mu(i)\mu(\ell)\mu(m)\mu(p)&\ell\ne m \ne p\in\overline{H};\\
Y_n+\gre_j+3\gre_k&\mbox{ w.p. }4\mu(j)\mu(k)^3;\\
Y_n-\gre_i+2\gre_{\ell} &\mbox{ w.p. }4\mu(j)\mu(\ell)^3&\ell\in \overline{H};\\
Y_n-\gre_i+2\gre_k &\mbox{ w.p. }2\mu(j)\mu(k)^2\mu(\ell)&\ell \in \overline{H}&\textrm{ends with }kk;\\
Y_n-\gre_i+\gre_j+\gre_k &\mbox{ w.p. }10\mu(j)\mu(k)^2\mu(\ell)&\ell \in \overline{H}&\textrm{otherwise};\\

Y_n-2\gre_i &\mbox{ w.p. }12\mu(j)\mu(k)\mu(\ell)^2&\ell\in \overline{H};\\

Y_n-\gre_i+2\gre_{\ell} &\mbox{ w.p. }2\mu(j)\mu(\ell)^2\mu(m)&\ell\ne m\in\overline{H}&\textrm{ends with }\ell\ell;\\
Y_n-2\gre_i &\mbox{ w.p. }10\mu(j)\mu(\ell)^2\mu(m)&\ell\ne m\in\overline{H}&\textrm{otherwise};\\

Y_n-\gre_i+\gre_j+\gre_k&\mbox{ w.p. }4\mu(j)\mu(k)\mu(\ell)\mu(m)&\ell \ne m \in \overline{H}&\textrm{ends with }jk;\\
Y_n-2\gre_i&\mbox{ w.p. }20\mu(j)\mu(k)\mu(\ell)\mu(m)&\ell \ne m \in \overline{H}&\textrm{otherwise};\\

Y_n-2\gre_i&\mbox{ w.p. }24\mu(j)\mu(\ell)\mu(m)\mu(p)&\ell\ne m \ne p\in\overline{H};\\

Y_n+2\gre_i+\gre_j+\gre_k &\mbox{ w.p. }12\mu(i)^2\mu(j)\mu(k);\\
Y_n+\gre_i&\mbox{ w.p. }12\mu(i)^2\mu(j)\mu(\ell)&\ell \in \overline{H};\\

Y_n+\gre_i &\mbox{ w.p. }12\mu(i)^2\mu(\ell)\mu(m)&\ell\ne m\in\overline{H};\\
%
Y_n+2\gre_j+2\gre_k &\mbox{ w.p. }6\mu(j)^2\mu(k)^2;\\
Y_n-2\gre_i &\mbox{ w.p. }6\mu(j)^2\mu(\ell)^2&\ell\in\overline{H};\\
Y_n-\gre_i+2\gre_j &\mbox{ w.p. }2\mu(j)^2\mu(\ell)\mu(m)&\ell\ne m\in\overline{H}&\textrm{ends with }jj;\\
Y_n-2\gre_i &\mbox{ w.p. }10\mu(j)^2\mu(\ell)\mu(m)&\ell\ne m\in\overline{H}&\textrm{otherwise};\\

%
%
Y_n-\gre_i+2\gre_j &\mbox{ w.p. }4\mu(j)^3\mu(\ell)&\ell\in\overline{H};\\
%
Y_n+\gre_i+\gre_j+2\gre_k&\mbox{ w.p. }12\mu(i)\mu(j)\mu(k)^2;\\
Y_n+\gre_j &\mbox{ w.p. }24\mu(i)\mu(j)\mu(k)\mu(\ell)&\ell \in \overline{H};\\
%
%

Y_n-\gre_i+2\gre_{\ell}&\mbox{ w.p. }4\mu(\ell)^3\mu(m)&\ell\ne m\in\overline{H};\\
Y_n-2\gre_i &\mbox{ w.p. }6\mu(\ell)^2\mu(m)^2&\ell\ne m\in \overline{H};\\

Y_n-\gre_i+2\gre_{\ell} &\mbox{ w.p. }2\mu(\ell)^2\mu(m)\mu(p)
&\ell\ne m\ne p \in\overline{H}&\textrm{ends with }\ell\ell;\\
Y_n-2\gre_i &\mbox{ w.p. }10\mu(\ell)^2\mu(m)\mu(p)
&\ell\ne m\ne p \in\overline{H}&\textrm{otherwise};\\
Y_n-2\gre_i&\mbox{ w.p. }24\mu(\ell)\mu(m)\mu(p)\mu(s)&\ell\ne m\ne p \ne s\in\overline{H}.
\end{array}\right.
\end{equation}
Then we decduce that,\begin{multline}
\Delta'_{i}:=\esp{Q\left(Y_{n+1}\right) - Q\left(Y_n\right) | Y_n = x_i.\gre_i}=\\
\begin{aligned}
&\;(8a_i+16)\mu^4(i) + 4(6a_i+10)\mu^3(i)\sum_{\ell\in\maV} \mu(\ell) + 6(4a_i+8)\mu^2(i)\sum_{\ell\in\maV} \mu^2(\ell) 
+ 4(2a_i+10)\mu(i)\sum_{\ell\in\maV} \mu^3(\ell)\\
&+16\sum_{\ell\in\maV} \mu^4(\ell) 
+ 24 \mu(i) \sum_{\substack{j\in H:\\ \ell\in \overline{H}}} \mu^2(j)\mu(\ell)
+ 24 \mu(i) \sum_{\substack{j\in H:\\ \ell\in \overline{H}}} \mu(j)\mu^2(\ell)
+ 24 \mu(i) \sum_{\substack{\ell,m\in\overline{H}}}\mu^2(\ell)\mu(m)\\
&+ 12 \mu(i) \sum_{\substack{j\in H:\\ \ell,m\in \overline{H}:\\\textrm{ends with }j}} \mu(j)\mu(\ell)\mu(m)
+ 36 \mu(i) \sum_{\substack{j\in H:\\ \ell,m\in \overline{H}:\\\textrm{otherwise}}} \mu(j)\mu(\ell)\mu(m)
+ 48 \mu(i) \sum_{\substack{\ell,m,p\in \overline{H}}} \mu(\ell)\mu(m)\mu(p)\qquad\\
&+44\sum_{\substack{j,k\in H}}\mu(j)\mu^3(k)-4(2a_i+5) \sum_{\substack{j\in H:\\\ell\in \overline{H}}} \mu(j) \mu^3(\ell)-2(2a_i+5) \sum_{\substack{j,k\in H:\\\ell\in\overline{H}:\\\textrm{ends with }kk}}  \mu(j)\mu^2(k)\mu(\ell)\\
&-10(2a_i+3)\sum_{\substack{j,k\in H:\\\ell\in\overline{H}:\\\textrm{otherwise}}} \mu(j)\mu^2(k)  \mu(\ell)
-12(4a_i+4) \mu(j)\mu(k) \sum_{\substack{\ell\in\overline{H}}} \mu^2(\ell)\\
&-2(2a_i+5) \sum_{\substack{j\in H:\\\ell\in\overline{H}:\\\textrm{ends with }\ell\ell}} \mu(j)\mu^2(\ell)\mu(m)
-10(4a_i+4) \sum_{\substack{j\in H:\\\ell\in\overline{H}:\\\textrm{otherwise}}} \mu(j)\mu^2(\ell)\mu(m)\\
&-4(2a_i+3) \mu(j)\mu(k)\sum_{\substack{\ell,m\in\overline{H}:\\\textrm{ends with }jk}} \mu(\ell)\mu(m)
-20(4a_i+4) \mu(j)\mu(k)\sum_{\substack{\ell,m\in\overline{H}:\\\textrm{otherwise}}} \mu(\ell)\mu(m)\\
&-24(4a_i+4) \sum_{\substack{j\in H:\\\ell,m\in\overline{H}}} \mu(j)\mu(\ell)\mu(m)\mu(p)
+ 12(4a_i+6) \mu^2(i)\mu(j)\mu(k)\\
&+12(2a_i+1) \mu^2(i)\sum_{\substack{j\in H:\\\ell\in\overline{H}}} \mu(j)\mu(\ell)+12(2a_i+1) \mu^2(i)\sum_{\substack{\ell,m\in\overline{H}}} \mu(\ell)\mu(m)
+48 \mu^2(j)\mu^2(k)\\
&-6(4a_i+4)\sum_{\substack{j\in H:\\\ell\in\overline{H}}} \mu^2(j)\mu^2(\ell)
-2(2a_i+5)\sum_{\substack{j\in H:\\\ell\in\overline{H}:\\\textrm{ends with }jj}} \mu^2(j)\mu(\ell)\mu(m)\\
&-10(4a_i+4)\sum_{\substack{j\in H:\\\ell\in\overline{H}:\\\textrm{otherwise}}} \mu^2(j)\mu(\ell)\mu(m)
-4(2a_i+5)\sum_{\substack{j\in H:\\\ell\in\overline{H}}} \mu^3(j)\mu(\ell)
+12(2a_i+6)\mu(i)\sum_{\substack{j,k\in H}}\mu(j)\mu^2(k)\\
&+48\mu(i)\mu(j)\mu(k)\sum_{\substack{\ell\in\overline{H}}}\mu(\ell)
-4(2a_i+5)\sum_{\substack{\ell,m\in\overline{H}}} \mu^3(\ell)\mu(m)
-6(4a_i+4)\sum_{\substack{\ell,m\in\overline{H}}} \mu^2(\ell)\mu^2(m)\\
&-2(2a_i+5)\sum_{\substack{\ell,m,p\in\overline{H}:\\ \textrm{ends with }\ell\ell}} \mu^2(\ell)\mu(m)\mu(p)
-10(4a_i+4)\sum_{\substack{\ell,m,p\in\overline{H}:\\\textrm{otherwise}}} \mu^2(\ell)\mu(m)\mu(p)\\
&-24(4a_i+4)\sum_{\substack{\ell,m,p,s\in\overline{H}}} \mu(\ell)\mu(m)\mu(p)\mu(s)
=\nu_{i}(\mu)x_{i}+\beta'_{i}(\mu),
\label{eq:InCLyapY1}
\end{aligned}
\end{multline}


\end{document}